\documentclass[a4paper,11pt,reqno]{amsart}
\usepackage[T1]{fontenc}
\usepackage[utf8]{inputenc}
\usepackage[margin=1.0in]{geometry}
\usepackage{lmodern}
\usepackage[english]{babel}
\usepackage{graphicx}
\usepackage{amsmath}
\usepackage{amsfonts}
\usepackage{amssymb} 
\usepackage{amsthm}
\usepackage{bbm}
\usepackage{bm} 
\usepackage{lipsum}
\usepackage{xcolor}
\usepackage{url}
\usepackage{hyperref}
\usepackage{mathabx}
\usepackage{dsfont,mathrsfs,stmaryrd}

\usepackage{scalerel,stackengine}
\stackMath
\newcommand\reallywidehat[1]{%
\savestack{\tmpbox}{\stretchto{%
  \scaleto{%
    \scalerel*[\widthof{\ensuremath{#1}}]{\kern-.6pt\bigwedge\kern-.6pt}%
    {\rule[-\textheight/2]{1ex}{\textheight}}
  }{\textheight}%
}{0.5ex}}%
\stackon[1pt]{#1}{\tmpbox}%
}
\allowdisplaybreaks

\theoremstyle{plain}
\newtheorem{thm}{Theorem}[section]
\newtheorem{metatheorem}[thm]{Meta-Theorem}
\newtheorem{lem}[thm]{Lemma}
\newtheorem{prop}[thm]{Proposition}
\newtheorem{cor}[thm]{Corollary}
\newtheorem{defn}[thm]{Definition}
\newtheorem{defn-prop}[thm]{Definition-Proposition}
\newtheorem{rem}[thm]{Remark}

\def \be {\begin{equation}}
\def \ee {\end{equation}}

\def \R {\mathbb{R}}
\def \bT {\mathbb{T}}

\def \fD {\mathfrak{P}}

\def \T {\mathbb{T}^d} 
\def \PP {\mathcal{P}(\mathbb{T}^d)}

 \renewcommand{\i}{{\mathrm{i}}}

\begin{document}

\title{Weak solutions to the master equation of potential mean field games}
\author{Alekos Cecchin}
\address[A. Cecchin]
{\newline \indent  University of Padova, Department of Mathematics ``Tullio Levi Civita'', 
	\newline \indent Via Trieste 63, 35121 Padova, Italy }
\email{alekos.cecchin@unipd.it}
\author{François Delarue}
\address[F. Delarue]
{\newline \indent Universit\'e C\^ote d'Azur, CNRS, 
	Laboratoire de Math\'ematiques  J.A. Dieudonn\'e,
	\newline 
	\indent 28 Avenue Valrose, 06108 Nice Cedex 2, France}
\email{francois.delarue@univ-cotedazur.fr} 

\thanks{ 
A. Cecchin benefited from the support of LABEX Louis Bachelier Finance and Sustainable Growth - project 
ANR-11-LABX-0019, ECOREES ANR Project, FDD Chair and Joint Research Initiative
FiME.	
F. Delarue is supported by French ANR project 
ANR-19-P3IA-0002 -- 3IA C\^ote d'Azur -- Nice -- Interdisciplinary Institute for Artificial Intelligence. 
}

\date{\today}

\keywords{Mean field games, master equation, 
weak one-sided Lipschitz solutions, 
mean field control problem, Hamilton-Jacobi-Bellman equation on the space of probability measures, displacement semi-concave solutions, Rademacher theorem on the space of probability measures.}

\subjclass[2020]{35L40, 35Q89, 49L12, 49N80, 60J60, 91A16}

\begin{abstract}
The purpose of this work is to introduce a notion of weak solution to the master equation 
of a potential mean field game and to prove that existence and uniqueness hold under quite general assumptions.  
Remarkably, this is achieved without any monotonicity constraint on the coefficients. 
The key point 
is to interpret the master equation in a conservative sense and then to adapt to the infinite dimensional setting 
earlier arguments for hyperbolic systems deriving from a Hamilton-Jacobi-Bellman equation. 
Here, the master equation is indeed regarded as an infinite dimensional system set on the space of probability measures
and is formally written as the derivative of the 
Hamilton-Jacobi-Bellman equation associated with the mean field control problem lying above the mean field game.
To make the analysis easier, we assume that the coefficients are periodic, which allows to represent probability measures through their Fourier coefficients. Most of the analysis then consists in rewriting the master equation and the corresponding Hamilton-Jacobi-Bellman equation for the mean field control problem as partial differential equations set on the  Fourier coefficients themselves. 
In the end, 
we establish existence and uniqueness of 
functions that are displacement semi-concave in the measure argument and that solve the Hamilton-Jacobi-Bellman equation in a suitable generalized sense
and, subsequently, we get existence 
and uniqueness of functions that solve the master equation in an appropriate weak sense and that satisfy a weak one-sided Lipschitz inequality. As another new result, we also prove that the  
optimal trajectories
of the associated mean field control problem are unique for almost every starting point, 
for a suitable probability measure on the space of probability measures. 
\end{abstract}

\maketitle 

\section{Introduction}
\label{se:introduction}

Mean field games is by now a well-established theory for 
the analysis of equilibria within a continuum of rational dynamic players, 
see for instance \cite{HuangCainesMalhame1,Huang2006,Lasry2006,LasryLions2,LasryLions,Lionscollege2} for pioneering  
contributions on the field and \cite{cardaliaguetporretta-cetraro,CarmonaDelarue_book_I,CarmonaDelarue_book_II,gomes_survey}
for a non-exhaustive list of surveys or monographs on the subject. 
Whilst most of the first works in the domain were dedicated to the formulation 
of the problem and to the study of existence and uniqueness of equilibria, 
more efforts have been spent recently on the analysis of the so-called master equation. 
Very briefly, the master equation
is a Partial Differential Equation (PDE) set on a space comprising both physical states and 
probability measures, and it provides an Eulerian description 
of the value of the game. In this sense, it corresponds, in the infinite dimensional mean field setting, 
to the usual standard Nash system in differential games. 
As a main feature of the theory, 
the characteristics 
of the
master equation write in the form of a forward-backward system, 
comprising either two coupled PDEs, one forward Fokker-Planck equation and one backward Hamilton-Jacobi-Bellman  (HJB) equation, 
or two coupled forward and backward Stochastic Differential Equations (SDEs)
of McKean-Vlasov type. 
The  forward-backward PDE system is usually referred to as the Mean Field Game (MFG) system 
and plays a key role in many of the aforementioned references 
since the forward component of any solution
of the MFG system 
encodes the statistical state of an equilibrium. 
\vskip 5pt

\textit{A brief review of the monotone setting.} 
Despite many recent progresses, 
the master equation remains only partially understood. 
The very first works in this direction were devoted to 
the analysis of classical solutions, 
see for instance \cite{CardaliaguetDelarueLasryLions,cha-cri-del_AMS,gan-swi,Lionsvideo}, see also 
\cite{ben-fre-yam2015} for a related presentation.  
While this looks a very natural and fundamental step in the 
study of the master equation, 
this leads in fact to results with a rather limited scope 
in practice because of the assumptions that they require. The master equation can be indeed regarded as a kind of system of hyperbolic nonlinear PDEs set on the space of probability measures and indexed by a continuum of states. As such, it
manifests the same phenomena as finite systems of nonlinear PDEs on the Euclidean space: solutions may develop
singularities in finite time.
{Accordingly, the analysis carried out in 
\cite{gan-swi}
just holds in small time (we refer to \cite{AmbroseMesazros,CardaliaguetCirantPorretta} for more recent studies over a small enough time interval, the former for MFGs with local interactions and the latter for MFGs with a common noise). Differently,} in the three contributions 
 \cite{Lionsvideo,CardaliaguetDelarueLasryLions,cha-cri-del_AMS}, 
 the absence of singularities
 {over an arbitrary time horizon} 
 is ensured under 
a suitable  
form of monotonicity, which is usually referred to as Lasry-Lions' condition. 
For instance (but this is by far not the only example), 
the
Lasry-Lions 
monotonicity 
condition is satisfied by the derivative of a convex 
function defined on the wider space of signed measures.  
Briefly, monotonicity is known to guarantee uniqueness of the equilibria to the corresponding MFG 
and also to enforce strong stability properties, which 
play a key role in the analysis of the regularity of the solutions to the master equation. 
Noticeably, several of the most recent works on the master equation are 
 also written within the same monotone framework, see for instance \cite{Bertucci,CardaliaguetSouganidis2,MouZhang} {(see also \cite{Bertucci:finite}, which is the counterpart of 
 \cite{Bertucci} but for MFGs on a finite state space)}. 
 In all these papers, the objective is to define, in the monotone setting, a relevant notion of solution to the master equation 
 that does not require the existence and the continuity of all the derivatives that appear therein. 
 As well expected, this allows to work with less regular coefficients. 
 However, monotonicity remains of a crucial use:
 {Obviously, as} it implies uniqueness of the equilibria, 
it makes
 trivial the definition of the value function of the game, with the latter being then the natural candidate for solving the master equation; 
 {Moreover, as it induces a form of stability of the equilibria, it 
 forces the value function of the game to be at least continuous. 
 Accordingly, in all the aforementioned references, the solution to the master 
equation is indeed continuous in the measure argument}.  
 
At this stage, it is worth mentioning that, although it is the most popular one,  
the Lasry-Lions condition is not the only assumption that enforces uniqueness and stability of the equilibria to MFGs. Indeed, another form of monotonicity, known as displacement 
monotonicity, has been also studied. To highlight the differences with the 
former notion used by Lasry and Lions,
one may 
think of the following illustrative example, in the same vein as before:  
the Wasserstein derivative of a function that is displacement convex on the space of probability measures is 
displacement monotone.
The very main point 
in this example is that neither the notion of derivative nor the type of convex perturbation are the same 
as in the corresponding example for the Lasry-Lions condition; 
in brief, convex perturbations (in the definition of the displacement convexity) 
are achieved in the space of random variables (living above the space of probability measures) 
and not in the space of measures itself. 
The fact that displacement monotonicity 
implies uniqueness of the equilibria was noticed in 
the earlier paper 
\cite{Ahuja}
and in
the book \cite{CarmonaDelarue_book_I}.
In 
\cite{cha-cri-del_AMS}, it 
was also shown to help in the analysis of the master equation
for cost coefficients that are also convex in the space variable.   
A more systematic study of the
master equation in the displacement monotone setting has been carried out recently, in 
\cite{GangboMeszaros}
(for potential mean field games with displacement convex costs) 
and in 
\cite{GangboMeszarosMouZhang} (which allows for a quite complicated `non-seperated' structure of the Hamiltonian), see also 
\cite{MeszarosMou} for a related study of uniqueness within a similar framework. 
Finally, we refer to the very recent 
work \cite{MouZhang2} for another notion of \textit{anti}-monotonicity, which is distinct from the aforementioned two  notions of monotonicity, and under which the 
master equation can be also shown to have a classical solution. 
\vskip 5pt

\textit{Road map for potential games.} 
The understanding of the master equation remains however 
much more limited when equilibria are no longer unique. 
The very first difficulty in this case is that the value of the game then ceases to be canonically defined
and just exists \textit{a priori} as a set-valued function. 
The properties 
of this set-valued function 
were studied recently in 
\cite{IseriZhang}. 
However, 
there has not been so far any systematic procedure for selecting, within
the set-valued function, one true function 
that could indeed solve 
the master equation in a relevant sense. A related 
difficulty is that, as soon as equilibria cease to be unique, 
the master equation can no longer have a classical solution. 
{As shown by the example studied in \cite{delfog2019}
(see also \cite{BayZhang,BayZhang-corr,cecdaifispel} for examples in the finite state case), even 
continuity may be lost, 
which clearly demonstrates the need to have a weaker form of solution.}
This is precisely 
our objective here to address these questions 
and to provide
in particular 
 a notion of weak solution to the 
master equation for which existence and uniqueness hold, 
and which, in turn, allows to select one {possibly discontinuous} value for the game. 
Our approach to do so is 
inspired 
from our earlier work 
\cite{Cecchin:Delarue:CPDE} {and in particular from Section 6 therein}, in which we addressed the same problem but for mean field games 
on a finite set. 
Here, the mean field game under study is 
set on ${\mathbb T}^d$, for a dimension $d \geq 1$,
with the choice to work in the periodic setting being explained 
later on. 
In fact, the key similarity with 
\cite{Cecchin:Delarue:CPDE}
is that we here restrict the analysis to so-called potential mean field games, 
 {which were first introduced in 
\cite{LasryLions2,LasryLions}
and then studied in detail in (among others)  \cite{BrianiCardaliaguet,CardaliaguetGraberPorrettaTonon}}.
Briefly, 
an MFG is said to be potential 
if there exists a Mean Field Control Problem (MFCP), \textit{i.e.} 
a control problem over dynamics taking values in the space of probability measures, 
whose optimal trajectories are equilibria 
of the MFG. 
Equivalently, 
the MFG then characterizes the
critical points of the functional underpinning the 
MFCP, with the following obvious but fundamental observation:
any minimizer of the MFCP is a critical point of the latter functional and hence an equilibrium of the MFG, but the converse may not be true. 
This provides a way to restrict the analysis to a smaller class of equilibria by focusing on the minimizers of the MFCP and not on the whole set of equilibria: for example, in 
\cite{Cecchin:Delarue:CPDE}, using the results from 
\cite{cannarsa}, it is shown that the MFCP has (under suitable conditions) 
a unique minimizer
for almost every initial condition {(namely, at points of differentiability of the value function)}. 
Another generic result from 
\cite{Cecchin:Delarue:CPDE}
is that, when the state space is finite, 
the (hence finite-dimensional) HJB equation\footnote{The reader who is not aware of the theories of 
MFCP and MFG should make a distinction between the HJB 
equation associated with the MFCP and the aforementioned 
HJB equation arising in the MFG system. They are not the same. 
The HJB equation arising in the MFG system is posed on the state 
space of the MFG, whilst the HJB equation associated with the MFCP is 
posed on the wider space of probability measures. In the rest of the introduction, the HJB equation refers to the HJB equation associated with the MFCP.}
 associated with the MFCP has a unique viscosity solution, which is 
also the unique function that is semi-concave in space and that solves the HJB equation almost everywhere; for sure, 
this unique solution is nothing but the value function of the MFCP. 
{Although this latter result is formulated in 
\cite{Cecchin:Delarue:CPDE}
in a way that 
fits exactly the framework of MFCPs on a finite state space, 
it has in fact a much longer history in the theory of HJB 
equations. 
Briefly, it goes back to the earlier notion of generalized semi-concave solutions  
for HJB equations in finite dimension, which was introduced before the development of viscosity solutions, see \cite{Douglis,Kruzhkov1960}, see also 
\cite{CardaSou2020}
for another application to MFGs\footnote{
{Notably, semi-concave almost everywhere solutions were indeed also employed in the recent work \cite{CardaSou2020}
 in order to study the stochastic HJB equation arising in an MFG system with a common noise but without idiosyncratic noise}.}}.
The equivalence between the two notions
of solutions 
was established later on, in 
\cite{Lions_HJB}. 
Compared with viscosity solutions, 
the very benefit of generalized
solutions 
manifests in the analysis of the MFG deriving from the MFCP. Indeed, in 
the framework of \cite{Cecchin:Delarue:CPDE},  
the almost everywhere derivative 
in space of the value function is a weak solution of the master equation, when the latter is written
in a conservative manner, which 
formulation indeed exists because of the potential structure of the MFG.
Following the earlier work 
 \cite{kruzkov},
 this weak formulation of the conservative master equation can be proved 
 to be uniquely solvable (within a suitable class of weakly one-sided Lipschitz functions) by proving that any such solution in fact derives from a semi-concave potential that 
solves the HJB equation almost everywhere. 
In 
\cite{Cecchin:Delarue:CPDE}, 
these results are shown to hold on a finite state space.
Here, 
we want to prove similar results when the state space is ${\mathbb T}^d$. 
\vskip 5pt

\textit{Probability measures with 
a finite Fourier expansion and discretization of the HJB equation.}
The first difficulty that we are facing in this work is that, to the best of our knowledge, 
the theory for 
HJB equations for MFCPs (over ${\mathbb T}^d$ or ${\mathbb R}^d$) is much less well established than
in the finite dimensional setting.  
To appreciate the difficulties, it is worth mentioning that the controlled dynamics we are dealing with in the paper describe the motion of particles forced by independent Brownian motions. 
This is an important feature. In the analysis of viscosity solutions associated with MFCPs, the 
stochastic case is indeed notoriously known to be more challenging than the deterministic one. 
For MFCPs 
set over deterministic trajectories, existence and uniqueness results for viscosity solutions to the HJB equation may be found {(among others) in \cite{CardaQuinca}, where a particular equation is studied, and in \cite{JiMaQu}, where solutions are defined in an intrinsic manner on the Wasserstein space. Another main contribution is 
\cite{GangboTudorascu}, where two definitions are given: an intrinsic one, relying on the notion of subdifferential introduced 
in \cite{AmbrosioGangbo,AGS}, and  an extrinsic one, formulated in a Hilbert space of random variables by means of the \textit{Lions lift} introduced in \cite{Lionscollege2} (see also \cite{CarmonaDelarue_book_I,cardaliaguetporretta-cetraro}). These two notions are in fact shown to be equivalent, which permits to invoke earlier results on first order viscosity solutions on Hilbert spaces in order to prove uniqueness.}
In a similar setting, 
we refer to the recent work 
\cite{jean:hal-03564787} for an alternative definition of solution based on 
suitable test functions. 
{In contrast, 
the theory of viscosity 
solutions
for MFCPs set over stochastic dynamics is not yet mature, as many contributions on the topic  have appeared recently.
Existence results, for notions of solutions formulated in the 
Crandall-Lions sense by means of appropriate test functions, are provided in \cite{Bandini,CossoPham,Pham-Wei},
but uniqueness remains a challenging question, with some partial results
or some complete results in specific cases available in 
\cite{Burzoni,WuZhang}. 
{Notably, the extrinsic approach based on Lions' lift is no longer appropriate to 
get uniqueness from 
earlier results for equations set on Hilbert spaces. 
This is due to the additional second order term that arises in the HJB equation when the controlled trajectories are stochastic.}
In the very recent work 
\cite{CossoGozziKharroubiPham}, updated just two months before the prepublication of ours, 
a general comparison principle 
is proven for the {intrinsic}
Crandall-Lions formulation (which means that the test functions are defined on the space of probability measures). This achievement is certainly an important milestone in the theory.  
Finally, in \cite{conforti}, 
another approach based on gradient flows is introduced to deal with
the stochastic case, as the heat equation (which describes the evolution of the marginal law of a Brownian motion) can be rewritten as a gradient flow 
deriving from an entropy, see \cite{JKO}. 
The latter yields a form of comparison principle, but existence of solutions 
to this formulation has not been shown yet. }

{Although we believe that the recent results from 
\cite{CossoGozziKharroubiPham} could be 
applied to our setting, and thus could permit to characterize the value function 
of the MFCP as the unique viscosity solution of the associated HJB equation, 
we follow another approach very much inspired from 
the aforementioned works \cite{Douglis,Kruzhkov1960,kruzkov,Lions_HJB} 
on generalized semi-concave solutions
to finite dimensional HJB equations}. Indeed, we prove here a tailor-made result of existence and uniqueness 
for the HJB equation associated with the MFCP
within the class of functions 
that are Lipschitz in time and space (the space variable living here in the 
space $\PP$ of probability measures on ${\mathbb T}^d$ and being equipped with a suitable topology) and 
displacement semi-concave in space. 
{This result is completely new and is disjoint from the comparison principle obtained in 
\cite{CossoGozziKharroubiPham}. It is in particular a crucial step in the analysis of the master equation to the related MFG and the same study would not be possible with 
the notion of viscosity solutions addressed in 
\cite{CossoGozziKharroubiPham}.}
{In fact, we are not aware of any equivalence result between viscosity and semi-concave generalized solutions for infinite dimensional equations, even in the simpler case of first order equations on a Hilbert space.}
In our approach, the 
HJB equation 
is understood in a generalized sense, since the derivatives that appear therein 
are not required to exist 
everywhere. 
Heuristically, we would like to say that, since the candidates for solving the HJB equation are Lipschitz continuous, 
their derivatives 
exist almost everywhere, 
whence the term \textit{generalized} in the notion of solution, 
but this requires a preliminary discussion about the choice of a
probability measure ${\mathbb P}$ on $\PP$ under which Rademacher's theorem is indeed true. 
{The latter choice is by far not canonical, see for instance 
\cite{DelloSchiavo} for a different example than ours together with the references therein for related questions. Here, 
we address directly the existence of such a measure ${\mathbb P}$, 
our construction being in fact dictated by our formulation of the HJB equation}.
Roughly speaking, the measure ${\mathbb P}$ is built as the limit of 
probability measures on finite dimensional slices of $\PP$ and, accordingly,
the HJB equation is reformulated
in the form of an approximate equation 
 on those 
finite dimensional slices.
In fact, 
this procedure looks like finding a finite dimensional approximation of the HJB equation and thus requires an appropriate discretization of the space variable, which is here an element of $\PP$. 
 While a natural idea 
for discretizing the space of probability measures 
would consist in approximating measures by means of discrete measures (think of uniform measures on finite sets, or equivalently 
of empirical measures, see for instance \cite{MouZhang}), we choose here to approximate probability measures by probability measures with a finite Fourier expansion, whence our choice 
to work on the torus. 
This looks a completely new idea in the field (although the second author already introduced part --but certainly not the main part-- of the idea in an earlier contribution on the long time behaviour  of McKean-Vlasov equations, see \cite{DelarueTse}), whose motivation is as follows.
The key observation in this regard is that the 
discretization 
of the HJB equation is very sensitive to the shape of the MFCP itself. 
Here, it is worth emphasizing again that  
the $\PP$-valued controlled dynamics
studied 
in the paper write in the (by now standard) form of a second-order Fokker-Planck equation:
the first order term therein contains the control, whilst the 
second order term is driven by a Laplacian (associated with the Brownian motions driving the underlying particles in the population). 
As we already mentioned, 
the Laplacian in the Fokker-Planck equation induces, 
in the HJB equation associated with the MFCP, 
 higher order derivatives of 
the solution: these derivatives are the most difficult terms to control in the HJB equation and, even more, they 
are very sensitive to 
the discretization procedure. 
Our claim is that
working with truncated 
Fourier expansions
is very well
adapted to our problem since the Fourier functions are precisely the eigenfunctions of 
the Laplacian. In other words, 
our discretization has a limited impact on the shape of the 
controlled trajectories and in turn on the 
HJB equation itself. 
In contrast, 
discretizing the measure by means of empirical mesures does not look, at least from the computations we did, 
a successful strategy, precisely because this does not combine well with the 
derivatives 
induced by the Laplacian in the Fokker-Planck equation (or, equivalently, by the Brownian motions in the associated cloud of particles).  
\vskip 5pt

\textit{Unique solvability of the HJB equation.}
To sum-up, our interpretation  
of the HJB 
equation stipulates
that solutions 
must satisfy an approximate version of the latter 
when reduced to probability 
measures with a finite number of non-zero Fourier coefficients, 
with the accuracy of the approximation 
(in the HJB equation) 
getting better and better as the 
level of truncation in the Fourier expansions increases.
As a first result,  
we show under appropriate regularity properties that the value 
function of the MFCP
satisfies our notion of solution.
The next step is to prove that the value function is in fact the only possible solution {within a class of functions that are displacement semi-concave in the measure argument. This is in the end the most demanding result
in the paper.
In fact, this is also} the point where the benefit for reducing the analysis to finite dimensional slices 
becomes clear. 
Roughly speaking, 
the core of the proof for uniqueness
is {to use semi-concavity in order to show that the 
characteristic flow generated on $\PP$
by any solution} cannot accumulate 
around some point. In finite dimension, this result can be easily formulated: 
basically, the marginal law of the flow must have a bounded density when the 
initial condition also has a bounded density. 
Formulating a similar statement in infinite dimension is certainly more challenging. 
Although obtaining such a statement remains an interesting direction of research, we felt it easier 
to reduce, as much as possible, the proof of uniqueness to the finite dimensional framework, 
whence our choice to discretize the problem. 
Notice, however, that in the end, we 
are able to obtain the expected properties on the whole space $\PP$ 
(and not on slices): 
uniqueness is formulated on the entire $\PP$ and, even more, 
we finally prove that we can construct a probability measure 
${\mathbb P}$ on $\PP$, with a full support, such that,
for any initial time and almost every initial measure, 
the MFCP has a unique solution. 
This second result looks also completely new
and complements 
earlier counter-examples about non-uniqueness. 
{Its proof is achieved in two steps. Uniqueness is first shown to hold at points where the value function has directional derivatives (using earlier results from \cite{BrianiCardaliaguet}) and is thus 
established
almost everywhere by the Rademacher-type result mentioned above. As we already said, this result is standard for finite dimensional systems (see \cite{cannarsa}). In fact, it also known to hold true 
under some conditions for infinite dimensional deterministic systems on Hilbert spaces.} 
 All this analysis holds true with a smooth Hamiltonian 
that is required to be strictly convex in the dual variable, but
{there is no need of convexity in the space or measure
arguments}.   
For the sake of completeness, it should be noted that our study of generalized solutions 
to the HJB  equation can be completed by the concomitant but independent work \cite{cardaliaguet-souganidis:2}, in which the
value function 
is in fact shown to be 
smooth on an open dense subset, 
under slightly more demanding assumptions on the coefficients.  
 \vskip 5pt
 
\textit{Weak solutions of the master equation}. The second main step in the paper is to study the conservative 
version of the master equation of the MFG. Again, this requires to have a proper form of weak solutions
and we do so by formulating approximated versions of the master equation on the same
finite dimensional slices as before. 
As in the finite dimensional setting addressed in \cite{Cecchin:Delarue:CPDE}, 
the key step is to prove that any solution to this weak formulation derives from a potential and that this potential 
is semi-concave in the space argument if the solution in hand satisfies a weak one-sided Lipschitz condition. 
This allows to identify the potential with the solution of the HJB equation of the MFCP and thus with the value function of the latter. 
In turn, the conservative version of the master equation is 
shown to have a unique weak solution, which is the almost everywhere (under the same probability ${\mathbb P}$ as before) 
derivative in space of the value function. 
This result also is new
and is certainly the first one 
to identify a solution to the master equation for a wide class of MFGs that are not required to be uniquely solvable. 
As a side result of our analysis, we show that any classical solution to the master equation is (up to a centering 
operation) 
a weak solution, which proves the consistency of our approach. 
\vskip 5pt

\textit{Further prospects.}
While we believe that this work is an interesting step forward 
toward a better understanding of the master equation, it 
obviously leaves open interesting extensions, which we feel better to address in 
future contributions. 
The main questions concern the choice of the controlled trajectories underpinning both the MFG and the MFCP. 
Here, the presence of the Brownian motion 
has a somewhat dramatic impact on the analysis. On the one hand, the very good point is that, because of the smoothing effect of the Laplacian, 
the MFG system has very strong regularity properties with respect to the finite dimensional variable (accounting for the private state of a tagged player in the continuum). These regularity properties are by now well documented in the theory of MFGs and we make here an intense use of them in order to control the decay of the various Fourier coefficients 
appearing in the analysis. On the other hand, the 
bad point is that the 
Brownian motion 
manifests in the HJB equation associated with the MFCP in the form of
a mixed second order derivative, obtained by taking the gradient (in ${\mathbb T}^d$) of the Wasserstein derivative of the solution. This second order derivative 
 creates substantial difficulties which, as we mentioned before, 
 were already
  reported in the analysis of viscosity solutions
  and which makes quite subtle any approximation by finite dimensional 
 derivatives. As we said, this is precisely the point where the Fourier approach
introduced in this paper becomes fully relevant. 
In turn, this raises the following two equations: 
$(i)$ What happens for 
a more complicated non-degenerate second order operator than the Laplacian itself? 
$(ii)$ What happens when the trajectories are no longer assumed to 
be forced by a noise? 
As for $(i)$, our guess is that our analysis could be 
adapted to more general non-degenerate cases, but the extension would certainly not be straightforward. As we already mentioned, the fact that the Fourier basis is diagonal for the Laplace operator 
is indeed very helpful here.
Anyhow,  even though the latter is no longer true
when the second order operator is more general, 
we believe that an approximation 
similar to the one we use here for handling the non-linearity in the HJB equation 
would solve the issue. In short, 
the action of the second order operator on a probability measure with a finite Fourier expansion 
would no longer have a finite Fourier expansion 
and, hence, 
should be truncated. 
As for question $(ii)$,  
the picture is more subtle since 
the aforementioned second order derivative would no longer appear in the new HJB equation. In turn, the interest of the Fourier approach becomes questionable in the deterministic case. 
While the reader may find it a drawback of our method, 
we recall that
we spotted 
a similar distinction between the stochastic and deterministic cases in the study 
of viscosity solutions, as the former is more difficult to handle. 
In brief, this should not be a surprise if the stochastic case is more demanding than the deterministic one.  
This is even more true that first order Fokker-Planck
equations (as those satisfied by the forward equation in the MFG system 
when the trajectories are deterministic) 
are notoriously known to 
preserve `empirical measures', such a stability property being false 
for second order Fokker-Planck equations. Even for models without control, 
this substantial difference is part of the standard knowledge in mean field theory, as it prevents
some of the arguments used in the deterministic case from being used in the stochastic case. 
As for MFCPs and MFGs, the fact that 
`empirical measures' are preserved by 
 first order Fokker-Planck
equations also plays a key role in
the analysis of the master equation 
carried out in 
\cite{GangboMeszaros} 
when the trajectories are deterministic 
and the costs are monotone. 
In the end, 
we strongly believe that, for potential MFGs driven by 
deterministic trajectories, 
we should use an empirical measure based approach instead of 
a Fourier based approach 
in order to construct 
the finite dimensional approximate HJB 
and master
equations.
In short, 
the point would be to discretize 
 elements of $\PP$ 
  by means of empirical measures
 and not by means of measures with a finite Fourier expansion. 
 However, this 
 claim  
should be clarified in a separate work, since, meanwhile, 
 most of the smoothness to the MFG system 
 is lost  
 when the noise is absent, which may cause other difficulties.  
 Back to models driven by stochastic trajectories, 
a related question is to wonder whether the Fourier based approach can help in any way for the understanding of viscosity solutions to the HJB equation describing the MFCP {and for connecting our approach with the one developed in 
\cite{CossoGozziKharroubiPham}}.
 
A very last remark concerns the extension to the Euclidean setting, i.e. to 
${\mathbb R}^d$. 
Certainly, we may think of using an orthonormal basis of 
$L^2({\mathbb R}^d,\mu)$,  
for a suitable measure $\mu$,  instead of the Fourier basis. For instance, a natural choice would be to work with $\mu$ being a Gaussian measure and thus with the Hermite instead of Fourier basis. This intuition would also deserve a careful inspection.

 \vskip 5pt

\paragraph{\bf Organization and references to the main results} 
The paper is organized as follows. 
The potential MFG under study and the related 
MFCP are introduced in Section 
\ref{se:2}. Therein, we clarify the shape of the master equation, both in 
non-conservative and conservative forms, and we also provide the form of the HJB equation associated with the MFCP. However, we do not provide yet, in this section, the precise definitions of the weaker solutions to these two equations. 
Instead, we just give in Meta-Theorems 
\ref{meta:1} 
and 
\ref{meta:2}
a meta form of the statements that 
will certainly help the reader 
to understand the main contributions 
of the paper. 
The rigorous versions of 
Meta-Theorems 
\ref{meta:1} 
and 
\ref{meta:2}
 require additional material about Fourier expansions of probability measures on ${\mathbb T}^d$. 
Most of this material is introduced in Section \ref{se:3}. Therein, we introduce in particular a mollification procedure that plays a key role in our analysis, see Definition 
\ref{def:mollification}. We also define the aforementioned probability 
measure ${\mathbb P}$, the main characteristics of which are summarized in 
the statement of Theorem 
\ref{thm:probability:probability}. 
In Theorem \ref{prop:rademacher}, we then provide the form of Rademacher's theorem that 
holds true on $(\PP,{\mathbb P})$ and which plays a key role in the identification of the weak solutions to the master equation. Although a bit lengthy, 
this preliminary Section \ref{se:3} is however crucial for the rest of the paper. 
Indeed,
using the tools introduced in this latter section, we address in Section \ref{sec:4} generalized solutions to the HJB
equation, the definition of which is given in 
Definition \ref{defn:HJB:gen}. Their existence and uniqueness are established in 
Theorem \ref{thm:uniqueness:HJB}. 
Application to potential MFGs is explained in 
Section \ref{sec:MFG}. The definition of a weak solution is 
given in Definition 
\ref{def:master:equation:gen}
and the main unique solvability result is stated in Theorem 
\ref{main:thm:MFG}. 
Section \ref{se:6} is an appendix, which is mostly devoted to the proofs of 
Theorems 
\ref{thm:probability:probability}
and
\ref{prop:rademacher}.

\vskip 5pt

\paragraph{\bf Useful notation} 
{\ }
\vskip 5pt

\textit{Standard notation.}
For an element $z \in {\mathbb C}$, we denote by $\overline{z}$ the conjugate of $z$, and by 
$\Re[z]$ and $\Im[z]$ the real and imaginary parts of $z$.
For $x$ and $y$ two elements of a set $E$, we denote by ${\mathbf 1}_{\{x=y\}}$ the symbol 
that is equal to $1$ if $x=y$ and to $0$ otherwise. 
Moreover, 
we denote by 
$\textrm{\rm Leb}_d$ the $d$-dimensional Lebesgue measure on ${\mathbb R}^d$
and by $I_d$ the identity matrix in dimension $d$. The Euclidean norm on ${\mathbb R}^d$ is denoted by $\vert \cdot \vert$
and the corresponding inner product between two vectors $x$ and $y$ is denoted by 
$\langle x, y \rangle$ or by $x \cdot y$ (depending on the context). 
For $X$ a random variable, defined on some probability space and taking values in some Polish space, 
we denote the law of $X$ by ${\mathcal L}(X)$. 
\vskip 5pt

\textit{Space of probability measures on ${\mathbb T}^d$.}
The space of probability measures on the torus ${\mathbb T}^d$ is denoted by ${\mathcal P}({\mathbb T}^d)$. 
When a probability measure $m$ has a density (with respect to the Lebesgue measure), we identify $m$ and its density $d m/dx$: we thus write $x \mapsto m(x)$ for the density. 

In the paper, 
we equip ${\mathbb T}^d$ with the Euclidean metric:
\begin{equation*}
d_{{\mathbb T}^d}(x,y) 
:= \inf_{k \in {\mathbb Z}^d} 
\vert x - y + k \vert
, \quad x,y \in {\mathbb T}^d,
\end{equation*}
with $\vert \cdot \vert$ denoting the standard $d$-dimensional Euclidean norm on 
${\mathbb R}^d$. Moreover, we use the following two standard distances on ${\mathcal P}({\mathbb T}^d)$: 
the total variation distance $d_{\rm TV}$ and the $1$-Wasserstein distance $d_{W_1}$, the definitions of
which are as follows: 
\begin{equation*}
d_{\rm TV}(m,m') := \sup_{\| \phi \|_\infty \leq 1} 
\biggl\vert \int_{{\mathbb T}^d} 
\phi(x) d \bigl( m - m' \bigr)(x) \biggr\vert,
\quad
d_{\rm W_1}(m,m') := \sup_{\| \phi \|_{1,\infty} \leq 1} 
\biggl\vert \int_{{\mathbb T}^d} 
\phi(x) d \bigl( m - m' \bigr)(x) \biggr\vert,
\end{equation*}
where, in the above two supremums, 
$\| \phi \|_\infty$ and $| \phi|_{1,\infty}$ denote (respectively) the $L^\infty$ norm and the Lipschitz constant 
of $\phi$. 
Since $\phi$ can always be assumed to satisfy $\phi(0)=0$, we notice that 
$| \phi |_{1,\infty} \leq 1 \Rightarrow \| \phi \|_\infty \leq c_0$, for a constant $c_0$ only depending on the dimension 
$d$. In particular,  $d_{W_1} \leq c_0 d_{\rm TV}$. We recall that the space 
$({\mathcal P}({\mathbb T}^d),d_{W_1})$ is compact.
Except when explicitly mentioned, the 
space ${\mathcal P}({\mathbb T}^d)$ is implicitly equipped with the $1$-Wasserstein distance.   
\vskip 5pt

\textit{Derivatives on the space of probability measures.}
For a real-valued function $\phi$ defined on ${\mathcal P}({\mathbb T}^d)$,
  $\phi$ is said to be G\^ateaux differentiable at $m$ if
there exists a bounded measurable function $x \mapsto [\delta \phi/\delta m](m)(x)$
such that, for any 
$\mu \in {\mathcal P}({\mathbb T}^d)$, 
\begin{equation}
\label{eq:delta m}
\lim_{\varepsilon \searrow 0} 
\frac{\phi\bigl((1-\varepsilon) m + \varepsilon \mu\bigr) - \phi(m)}{\varepsilon} 
= \int_{{\mathbb T}^d} 
\frac{\delta \phi}{\delta m}(m)(x) d \bigl( \mu - m \bigr) (x).
\end{equation}
The function $\phi$ is said to be continuously differentiable if the function 
$(m,x) \mapsto 
[\delta \phi/\delta m](m)(x)$ is continuous (and thus uniformly continuous), when 
${\mathcal P}({\mathbb T}^d)$ is equipped with $d_{W_1}$. 
In that case, the above limit is uniform in 
$(m,\mu)$. 
Moreover, when 
$x \mapsto 
[\delta \phi/\delta m](m)(x)$ is
differentiable in $x$ (for a given $m$), we let 
\begin{equation}
\label{eq:deltam:partialmu}
\partial_{\mu} \phi(m)(x) := \partial_x \frac{\delta \phi}{\delta m}(m,x). 
\end{equation}
Notice that 
$\delta \phi/\delta m$ is uniquely defined up to a constant. We may take 
\begin{equation}
\label{eq:centring}
\int_{{\mathbb T}^d} 
\frac{\delta \phi}{\delta m}(m)(x) dx= 0
\end{equation}
as centering condition but this specific choice 
is mostly for convenience (see in particular Definition 
\ref{def:master:equation:gen})
 and other choices would be fine.  
We recall the following property, proven in \cite[Appendix]{CardaliaguetDelarueLasryLions}. 
If $\partial_\mu \phi$ is continuous on $\PP \times {\mathbb T}^d$, 
then, 
for any two square integrable ${\mathbb R}^d$-valued random variables $X,Y$ (defined on some arbitrary probability space $(\Omega,{\mathcal A},{\mathbf P})$), 
\begin{equation}
\label{eq:Lions:torus}  
\phi \bigl( {\mathcal L}(Y) \bigr) 
- 
\phi \bigl( {\mathcal L}(X) \bigr) 
=
{\mathbf E} 
\int_0^1 \biggl[ 
\partial_\mu \phi \Bigl( {\mathcal L} \bigl( \lambda Y + (1- \lambda) X \bigr)
\Bigr)\bigl( 
\lambda Y + (1- \lambda) X \bigr)
\cdot \bigl( Y - X \bigr) 
\biggr] d \lambda,
\end{equation}
where, for $Z$ taking values in ${\mathbb R}^d$, 
${\mathcal L}(Z)$ is here understood as the trace of the law of $Z$ on ${\mathbb T}^d$, namely as the 
probability measure $m \in {\mathcal P}({\mathbb T}^d)$ defined by $m(A) = {\mathbf P}(\{ \tau_X(0) \in A \})$ for any Borel subset $A$ of ${\mathbb T}^d$, with $(\tau_y)_{y \in {\mathbb R}^d}$ denoting the group of translations on 
${\mathbb T}^d$.  
Moreover, in this writing, 
the function $x \mapsto 
\partial_\mu \phi ( {\mathcal L} ( \lambda Y + (1- \lambda) X )
)(x)$ is implicitly regarded as a periodic function on ${\mathbb R}^d$. 
\vskip 5pt

\textit{Functional spaces.}
We denote by ${\mathcal C}({\mathbb T}^d)$ the space of continuous functions on 
${\mathbb T}^d$. 
We  also make use of the classical space and time-space H\"older spaces: for $\gamma \in (0,1)$, we let ${\mathcal C}^\gamma({\mathbb T}^d)$ be the space of $\gamma$-H\"older continuous functions in space 
and, for $T>0$, we let ${\mathcal C}^{\gamma/2,\gamma}([0,T] \times {\mathbb T}^d)$
be the space of time-space functions that are $\gamma/2$-H\"older in time and $\gamma$-H\"older in space.
The corresponding H\"older norms are respectively denoted
by $\vvvert \cdot \vvvert_\gamma$ and $\vvvert \cdot \vvvert_{\gamma/2,\gamma}$. We also let ${\mathcal C}^{2+\gamma}({\mathbb T}^d)$ be the space of 
twice differentiable functions $\phi$ on ${\mathbb T}^d$ with $\gamma$-H\"older continuous derivatives of order 1 and 2, and 
${\mathcal C}^{1+\gamma/2,2+\gamma}([0,T] \times {\mathbb T}^d)$
be 
the space of time-space functions that are once differentiable in time and twice in space, with 
$u$, $\partial_x u$, $\partial_x^2 u$
and $\partial_t u$ belonging to ${\mathcal C}^{\gamma/2,\gamma}
([0,T] \times {\mathbb T}^d)$. 
We use the similar notation 
${\mathcal C}^{1+\gamma/2,2+\gamma}((0,T] \times {\mathbb T}^d)$
for functions whose restriction to $[\epsilon,T] \times {\mathbb T}^d$ belongs to 
${\mathcal C}^{1+\gamma/2,2+\gamma}([\epsilon,T] \times {\mathbb T}^d)$
for any $\epsilon \in (0,T)$. 

The $L^1$ and $L^2$ norms on the torus are denoted by $\| \cdot \|_1$ and 
$\| \cdot \|_2$.

\section{Mean field game and control problem}
\label{se:2}

\subsection{Mean field game and standing assumption}
\label{subse:MFG}
The Mean Field Game (MFG) under study is described by means of the following (standard) forward-backward system of PDEs 
\begin{equation}
\label{eq:MFG:system}
\begin{split}
&\partial_t m_t(x) 
- \textrm{\rm div}_x \Bigl( \partial_p H\bigl(x,\nabla_x u_t(x)\bigr) m_t(x) \Bigr)
- \frac12 \Delta_x m_t(x) = 0,
\\
&\partial_t u_t (x)+ \frac12 \Delta_x u_t (x) - H\bigl( x, \nabla_x u_t(x) \bigr) 
+ f(x,m_t) =0,
\end{split}
\end{equation}
for $(t,x) \in [0,T] \times {\mathbb T}^d$, 
with $u_T(x) = g(x,m_T)$ as terminal boundary condition. 
Above, the boundary condition for the dynamics of $(m_t)_{0 \leq t \leq T}$ is prescribed at 
time $0$: $m_0$ is an element of ${\mathcal P}({\mathbb T}^d)$. 
The coefficients $f$ and $g$ are 
real-valued
Lipschitz continuous functions on $[0,T] \times {\mathcal P}({\mathbb T}^d)$. 
We assume $f$ and $g$ to be differentiable in $x$, and 
the derivatives 
$\partial_x f$ and $\partial_x g$ 
to be 
Lipschitz continuous on $[0,T] \times {\mathcal P}({\mathbb T}^d)$
when 
$\PP$ is equipped with $d_{W_1}$. 
We also require 
$g$ to satisfy 
\begin{equation*}
\sup_{m \in \PP} \vvvert g(\cdot,m)\vvvert_{2+\gamma} < \infty,
\end{equation*}
for some $\gamma >0$. 

Very importantly, we do not assume that $f$ and $g$ are monotone, but we assume below that they derive from a potential, namely
\begin{equation}
\label{eq:potential:structure}
f(x,m) - \int_{{\mathbb T}^d} f(y,m) dy  = \frac{\delta F}{\delta m}(m)(x), \quad g(x,m) 
- \int_{{\mathbb T}^d} g(y,m) dy 
= \frac{\delta G}{\delta x}(m)(x),
\end{equation}
for two differentiable functions $F : \PP \rightarrow  {\mathbb R}$ and $G : \PP \rightarrow {\mathbb R}$.

As for the Hamiltonian $H$, it is assumed to be 
twice continuously differentiable in $(x,p)$, periodic in $x$ and 
strictly convex in the variable 
$p$, uniformly in $x$, 
with the following bounds: 
\begin{equation*}
\begin{split}
&\inf_{x \in {\mathbb T}^d,p \in {\mathbb R}^d} 
\inf_{\xi \in {\mathbb R}^d : \vert \xi \vert =1}
\langle \xi,
\partial^2_{pp} H(x,p) \xi \rangle >0,
\\
&\sup_{x \in {\mathbb T}^d,p \in {\mathbb R}^d}
\Bigl[
\vert \partial^2_{pp} H(x,p)  \vert 
+ \vert \partial_{x} H(x,p)  \vert 
+ \vert \partial^2_{xx} H(x,p)  \vert 
\Bigr] < \infty.
\end{split}
\end{equation*}
In particular, 
\begin{equation*}
\begin{split}
&\sup_{x \in {\mathbb T}^d,p \in {\mathbb R}^d}
\Bigl[
\frac{
\vert \partial_{p} H(x,p)  \vert }{1 + \vert p \vert}
\Bigr] < \infty.
\end{split}
\end{equation*}

We recall the following statement on existence of classical solutions; see for instance 
\cite[Theorems 1.4 \& 1.5, pages 29 and 33]{cardaliaguetporretta-cetraro}. 
\begin{prop}
\label{thm:solvability:MFG}
Under the standing assumption, for any initial condition $m_0\in \PP$ to the forward equation,  the system 
\eqref{eq:MFG:system}
has at least one classical solution $(m,u)$, i.e. such that 
$m \in {\mathcal C}^{1+\gamma/2,2+\gamma}( (0,T] \times \T)$ and 
$u \in {\mathcal C}^{1+\gamma/2,2+\gamma}( [0,T] \times \T)$.
\end{prop}

Importantly, 
any solution to \eqref{eq:MFG:system}, as given by Proposition 
\ref{thm:solvability:MFG}, can be interpreted as a fixed point of a mapping that sends a continuous 
path $(\mu_t)_{0\leq t \leq T}$ with values
in $\PP$ onto the flow of marginal laws of 
a stochastic control problem. 
In order to clarify this interpretation (which is in fact standard in 
MFG theory), we fix a  probability space $(\Omega, \mathcal{F}, \mathbf{P})$ and  a $d$-dimensional Brownian motion $(B_s)_{0\leq s\leq T}$. 
We also define the convex conjugate $L$ of $H$ as 
\begin{equation}
\label{eq:Lagrangian}
L(x,\alpha) := \sup_{p \in {\mathbb R}^d} 
\Bigl[ - p \cdot \alpha - H(x,p) \Bigr], \quad x \in {\mathbb T}^d,  \ \alpha \in {\mathbb R}^d,
\end{equation}
which is of quadratic growth in $\alpha$, uniformly in $x$, 
locally Lipschitz in $\alpha$, the Lipschitz constant being of linear growth in 
$\alpha$, 
uniformly in $x$, and 
Lipschitz in $x$, uniformly in $\alpha$.  
 {By 
\cite{Zalinescu}, 
we know that 
$L$ is 
differentiable and
uniformly convex in $\alpha$, i.e., there exists $c_0>0$ such that} 
\be 
\label{eq:111}
\forall \alpha,\alpha' \in {\mathbb R}^d, \quad {L(x,{\alpha}')-L(x,\alpha) \geq \partial_\alpha L(x,\alpha) \cdot ({\alpha}'-\alpha) + c_0 |{\alpha}' - \alpha|^2 .}
\ee 
{We also recall the standard formula
\begin{equation}
\label{eq:formula:L:H}
L\bigl(x,-\partial_pH(x,p)\bigr) = 
p \cdot \partial_pH(x,p) - H(x,p), \quad x \in {\mathbb T}^d,  \ p \in {\mathbb R}^d.
\end{equation}
Then, for $(m_t,u_t)_{0 \leq t \leq T}$ a solution 
to  \eqref{eq:MFG:system}, $u_t(x)$ can be written, for any $(t,x) \in [0,T) \times {\mathbb T}^d$, as the optimal cost 
\begin{equation}
\label{eq:u_t:probabilistic:representation}
u_t(x) = \inf_{(\pi_s)_{t \leq s \leq T}} {\mathbf E} \biggl[ g(X_T,m_T) + \int_t^T \Bigl( f(X_s,m_s) + L(X_s,\pi_s) \Bigr) ds 
\biggr],
\end{equation}
the infimum being taken over controlled trajectories 
\begin{equation*}
dX_s= \pi_s  ds + dB_s,
\end{equation*}
starting from $X_t \sim m_t$ and 
driven by ${\mathbb R}^d$-valued processes $(\pi_s)_{t \leq s \leq T}$
that are 
progressively-measurable
with respect to 
the (${\mathbf P}$-completion of the)
filtration generated by $(B_s-B_t)_{t \leq s \leq T}$ and that satisfy 
\begin{equation}
\label{eq:pi:integrability:square}
{\mathbf E} \biggl[ \int_t^T \vert \pi_s \vert^2 ds \biggr] < \infty. 
\end{equation}
The optimal feedback is the function $(s,x) \mapsto - \partial_p H(x,\nabla_x u_s(x))$ and the marginal laws of the optimal trajectory solves the forward Fokker-Planck equation
in 
\eqref{eq:MFG:system}.}

Noticeably, the MFG system \eqref{eq:MFG:system} can be seen as the system of characteristics to the so-called master equation, which is a PDE stated on $[0,T]\times \T \times \PP$: 
\be 
\label{master}
\begin{split}
	&\partial_t U (t,x,m) + \frac12 \Delta_x U(t,x,m) - H\bigl(x,\nabla_x U(t,x,m) \bigr) + f(x,m) \\
	&\quad -   \int_{\T} \partial_p H\bigl( y, \nabla_y U(t,y,m) \bigr) \cdot  \partial_\mu U(t,x,m)(y) m(dy) 
	\\
	&\quad + \frac12 \int_{\T} \mathrm{Tr}\bigl[ \partial_y \partial_\mu U(t,x,m)(y) \bigr] m(dy) =0 , \\
	& U(T,x,m) =g(x,m) .
\end{split} 
\ee 
Existence of classical solutions, when $f$ and $g$ are monotone and smooth in the measure argument, is established in \cite{CardaliaguetDelarueLasryLions}. Here, we address weak solutions in the sense of distributions, which is the main objective of this paper.  The main result in this regard is Theorem 
\ref{main:thm:MFG}, which ensures existence and uniqueness of such weak solutions when the equation is understood 
in a conservative form (see 
Subsection 
\ref{subse:connection:MFG}
for more details about this conservative formulation).

\subsection{Mean field control problem}
\label{subse:mfc:presentation}
The key assumption in 
our analysis is 
\eqref{eq:potential:structure}. 
It says that the MFG in hand is potential, meaning that it derives from a Mean Field Control Problem (MFCP). At this stage, we  first state some useful properties about MFCPs, independently of the MFG system 
\eqref{eq:MFG:system}. In particular, we state the 
MFCP in both  stochastic and deterministic formulations, and then show that the two formulations are equivalent; notably, we make use of both of them in the paper. 

To state the MFCP in the stochastic strong formulation with open-loop controls, we use the same  probability space $(\Omega, \mathcal{F}, \mathbf{P})$ and  the same $d$-dimensional Brownian motion $(B_s)_{0\leq s\leq T}$ as before. 
The MFCP then consists in minimizing 
\be 
\label{cost:J:pi}
\mathcal{J}_{\textrm{\rm sto}} (\pi) :=  G \bigl( {\mathcal L}(X_T)\bigr) + \int_0^T \Bigl( F\bigl( {\mathcal L}(X_s) \bigr) + 
 {\mathbf E}
\Bigl[ 
L\bigl(X_s,\pi_s\bigr) 
\Bigr]
\Bigr) ds ,
\ee 
over processes $(X_s)_{0 \leq s \leq T}$ solving dynamics of the form
\begin{equation}
\label{eq:dynamics:L2}
dX_s = \pi_s ds + dB_s,
\end{equation}
where $(\pi_s)_{0\leq s\leq T}$ is a progressively measurable process with respect to the (${\mathbf P}$-completion of the) filtration generated by $(B_t)_{0 \leq t \leq T}$ and by $X_0$, and 
is required to be square-integrable over $([0,T] \times \Omega, \textrm{\rm Leb}_1 \otimes {\mathbf P})$, namely 
\eqref{eq:pi:integrability:square}
holds true with $t=0$ therein.

If the controls are in Markovian feedback form, that is $\pi_s = \alpha_s(X_s)$ for a measurable function $\alpha: [0,T] \times \T \rightarrow \R^d$, then the MFCP can be reformulated in a deterministic way. We take $\alpha$ to be bounded and measurable, so that \eqref{eq:dynamics:L2} admits a unique strong solution. Then the deterministic formulation of the MFCP consists in minimizing 
the cost functional 
\begin{equation}
\label{eq:cost}
{\mathcal J}_{\textrm{\rm det}}
(\alpha) = 
\int_0^T \biggl( F(m_t) +   \int_{{\mathbb T}^d} L \bigl( x, \alpha_t(x) \bigr) dm_t(x) 
\biggr) dt + G(m_T)
\end{equation}
over trajectories given by the Fokker-Planck equation 
\begin{equation}
\label{eq:control:FKP}
\partial_t m_t(x) + \textrm{div}_x \Bigl( \alpha_t(x) m_t(x) \Bigr) - \frac12 \Delta_x m_t(x) = 0,
\quad t \in [0,T], \ x \in {\mathbb T}^d, 
\end{equation}
and bounded feedback functions $\alpha : [0,T] \times {\mathbb T}^d \rightarrow {\mathbb R}^d$.

\begin{prop}
\label{prop:2:2}
Under the standing assumption, the functions $F$ and $G$ are $d_{W_1}$-Lipschitz-continuous. Denoting their Lipschitz constants by $\ell_F$ and $\ell_G$, there exists an optimal control for the MFCP in the stochastic formulation \eqref{cost:J:pi}-\eqref{eq:dynamics:L2}. Moreover, any optimal control is in Markovian feedback form and is bounded by 
a fixed constant
{$M$, which only depends on $T$, $ \ell_F$, $\ell_G$, the constant 
$c_0$ in \eqref{eq:111}, the Lipschitz constant of $L$ in $x$ and the supremum norm of 
$\partial_{\alpha} L(\cdot,0)$.}

Therefore,
 any optimal control to 
the stochastic formulation \eqref{cost:J:pi}-\eqref{eq:dynamics:L2}
induces an optimal control to the deterministic formulation 
\eqref{eq:cost}--\eqref{eq:control:FKP}. Also, 
 the stochastic formulation \eqref{cost:J:pi}-\eqref{eq:dynamics:L2}
 and the deterministic formulation \eqref{eq:cost}-\eqref{eq:control:FKP} are equivalent, in the sense that 
\[
 \min_\pi \mathcal{J}_{\textrm{\rm sto}}(\pi) = \min_{\alpha} \mathcal{J}_{\textrm{\rm det}}(\alpha). 
\]
\end{prop}

\begin{rem}
\label{rem:minimization}
The reader may wonder 
why we restrict 
the minimization of ${\mathcal J}_{\textrm{\rm det}}$ 
to bounded feedback functions. 
Obviously, we could consider 
more general feedback functions 
such that the 
cost 
\eqref{eq:cost}
remains well-defined and the 
Fokker-Planck equation 
\eqref{eq:control:FKP}
remains solvable. 
For instance, so is the case in 
\cite{BrianiCardaliaguet}. Therein, the MFCP is formulated in terms of 
a pair $(m_t,w_t)_{0 \leq t \leq T}$, with $w_t$ being 
a vector valued signed measure with a finite mass that 
is absolutely continuous with respect to 
$m_t$, in which case $\alpha_t$ is identified with the 
density $dw_t/dm_t$. 
In fact, there would not be any strong interest in doing so
in our context. 
Indeed, 
\cite[Proposition 3.1]{BrianiCardaliaguet} shows that, in the end,  
the optimizers over such a wider class of controls are in fact 
driven by bounded feedback functions.
In this respect, the statement right above permits to identify 
a bound for the optimal controls. 
 
Last but not least, the reader must understand our main concern here. 
We want to have a simple framework in which 
the deterministic and stochastic formulations are equivalent, 
hence allowing us to use next both representations. 
The reader will find 
a more up-to-date review of the connection between the two formulations in 
\cite{Lacker_superposition}. 
{The reader may also notice that  the proof below requires in fact much weaker regularity on the coefficients than
what we stated in the standing assumptions. Actually, the same proposition also holds true for more general MFCPs (not only those related to potential MFGs) for which the cost does not split as a function of $\alpha$ plus a function of $m$. We chose to write here this general simple result, whose proof does not employ the MFG system, and to state
later on the deeper results that are truly related to the potential structure, see Proposition \ref{prop:Briani}. }

\end{rem}

\begin{proof}
We start with the Lipschitz property of $F$. It follows from the identity 
\begin{equation*}
F(m') - F(m) = \int_0^1 \frac{\delta F}{\delta m} \Bigl( \lambda m' + (1- \lambda) m \Bigr)(x) d \bigl( m' - m \bigr)(x)
\end{equation*}
and then from the Lipschitz property of $f=[\delta F/\delta m]$ in the variable $x$. The same argument holds for $G$. 

The fact that an optimal control exists and is in feedback form is proved in \cite[Theorems 2.2 and 2.3]{Lacker2017}. 
Thanks to the uniform convexity of the Lagrangian, 
the same proof shows 
that any optimal control is in fact in feedback form. 
Note however that those statements address MFCPs formulated in a weak sense (which means that the probability space is part of the unknown), but the results also hold true for the strong formulation (with a fixed probability space) if any optimal control can be shown to be bounded (since the 
dynamics 
\eqref{eq:dynamics:L2}
become 
a strongly well-posed stochastic differential equation when $\pi_s=\alpha_s(X_s)$ for a bounded measurable feedback 
function $\alpha$). 
In order to prove the latter (together with the fact that the bound only depends on the parameters quoted in the statement of Proposition \ref{prop:2:2}), suppose by contradiction that $\pi$ is optimal and such that 
$(\textrm{\rm Leb}_1 \otimes \mathbf{P}) \{ |\pi_t|> K\} >0$ for an arbitrary $K>M$, with $M$ a constant as in the statement whose value is fixed next.  We show that the truncated control 
$\pi^K_t = \pi_t \mathbbm{1}_{\{ |\pi_t| \leq K \}}$ is such that $\mathcal{J}_{\textrm{\rm sto}}( \pi^K) < \mathcal{J}_{\textrm{\rm sto}} (\pi)$, which contradicts the optimality and concludes the proof. 

Denote by $X$ and $X^K$ the solutions to \eqref{eq:dynamics:L2} with controls $\pi$ and $\pi^K$ respectively. 
We clearly have  
\[
\mathbf{E} \Big[ \sup_{0\leq t\leq T} |X^K_t - X_t| \Big] \leq 
\mathbf{E} \int_0^T | \pi^K_t - \pi_t| dt = 
\mathbf{E} \int_0^T | \pi_t| \mathbbm{1}_{\{|\pi_t| >K\}}dt. 
\]
The $d_{W_1}$-Lipschitz-continuity of $F$ and $G$
and then the above estimate yield
\begin{align*}
&\mathcal{J}_{\textrm{\rm sto}}( \pi) - \mathcal{J}_{\textrm{\rm sto}} (\pi^K) 
  \\
&\hspace{15pt}  \geq
   { \int_0^T \mathbf{E} \Big[L(X_t,\pi_t) - L(X^K_t, \pi^K_t)\Big] dt
 - \ell_F \int_0^T \mathbf{E} | X_t - X^K_t|dt - \ell_G \mathbf{E} |X_T - X^K_T| }.
 \end{align*}
 We now use the fact that 
 $L$ is Lipschitz continuous in $x$, uniformly in $\alpha$ (we denote by 
 $\ell_{L,x}$ the related Lipschitz constant) and we get 
\begin{align*}
&\mathcal{J}_{\textrm{\rm sto}}( \pi) - \mathcal{J}_{\textrm{\rm sto}} (\pi^K) 
  \\
&\hspace{15pt}  
\geq  { \mathbf{E} \int_0^T  \Bigl[  L(X_t, \pi_t)
- L(X_t,0)\Bigr] 
\mathbbm{1}_{\{|\pi_t|>K\}}
 dt  - \Bigl(T \ell_F + T \ell_{L,x} + \ell_G \Bigr) \mathbf{E} \int_0^T | \pi_t| \mathbbm{1}_{\{|\pi_t| >K\}}dt }.
 \end{align*}
 And then, by
\eqref{eq:111} (with $\alpha=0$ therein),
\begin{align*} 
&\mathcal{J}_{\textrm{\rm sto}}( \pi) - \mathcal{J}_{\textrm{\rm sto}} (\pi^K) 
  \\
  & \geq c_0 \mathbf{E} \int_0^T |\pi_t|^2 \mathbbm{1}_{\{|\pi_t|>K\}} dt 
  - \Big[ T \bigl(\ell_F +\ell_{L,x} +  \| \partial_\alpha L(\cdot,0)\|_\infty \bigr) +   \ell_G \Bigr] \int_0^T | \pi_t| \mathbbm{1}_{\{|\pi_t| >K\}}dt, 
 \end{align*}
and then,
  \begin{align*} 
\mathcal{J}_{\textrm{\rm sto}}( \pi) - \mathcal{J}_{\textrm{\rm sto}} (\pi^K) 
 & \geq \Bigl[ c_0 K - T\bigl(\ell_F +\ell_{L,x} + 
 \| \partial_\alpha L(\cdot,0)\|_\infty
 \bigr) -\ell_G \Bigr] \mathbf{E} \int_0^T  | \pi_t|  \mathbbm{1}_{\{|\pi_t| >K\}}dt
  \\
 & \geq  \Bigl[ c_0 K - T\bigl(\ell_F +\ell_{L,x} 
 +
 \| \partial_\alpha L(\cdot,0)\|_\infty
\bigr) -\ell_G\Bigr] K (\textrm{\rm Leb}_1 \otimes \mathbf{P}) \{ |\pi_t|> K\} , 
\end{align*} 
which is strictly positive, 
if we choose $M:=[T(\ell_F +\ell_{L,x} 
 +
 \| \partial_\alpha L(\cdot,0)\|_\infty) +\ell_G]/c_0$ and then recall that $K>M$.  
\end{proof}


In light of the deterministic formulation \eqref{eq:cost}-\eqref{eq:control:FKP}, we define the \emph{value function} $V:[0,T]\times \PP \rightarrow \R$ of the MFCP as 
\be 
\label{eq:MKV:V}
\begin{split}
&V(t,m) := \inf_{\alpha \in \mathcal{A}_t} \mathcal{J}_{\textrm{\rm det}}(t,m, \alpha),
\\
&\mathcal{J}_{\textrm{\rm det}}(t,m,\alpha) :=  G(m_T)
+
\int_t^T \biggl( F(m_s) +  \int_{{\mathbb T}^d} L\bigl(x,\alpha_s(x) \bigr) dm_s(x) 
\biggr) ds ,
\end{split}
\ee  
where $(m_s)_{t\leq s \leq t}$ solves the Fokker-Planck equation 
\eqref{eq:control:FKP} starting at $m_t =m$ and the infimum is taken over the set $\mathcal{A}_t$ of feedback functions $\alpha:[t,T]\times \T \rightarrow \R^d$ that are bounded and measurable. 
By the above proposition, it is equivalent to 
restrict the optimization problem to feedback functions that are bounded by $M$. 

Moreover, the value function can be equivalently defined in the open-loop framework. On the probability space $(\Omega, \mathcal{F}, \mathbf{P})$, in addition to $B$, we assume that there is a sub-$\sigma$-algebra $\mathcal{G}$ independent of $B$ and  rich enough so that $\PP = \{ \mathcal{L}(\xi) : \xi : \Omega \rightarrow {\mathbb T}^d$ is $\mathcal{G}$-measurable$\}$. For any $0\leq t <T$, we then denote by $\Pi_t$ the collection of $\mathbb{F}^t$-progressively-measurable 
${\mathbb R}^d$-valued 
processes $(\pi_s)_{t\leq s\leq T}$ that are square-integrable (see 
\eqref{eq:pi:integrability:square}), where $\mathbb{F}^t$ is the $\mathbf{P}$-completion of the filtration gererated by $\mathcal{G}$ and $(B_s)_{t\leq s\leq T}$. We have 
\be 
\begin{split}
&V(t,m) = \inf_{\pi\in \Pi_t} \mathcal{J}_{\textrm{\rm sto}} (t,m, \pi), 
\\
&\mathcal{J}_{\textrm{\rm sto}} (t,m, \pi) :=  G \bigl( {\mathcal L}(X_T)\bigr) + \int_t^T \Bigl( F\bigl( {\mathcal L}(X_s) \bigr) + 
  {\mathbf E}\bigl[ L(X_s,\pi_s) \bigr] \Bigr) ds,
  \end{split}
\ee 
over processes $(X_s)_{t \leq s \leq T}$ solving \eqref{eq:dynamics:L2} 
with $\mathcal{L}(X_t)= m$, with $X_t$ being $\mathcal{G}$-measurable.



The interest of this representation is that it permits to prove easily some important properties of the value function.  
The first one is related to the following notion of semi-concavity:

\begin{defn-prop}
\label{defn-prop:semiconcave}
Under the standing assumption, the functions $F$ and $G$ are displacement semi-concave, 
i.e. 
there exist two constants $C_F$ and $C_G$ such that, for any two random variables in 
$L^ 2(\Omega,{\mathcal A},{\mathbf P};{\mathbb R}^d)$, 
\begin{equation}
\label{semiconcave:F}
F \bigl( {\mathcal L}(\xi+Y) \bigr) + 
F \bigl( {\mathcal L}(\xi-Y) \bigr)
- 2 F \bigl( {\mathcal L}(\xi) \bigr)
\leq C_F {\mathbf E} \bigl[ \vert Y \vert^2 \bigr], 
\end{equation} 
and similarly for $G$ and $C_G$, with the same convention as in 
\eqref{eq:Lions:torus}
 for 
the various laws that appear in the above inequality. 
\end{defn-prop}

 The definition of displacement semi-concavity stated in \eqref{semiconcave:F} is shown to be equivalent to the usual notion used in optimal transport, which is formulated by means of geodesics in the Wasserstein space, see \cite[Lemma 3.6]{GangboMeszaros}.  In fact, \eqref{semiconcave:F} means that the lift of $F$ in the space of $L^2$ random variables is semi-concave, see again \cite{GangboMeszaros}.

\begin{proof}
We use 
\eqref{eq:Lions:torus}  to get: 
\begin{equation*}
F \bigl( {\mathcal L}(\xi \pm Y) \bigr) 
- 
F \bigl( {\mathcal L}(\xi) \bigr) 
=
\pm {\mathbf E} 
\int_0^1 \Bigl[ 
\partial_\mu F \Bigl( {\mathcal L} \bigl( \xi \pm  \lambda Y \bigr)
\Bigr)\bigl( 
\xi \pm \lambda Y  \bigr)
\cdot 
Y 
\Bigr] d \lambda,
\end{equation*}
and then 
\begin{equation*}
\begin{split}
&F \bigl( {\mathcal L}(\xi  + Y) \bigr) 
+
F \bigl( {\mathcal L}(\xi  - Y) \bigr) 
- 
2 F \bigl( {\mathcal L}(\xi) \bigr) 
\\
&=
{\mathbf E} 
\int_0^1 \biggl[ 
\biggl( 
\partial_\mu F \Bigl( {\mathcal L} \bigl( \xi +  \lambda Y \bigr)
\Bigr)\bigl( 
\xi + \lambda Y  \bigr)
-
\partial_\mu F \Bigl( {\mathcal L} \bigl( \xi -  \lambda Y \bigr)
\Bigr)\bigl( 
\xi - \lambda Y  \bigr)
\biggr)
\cdot 
Y 
\biggr] d \lambda.
\end{split}
\end{equation*}
Recalling that $\partial_\mu F=\partial_x f$ is Lipschitz continuous on ${\mathbb T}^d \times {\mathcal P}({\mathbb T}^d)$, we complete the proof. 
\end{proof}



\begin{prop}
\label{prop:semi-concavity}
Under the standing assumption, the value function $V : [0,T] \times {\mathcal P}({\mathbb T}^d) \rightarrow {\mathbb R}$ is
Lipschitz continuous  in space with respect to $d_{W_1}$ and displacement semi-concave, uniformly in time, in the sense that
there exists a constant $C_0$ such that, for 
any $t \in [0,T]$ and
any two $\mathcal{G}$-measurable random variables $X$ and $Y$, 
\begin{equation}
\label{semiconcave:V}
V \bigl(t, {\mathcal L}(\xi+Y) \bigr) + 
V \bigl(t, {\mathcal L}(\xi-Y) \bigr)
- 2 V \bigl(t, {\mathcal L}(\xi) \bigr)
\leq C_0 {\mathbf E} \bigl[ \vert Y \vert^2 \bigr]. 
\end{equation} 
\end{prop}


\begin{proof}
The first step in the proof is to notice that the Lagrangian $L$ in 
\eqref{eq:Lagrangian}
has bounded first and second order derivatives in $x$. This is a mere consequence of the uniform strict convexity of $H$ in 
$p$  and of the bounds for $\partial_x H$ and 
$\partial^2_{xx} H$, see
for instance 
\cite[Lemma 5.44]{CarmonaDelarue_book_II}. 
Then, 
the Lipschitz property of $V$ in space is almost immediate. 
Fix $t$ and take $m, \tilde{m} \in \PP$ and two random variables $\xi, \tilde{\xi}$ such that $d_{W_1}(m, \tilde{m}) = \mathbf{E}[
d_{{\mathbb T}^d}( \xi,\tilde{\xi})]$. Fixing an optimal open-loop  control $\pi$ for $(t,m)$, the Lipschitz property of $F$ and $G$ 
in $m$ and the Lipschitz property of $L$ in $x$
give 
\[
V(t,\tilde{m}) - V(t, m) \leq \mathcal{J}_{\textrm{\rm sto}}(t,\tilde{m}, \pi) - 
\mathcal{J}_{\textrm{\rm sto}}(t,m, \pi) \leq C \mathbf{E}\bigl[ d_{{\mathbb T}^d}( \xi,\tilde{\xi})\bigr],
\] 
which provides the Lipschitz-continuity of $V$. 
To show the semi-concavity, let again $\pi \in \Pi_t$ be optimal for $(t,m)$ and $\xi$ such that $\mathcal{L}(\xi) =m$. 
Denote  the  processes in \eqref{eq:dynamics:L2} corresponding to the control $\pi$ and the initial conditions $\xi$ and $\xi \pm Y$ at time $t$ by
(respectively) 
  $(X_s := \xi + \int_t^s \pi_r dr + B_s-B_t)_{t \leq s \leq T}$
  and 
  $(X^{\pm}_s := \xi \pm Y + \int_t^s \pi_r dr + B_s-B_t)_{t \leq s \leq T}$.
Then, 
\begin{align*} 
V &\bigl(t, {\mathcal L}(\xi+Y) \bigr) + 
V \bigl(t, {\mathcal L}(\xi-Y) \bigr)
- 2 V \bigl(t, {\mathcal L}(\xi) \bigr)  \\
& \leq \mathcal{J}_{\textrm{\rm sto}}\bigl(t, {\mathcal L}(\xi+Y), \pi \bigr) +  \mathcal{J}_{\textrm{\rm sto}}\bigl(t, {\mathcal L}(\xi-Y), \pi \bigr)
-2  {\mathcal J}_{\textrm{\rm sto}}\bigl(t, {\mathcal L}(\xi+Y), \pi \bigr)  
\\
& 
=
 G\bigl( \mathcal{L}(X^+_T)\bigr) + G\bigl( \mathcal{L}(X^-_T)\bigr) 
-2 G\bigl( \mathcal{L}(X_T)\bigr)
 + 
\int_t^T \Bigl[ F\bigl( \mathcal{L}(X^+_s)\bigr) + F\bigl( \mathcal{L}(X^-_s)\bigr) 
-2 F\bigl( \mathcal{L}(X_s)\bigr) \Bigr] ds
\\
&\hspace{15pt}+
\int_t^T \Bigl[ L\bigl(X^+_s,\pi_s\bigr) + 
L\bigl(X^-_s,\pi_s\bigr)
- 
2 L\bigl(X_s,\pi_s\bigr)
 \Bigr] ds  
 \\
&\leq C \mathbf{E} |Y|^2, 
\end{align*} 
the last line following from 
Definition-Proposition 
\ref{defn-prop:semiconcave}
and from the identity
 $X^{\pm}_s =X_s \pm Y$, for 
 $s \in [t,T]$. 
\end{proof}

An important tool to study control problems is the \emph{dynamic programmin principle}. We make use of it only for the deterministic formulation of the MFCP for which its proof is straightforward: we have, for any $0\leq t\leq \tau \leq T$, 
\be 
V(t,m) = \inf_{\alpha\in \mathcal{A}_t} 
\biggl\{ 
V(\tau, m_\tau) 
+
\int_t^\tau \biggl( F(m_s) +  \int_{{\mathbb T}^d} L \bigl(x, \alpha_s(x)\bigr)   dm_s(x) 
\biggr) ds  
\biggr\}.
\label{eq:DPP}
\ee 
We refer to \cite{bayraktar2018randomized} and \cite{djete2019mckean} for the statement and proof of the dynamic programming principle for the stochastic formulation of the MFCP, in a much more general framework. 
We refer to 
\cite{Lauriere-Pironneau,Pham-Wei} for 
results directly formulated with controls in feedback form.

Formally, if $V$ is smooth, then, by the dynamic programming principle, it solves the HJB equation
\begin{equation}
\label{eq:HJB}
\begin{split}
&\partial_t V(t,m) -   \int_{{\mathbb T}^d} H \bigl( y, \partial_\mu V(t,m) (y) \bigr) d m(y) 
+ \frac12 \int_{{\mathbb T}^d} 
\text{Tr} \Bigl[ 
\partial_y \partial_\mu V(t,m,y) \Bigr] d m(y) + F(m) =0, \\
&V(T,m)=G(m),
\end{split}
\end{equation}
see for instance 
\cite[Section 3.7]{CardaliaguetDelarueLasryLions}. 

In Section \ref{sec:4}, we define a notion of generalized solution to the HJB equation
and we prove an existence and uniqueness result for it, see 
Theorem \ref{thm:uniqueness:HJB}, which is one of the main results of this paper. 


\subsection{Connection with the MFG system}
\label{subse:connection:MFG}

The analysis of the 
HJB equation 
\eqref{eq:HJB}
 plays a key role in the derivation of existence and uniqueness of weak solutions to the master equation
 \eqref{master}. 
Before we clarify the precise form of the master equation that we study next, 
we first recall the connection between the MFCP and the
potential MFG 
introduced in 
Subsection 
\ref{subse:MFG}.
%
%
%
The following key result is taken from \cite[Proposition 3.1]{BrianiCardaliaguet}:

\begin{prop} 
	\label{prop:Briani}
	For any initial condition $m_0 \in {\mathcal P}({\mathbb T}^d)$
	and for any optimal feedback function $\alpha$ (whose existence is already
	guaranteed by Proposition \ref{prop:2:2}) with corresponding optimal trajectory $m$, there exists $u \in {\mathcal C}^{2+\gamma}([0,T] \times \T)$
	such that $(m,u)$ is a classical solution to the MFG sytem \eqref{eq:MFG:system} and $\alpha_t(x) = - \nabla u_t(x)$ for any $0<t\leq T$ and any $x\in \T$. 
	
	As a consequence, the infimum in the definition of the value function \eqref{eq:MKV:V} can be equivalently taken over feedback functions that are bounded (by $M$
	in
Proposition \ref{prop:2:2}	
	) and Lipschitz continuous in $x$ (with a fixed Lipschitz constant determined by the data only).

\end{prop}

Below, we say (abusively) that 
$(m_t,u_t)_{0 \leq t \leq T}$
in 
Proposition 
\ref{prop:Briani}
is an MFG solution that minimizes the 
cost functional ${\mathcal  J}_{\textrm{\rm det}}$.
As we already noticed in Remark 
\ref{rem:minimization}, the optimization problem is in fact defined 
in 
\cite{BrianiCardaliaguet}
over a wider class of controls, namely over measures of the form $\alpha(t,x) m(t, dx)$. 
Obviously, this does not change the validity of Proposition \ref{prop:2:2}. 
 We also stress the following point: 
 In 
  \cite{BrianiCardaliaguet}, 
  $f$ and $g$ are satisfied to require 
  $f(x,m)=[\delta F/\delta m](m)(x)$ and 
  $g(x,m)=[\delta G/\delta m](m)(x)$, which is in fact stronger 
  than 
 the condition
 stated in  \eqref{eq:potential:structure}
 as it forces $\int_{{\mathbb T}^d} f(x,m) dx$ and $\int_{{\mathbb T}^d} g(x,m) dx$ to be
 equal to $0$, see
 \eqref{eq:deltam:partialmu}.   
 However, as remarked in \cite{BrianiCardaliaguet}, 
 there is in fact no need
to require $f$ and $g$ to be centered with respect to 
$m$ since, in the end, $f$ and $g$ are always integrated against a difference of two probability measures. 

Importantly, let us finally observe that, formally, if $V$ is $C^2$ in the measure argument, 
then,
by
applying the Schwarz identity 
(see \cite[§2.2.2]{CardaliaguetDelarueLasryLions}, but with a different centering condition)
$$\frac{\delta^2 V}{\delta m^2}(t,m)(x,y) = \frac{\delta^2 V}{\delta m^2}(t,m)(y,x),$$
its derivative 
$$\tilde U(t,x,m) : =\frac{\delta V}{\delta m}(t,m)(x), \quad (t,x,m) \in [0,T] \times {\mathbb T}^d,$$ 
satisfies the following two equations:
\begin{align}
\label{master:cons}
&\partial_t \tilde U(t,x,m) 
\\
&\quad+  \frac{\delta}{\delta m} \bigg\{-  \int_{{\mathbb T}^d} 
H \bigl( y,  \nabla_y \tilde U(t,y,m) \bigr) d m(y) 
+ \frac12 \int_{{\mathbb T}^d} 
\text{Tr} \Bigl[ 
\partial^2_y \tilde U(t,y,m) \Bigr] d m(y) + F(m)  \bigg\}(x) = 0, \nonumber
\\
&\tilde U(T,m)=\frac{\delta G}{\delta m}(m)(x), \nonumber
\end{align}
together with 
\begin{align}
	&\partial_t \tilde U (t,x,m) + \frac12 \Delta_x \tilde U(t,x,m) - H\bigl(x,\nabla_x \tilde U(t,x,m) \bigr) + f(x,m) 
	\nonumber
	\\
	&\quad -   \int_{\T} \partial_p H\bigl( y, \nabla_y \tilde U(t,y,m) \bigr) \cdot  \partial_\mu \tilde U(t,x,m)(y) m(dy) 
  + \frac12 \int_{\T} \mathrm{Tr}\bigl[ \partial_y \partial_\mu \tilde U(t,x,m)(y) \bigr] m(dy) \nonumber
  \\
  &\quad - C_t(m) =0, \phantom{\frac12} \label{master:centred} 
  \\
	&\tilde U(T,x,m) =g(x,m) - \int_{{\mathbb T}^d} g(y,m) dm(y),
	\nonumber
\end{align}
where $C_t(m)$ is a penalization term that makes the above equation consistent with 
the condition $\int_{{\mathbb T}^d} \tilde U(t,x,m) dx = 0$. 
Equivalently, 
\eqref{master:centred}
says that 
the gradient 
$\nabla_x \tilde U$ solves the derivative of the master  equation \eqref{master}. 
The fact that 
\eqref{master:centred} 
only coincides
with 
\eqref{master}
up to the additional remainder 
$C_t(m)$
was already reported in 
\cite{Cecchin:Delarue:CPDE}
 for games on a finite 
set.
In fact, the key point in this formulation 
is that it allows to identify (at least formally)  
$\nabla_x \tilde U$ in 
\eqref{master:centred} 
with 
$\nabla_x  U$ 
in \eqref{master}: this is crucial since 
$\nabla_x \tilde U$ is precisely the term inside the nonlinearity; once the nonlinear term in \eqref{master} has been solved, the equation becomes easier to handle.  

The fact that $\tilde U$ solves both 
\eqref{master:cons}
and 
\eqref{master:centred} prompts us to call Equation 
\eqref{master:cons} `the conservative form of the master equation'. 
This is precisely this version for which we prove an existence and uniqueness result in Theorem 
\ref{main:thm:MFG}: in short, the unique solution is shown to be the derivative of the value function 
$V$, the notion of derivative being understood in a relaxed sense. 

In order to check the accuracy of our result, we prove that the centered version of 
any smooth solution to the 
master equation \eqref{master}
is in fact a solution of the conservative form of the equation, see Proposition 
\ref{prop:verification:classical:conservative}. As a by-product, 
the proof clarifies the form of the penalization term $C_t(m)$ in \eqref{master:centred}.

\subsection{Summary of the main results}

We here provide a short summary of the mains results that are established in the paper. We draw reader's attention to the fact that these are only meta-statements: proper versions can only be given next, once the required tools have been 
carefully introduced. 
We start with the MFCP:

\begin{metatheorem}
\label{meta:1}
There exists a notion of generalized solution to the HJB equation
\eqref{eq:HJB}, under which a solution must satisfy 
almost everywhere
finite-dimensional 
approximations of \eqref{eq:HJB}
and
for which 
 existence and uniqueness 
 hold true. 
 The unique solution is the value function $V$ 
 defined in 
 \eqref{eq:MKV:V}. 
 
 Moreover, there exists a probability measure 
 ${\mathbb P}$ on $\PP$, with full support, 
 such that, for
any $t \in [0,T]$ and 
 ${\mathbb P}$ 
 almost every $m \in \PP$,  the MFCP 
 \eqref{eq:MKV:V}
 issued from $(t,m)$
 has a unique optimal trajectory.
 In fact, we also have that, for
any $t \in [0,T]$ and 
 ${\mathbb P}$ 
 almost every $m \in \PP$, 
 {$V(t,\cdot)$ has directional derivatives at $m$ 
 along $\mu-m=\cos(2 \pi k \cdot)$ and $\mu-m=\sin(2 \pi k \cdot)$, for any $k \in {\mathbb Z}^d \setminus \{0\}$, in 
 the sense specified in 
 \eqref{eq:delta m}.}

\end{metatheorem} 
This meta-statement is addressed in Section 
\ref{sec:4}. 
The notion of generalized solution is 
defined in 
Definition \ref{defn:HJB:gen}. 
The mains rigorous results 
encompassing 
Meta-Theorem 
\ref{meta:1}
are
Theorems
\ref{thm:value:is:gen:HJB}
and \ref{thm:uniqueness:HJB}. 
The definition and the properties of the measure $\PP$
are given in 
Theorem 
\ref{thm:probability:probability}. The existence of directional derivatives to $V$ are guaranteed by a tailor-made
version of  
Rademacher's theorem on 
$(\PP,{\mathbb P})$, which is rigorously stated in 
Theorem \ref{prop:rademacher}.

As for the MFG, we have:
\begin{metatheorem}
\label{meta:2}
There exists a notion of weak solution to the conservative master equation
\eqref{master:cons}
for which 
 existence and uniqueness 
 hold true, uniqueness being understood
everywhere in time, ${\mathbb P}$ almost everywhere in the measure argument and everywhere in the space argument, 
 for the same ${\mathbb P}$ as in the statement of 
 Meta-Theorem \ref{meta:1}. 
The hence unique solution coincides with the 
derivative of the value function $V$ 
with respect 
to the measure argument, which is known 
to exist
 in a directional sense from Meta-Theorem \ref{meta:1},  
everywhere in time and
${\mathbb P}$ almost everywhere 
in the measure argument.

Moreover, any classical 
solution to the 
master equation \eqref{master} is, after centering,  a solution of the conservative form of the equation.
\end{metatheorem} 
 The definition of a weak solution is 
given in Definition 
\ref{def:master:equation:gen}
and the main unique solvability result is stated in Theorem 
\ref{main:thm:MFG}. 
The last part of the statement is shown in Proposition 
\ref{prop:verification:classical:conservative}.

\section{Elements of Fourier Analysis}
\label{se:3}

We provide some elements of Fourier analysis that are needed in the rest of the work. 
The very basic idea of our approach is to regard a function $\phi : {\mathcal P}({\mathbb T}^d) \rightarrow {\mathbb R}$ as a function of the Fourier coefficients $(\widehat{m}^k)_{k \in {\mathbb Z}^d}$, defined by 
\begin{equation*}
\widehat{m}^k := \int_{{\mathbb T}^d} 
e^{\i 2 \pi k \cdot y} dm(y),
\end{equation*}
with $\i^2=-1$. For simplicity, we write $e_k$ for $e_k(x) = e^{\i 2 \pi k \cdot x }$. 
From time to time, we also write $e_k(\cdot)$ instead of $e_k$ in order to emphasize the fact that 
$e_k$ is a function (on the torus).

\subsection{Probability measures with finitely many non-zero Fourier coefficients}

\subsubsection{Bochner's theorem.}
We recall Bochner's theorem:  A complex-valued sequence $(\widehat{m}^k)_{k \in {\mathbb Z}^d}$ is the sequence of Fourier coefficients of 
a probability measure if and only if
\vskip 5pt

$(i)$ $\widehat{m}^0=1$, $\widehat{m}^{-k} = \overline{\widehat{m}^k}$ for $k \in {\mathbb Z}^d \setminus \{0\}$,  
\vskip 5pt

$(ii)$ for any $N \geq 1$, any $(z_k)_{k \in \{1,\cdots,N\}^d} \in {\mathbb C}^{\vert N \vert^d}$,
$\displaystyle 
\sum_{k,k' \in {\{1,\cdots,N\}^d}} \widehat{m}^{k-k'} z_k \overline z_{k'} \geq 0$.

Obviously, for $k,k' \in \{1,\cdots,N\}^d$, 
$k-k'$ belongs to $F_N:=\{-N+1,\cdots,N-1\}^d$.
Next, we introduce a few more notations in order to 
reformulate the above condition.  
For an element $k \in {\mathbb N}^d \setminus \{0\}$, we define $\sharp(k):=\inf\{ j\in \{1,\cdots,d\} : k_j \not =0\}$ (i.e., 
$\sharp(k)$ is the first-non zero coordinate in $k$). 
For $j \in \{1,\cdots,d\}$, we call 
\begin{equation*}
F_N^{+,j} := \bigl\{ k \in F_N : \sharp(k)=j \ \text{and} \  k_j >0\}, 
\end{equation*}
which allows us to let 
$F_N^+ :=
\bigcup_{j=1}^d 
F_N^{+,j}
=
 \{ k \in F_N \setminus \{0\} : k_{\sharp(k)} >0 \}$. 
 This notation is very convenient to eliminate the conjugaison constraints in 
 $(i)$. 
Indeed, we can rewrite 
\begin{equation*}
\begin{split}
\sum_{k,k' \in \{1,\cdots,N\}^d} \widehat{m}^{k-k'} z_k \overline z_{k'} 
 &=  \sum_{k \in \{1,\cdots,N\}^d} \vert z_k \vert^2
+ 
 \sum_{k,k' \in \{1,\cdots,N\}^d : k-k' \in F_N^{+}}
 \bigl(
  \widehat{m}^{k-k'}  z_k \overline z_{k'} 
 + 
 \overline{ \widehat{m}^{k-k'}  z_k \overline z_{k'} } \bigr).
 \end{split}
\end{equation*}
And then, we let
\begin{equation*}
\begin{split}
{\mathcal O}_N &:= \biggl\{ \bigl( \widehat{m}^k \bigr)_{k \in F_N^+} :
\inf_{\sum_{k \in  ({\mathbb N} \setminus \{0\})^d} \vert z_k \vert^2 \leq 1} 
 2 \Re \biggl[
  \sum_{k,k' \in ({\mathbb N} \setminus \{0\})^d : k-k' \in F_N^{+}}
  \widehat{m}^{k-k'}  z_k \overline z_{k'} 
\biggr] > -1 
\biggr\}.
\end{split}
\end{equation*}
It must be stressed that in the above definition, the infimum is taken over all sequences 
$(z_k)_{k \in 
 ({\mathbb N} \setminus \{0\})^d}$ such that $\sum_{k \in ({\mathbb N} \setminus \{0\})^d}
 \vert z_k\vert^2 \leq 1$. 
 In particular, 
 if we are given $(\widehat{m}^k)_{k \in F_N^+} \in {\mathcal O}_N$, we can complete the collection into a wider collection 
$(\widehat{m}^k)_{k \in {\mathbb Z}^d}$
by letting
\begin{equation}
\label{eq:Fourier:extensions}
\widehat{m}^0=1 \, ; \quad \widehat{m}^{-k}=\overline{\widehat{m}^k} \ \textrm{\rm if} \ k \in F_N^+ \, ; 
\quad \widehat{m}^k = 0 \ \textrm{\rm if} \ k \not \in F_N.
\end{equation}
By Bochner's theorem, 
the extended collection 
$(\widehat{m}^k)_{k \in {\mathbb Z}^d}$
hence defines a probability measure, which we may denote by 
${\mathscr I}_N((\widehat{m}^k)_{k \in F_N^+})$.
Below, we thus say that a probability measure $m \in {\mathcal P}({\mathbb T}^d)$ belongs to 
${\mathcal P}_N$ if 
$(\widehat{m}^k)_{k \in F_N^+} \in {\mathcal O}_N$
and 
$\widehat{m}^k=0$ if $k \not \in F_N$.  
In this sense, 
${\mathcal O}_N$ and ${\mathcal P}_N$ are one-to-one (through the mapping ${\mathscr I}_N$). 

Equivalently, elements 
of ${\mathcal P}_N$ can be identified with strictly positive density  
measures on ${\mathbb T}^d$ whose Fourier coefficients above $N$ are equal to $0$. 
Indeed, if $m \in {\mathcal P}_N$, then we can find $c>1$ such that 
\begin{equation*}
\inf_{\sum_{k \in  ({\mathbb N} \setminus \{0\})^d} \vert z_k \vert^2 \leq 1} 
 2 \Re \biggl[
  \sum_{k,k' \in ({\mathbb N} \setminus \{0\})^d : k-k' \in F_N^{+}}
  \widehat{m}^{k-k'}  z_k \overline z_{k'} 
\biggr] > -1/c,
\end{equation*} 
which, in turn, yields from Bochner's theorem 
that $1 + c (m-1)$ is a probability measure, namely 
$m \geq 1-1/c$. In the end, 
${\mathcal O}_N$ can be reformulated as
\begin{equation}
\label{eq:writing:2:ON}
\begin{split}
{\mathcal O}_N &= \biggl\{ \bigl( \widehat{m}^k \bigr)_{k \in F_N^+} :
\inf_{x \in {\mathbb T}^d}
 \biggl[ 
 1 + 
  2 \Re \biggl(
  \sum_{k \in  F_N^{+}}
  \widehat{m}^{k} e_{-k}(x) 
  \biggr)
\biggr] 
> 0 
\biggr\}.
\end{split}
\end{equation}

Next, we clearly have
\begin{prop}
\label{prop:3}
When complex numbers are identified with 
two-dimensional real vectors, 
${\mathcal O}_N$ is an 
open subset of real dimension $D_N:=2\vert F_N^+\vert = 2 \sum_{j=1}^d \vert F_N^{+,j} \vert$. 
The dimension $D_N$ is less than $2 (2 N-1)^d$.
\end{prop}

\begin{proof}
The proof for the upper bound of $D_N$ is as follows. We observe that 
$F_N^+ \subset F_N$ and $\vert F_N \vert \leq (2N-1)^d$. 
Therefore, there are at most $(2N-1)^d$ complex-valued entries in 
elements of ${\mathcal O}_N$ and, therefore, at most 
$2 (2N-1)^d$ real-valued entries. 
\end{proof}

\subsubsection{Bochner-Herglotz' Theorem}
\label{subsubse:BH}
In fact, there is a systematic construction of elements of ${\mathcal P}_N$. 
This is based on Bochner-Herglotz' theorem, which we recall now.  
The
function 
\begin{equation*}
f_N : \theta \in {\mathbb T}^d \mapsto 
\frac1{N^d} \sum_{k,k' \in \{1,\cdots,N\}^d}
\exp\bigl(  \i 2 \pi (k-k') \cdot \theta \bigr)
=
\frac1{N^d} 
\biggl\vert \sum_{k \in \{1,\cdots,N\}^d}
\exp\bigl(  \i 2 \pi k \cdot \theta \bigr)
\biggr\vert^2 
\end{equation*}
is a density on ${\mathbb T}^d$. Its Fourier coefficients are given by  
 \begin{equation*}
\begin{split}
&(i) \quad \widehat{f}^k_N = \prod_{j=1}^d \bigl( 1 - \frac{\vert k_j\vert}{N} \bigr), \quad k \in F_N,
\\
&(ii) \quad \widehat{f}^k_N =0, \quad k \not \in F_N.
\end{split}
\end{equation*}
If $m \in {\mathcal P}({\mathbb T}^d)$, then 
the convolution 
$m * f_N$ is (obviously) a probability measure and
its Fourier coefficients are given by 
\begin{equation*}
\begin{split}
&(i) \  \widehat{m * f_N}^k = \prod_{j=1}^d \bigl( 1 - \frac{\vert k_j\vert}{N} \bigr) \widehat{m}^k, \quad
\textrm{\rm if} \ k \in F_N \ ; 
\quad 
(ii) \  
\widehat{m * f_N}^k =0, \quad \textrm{\rm if} \ k \not \in F_N.
\end{split}
\end{equation*}
Obviously,
by 
\eqref{eq:writing:2:ON},
 we have
\begin{prop}
Let $m \in {\mathcal P}({\mathbb T}^d)$. 
Then, 
$m * f_N \in {\mathcal P}_N$ if and only if 
$\inf_{x \in {\mathbb T}^d} (m*f_N)(x)>0$.
\end{prop}


\subsubsection{Weak convergence of $(f_N)_{N \geq 1}$}
Clearly, $(f_N)_{N \geq 1}$ converges in the weak sense to $\delta_0$. 
The following result is completely standard. Since we will use it many times, we feel it more convenient to have it in the form of a lemma.
\begin{lem}
\label{lem:weak:convergence}
Let ${\mathcal K}$ be a compact subset of ${\mathcal C}({\mathbb T}^d)$. Then 
\begin{equation*}
\lim_{N \rightarrow \infty} 
\sup_{h \in {\mathcal K}} 
\sup_{x \in {\mathbb T}^d} 
\bigl\vert \bigl( h * f_N\bigr) (x) - h(x) \bigr\vert =0.
\end{equation*}
\end{lem}

\subsection{From probability measures on ${\mathcal O}_N$ to a probability measure on ${\mathcal P}({\mathbb T}^d)$}
\label{subse:PP}
In Subsection \ref{sec:6}, we construct a probability measure on ${\mathcal P}({\mathbb T}^d)$ that satisfies Theorem 
\ref{thm:probability:probability}
 below. 
Before we make the statement clear, we recall that, except when explicitly mentioned, the 
space ${\mathcal P}({\mathbb T}^d)$ is equipped with the $1$-Wasserstein distance.  
The Borel $\sigma$-algebra 
 on 
 ${\mathcal P}({\mathbb T}^d)$ is the $\sigma$-algebra generated by 
 the mappings
 $m \in {\mathcal P}({\mathbb T}^d) \mapsto \int_{{\mathbb T}^d} h(x) dm(x)$, for 
 $h$ in ${\mathcal C}({\mathbb T}^d)$. 
 Quite often, we use the following characterization of the Borel $\sigma$-algebra:
 \begin{lem}
 \label{lem:Borel}
 The Borel $\sigma$-algebra  
  ${\mathcal B}({\mathcal P}({\mathbb T}^d))$
  is generated by the mappings 
 $m \in {\mathcal P}({\mathbb T}^d) \mapsto \widehat m^k$, for $k \in {\mathbb N}^d$.  
 \end{lem}
 \begin{proof}
It suffices to observe that, for any $h \in {\mathcal C}({\mathbb T}^d)$
and any $m \in {\mathcal P}({\mathbb T}^d)$, 
\begin{equation*}
 \int_{{\mathbb T}^d} h(x) dm(x) = 
 \lim_{N \rightarrow \infty} 
 \int_{{\mathbb T}^d} h(x) d\bigl( m* f_N \bigr) (x)
=
 \lim_{N \rightarrow \infty} 
\sum_{k \in F_N} \widehat{m}^k \widehat{f}_N^k \widehat{h}^{-k}.  
 \end{equation*}
 \end{proof}

 In brief, the probability measure provided by Theorem \ref{thm:probability:probability} 
 below is obtained by means of a limiting argument, 
which relies on the  following notations. For any $N \geq 1$, we equip 
${\mathcal O}_N$ with the (truncated) Gaussian density 
(in the definition below,  $\widehat{m}^k$ is implicitly regarded as 
a two-dimensional vector, namely $(\Re(\widehat{m}^k),\Im(\widehat{m}^k))$):
\begin{equation}
\label{eq:Gamma_N}
\Gamma_N \Bigl( \bigl( \widehat{m}^k \bigr)_{k \in F_N^+} \Bigr)
:=\frac1{Z_N} 
{\mathbf 1}_{{\mathcal O}_N}
\Bigl( \bigl( \widehat{m}^k \bigr)_{k \in F_N^+} \Bigr)
\exp \Bigl( - \sum_{k \in F_N^+} \vert k\vert^{2pd} { \vert \widehat{m}^k \vert^2} \Bigr),
\end{equation}
for a real $p \geq 5$ that is fixed throughout the paper, and define
on ${\mathcal P}({\mathbb T}^d)$  
${\mathbb P}_N:= \Gamma_N \circ {\mathscr I}_N^{-1}$,
where we recall that 
\begin{equation}
\label{eq:I_N}
{\mathscr I}_N : 
\bigl( \widehat{m}^k \bigr)_{k \in F_{N}^+}  \in {\mathcal O}_N \mapsto 
m =
1+ 2 
\sum_{k \in F_N^+} 
\Re \bigl[ 
\widehat{m}^k e_{-k}(\cdot) 
\bigr]
=
 \sum_{k \in F_N} \widehat{m}^k e_{-k}(\cdot) \in {\mathcal P}({\mathbb T}^d), 
\end{equation}
with the convention that $\widehat{m}_0=1$ and 
$\widehat{m}^{-k}= \overline{\widehat{m}^k}$ for 
$k \in F_N \setminus \{0\}$. (Since ${\mathscr I}_N$ is obviously continuous, ${\mathbb P}_N$ is well-defined.)
Since 
${\mathcal P}({\mathbb T}^d)$  
is compact, the sequence 
$({\mathbb P}_N)_{N \geq 1}$ has weak limits. These weak limits are the measures on ${\mathcal P}({\mathbb T}^d)$ that we are interested in. 
\vskip 5pt

Here is now our main statement in this subsection:
\begin{thm}
\label{thm:probability:probability}
The sequence $({\mathbb P}_N)_{N \geq 1}$ is weakly converging 
to a probability measure ${\mathbb P}$ on the Borel space ${\mathcal P}({\mathbb T}^d)$ equipped
with the $1$-Wasserstein measure. The latter satisfies the following items: 
\vskip 4pt

$(i)$ 
${\mathbb P}$ has full support and, for  
${\mathbb P}$ almost every $m \in {\mathcal P}({\mathbb T}^d)$, 
$m$ has a continuously differentiable strictly positive density. 
\vskip 4pt

$(ii)$ For any $N \geq 1$, 
the image of ${\mathbb P}$ by the projection mapping 
\begin{equation*}
\pi_N^{(1)} : m 
\in {\mathcal P}({\mathbb T}^d)
\mapsto 
\bigl( \widehat{m}^k \widehat{f}_{N}^k \bigr)_{k \in F_{N}^+} 
\in 
{\mathcal O}_{N} 
\end{equation*}
is absolutely continuous with respect to the Lebesgue measure (on ${\mathcal O}_{N}$).
Similarly, the image of the subprobability measure 
${\mathbf 1}_{{\mathcal P}_{N}}(m) \cdot {\mathbb P}$ by the projection mapping
\begin{equation*}
\pi_N^{(2)} : m 
\in {\mathcal P}({\mathbb T}^d)
\mapsto 
\bigl( \widehat{m}^k  \bigr)_{k \in F_{N}^+} 
\in 
{\mathbb R}^{2 \vert F_N^+ \vert} 
\end{equation*}
is supported by ${\mathcal O}_N$ and 
is
absolutely continuous with respect to the Lebesgue measure (on ${\mathcal O}_{N}$).
\vskip 4pt

$(iii)$ 
\begin{equation*}
\lim_{N \rightarrow \infty} 
\sup_{\phi : {\mathbb R}^{{D_{N}}} \rightarrow [-1,1]}
\biggl\vert \int_{{\mathcal P}({\mathbb T}^d)}
\phi\Bigl( (\widehat{m}^k)_{k \in F_{N}^+ } \Bigr) d\bigl( {\mathbb P} - {\mathbb P}_{N} \bigr)(m)
\biggr\vert
=0,
\end{equation*}
where, for any $N \geq 1$, $\phi$ in the argument of the supremum is required to be measurable. 
\end{thm}

Items $(i)$ and $(iii)$ will be used next, when addressing generalized and weak solutions to 
the HJB equation of the MFCP and the master equation of the MFG. 
In the rest of this section, we just make use of item $(ii)$. It serves us 
as a way to pass from properties that are almost everywhere satisfied on 
${\mathcal O}_{N}$ (under the finite-dimensional Lebesgue measure) 
to properties that are almost everywhere satisfied 
on 
${\mathcal P}({\mathbb T}^d)$ under ${\mathbb P}$ (or under ${\mathbf 1}_{{\mathcal P}_{N}} \cdot 
{\mathbb P}$). In this regard, it is worth observing that
 the need for considering the subprobability measure 
${\mathbf 1}_{{\mathcal P}_{N}} \cdot {\mathbb P}$ instead of the measure 
${\mathbb P}$ itself in item $(ii)$ comes from the fact that, for 
$m \in 
{\mathcal P}({\mathbb T}^d) \setminus 
{\mathcal P}_{N}$, 
the collection 
$( \widehat{m}^k )_{k \in F_{N}^+} $
may not be in ${\mathcal O}_{N}$.  

A useful tool to pass from 
Lebesgue almost everywhere properties on 
$({\mathcal O}_{N})_{N \geq 1}$ to ${\mathbb P}$ almost sure properties on ${\mathcal P}({\mathbb T}^d)$ is the following lemma: 

\begin{lem}
\label{lem:23}
Let ${\mathbb P}$ be as in the statement of Theorem 
\ref{thm:probability:probability}.
In addition, let, for any $N\geq 1$, $A_N$ be a full Borel subset of ${\mathcal O}_N$.

Then, we can find 
an event $A_\infty \in {\mathcal B}({\mathbb T}^d)$ such that
${\mathbb P}(A_\infty)=1$ and $\pi_N^{(1)}(A_\infty) \subset A_N$ for any $N \geq 1$. 
\end{lem}

\begin{proof}
The measurability of the mapping $\pi_N^{(1)}$ is absolutely obvious since it is continuous. 
Then, we notice that 
\begin{equation*}
{\mathbb P}\Bigl( \bigl\{ m \in {\mathcal P}({\mathbb T}^d) : \pi_N^{(1)} (m) \in A_N^{\complement} \bigr\} \Bigr) 
= {\mathbb P} \circ \bigl( \pi_N^{(1)} \bigr)^{-1}\bigl(A_N^{\complement}\bigr) =0,
\end{equation*}
since
${\mathbb P} \circ (\pi^{(1)}_N)^{-1}$ is absolutely continuous with respect to the Lebesgue measure on 
${\mathcal O}_N$. 
Then, it suffices to let 
\begin{equation*}
A_{\infty} := \bigcap_{N \geq 1}
\Bigl\{ m \in {\mathcal P}({\mathbb T}^d) : \pi_N^{(1)}(m) \in A_N  \Bigr\}.  
\end{equation*}
\end{proof}

\subsection{Restriction of a function defined on ${\mathcal P}({\mathbb T}^d)$}

Throughout this subsection, we take a function $\phi : {\mathcal P}({\mathbb T}^d) \rightarrow {\mathbb R}$, which we regard as a function 
of $(\widehat{m}^k)_{k \in {\mathbb Z}^d}$. 
Under the assumption that $\phi$ has a linear functional derivative, we are interested in the shape of the derivatives of $\phi$, when computed with respect to the Fourier coefficients. 

\subsubsection{Connection between flat derivatives and derivatives w.r.t. Fourier coefficients}
We start with
\begin{prop}
\label{prop:3:7}
Assume that 
$\phi$ has a linear functional derivative at $m \in {\mathcal P}_N$. 
Then, for
$(\hat{r}^k)_{k \in F_N^+} \in {\mathbb R}^{2 \vert F_N^+\vert}$, denote
by $r$ the function
\begin{equation*}
r(\cdot) :=
2
 \sum_{k \in F_N^+} 
 \Re \bigl[ \widehat{r}^k 
e_{-k}(\cdot) \bigr] =
 \sum_{k \in F_N \setminus \{0\}} \widehat{r}^k 
e_{-k}(\cdot), 
\end{equation*}
with the convention that,  
in the above sum, 
$\widehat{r}^{-k} 
= \overline{\widehat{r}^k}$ 
if $k \in F_N^+$. Then,
for $\varepsilon$ small enough, 
$m + \varepsilon r \in {\mathcal P}_N$ and 
\begin{equation*}
\begin{split}
&\frac{d}{d \varepsilon}_{\vert \varepsilon =0} 
\phi \Bigl( 
 m + 
\varepsilon r \Bigr) 
\\
&= \sum_{k \in F_N} 
\widehat{r}^k
\int_{{\mathbb T}^d} \frac{\delta \phi}{\delta m} (m)(y) e_{-k}(y) dy
\\
&= 
\sum_{k \in F_N \setminus \{0\}} 
\widehat{r}^k
\reallywidehat{\displaystyle \frac{\delta \phi}{\delta m} (m)(\cdot)}^{-k}
= 
2\sum_{k \in F_N^+}
\Bigl( 
\Re \bigl[ 
\widehat{r}^k
\bigr]
\Re \Bigl[ 
\reallywidehat{\displaystyle \frac{\delta \phi}{\delta m} (m)(\cdot)}^{k}
\Bigr] 
+
\Im \bigl[ 
\widehat{r}^k
\bigr]
\Im \Bigl[ 
\reallywidehat{\displaystyle \frac{\delta \phi}{\delta m} (m)(\cdot)}^{k}
\Bigr] 
\Bigr). 
\end{split}
\end{equation*}
\end{prop}
The proof is quite straightforward. 
Below, 
we will write
\begin{equation}
\label{eq:prop:3:7:identification}
\begin{split}
&
\partial_{\Re[\widehat{m}^k]} \phi(m) 
 := 2 \Re \biggl[  \reallywidehat{\displaystyle \frac{\delta \phi}{\delta m} (m)(\cdot)}^{k} \biggr],
 \quad \partial_{\Im[\widehat m^k]} \phi(m)  
 := 2 \Im \biggl[
 \reallywidehat{\displaystyle \frac{\delta \phi}{\delta m} (m)(\cdot)}^{k} 
 \biggr],
 \end{split}
\end{equation}
which is rather abusive as it means that the function $\phi$ is implicitly regarded as a function 
of the Fourier coefficients.
Accordingly, we let, for $k \in F_N^+$, 
\begin{equation}
\label{eq:gradient:complex:000}
\begin{split}
\partial_{\widehat{m}^k} \phi(m) &:= 
\frac12 \partial_{\Re[\widehat{m}^k]} \phi(m)
-  \frac{\i}2 \partial_{\Im[\widehat{m}^k]} \phi(m)
=
 \reallywidehat{\displaystyle \frac{\delta \phi}{\delta m} (m)(\cdot)}^{-k} ,
\end{split}
\end{equation}
so that 
\begin{equation}
\label{eq:gradient:complex}
\begin{split}
\partial_{\widehat m^k} \phi(m) \widehat{r}^k
+ 
\overline{\partial_{\widehat m^k} \phi(m) \widehat{r}^k}
= 
\partial_{\Re[\widehat{m}^k]} \phi(m)
\Re[ \widehat{r}^k]
+ 
\partial_{\Im[\widehat{m}^k]} \phi(m)
\Im[ \widehat{r}^k]. 
\end{split}
\end{equation}
When $-k \in F_N^+$, we let 
$\partial_{\widehat{m}^k} \phi(m)=
\overline{\partial_{\widehat{m}^{-k}} \phi(m)}$. 
However,
the symbol $\partial_{\widehat{m}^k}$ is by no means an holomorphic derivative in the complex plain.
This is just a shorten notation that refers to (real) differentiability with respect to $(\Re(\widehat{m}^k),\Im(\widehat{m}^k))$. This
notation is fully legitimated by the fact that, each time this term appears, its complex conjugate also appears (in a common sum).

\subsubsection{Almost everywhere differentiability of restrictions of Lipschitz functions}
A function $\phi$ defined on ${\mathcal P}({\mathbb T}^d)$ can be restricted to ${\mathcal P}_N$. 
This induces a function, denoted by $\widetilde{\phi}_N$ and defined 
by
\begin{equation}
\label{eq:tilde:phi:N}
\widetilde \phi_N : \bigl(\widehat{m}^k \bigr)_{k \in F_N^+} 
\in {\mathcal O}_N 
\mapsto 
\phi
\circ {\mathscr I}_N
\Bigl(
 \bigl(\widehat{m}^k \bigr)_{k \in F_N^+} 
\Bigr)
= \phi
 \biggl( 
\sum_{k \in F_N}
\widehat{m}^k 
e_{-k}
\biggr),
\end{equation}
with $(\widehat{m}^k)_{k \in F_N}$
as in  \eqref{eq:Fourier:extensions} and 
${\mathscr I}_N$ as in 
\eqref{eq:I_N}.
When $m \in {\mathcal P}_N$, 
$\phi(m)$ and 
$\widetilde{\phi}_N((\widehat{m}^k )_{k \in F_N^+})$
(with 
$(\widehat{m}^k )_{k \in F_N^+}$
being here regarded as 
the Fourier coefficients of $m$)
coincide. For this reason, we sometimes merely write 
$\widetilde{\phi}_N((\widehat{m}^k )_{k \in F_N^+})$
as 
${\phi}((\widehat{m}^k )_{k \in F_N^+})$, 
even though the notation is rather abusive. 
Moreover, it is quite straightforward to see 
from the formula stated in Proposition 
\ref{prop:3:7}
that
the restriction of $\phi$ to ${\mathcal P}_N$ is differentiable at $m \in {\mathcal P}_N$
if and only if
 $\widetilde{\phi}_N$ is 
 differentiable (in the real sense) at $(\widehat{m}^k )_{k \in F_N^+}$.

The following statement is straightforward, but very useful:
\begin{prop}
\label{prop:3:8}
Let $\phi$ be a Lipschitz function with respect to the total variation distance. Then, for 
any $N \geq 1$, 
the 
function 
$\widetilde \phi_N$
is almost everywhere differentiable on 
${\mathcal O}_N$, when 
the latter is equipped with the Lebesgue measure on ${\mathcal O}_N$. 

In particular, the restriction of the function 
$\phi$ to ${\mathcal P}_N$ is almost everywhere differentiable on 
${\mathcal P}_N$, when the latter is equipped with the image 
of the Lebesgue measure on ${\mathcal O}_N$ by  
the canonical embedding ${\mathscr I}_N : (\widehat{m}^k)_{k \in F_N^+} \mapsto 
 \sum_{k \in F_N} \widehat{m}^k e_{-k}$, with $(\widehat{m}^k)_{k \in F_N}$
 being 
as in  \eqref{eq:Fourier:extensions}. 
%
%
\end{prop}

Notice that 
$\phi$ is 
assumed to be Lipschitz continuous with respect to the total variation distance. 
This includes the case when 
$\phi$ is Lipschitz continuous with respect to the $1$-Wasserstein distance.

\begin{proof}
It suffices to notice that, for two probability measures $m_1$ and $m_2$ in ${\mathcal P}_N$, 
\begin{equation*}
\bigl\vert \phi(m_1)- \phi(m_2) \bigr\vert 
\leq C \textrm{d}_{\textrm{TV}}(m_1,m_2) \leq C \biggl( \sum_{k \in F_N} \vert \widehat{m}_1^k - \widehat{m}_2^k \vert^2
\biggr)^{1/2},
\end{equation*}
where $C$ is the Lipschitz constant of the function $\phi$. 
We then apply standard (finite-dimensional) Rademacher's theorem. It says that $\phi$, when regarded as a function on ${\mathcal O}_N$, 
is almost everywhere differentiable.  
\end{proof}

By Lemma
\ref{lem:23}, we have the following corollary:

\begin{cor}
Let ${\mathbb P}$ be as in the statement of Theorem 
\ref{thm:probability:probability} and 
$\phi$ be a Lipschitz function with respect to the total variation distance.
Then, for ${\mathbb P}$-almost every
$m \in {\mathcal P}({\mathbb T}^d)$, 
for any $N \geq 1$, the restriction of the function $\phi$ to 
${\mathcal P}_N$ is differentiable at $m*f_N$.  
\end{cor}

In fact, we have the following version of Rademacher's theorem on the entire $\PP$: 
\begin{thm}
\label{prop:rademacher}
Let ${\mathbb P}$ be as in the statement of Theorem 
\ref{thm:probability:probability} and 
$\phi$ be a Lipschitz function with respect to the total variation distance.
Then, 
for ${\mathbb P}$-almost every
$m \in {\mathcal P}({\mathbb T}^d)$, 
$\phi$ is differentiable at $m$ 
in the directions $\Re[e_k]$ and $\Im[e_k]$ (or equivalently with respect to $\Re[\widehat{m}^k]$ and 
$\Im[\widehat{m}^k]$) for any $k \in {\mathbb Z}^d \setminus \{0\}$. 

Moreover, the following convergences hold true (as $N$ tends to $\infty$) for the weak-star topology
$\sigma^*((L^\infty (\PP,{\mathbb P});(L^1 (\PP,{\mathbb P}))$:
\begin{equation}
\label{weak:convergence:derivatives}
\forall k \in {\mathbb Z}^d 
\setminus \{0\}, 
\quad 
\partial_{\widehat{m}^k} \phi \bigl( m * f_N) 
\underset{N \rightarrow \infty}{\rightharpoonup}
\partial_{\widehat{m}^k} \phi( m),
\end{equation}
where the derivative in the left-hand side is understood as 
the Lebesgue almost everywhere derivative of 
the restriction of $\phi$ to ${\mathcal P}_N$ whilst the derivative in the right-hand side is understood as the ${\mathbb P}$ almost everywhere derivative 
of $\phi$ (as given by the hence stated form of Rademacher's theorem on $\PP$). 
\end{thm}
Theorem 
\ref{prop:rademacher}
is proved in 
Subsection 
\ref{subse:proof:rademacher}. 
The identification of the 
${\mathbb P}$ almost everywhere derivative as the limit of finite dimensional derivatives, as stated in 
\eqref{weak:convergence:derivatives}, plays a key role in the identification of 
the solution of the weak formulation to the master equation given below.

\subsection{Projection of a function defined on ${\mathcal P}({\mathbb T}^d)$}

\begin{prop}
\label{prop:convolution:f_N}
Let $m \in {\mathcal P}({\mathbb T}^d)$ and $\phi : {\mathcal P}({\mathbb T}^d) \rightarrow {\mathbb R}$ be such that the restriction 
  of $\phi$ to ${\mathcal P}_N$ is differentiable at $m * f_N$, then 
the function 
\begin{equation*}
\phi \star f_N (\cdot) := \phi \bigl( f_N * \cdot \bigr),
\end{equation*}
is differentiable in any direction at $m$ and, for any 
$r \in {\mathcal P}({\mathbb T}^d)$,  
\begin{equation*}
\frac{d}{d \varepsilon}_{| \varepsilon =0} 
\phi \star f_N \bigl( (1- \varepsilon) m + \varepsilon r \bigr)  
= \int_{{\mathbb T}^d} \biggl( \int_{{\mathbb T}^d} \sum_{k \in F_N \setminus \{0\}} (\partial_{\widehat{m}^k} \phi) (m * f_N)  e_{k}(y) f_N(x-y) dy  \biggr) d \bigl( r-m\bigr)(x).
\end{equation*}
\end{prop}
In words, the above statement allows to identify 
$\frac{\delta}{\delta m}[ \phi \star f_N (m)]$ as
\begin{equation*}
 \frac{\delta}{\delta m} \phi \star f_N (m)(\cdot ) =  \Bigl( \sum_{k \in F_N  \setminus \{0\}} (\partial_{\widehat{m}^k}  \phi)(m * f_N)  e_{k} \Bigr) * f_N(\cdot).
\end{equation*}

\begin{proof}
It suffices to observe 
from 
Proposition 
\ref{prop:3:7}
and
\eqref{eq:gradient:complex:000}
that 
\begin{equation*}
\begin{split}
\frac{d}{d \varepsilon}_{| \varepsilon =0} 
\phi \star f_N \bigl( (1- \varepsilon) m + \varepsilon r \bigr)  
&=
\sum_{k \in F_N \setminus \{0\}}
\partial_{\widehat{m}^k} \phi (m*f_N) \bigl(
\widehat{r}^k - \widehat{m}^k \bigr) \widehat{f}_N^k
\\
&= \int_{{\mathbb T}^d} 
\sum_{k \in F_N \setminus \{0\}}
\partial_{\widehat{m}^k} \phi (m*f_N)  \widehat{f}_N^k e_{k}(x) d \bigl( r- m\bigr)(x),
\end{split}
\end{equation*}
and the identity easily follows.
\end{proof}

\subsection{Mollification}
In line with the material 
that we have introduced so far, we introduce a mollification procedure
of real-valued functions
on ${\mathcal P}({\mathbb T}^d)$ 
 based on the Fourier coefficients 
of the measure. 

\begin{defn}
\label{def:admiss:threshold:smoothing}
For $\varepsilon \in (0,1)$ and $N \in {\mathbb N}$, we call $\delta_{N,\varepsilon}$ given by
\begin{equation*}
\delta_{N,\varepsilon} := \frac{\varepsilon}{d N^2 \vert F_N\vert}
\end{equation*}
the regularization threshold.
\end{defn}

The regularization threshold plays a key role in the whole paper. In particular, it manifests in the following definition:
\begin{defn}
\label{def:mollification}
Let $\rho$ be a smooth density on ${\mathbb R}^2$ with a support included in the ball of center $0$ and radius $\delta_{N,\varepsilon}$, for 
$\varepsilon$ and $N$ as in Definition \ref{def:admiss:threshold:smoothing}. Then, for a real-valued function $\phi$ defined on ${\mathcal P}({\mathbb T}^d)$, we call mollification 
$\phi^{N,\varepsilon,\rho}$ of $\phi$ the function 
\begin{equation*}
\phi^{N,\varepsilon,\rho}(m) := \int_{{\mathbb R}^{2\vert F_N^+\vert}}
\phi \Bigl( m^{N,\varepsilon}(r) * f_N \Bigr) \prod_{j \in F_N^+}   \rho\bigl(\Re[ \widehat{r}^j],\Im[ \widehat{r}^j] \bigr)  
\bigotimes_{j \in F_N^+} 
d \bigl( \Re[\widehat{r}^j],\Im[\widehat{r}^j] \bigr),
\end{equation*}
with the notation
\begin{equation*}
\begin{split}
m^{N,\varepsilon}(r) &:= 
\varepsilon \textrm{\rm Leb}_{d} 
+ (1-\varepsilon) 
\Bigl( m + 2 \sum_{j \in F_N^+} \Re \bigl[ \widehat{r}^j e_{-j} \bigr] \Bigr)
=
\varepsilon \textrm{\rm Leb}_{d} 
+ (1-\varepsilon) 
\Bigl( m + \sum_{j \in F_N \setminus \{0\}} \widehat{r}^j e_{-j} \Bigr),
\end{split}
\end{equation*}
and 
with the convention that $\widehat{r}^{-j} = \overline{\widehat{r}^j}$.  
\end{defn}

For the sake of convenience, we often write
\begin{equation*}
\phi^{N,\varepsilon,\rho}(m) = \int_{{\mathbb R}^{2\vert F_N^+\vert}}
\phi \Bigl( m^{N,\varepsilon}(r) * f_N \Bigr) 
\prod_{j \in F_N^+} \rho\bigl(  \widehat{r}^j \bigr)
\bigotimes_{j \in F_N^+} 
  d \widehat{r}^j ,
\end{equation*}
but again the notation is abusive. 

Of course, it should be stressed that, in the above proposition, 
the Fourier coefficients that belong to the support of 
$\rho$ satisfy the inequality
\begin{equation*}
\Bigl\vert \sum_{j \in F_N \setminus \{0\}} \widehat{r}^j e_{-j}(x) \Bigr\vert 
< \delta_{N,\varepsilon}  \vert F_N\vert < \varepsilon,
\end{equation*}
which guarantess that $m^{N,\varepsilon}(r)$ is a positive density. 

The key fact is that $\phi^{N,\varepsilon,\rho}$ only depends on $m$ through the Fourier coefficients 
$(\widehat{m}^k)_{k \in F_N}$. Obviously, it is a smooth function of the latter, which proves that 
$\phi^{N,\varepsilon,\rho}$ is differentiable with respect to $m$. 
We have:  
\begin{prop}
\label{prop:mollification}
For an integer $N$, a real $\varepsilon \in (0,1)$
and a smooth density $\rho$ with a support of radius less than $\delta_{N,\varepsilon}$, 
and for a bounded measurable real-valued function 
$\phi$ on 
${\mathcal P}({\mathbb T}^d)$, 
the
function $\phi^{N,\varepsilon,\rho}$ is continuously differentiable with respect to 
$m$, with the following writing
for the derivative: 
\begin{equation*}
\begin{split}
&\frac{\delta \phi^{N,\varepsilon,\rho}}{\delta m}(m)(x)
 = -  
\int_{{\mathbb R}^{2 \vert F_N^+\vert}}
\sum_{k \in F_N  \setminus \{0\}}  \Bigl[ \Bigl(\phi \bigl(m^{N,\varepsilon}(r) * f_N\bigr) e_{k} \Bigr) * f_N \Bigr] (x)
 \partial_{\widehat{r}^k}
 \biggl[
 \prod_{j \in F_N^+} \rho\bigl(  \widehat{r}^j \bigr) 
 \biggr]
\bigotimes_{j \in F_N^+} 
 d \widehat{r}^j, 
\end{split}
\end{equation*}
for $x \in {\mathbb T}^d$. 
Moreover, if
$\phi$ is Lipschitz continuous with respect to 
$d_{\textrm{\rm TV}}$, 
then the following formula also holds true: 
\begin{equation*}
\begin{split}
&\frac{\delta \phi^{N,\varepsilon,\rho}}{\delta m}(m)(x)
\\
&= (1-\varepsilon) 
\int_{{\mathbb R}^{2\vert F_N^+\vert}}
\sum_{k \in F_N  \setminus \{0\}} \Bigl[ \Bigl(\partial_{\widehat{m}_k}  \phi \bigl(m^{N,\varepsilon}(r) * f_N\bigr) e_{k} \Bigr) * f_N\Bigr] (x) 
\prod_{j \in F_N^+}  \rho\bigl(  \widehat{r}^j \bigr)  
\bigotimes_{j \in F_N^+} 
 d \widehat{r}^j,
\end{split}
\end{equation*} 
for $x \in {\mathbb T}^d$. 
\end{prop}
\begin{proof} 
The proof of a very similar result is already exposed in the appendix of 
\cite{DelarueTse}. Briefly, 
the first identity in the statement can be proved in two steps. The first step is
to
write $\phi(m^{N,\varepsilon}(r)*f_N)$ in the definition of 
$\phi^{N,\varepsilon,\rho}(m)$ as a function 
of $(\widehat{m}^k+\widehat{r}^k)_{k \in F_N^+}$
and then to make a change a variable (with respect to 
$(\widehat{r}^k)_{k \in F_N^+}$) in order to pass the dependence upon 
$(\widehat{m}^k)_{k \in F_N^+}$ into the function $\rho$.
We hence get that 
$\phi^{N,\varepsilon,\rho}$ is smooth with respect to 
 $(\widehat{m}^k)_{k \in F_N^+}$.
 The second step 
is then to check, very much in the spirit of Proposition 
\ref{prop:convolution:f_N},
that 
the differentiability of 
$\phi^{N,\varepsilon,\rho}$ with respect to 
 $(\widehat{m}^k)_{k \in F_N^+}$
 implies the existence of $\delta \phi^{N,\varepsilon,\rho}/\delta m$.
 Once the first identity has been proved, the second one easily follows by a standard integration by parts, using the fact that the derivatives of 
 $\phi$ with respect to 
$(\widehat{m}^k )_{k \in F_N^+}$
exist almost everywhere. 
\end{proof}

As for the convergence of 
$\phi^{N,\varepsilon,\rho}$ to $\phi$, it is ensured under a standard continuity property:
\begin{lem}
\label{lem:convergence:mollif}
Assume that the function $\phi$
in Definition 
\ref{def:mollification}
is continuous for $d_{W^1}$ at some $m \in {\mathcal P}({\mathbb T}^d)$. Then,
$\phi^{N,\varepsilon,\rho}(m)$ converges to $\phi(m)$ 
when $(N,\varepsilon)$ tends to $(\infty,0)$ (it being understood that $\rho$ varies along the prescription of 
Definition 
\ref{def:mollification}). 
\end{lem}

\begin{proof}
We first compute $d_{\textrm{\rm TV}}(m^{N,\varepsilon}(r),m^{N,\varepsilon}(0))$, for $r$ in the support of $\rho$:
\begin{equation*}
\begin{split}
d_{\textrm{\rm TV}}\Bigl(m^{N,\varepsilon}(r),m^{N,\varepsilon}(0)\Bigr)
&= (1-\varepsilon)  \sup_{\|  h\|_\infty \leq 1}
\sum_{j \in F_N \setminus \{0\}}
\int_{{\mathbb T}^d} h(x) \widehat{r}^j e_{-j}(x) dx
\leq \delta_{N,\varepsilon} \vert F_N\vert \leq \varepsilon.
\end{split}
\end{equation*}
 Next we have, in a similar way,  $d_{\textrm{\rm TV}}(m^{N,\varepsilon}(0),m)
\leq 2 \varepsilon$, which gives
$d_{\textrm{\rm TV}}(m^{N,\varepsilon}(r),m)
\leq 3 \varepsilon$,
and then 
$d_{\textrm{\rm TV}}(m^{N,\varepsilon}(r)*f_N,m*f_N)
\leq 3 \varepsilon$.
It remains to recall from
Lemma \ref{lem:weak:convergence}
that $d_{W_1}(m*f_N,m)$
tends to $0$ as 
$N$ tends to $\infty$. 
Using the fact that $d_{W_1} \leq c_0 d_{\rm TV}$ (with 
$c_0$ as in the notation introduced in Section \ref{se:introduction}), 
$d_{W_1}(m^{N,\varepsilon}(r)*f_N,m)$
(with $r$ in the support of $\rho$)
tends to 0
as
$(N,\varepsilon)$ tends to $(\infty,0)$.  This completes the proof.  
\end{proof}

Below, we apply Proposition 
\ref{prop:mollification}
to time-space dependent functions $\varphi : [0,T] \times {\mathcal P}({\mathbb T}^d) \rightarrow {\mathbb R}$. 
Following 
Definition 
\ref{def:mollification}, 
we let:
\begin{equation*}
\varphi^{N,\varepsilon,\rho}(t,m) := \int_{{\mathbb R}^{2\vert F_N^+\vert}}
\varphi \Bigl(t, m^{N,\varepsilon}(r) * f_N \Bigr) \prod_{j \in F_N^+}  \rho\bigl(\Re[ \widehat{r}^j],\Im[ \widehat{r}^j] \bigr)  
\bigotimes_{j \in F_N^+} 
d \bigl( \Re[\widehat{r}^j],\Im[\widehat{r}^j] \bigr),
\end{equation*} 
for $(t,m) \in [0,T] \times {\mathcal P}({\mathbb T}^d)$. 
We then have the following corollary:

\begin{cor}
\label{cor:mollif:time-space}
For an integer $N$, an $\varepsilon \in (0,1)$
and a smooth density $\rho$ with a support of radius less than $\delta_{N,\varepsilon}$, 
and for a continuous real-valued function 
$\varphi$ on 
$[0,T] \times 
{\mathcal P}({\mathbb T}^d)$, 
the
function $\varphi^{N,\varepsilon,\rho}$ is differentiable with respect to 
$m$
and the derivative $\delta \varphi^{N,\varepsilon,\rho}/\delta m$
is continuous on $[0,T] \times {\mathcal P}({\mathbb T}^d)$. 
Moreover, for any 
$t \in [0,T]$, 
the following formula holds true, for any $m \in {\mathcal P}({\mathbb T}^d)$ and any $x \in {\mathbb T}^d$: 
\begin{equation*}
\begin{split}
&\frac{\delta \varphi^{N,\varepsilon,\rho}}{\delta m}(t,m)(x)
\\
&= (1-\varepsilon) 
\int_{{\mathbb R}^{2\vert F_N^+\vert}}
 \sum_{k \in F_N  \setminus \{0\}} \biggl[ \biggl(\partial_{\widehat{m}_k}  \varphi \Bigl(t, m^{N,\varepsilon}(r) * f_N\Bigr) e_{k} \biggr) * f_N\biggr] (x) 
\prod_{j \in F_N^+}  \rho\bigl(  \widehat{r}^j \bigr)  
\bigotimes_{j \in F_N^+} 
 d \widehat{r}^j.
\end{split}
\end{equation*} 
Lastly,
if the function $\varphi$ is also Lipschitz continuous in $t$, uniformly in $m$, then, 
for almost every $t \in [0,T]$, 
for any $m \in {\mathcal P}({\mathbb T}^d)$, 
the function $\varphi^{N,\varepsilon,\rho}$ is differentiable 
at $(t,m)$ and 
$\partial_t [\varphi^{N,\varepsilon,\rho}](t,m)=
[\partial_t \varphi(t,\cdot)]^{N,\varepsilon,\rho}(m)$.
\end{cor}

\subsection{Weak formulation of semi-concavity}
This step is crucial in our analysis. We provide a reformulation of the semi-concavity property through the Fourier coefficients. 
We start with the following reformulation of (displacement) semi-concavity (the proof may be skipped ahead on a first reading):
\begin{prop}
\label{prop:weak:semi-concavity}
If $\phi : {\mathcal P}({\mathbb T}^d) \rightarrow {\mathbb R}$ is displacement semi-convave and 
$\phi$ is Lipschitz continuous with respect to the measure 
$d_{-2}$ on 
${\mathcal P}({\mathbb T}^d)$ defined by
\begin{equation*}
d_{-2}(\mu,\nu) := \sup_{h \in {\mathcal C}^2 : \| \nabla^2 h \|_\infty 
\leq 1}
\biggl\vert \int_{{\mathbb T}^d} 
h(x) d \bigl( \mu - \nu \bigr)(x)\biggr\vert,
\end{equation*}
then, for any constant $c>1$,
there exists a constant $C$ 
depending on $c$ and on $\phi$ through its Lipschitz constant only
such that, for
 any 
probability measure $m$ with a density such that 
$m \geq 1/c$, any smooth vector field $b : {\mathbb T}^d \rightarrow {\mathbb R}^d$, 
any integer $N$ and any real $\varepsilon \in (0,1)$
such that 
$\varepsilon / (dN^2) < 1/(2c)$, and any smooth density $\rho$ with a support of radius $\delta$ less than $\delta_{N,\varepsilon}$ in 
Definition \ref{def:admiss:threshold:smoothing}, 
\begin{equation*}
\frac{d^2}{dt^2}_{\vert t=0}
\int_{{\mathbb R}^{2 \vert F_N^+\vert}} \phi\Bigl(m^{N,\varepsilon}(r) * f_N\Bigr)
\biggl(
 \prod_{j \in F_N^+}
 \rho\bigl( \widehat{r}^j + t \widehat{\textrm{\rm div} b}^j \bigr) 
\biggr)
 \bigotimes_{j \in F_N^+}
 d \widehat{r}^j  
\leq C \| b \|_\infty^2. 
\end{equation*}
\end{prop}

\begin{proof}
\textit{First Step.}
For a smooth vector field $b : {\mathbb T}^d \rightarrow {\mathbb R}^d$ and for an initial measure 
$m_0 \in {\mathcal P}({\mathbb T}^d)$ that is smooth and positive, we consider
the stochastic process
\begin{equation*}
X_t^{\pm} := X_0 \pm t \frac{b}{m_0}(X_0), \quad t \geq 0, 
\end{equation*}
with $X_0$ a random variable of law $m_0$. Then, we let  
$m_t^{\pm}:={\mathcal L}(X_t^{\pm})$, $t \geq 0$. 
For $t$ small enough, $m_0 \pm t \textrm{\rm div}_x(b)$ is a probability measure (using the fact that $m_0>0$ and $b$ is smooth). 
Since the function $\phi$ in the statement is $d_{-2}$-Lipschitz, we have
\begin{equation*}
\begin{split}
&\Bigl\vert 
\phi \bigl( m_t^+ \bigr) - \phi \Bigl( m_0 - t \textrm{\rm div}_x(b)
 \Bigr) 
\Bigr\vert
\\
&\leq C 
\sup_{{\| \nabla^2 h \|_\infty \leq 1}}
\biggl\vert
\int_{{\mathbb T}^d}
\Bigl\{ \Bigl[ 
h\Bigl(x 
+ t \frac{b}{m_0}(x)
\Bigr) 
 - h(x) 
\Bigr]
m_0(x) 
+ t h(x) 
\textrm{\rm div}_x (b) 
 (x) 
 \Bigr\}
dx 
\biggr\vert
\\
&= C 
\sup_{{\| \nabla^2 h \|_\infty \leq 1}}
\biggl\vert
\int_{{\mathbb T}^d}
\Bigl\{ \Bigl[ 
h\Bigl(x 
+ t \frac{b}{m_0}(x)
\Bigr) 
 - h(x) 
\Bigr]
m_0(x) 
- t \nabla h(x)
\cdot b(x)
 \Bigr\}
dx 
\biggr\vert
\\
&\leq C \| b \|_\infty^2 t^2,
\end{split}
\end{equation*}
where $C$ 
only depends on the Lipschitz constant of $\phi$ and on $c$ such that $m_0\geq 1/c$. 

Proceeding similarly with $m_t^-$ (and allowing the constant $C$ to vary from line to line), we 
get
\begin{equation*}
\begin{split}
\phi \Bigl( m_0 + t \text{\rm div}_x(b) \Bigr) 
+
\phi \Bigl( m_0 - t \text{\rm div}_x(b) \Bigr)
- 2 \phi ( m_0 ) 
&\leq
\phi \bigl( m_t^- \bigr) + 
\phi \bigl( m_t^+ \bigr) 
- 2 \phi(m_0)
\leq C \| b \|_\infty^2  t^2.
\end{split}
\end{equation*}
\vspace{5pt}

\textit{Second Step.} We now want to replace $m_0$ by $m^{N,\varepsilon}(r) * f_N$, for 
$N \geq 1$, $\varepsilon \in (0,1)$, 
$m$ as in the statement
and $r \in {\mathbb C}^{\vert F_N^+\vert}$ such that
$(\widehat{r}^{-k} = \overline{\widehat{r}^k})_{k \in F_N^+}$
and  
\begin{equation*}
\sum_{k \in F_N \setminus \{0\}} \bigl\vert \widehat{r}^k \bigr\vert \leq \frac{1}{2c}.
\end{equation*}
For such an $r$, we have $m^{N,\varepsilon}(r) \geq 1/(2c)$ and 
$m^{N,\varepsilon}(r)*f_N \geq 1/(2c)$. Modifying $b$ into $b*f_N$, we get
(using the same notation as in 
\eqref{eq:tilde:phi:N}
for 
$\widetilde \phi_N$, which is a real-valued 
function on 
${\mathcal O}_N$)
\begin{equation*}
\begin{split}
&\widetilde \phi_N \biggl( \Bigl\{\reallywidehat{\displaystyle m^{N,\varepsilon}(r) * f_N}^k + t \reallywidehat{\displaystyle \textrm{\rm div}( b * f_N)}^k \Bigr\}_{k \in F_N^+} \biggr)  + 
\widetilde \phi_N \biggl( \Bigl\{\reallywidehat{\displaystyle m^{N,\varepsilon}(r) * f_N}^k - t  \reallywidehat{\displaystyle \textrm{\rm div} (b * f_N)}^k \Bigr\}_{k \in F_N^+} \biggr)
\\
&\hspace{15pt}- 2 \widetilde \phi_N \biggl( \Bigl\{\reallywidehat{\displaystyle m^{N,\varepsilon}(r) * f_N}^k  \Bigr\}_{k \in F_N^+} \biggr)
\\
&\leq C \| b \|_\infty^2 t^2,
\end{split}
\end{equation*}
for a constant $C$, only depending 
on the Lipschitz constant of $\phi$ 
and on $c$ such that $m \geq 1/c$. 
Then, replacing 
$m^{N,\varepsilon}(r)$ by its definition, we get 
\begin{equation*}
\begin{split}
&\widetilde \phi_N \biggl( \Bigl\{ \Bigl(\varepsilon \delta_{k,0} + (1-\varepsilon) \bigl(\widehat m^k + \widehat{r}^k\bigr) + t \widehat{\textrm{\rm div} b}^k \Bigr)  \widehat{f}_N^k \Bigr\}_{k \in F_N^+} \biggr) 
\\
&\hspace{15pt} +
\widetilde \phi_N \biggl( \Bigl\{ \Bigl(\varepsilon \delta_{k,0} + (1-\varepsilon) \bigl(\widehat m^k + \widehat{r}^k\bigr) - t \widehat{\textrm{\rm div} b}^k \Bigr)  \widehat{f}_N^k \Bigr\}_{k \in F_N^+} \biggr)  
\\
&\hspace{15pt}
 - 2 
\widetilde \phi_N \biggl( \Bigl\{ \Bigl(\varepsilon \delta_{k,0} + (1-\varepsilon) \bigl(\widehat m^k + \widehat{r}^k\bigr)  \Bigr)  \widehat{f}_N^k \Bigr\}_{k \in F_N^+} \biggr)  
  \leq C\| b \|_\infty^2 t^2.
\end{split}
\end{equation*}
Integrating with respect to 
$\rho$ as in Definition 
\ref{def:mollification},
with the radius $\delta$ of the support of $\rho$ satisfying $\rho < 1/(2 c \vert F_N \vert)$, 
which is indeed the case if $\varepsilon / (dN^2) < 1/(2c)$, 
\begin{equation*}
\begin{split}
&\int_{{\mathbb R}^{2 \vert F_N^+\vert}}
\widetilde \phi_N \biggl( \Bigl\{ \Bigl(\varepsilon \delta_{k,0} + (1-\varepsilon) \bigl(\widehat m^k + \widehat{r}^k\bigr) + t \widehat{\textrm{\rm div} b}^k \Bigr)  \hat{f}_N^k \Bigr\}_{k \in F_N^+} \biggr) 
\prod_{j \in F_N^+} \rho\bigl(  \widehat{r}^j \bigr)
\bigotimes_{j \in F_N^+} 
  d \widehat{r}^j 
\\
&\hspace{15pt}
+\int_{{\mathbb R}^{2 \vert F_N^+\vert}} 
\widetilde \phi_N \biggl( \Bigl\{ \Bigl(\varepsilon \delta_{k,0} + (1-\varepsilon) \bigl(\widehat m^k + \widehat{r}^k\bigr) - t \widehat{\textrm{\rm div} b}^k \Bigr)  \hat{f}_N^k \Bigr\}_{k \in F_N^+} \biggr) 
\prod_{j \in F_N^+} \rho\bigl(  \widehat{r}^j \bigr)
\bigotimes_{j \in F_N^+} 
  d \widehat{r}^j 
 \\
 &\hspace{15pt}- 2\int_{{\mathbb R}^{2 \vert F_N^+\vert}} 
 \widetilde \phi_N \biggl( \Bigl\{ \Bigl(\varepsilon \delta_{k,0} + (1-\varepsilon) \bigl(\widehat m^k + \widehat{r}^k\bigr)   \Bigr)  \hat{f}_N^k \Bigr\}_{k \in F_N^+} \biggr) 
\prod_{j \in F_N^+} \rho\bigl(  \widehat{r}^j \bigr)
\bigotimes_{j \in F_N^+} 
  d \widehat{r}^j \leq C \| b \|_\infty^2 t^2,
\end{split}
\end{equation*}
which yields to (implicitly replacing $b$ by $(1-\varepsilon)b$)
\begin{equation*}
\frac{d^2}{dt^2}_{\vert t=0}
\int_{{\mathbb R}^{2 \vert F_N^+\vert}} \phi\Bigl(m^{N,\varepsilon}(r) * f_N\Bigr) \prod_{j \in F_N^+}
 \rho\bigl( \widehat{r}^j + t \widehat{\textrm{\rm div} b}^j \bigr) 
\bigotimes_{j \in F_N^+}  d\widehat{r}^j  
\leq C \| b \|_\infty^2. 
\end{equation*}
This completes the proof. 
\end{proof}
We have the following corollary:

\begin{cor}
\label{cor:weak:semi-concavity}
In the framework of 
Proposition 
\ref{prop:weak:semi-concavity},  it holds
for any collection  of complex numbers $(z^k)_{k \in F_N^+}$ and for any non-negative symmetric matrix $S$ of dimension $d$, 
\begin{equation*}
\begin{split}
&\int_{{\mathbb R}^{2 \vert F_N^+\vert}} 
\sum_{k,\ell \in F_N^+}
\phi\Bigl(m^{N,\varepsilon}(r) * f_N\Bigr)
 \sum_{i=1}^d 
\textrm{\rm Trace} \biggl\{
\biggl(k_i  
\overline{z^k} \otimes 
\bigl[ S \ell \bigr]_i 
\overline{z^\ell}
\biggr)^{\top}
\partial^2_{\widehat{r}^k,\widehat{r}^\ell} \biggl[ \prod_{j \in F_N^+}
 \rho\bigl( \widehat{r}^j  \bigr)
\biggr]\biggr\}
\bigotimes_{j \in F_N^+}
 d \widehat{r}^j  
 \\
&\leq C
\vert S \vert
\sum_{q=1}^d
\biggl(
\sum_{k \in F_N^+} \vert k_q \vert \, \vert z^k
\vert \biggr)^2,
\end{split}
\end{equation*}
with $C$ depending on $c$ such that $m \geq 1/c$. 
\end{cor}

In the above formula, 
we used the following notation
(very much in the spirit of 
\eqref{eq:gradient:complex:000}
and
\eqref{eq:gradient:complex}):
for a smooth function $\phi$, 
$\partial^2_{\widehat{r}^k,\widehat{r}^\ell}
\phi$ 
denotes 
the $2 \times 2$ matrix
\begin{equation*}
\partial^2_{\widehat{r}^k,\widehat{r}^\ell}
\phi
= \frac14 \left( 
\begin{array}{ll}
\partial^2_{\Re[\widehat{r}^k],\Re[\widehat{r}^\ell]} \phi 
&-\partial^2_{\Re[\widehat{r}^k],\Im[\widehat{r}^\ell]} \phi 
\\
-\partial^2_{\Im[\widehat{r}^k],\Re[\widehat{r}^\ell]} \phi
&\partial^2_{\Im[\widehat{r}^k],\Im[\widehat{r}^\ell]} \phi  
\end{array}
\right).
\end{equation*}
Moreover, 
$k_i  
z^k \otimes 
\ell_i 
z^\ell
$ denotes the $2 \times 2$ matrix 
\begin{equation*}
k_i  
z^k \otimes 
\ell_i 
z^\ell
= \left( 
\begin{array}{ll}
\Re[k_i z^k] 
\Re[\ell_i z^\ell]
&\Re[k_i z^k] 
\Im[\ell_i z^\ell]
\\
\Im[k_i z^k] 
\Re[\ell_i z^\ell]
&\Im[k_i z^k] 
\Im[\ell_i z^\ell]
\end{array}
\right),
\end{equation*}
and similarly for $k_i \overline{z^k} \otimes [ S \ell]_i \overline{z^\ell}$. 
In particular, 
using
$\fD$
as a generic notation for $\Re,\Im$, the main  
inequality in the statement of Corollary 
\ref{cor:weak:semi-concavity}
can be rewritten as
(with the dot product in the second line below being an inner product)
\begin{equation}
\label{eq:corollary:weak:semi-concavity:2}
\begin{split}
&\int_{{\mathbb R}^{2 \vert F_N^+\vert}} 
\biggl\{
\phi\Bigl(m^{N,\varepsilon}(r) * f_N\Bigr)
\\
&\hspace{45pt} \times 
\sum_{\fD,\widetilde{\fD} = \Re,\Im}
\sum_{k, \ell \in F_N^+}
\biggl[
\Bigl( {\fD} \bigl[ z^k  \bigr]
k \Bigr)
\cdot 
\Bigl( \widetilde{\fD} \bigl[ z^\ell\bigr]
S \ell \Bigr)
\partial_{\fD[\widehat{r}^{k}],\widetilde{\fD}[\widehat{r}^{\ell}]}^2
\Bigl( \prod_{j \in F_N^+ }
  \rho\bigl(\Re\bigl[ \widehat{r}^j\bigr],\Im\bigl[ \widehat{r}^j\bigr] \bigr) 
  \Bigr)
 \biggr]
 \biggr\}
 \\
&\hspace{30pt} \bigotimes_{j \in F_N^+}
 d \Bigl( \Re\bigl[ \widehat{r}^j\bigr],\Im \bigl[ \widehat{r}^j  \bigr]
 \Bigr) 
 \\
&\leq C
\vert S \vert
\sum_{q=1}^d
\biggl(
\sum_{k \in F_N^+} \vert k_q \vert \, \vert z^k
\vert \biggr)^2.
\end{split}
\end{equation}

\begin{proof}
Given the same non-negative symmetric matrix $S$ as in the statement, we
call $S^{1/2}$
a symmetric (non-negative) square root of $S$. 
We then
 apply 
Proposition 
\ref{prop:weak:semi-concavity}
with the following vector field 
\begin{equation*}
b^{q'}(x) = \bigl[ S^{1/2} \bigr]_{q,q'} 
\sum_{k \in F_N \setminus \{0\}} z^k e_{-k}(x), \quad 
q'= 1,\cdots,d,
\end{equation*}
with 
$z^{-k}=\overline{z^k}$, where $q$ is a frozen integer in $\{1,\cdots,d\}$
and 
$[S^{1/2}]_{q,q'}$ the element $(q,q')$ of the matrix $S^{1/2}$. 
Then, we have 
\begin{equation*}
\textrm{\rm div}(b)(x) = - \i 2 \pi \sum_{k \in F_N \setminus \{0\}} \biggl( \sum_{q'=1}^d \bigl[ 
S^{1/2} \bigr]_{q,q'} k_{q'} \biggr)  z^k e_{-k}(x) = - \i 2 \pi 
\sum_{k \in F_N \setminus \{0\}}
 \bigl( S^{1/2} k \bigr)_q z^k e_{-k}(x).
\end{equation*}
In turn, 
\begin{equation*}
\widehat{\textrm{\rm div}(b)}^k 
=
- \i 2 \pi 
 \bigl( S^{1/2} k \bigr)_q z^k.
\end{equation*}
We now apply 
Proposition 
\ref{prop:weak:semi-concavity}, 
from which we get 
\begin{equation*}
\begin{split}
&\int_{{\mathbb R}^{2 \vert F_N^+\vert}} 
\biggl\{
\phi\Bigl(m^{N,\varepsilon}(r) * f_N\Bigr)
\\
&\hspace{45pt} \times 
\sum_{\fD,\widetilde{\fD} = \Re,\Im}
\sum_{k, \ell \in F_N^+}
\biggl[
\Bigl( {\fD} \bigl[ \widehat{\textrm{\rm div}(b)}^k  \bigr]
 \Bigr)
\Bigl( \widetilde{\fD} \bigl[ \widehat{\textrm{\rm div}(b)}^\ell \bigr]
 \Bigr)
\partial_{\fD[\widehat{r}^{k}],\widetilde{\fD}[\widehat{r}^{\ell}]}^2
\Bigl( \prod_{j \in F_N^+ }
  \rho\bigl(\Re\bigl[ \widehat{r}^j\bigr],\Im\bigl[ \widehat{r}^j\bigr] \bigr) 
  \Bigr)
 \biggr]
 \biggr\}
 \\
&\hspace{30pt} \bigotimes_{j \in F_N^+}
 d \Bigl( \Re\bigl[ \widehat{r}^j\bigr],\Im \bigl[ \widehat{r}^j  \bigr]
 \Bigr) 
 \\
&\leq C
\vert S \vert
\biggl(
\sum_{k \in F_N^+} \vert k_q \vert \, \vert z^k
\vert \biggr)^2,
\end{split}
\end{equation*}
and then,
\begin{equation*}
\begin{split}
&\int_{{\mathbb R}^{2 \vert F_N^+\vert}} 
\biggl\{
\phi\Bigl(m^{N,\varepsilon}(r) * f_N\Bigr)
\\
&\hspace{45pt} \times 
\sum_{\fD,\widetilde{\fD} = \Re,\Im}
\sum_{k, \ell \in F_N^+}
\ {\fD} \bigl[ \i z^k\bigr]
  \widetilde{\fD} \bigl[ \i z^\ell \bigr]
 \Bigl( S^{1/2}  k \Bigr)_q
 \Bigl( S^{1/2}  \ell \Bigr)_q 
\partial_{\fD[\widehat{r}^{k}],\widetilde{\fD}[\widehat{r}^{\ell}]}^2
\Bigl( \prod_{j \in F_N^+ }
  \rho\bigl(\Re\bigl[ \widehat{r}^j\bigr],\Im\bigl[ \widehat{r}^j\bigr] \bigr) 
  \Bigr)
 \biggr]
 \biggr\}
 \\
&\hspace{30pt} \bigotimes_{j \in F_N^+}
 d \Bigl( \Re\bigl[ \widehat{r}^j\bigr],\Im \bigl[ \widehat{r}^j  \bigr]
 \Bigr) 
 \\
&\leq C
\vert S \vert
\biggl(
\sum_{k \in F_N^+} \vert k_q \vert \, \vert z^k
\vert \biggr)^2.
\end{split}
\end{equation*}
By summing over $q \in \{1,\cdots,d\}$ and by changing 
$(z^k)_{k \in F_N^+}$ into 
$(-\i z^k)_{k \in F_N^+}$, 
we easily complete the proof.
\end{proof}

\section{Generalized solutions to the HJB equation}
\label{sec:4}

The purpose of this section is to clarify and to prove Meta-Theorem 
\ref{meta:1}.
.

\subsection{Interpretation of the spatial derivatives}
\label{subse:spatial:deri:interpretation}

We start with an informal discussion to explain our ideas.  

Basically, 
the purpose is to reformulate the equation at points 
$m \in {\mathcal P}_N$ at which 
the restriction at level $N$ (see 
\eqref{eq:tilde:phi:N})
of 
a candidate $W$ for solving 
\eqref{eq:HJB} is differentiable.  Here we use $W$ (and not $V$) as a notation
for a generic candidate 
 since, at this stage, 
$W$ may not be the value function of the 
MFCP
\eqref{eq:MKV:V}. 
To ease the exposition, we do throughout this informal subsection 
as if 
$W$ were smooth (even though it is obviously not true). 
\vspace{5pt}

\paragraph{\it Crossed derivatives.}
We here address the term
\begin{equation*}
T[m,W] := \frac12 \int_{{\mathbb T}^d} 
\text{Tr} \Bigl[ 
\partial_y \partial_\mu W(t,m)(y) \Bigr] d m(y)
\end{equation*}
in 
\eqref{eq:HJB}. 
If $m \in {\mathcal P}_N$, then
we can identify the measure and its density and write 
(using an obvious integration by parts)
\begin{equation*}
\begin{split}
T[m,W]& = \frac12 
\int_{{\mathbb T}^d} \Bigl[ 
\frac{\delta}{\delta m} W(t,m)(y) \Bigr] \Delta m(y) dy
=
 - \sum_{k \in F_N} 2 \pi^2 \vert k \vert^2 
 \widehat{m}^k 
\reallywidehat{\displaystyle \frac{\delta}{\delta m} W(t,m) }^{-k}.
\end{split}
\end{equation*}
Recall
from Proposition 
\ref{prop:3:7}
 the formula
\begin{equation*}
\partial_{\widehat m^k} W(t,m) = \reallywidehat{\displaystyle \frac{\delta W}{\delta m} (t,m)}^{-k},
\quad k \in F_N^+,
\end{equation*}
from which we get 
\begin{equation*}
\begin{split}
T[m,W]& = \frac12 
\int_{{\mathbb T}^d} \Bigl[ 
\frac{\delta}{\delta m} W(t,m)(y) \Bigr] \Delta m(y) dy
= 
 - \sum_{k \in F_N} 2 \pi^2 \vert k \vert^2 
\partial_{ \widehat{m}^{k} }
W(t,m) \widehat{m}^{k}.
\end{split}
\end{equation*}
We feel useful to stress the interest of the above formula: 
the quantity
$T[m,W]$ can be reformulated in terms of the sole derivatives of 
$W$ with respect to $\widehat{m}^k$, for $k \in F_N$. 
This property may be recast in another way:
the differential operator 
that maps 
a real-valued smooth function $\phi$
on ${\mathcal P}({\mathbb T}^d)$ onto the function 
\begin{equation*}
m \in {\mathcal P}({\mathbb T}^d) \mapsto \frac12 
\int_{{\mathbb T}^d}
\textrm{\rm Trace}
\Bigl[ \partial_y \partial_\mu \phi(m)(y) 
\Bigr] d m(y)
\end{equation*}
has the heat equation as characteristics, that is 
\begin{equation*}
\frac{d}{dt} \phi (m_t) = 
\frac12 
\int_{{\mathbb T}^d}
\textrm{\rm Trace}
\Bigl[ \partial_y \partial_\mu \phi\bigl(m_t\bigr)(y) 
\Bigr] d m_t(y),\quad 
\textrm{\rm for} \quad
\partial_t m_t(y) - \frac12 \Delta_y m_t(y) =0,
\end{equation*}
where 
$t \geq 0$ and $m_0 \in {\mathcal P}({\mathbb T}^d)$. 
Since the heat equation
 keeps the space 
${\mathcal P}_N$ invariant, this strongly advocates for expanding probability measures in Fourier series.
\\
\paragraph{\bf Quadratic term.}
We now address the term
\begin{equation*}
S[m,W] := 
 \int_{{\mathbb T}^d} 
H \bigl( y,  \partial_\mu W(t,m) (y) \bigr) d m(y)
\end{equation*}
in \eqref{eq:HJB}. 
The big difficulty with this term is that it contains all the Fourier modes of $\delta W/\delta m$:
differently from $T[m,W]$, 
$S[m,W]$ cannot be expressed 
in terms of the sole 
derivatives of 
$W$ with respect to $\widehat{m}^k$, for $k \in F_N$. This is one of the main difficulty in our 
analysis. 
Anyway, it makes perfect sense in our context to assume that, whenever it exists, $\partial_\mu W(t,m)(y)$ is H\"older continuous in $y$: 
the reason is that the function $\partial_\mu W$ is expected to benefit from the presence of the heat operator in the direction $y$ (this point is clarified in Proposition \ref{prop:4:7} below).  In turn, we can assume that the $d$ components 
of $y \mapsto 
\partial_\mu W(t,m)(y)$ are in a compact set of $L^2({\mathbb T}^d)$. We then recall that, for any compact subset 
${\mathcal K}$ of 
$L^2({\mathbb T}^d)$, 
\begin{equation}
\label{eq:compact:Fourier}
\lim_{N \rightarrow \infty} 
\sup_{f \in {\mathcal K}}
\sum_{k \not \in F_N}
| \widehat{f}^k |^2 = 0.
\end{equation}
In that case, we can write 
\begin{equation*}
\begin{split}
S[m,W] =  \int_{{\mathbb T}^d} H\biggl(y, - \i 2\pi \biggl( \sum_{k \in F_N} k \reallywidehat{\displaystyle \frac{\delta W}{\delta m}(t,m)}^k e_{-k}(y) +
  \sum_{k \not \in F_N} k \reallywidehat{\displaystyle \frac{\delta W}{\delta m}(t,m)}^k e_{-k}(y)
  \biggr)
\biggr) d m(y).
\end{split}
\end{equation*}
Therefore, by expanding the Hamiltonian, we can reasonably postulate
\begin{equation*}
S[m,W] = \int_{{\mathbb T}^d} H
\biggl(y, - \i 2 \pi \sum_{k \in F_N} k \reallywidehat{\displaystyle \frac{\delta W}{\delta m}(t,m)}^k e_{-k}(y) \biggr) dm(y) + 
O(\eta_N),
\end{equation*}
where $\eta_N$ tends to $0$ as $N \rightarrow \infty$, uniformly on measures $m$ whose density is bounded by the same constant. 
Finally, 
by 
\eqref{eq:gradient:complex:000},
\begin{equation*}
S[m,W] =  \int_{{\mathbb T}^d} H \biggl( y, \i 2 \pi  \sum_{k \in F_N} k \partial_{\widehat m^k} W(t,m) e_{k}(y) \biggr) dm(y) + 
O(\eta_N),
\end{equation*}
Beware that this is not a proof. This is just a hint that motivates the definition right below. 
\subsection{Definition of a generalized solution}
\begin{defn}
\label{defn:HJB:gen}
For a constant $c>1$, 
let
$A_N(c):=[0,T] \times B_N(c)$ with $B_N(c):= {\mathcal P}_N \cap B(c)$ and $B(c):= \{ m :  \sup_{x \in \bT^d} \vert \nabla m(x)  \vert \leq c, \ \inf_{x \in \bT^d}  m(x)   \geq 1/c \}$.
Then, 
we call a generalized solution to the HJB equation 
\eqref{eq:HJB} a function $W : [0,T] \times {\mathcal P}({\mathbb T}^d) \rightarrow 
{\mathbb R}$ that is time-space Lipschitz continuous when 
${\mathcal P}({\mathbb T}^d)$ is equipped with 
$d_{\rm TV}$
and that satisfies the following property: For any $c >1$,
there exists a sequence $(\eta_N(c))_{N \geq 1}$ converging to $0$, 
such that, for any $N \geq 1$, 
$\textrm{\rm Leb}_1 \otimes {\mathbb P}_N$ almost everywhere on $A_N(c)$, the following bound holds true: 
\begin{equation}
\begin{split}
\biggl\vert \partial_t W(t,m) &- 
 \int_{{\mathbb T}^d}
 H \biggl( y, \i 2 \pi  \sum_{k \in F_N} k \partial_{\widehat m^k} W(t,m) e_{k}(y) \biggr)  d m(y)
 \\
&\hspace{15pt} - \sum_{k \in F_N} 2 \pi^2 \vert k\vert^2 
\partial_{ \widehat{m}^{k} }
W(t,m) \widehat{m}^{k}  + F(m)\biggr\vert 
\leq \eta_N(c).
\label{eq:HJB:gen}
\end{split}
\end{equation}
\end{defn}
Of course, in this definition, the
`$\textrm{\rm Leb}_1 \otimes {\mathbb P}_N$ almost everywhere' should be 
equivalently understood as `almost everywhere' for the product of the Lebesgue measure on $[0,T]$ and of 
the image of the Lebesgue measure on ${\mathcal O}_N$ by 
the 
canonical embedding ${\mathscr I}_N$ in 
\eqref{eq:I_N}. 
In this regard, the need for $W$ to be Lipschitz continuous is clear: this is a way to guarantee the existence of the
derivatives, almost everywhere (see Proposition \ref{prop:3:8}).
In fact, the interest of this formulation is twofold: Not only does it permit to use 
directly the standard finite-dimensional Rademacher's theorem 
(even though 
Proposition 
\ref{prop:rademacher}
provides 
an infinite-dimensional version in our context), but it also 
allows us to implement next the finite-dimensional regularization procedure
introduced in Definition 
\ref{def:mollification}.
The latter is a key ingredient in the 
proof of uniqueness, see in particular
the forthcoming  Proposition 
\ref{prop:generalized:solution}.

Very interestingly, by means of 
Theorem 
\ref{thm:probability:probability}, the 
`almost-everywhere' that is here understood for each $N \geq 1$ can be `factorized' into a single 
`almost-everywhere'
(below, 
${\mathbb P}$ denotes the same probability measure as in 
the statement of 
Theorem 
\ref{thm:probability:probability}):
\begin{prop}
Assume that $W$ is a generalized solution 
to \eqref{eq:HJB}. Then,
for any $\varepsilon \in (0,1)$, there exist
a Borel subset $D_{\varepsilon}$ of 
$[0,T] \times {\mathcal P}({\mathbb T}^d)$, 
such that 
$[\textrm{\rm Leb}_1 \otimes {\mathbb P}](D_{\varepsilon}^{\complement}) \leq \varepsilon$,
and
 a sequence $(\eta_N(\varepsilon))_{N \geq 1}$ converging to $0$ such that, for any 
 $(t,m) \in D_{\varepsilon}$ and any $N \geq 1$, 
\begin{equation}
\label{eq:HJB:gen:m*f}
\begin{split}
&\biggl\vert \partial_t W(t,m*f_N) - 
  \int_{{\mathbb T}^d} H\biggl( y,   \i 2 \pi \sum_{k \in F_N} k \partial_{\widehat m^k} W(t,m * f_N) e_{k}(y) \biggr) d(m*f_N)(y)
 \\
 &\hspace{15pt} - \sum_{k \in F_N} 2 \pi^2 \vert k\vert^2 
\partial_{ \widehat{m}^{k} }
W(t,m*f_N) \widehat{m}^{k} \widehat{f}_N^{k} + F(m * f_N)\biggr\vert \leq \eta_N(\varepsilon).
\end{split}
\end{equation}
\end{prop}
In other words, the combination of Theorem 
\ref{thm:probability:probability}
and of the mollification procedure 
reported in \S \ref{subsubse:BH}
makes it possible to have 
a definition that holds at `almost every' point of the state space. 
Of course, what makes this definition interesting in this regard is that the measure ${\mathbb P}$ has a full support.

\begin{proof}
We call $E$ the set of points $m \in {\mathcal P}({\mathbb T}^d)$ such that 
$m$ has a (strictly) positive continuously differentiable density. 
Theorem \ref{thm:probability:probability}
says that ${\mathbb P}(E) = 1$. 
For any $m \in E$, we have $\inf_{N \geq 1} \inf_{x \in {\mathbb T}^d} 
m*f_N(x) >0$ and 
$\sup_{N \geq 1} \sup_{x \in {\mathbb T}^d} 
\vert \nabla m * f_N(x) \vert < \infty$. In particular, 
\begin{equation*}
\lim_{c \rightarrow \infty} 
{\mathbb P} 
\biggl( 
\Bigl\{ 
\inf_{N \geq 1} \inf_{x \in {\mathbb T}^d} 
m*f_N(x) \geq 1/c\Bigr\}
\cap
\Bigl\{
\sup_{N \geq 1} \sup_{x \in {\mathbb T}^d} 
\vert \nabla m * f_N(x) \vert \leq c
\Bigr\}
\biggr)
= 1. 
\end{equation*}
Hence, for any $\varepsilon \in (0,1)$,
we can find $c_{\varepsilon}>1$ such that 
${\mathbb P}(E_{\varepsilon}) \geq 1- \varepsilon$, 
with 
\begin{equation*}
E_{\varepsilon}
:= 
\Bigl\{ 
\inf_{N \geq 1} \inf_{x \in {\mathbb T}^d} 
m*f_N(x) \geq 1/c_{\varepsilon}\Bigr\}
\cap
\Bigl\{
\sup_{N \geq 1} \sup_{x \in {\mathbb T}^d} 
\vert \nabla m * f_N(x) \vert \leq c_{\varepsilon}
\Bigr\}.
\end{equation*}

Now, for any $c \geq 1$ and any $N \geq 1$, we call $\widetilde{A}_N(c)$ the collection of points $(t,m) \in A_N(c)$ 
(see Definition 
\ref{defn:HJB:gen} for the notation)
at which 
\eqref{eq:HJB:gen}
does not 
hold. Clearly $\textrm{\rm Leb}_1 \otimes 
{\mathbb P}_N(\widetilde{A}_N(c))=0$. 
We then make use of item 
$(ii)$ in the statement of Theorem 
\ref{thm:probability:probability}. 
By combining with an obvious adaptation of Lemma 
\ref{lem:23}, we deduce that there exists a full subset $D$ of 
$[0,T] \times {\mathcal P}({\mathbb T}^d)$ such that, 
for any $(t,m) \in D$
and for any $N \geq 1$, 
the point $(t,\pi_N^{(1)}(m))$ belongs to $\cap_{c \in {\mathbb N} \setminus \{0\}} 
(\widetilde{A}_N(c))^{\complement}$. 
We thus let
\begin{equation*}
D_{\varepsilon} := \bigl( [0,T] \times E_{\varepsilon} \bigr) \cap D. 
\end{equation*}
Clearly, 
\begin{equation*}
[\textrm{\rm Leb}_1 \otimes {\mathbb P}](D_{\varepsilon}^{\complement}) \leq T \varepsilon.
\end{equation*}
If $(t,m) \in D_{\varepsilon}$,
then
$(t,m*f_N)   \in A_N(c_{\varepsilon})
\cap (\widetilde{A}_N(c_{\varepsilon}))^{\complement}$.
In particular, 
\eqref{eq:HJB:gen}
holds true with respect to 
the sequence $(\eta_N(c_\varepsilon))_{N \geq 1}$, which we simply denote
by 
$(\eta_N(\varepsilon))_{N \geq 1}$ in the statement. 
\end{proof}

\begin{rem}
\label{rem:formulation:P:as}
A natural question is to determine whether we can pass to the limit 
over $N$ in 
\eqref{eq:HJB:gen:m*f} and then get that 
\eqref{eq:HJB} 
is satisfied 
$\textrm{\rm Leb}_1 \otimes 
{\mathbb P}$
almost everywhere.

The answer is twofold. On the one hand, it seems that we cannot do so \emph{a priori}. The reason is that 
the passage from the
 derivative
 $\partial_{\widehat{m}^k} W(t,m*f_N)$
 to the 
 derivative
 $\partial_{\widehat{m}^k} W(t,m)$
 (with the existence of the latter following from 
 our own version 
of Rademacher's theorem, 
see Theorem \ref{prop:rademacher}) 
is just known to hold in the weak sense, see again 
the statement of Theorem
\ref{prop:rademacher}. In particular, it seems impossible to 
address (at least in the current framework) the limit of the gradients inside the Hamiltonian $H$ driving the HJB equation. 
On the other hand, it seems that the same passage to the limit can be in fact achieved \emph{a posteriori}, once the 
generalized solutions are known to coincide with the value function 
$V$, which identification 
follows from Theorem 
\ref{thm:uniqueness:HJB}
if we restrict ourselves to 
generalized solutions with sufficient regularity properties. 
The proof would follow from a combination of 
Theorem 
\ref{thm:value:is:gen:HJB},
Lemma 
\ref{lem:superjet}
and Remark 
\ref{rem:4.5}
right below. In short, 
the former statement 
says that,
for 
$\textrm{\rm Leb}_1 \otimes {\mathbb P}$ almost every initial condition, 
the MFCP has a unique 
solution. 
In turn, 
we conjecture that, for
$\textrm{\rm Leb}_1 \otimes {\mathbb P}$ almost every 
$(t,m) \in [0,T] \times \PP$, any sequence of optimal trajectories issued from $((t,m*f_N))_{N \geq 1}$
should
converge in a suitable sense 
to the optimal trajectory issued from 
$(t,m)$. 
However, since optimal trajectories 
satisfy the MFG system 
\eqref{eq:MFG:system}, this convergence property
can be certainly reformulated as a convergence 
property on any sequence of solutions to 
\eqref{eq:MFG:system}
that start from $((t,m*f_N))_{N \geq 1}$
and that minimize 
${\mathcal J}_{\emph{det}}$. It then remains to 
observe that 
Lemma 
\ref{lem:superjet}
and Remark 
\ref{rem:4.5}
below 
permit to identify 
 $\partial_{\widehat{m}^k} V(t,m*f_N)$
and
 $\partial_{\widehat{m}^k} V(t,m)$
with solutions 
of the backward equation 
in the MFG system
\eqref{eq:MFG:system}
when initialized 
from $(t,m*f_N)$ and $(t,m)$ respectively. 
Combined with the 
convergence property of the 
MFG system, this should imply in the end that 
 $\partial_{\widehat{m}^k} V(t,m*f_N)$
converges to 
 $\partial_{\widehat{m}^k} V(t,m)$ 
 $\textrm{\rm Leb}_1 \otimes {\mathbb P}$ almost everywhere, which is almost what we need to pass to the limit in the Hamiltonian.  
 Indeed, 
 once the convergence of 
 the derivatives with respect to any 
 $\widehat{m}^k$ has been shown, the 
 convergence of the Fourier series encoding the 
 derivatives with respect to $m$ is quite easy to address, see for instance 
 Proposition \ref{prop:4:7} which provides strong bounds on the decay of  
 $\partial_{\widehat{m}^k} V(t,m)$.

Of course, this result should be formalized in the form of 
a more rigorous statement. However, 
we feel useless to address the details here since 
 the real benefit for our analysis 
would be small in the end. Indeed, the result would come only in the end, once uniqueness of the generalized solution has been proven. 
Moreover, our feeling is precisely that (at least at this stage of our understanding) it is in fact much easier to work with 
finite dimensional 
formulations of the HJB equation. 
 \end{rem}

\subsection{The value function as a generalized solution}

We have the following main statement:

\begin{thm}
\label{thm:value:is:gen:HJB}
The value function $V$, as defined in \eqref{eq:MKV:V}, is a generalized solution of the HJB equation. It is Lipschitz continuous in time and space when the space ${\mathcal P}({\mathbb T}^d)$ is equipped with respect to $d_{-2}$ (see the definition in Proposition 
\ref{prop:weak:semi-concavity}) and it is semi-concave (see Proposition \ref{prop:semi-concavity}). 


 {Moreover, 
if, for some $(t,m) \in [0,T] \times \PP$,
$V(t,\cdot)$ has directional derivatives at $m$ along $\Re[e_k]$ and $\Im[e_k]$, for any $k \in {\mathbb Z}^d$, then 
 	the 
 	MFCP 
 	introduced in Subsection 
 	\ref{subse:mfc:presentation}
 	has a unique optimizer when the trajectories
 	\eqref{eq:control:FKP}
 	are initialized from $m$ at time $t$. In particular, uniqueness holds 
 for $\textrm{\rm Leb}_1 \otimes {\mathbb P}$-almost every point $(t,m) \in [0,T] \times 
 {\mathcal P}({\mathbb T}^d)$, 
 with
 ${\mathbb P}$ denoting the same probability measure as in the statement of   
 Theorem 
 \ref{thm:probability:probability}.}

\end{thm}

\begin{rem}
\label{rem:counter-ex}
The reader may compare the almost everywhere uniqueness result stated right above 
with the counter-example 
to uniqueness 
provided in 
\cite[Subsection 3.2]{BrianiCardaliaguet}. 
In the latter reference, non-unique solutions are constructed at probability 
measures $m$ whose density is symmetric with respect to 
the first coordinate 
$x_1$ of $x \in {\mathbb T}^d$.
In turn, this says that 
those measures satisfy  
$\Im[\widehat{m}^k]=0$ when $k=(1,0,\cdots)$. 
Obviously, the collection of such probability measures has zero mass under 
${\mathbb P}_N$, when $N \geq 2$, and also under ${\mathbb P}$ by 
$(iii)$ in the statement of Theorem 
\ref{thm:probability:probability}. 
\end{rem}

In preparation for the proof, we state
the following three lemmas (the proofs of which are given in the next subsection):
\begin{lem}
\label{lem:superjet}
{For $N \geq 1$, 
take a point $m \in {\mathcal P}_N$} such that the restriction $\widetilde V_N(0,\cdot)$ of $V(0,\cdot)$ to ${\mathcal O}_N$ is differentiable
at $(\widehat{m}^k)_{k \in F_N^+}$. 
Then, 
whatever the solution $(m_t,u_t)_{0 \leq t  \leq T} $ to the MFG system 
\eqref{eq:MFG:system}
(see Proposition \ref{thm:solvability:MFG}, with $m_0=m$) that minimizes 
${\mathcal J}_{\textrm{\rm det}}$ (see Proposition \ref{prop:Briani}),
it holds that 
\begin{equation*}
\partial_{\widehat{m}^k} V(0,m) = \widehat {u}_0^{-k}, \quad k \in F_N \setminus \{0\}. 
\end{equation*}
\end{lem}

{\begin{rem}
\label{rem:4.5}
In fact, a similar statement holds when 
$V(0,\cdot)$ has a directional derivative along $\Re[e_k]$ and $\Im[e_k]$, for any $k \in {\mathbb Z}^d \setminus \{0\}$, at a point $m \in {\mathcal P}({\mathbb T}^d)$. The same proof shows that, in that case, 
\begin{equation*}
\partial_{\widehat{m}^k} V(0,m) = \widehat {u}_0^{-k}, \quad k \in {\mathbb Z}^d \setminus \{0\}. 
\end{equation*}
We will use this extension 
in order to prove uniqueness of the optimal trajectories for almost every starting point, as described in the second part of 
the statement of Theorem 
\ref{thm:value:is:gen:HJB}.
\end{rem}}

\begin{lem}
\label{lem:small:time:expansion}
For any constant $c>0$, 
there exist two positive-valued sequences 
$(\epsilon_N)_{N \geq 1}$
and
$(\eta_N)_{N \geq 1}$, with 
$(\eta_N)_{N \geq 1}$ converging to $0$ and with both satisfying the following property. 
If, 
for an initial condition 
$m_0$ such that 
$\| \nabla m_0 \|_\infty \leq c$ and 
$\widehat{m}_0^k =0$ when $k \not \in F_N$,  
for a
control 
$\alpha : [0,T] \times {\mathcal P}({\mathbb T}^d)
\rightarrow {\mathbb R}^d$
with $\sup_{t \in [0,T]} \| \nabla \alpha_t \|_\infty \leq c$, 
and for 
$(m_t)_{0 \leq t \leq T}$ the $\alpha$-controlled trajectory defined in 
\eqref{eq:control:FKP}, we let 
$(\mu_t)_{0 \leq t \leq T}$ be the 
function-valued trajectory defined by 
\begin{equation}
\label{eq:truncated:trajectories}
\widehat{\mu}^k_t: =
\left\{
\begin{array}{ll}
\widehat{m}^k_t \quad &\text{if} \quad k \in F_N,
\\
0 \quad &\text{otherwise},
 \end{array}
 \right.
\end{equation}
then, 
for any 
$t \in [0,\epsilon_N]$, 
$\mu_t \in {\mathcal P}({\mathbb T}^d)$ 
 and 
$d_{W_1}(m_t,\mu_t) \leq   \eta_N t$.
\end{lem}

\begin{lem}
\label{le:d-2:V}
The value function is time and space Lipschitz continuous, when the space of probability measures is equipped with $d_{-2}$ (see Proposition 
\ref{prop:weak:semi-concavity} for the definition of the latter). 
\end{lem}

For the time being, 
we switch to the proof of our main theorem. 

\begin{proof}[Proof of Theorem \ref{thm:value:is:gen:HJB}]

By Proposition 
\ref{prop:semi-concavity}, the value function is  Lipschitz in space, when 
${\mathcal P}({\mathbb T}^d)$ is equipped with $d_{W^1}$. 
In fact, 
Lipschitz property 
in
time and space, with ${\mathcal P}({\mathbb T}^d)$ being equipped with 
$d_{-2}$, is a consequence of Lemma 
 \ref{le:d-2:V}, whose proof is independent. Obviously, 
 Lemma 
 \ref{le:d-2:V}
 subsumes Proposition 
\ref{prop:semi-concavity}. 
\vspace{5pt}

\textit{First step.}
We consider a probability measure $m$ at which the restriction of $V$ 
to $[0,T] \times {\mathcal P}_N$
is time-space differentiable at $(0,m*f_N)$ (i.e., using the notation of 
\eqref{eq:tilde:phi:N}, the function 
$\widetilde V_N$ is time-space differentiable at 
$(0,(\widehat m^k \widehat f^k_N)_{k \in F_N^+})$. 
 
For simplicity, we let $m_0 = m * f_N$. 
We fix $c>0$ 
such that 
$\| \nabla m_0 \|_\infty \leq c$
and we choose 
a
control 
$\alpha : [0,T] \times {\mathcal P}({\mathbb T}^d) \rightarrow {\mathbb R}^d$
with $\sup_{t \in [0,T]} \| \nabla \alpha_t \|_\infty \leq c$. 
Then, 
for 
$(m_t)_{0 \leq t \leq T}$ the corresponding $\alpha$-controlled trajectory
and  
for
$(\mu_t)_{0 \leq t \leq T}$ being defined as in the statement of Lemma 
\ref{lem:small:time:expansion}, 
$\mu_t$ is a probability measure for $t \in [0,\epsilon_N]$ and
\begin{equation*}
 d_{W_1}(m_t,\mu_t) \leq   \eta_N t, \quad t \in [0,\epsilon_N], 
\end{equation*}
where, throughout the proof, 
$(\epsilon_N)_{N \geq 1}$
and 
$(\eta_N)_{N \geq 1}$ denote generic positive-valued sequences that tend to $0$ and whose choices only depend on the value of $c$. 
Since $V$ is time-space Lipschitz continuous (when 
${\mathcal P}({\mathbb T}^d)$ is equipped with $d_{W_1}$), we obtain
\begin{equation}
\label{eq:proof:thm:4:3:1}
\vert V(t,m_t) - V(t,\mu_t) \vert \leq C \eta_N t. 
\end{equation}
Using the fact that the restriction of $V$ is differentiable at $(0,m_0)$ together with the fact that  
$(\mu_t)_{0 \leq t \leq \epsilon_N}$ takes values in ${\mathcal P}_N$, we get, 
for $t \in [0,\epsilon_N]$, 
 \begin{equation}
 \label{eq:truncated:trajectories:application}
 \begin{split}
 V(t,\mu_t) &= V(0,m_0) +
 t \partial_t V(0,m_0)
 + \sum_{k \in F_N} \partial_{\widehat{m}^k} V(0,m_0)     \bigl( \widehat{\mu}_t^k - \widehat{m}_0^k \bigr) + o_N(t),
 \\
 &=  V(0,m_0) +
 t \partial_t V(0,m_0)
 +  \sum_{k \in F_N} \partial_{\widehat{m}^k} V(0,m_0) \bigl( \widehat{m}_t^k - \widehat{m}_0^k \bigr) + o_N(t),
 \end{split}
\end{equation}
where we used 
\eqref{eq:truncated:trajectories}
to pass from the first to the second line. 
In the above notation, the bound in the Landau symbol $o_N(\cdot)$ may depend on $N$ (as the differentiation is performed in finite-dimension). 
Next, we notice that, for $k \in F_N$, 
\begin{equation}
\label{eq:proof:thm:4:3:4}
\begin{split}
\frac{d}{dt} \widehat{m}_t^k &= -2 \pi^2 \vert k \vert^2 \widehat{m}_t^k  + \i 2 \pi  \int_{{\mathbb T}^d} k \cdot \alpha_t(x) e_{k}(x) dm_t(x)
\\
&=  -2 \pi^2 \vert k \vert^2 \widehat{m}_0^k  +  \i 2 \pi  \int_{{\mathbb T}^d}k \cdot \alpha_t(x) e_{k}(x) m_0(x) dx + O_N(t^{1/2}),
\end{split}
\end{equation}
where we used in the second line the obvious fact that $\| m_t - m_0 \|_2 \leq C \sqrt{t}$. The latter follows from the reformulation of the Fokker-Planck equation
\eqref{eq:control:FKP} 
in a non-divergence form, namely
\begin{equation*}
\partial_t m_t(x) + \alpha_t(x) \cdot \nabla_x m_t(x) - \frac12 \Delta_x m_t(x) + \textrm{\rm div}_x\bigl( \alpha_t\bigr)(x) m_t(x) = 0,
\quad t \in [0,T], \ x \in {\mathbb T}^d. 
\end{equation*}
Since 
$\sup_{t \in [0,T]} \| \nabla \alpha_t \|_\infty \leq c$ and 
$\| \nabla m_0 \|_\infty \leq c$, standard Schauder's estimates guarantee that 
the solution to the above equation is Lipschitz continuous in space and $1/2$-H\"older continuous in time, which proves the announced estimate for
$\|m_t-m_0\|_2$.  

Therefore, 
inserting 
\eqref{eq:proof:thm:4:3:4}
into \eqref{eq:truncated:trajectories:application}, we obtain 
\begin{equation*}
\begin{split}
 V(t,\mu_t) &= V(0,m_0) +
 t \partial_t V(0,m_0)
 - t 2 \pi^2 \sum_{k \in F_N} \partial_{\widehat{m}^k} V(0,m_0) \vert k\vert^2  \widehat{m}^k_0
 \\
&\hspace{15pt}
+
\i 2 \pi \int_0^t  \int_{{\mathbb T}^d} \Bigl( \sum_{k \in F_N}  k \partial_{\widehat{m}^k} V(0,m_0)  e_{k}(x) \Bigr) \cdot \alpha_s(x) m_0(x) dx
ds
 + o_N(t). 
\end{split}
\end{equation*}
Let now 
\begin{equation}
\label{eq:proof:thm:4:3:3}
v^N(x) :=   \i 2 \pi \sum_{k \in F_N}  k \partial_{\widehat{m}^k} V(0,m_0)  e_{k}(x) \in {\mathbb R}^d.
\end{equation}
Then, combining with 
\eqref{eq:proof:thm:4:3:1}, we obtain, for $t \in [0,\epsilon_N]$, 
\begin{equation}
\label{eq:proof:thm:4:3:3:b:b:b}
\begin{split}
&\biggl\vert  V(t,m_t) - \biggl( V(0,m_0) +
 t \partial_t V(0,m_0)
 - t 2 \pi^2 \sum_{k \in F_N} \partial_{\widehat{m}^k} V(0,m_0) \vert k\vert^2  \widehat{m}^k_0
 \\
&\hspace{15pt}
+
 \int_0^t \int_{{\mathbb T}^d}  v^N(x) \cdot \alpha_s(x) m_0(x) dx  ds \biggr) \biggr\vert \leq 
 C \eta_N t + o_N(t).
\end{split}
\end{equation}
Importantly, we recall from 
Lemma 
\ref{lem:superjet}
that $v^N(x)= \i 2 \pi \sum_{k \in F_N} k \widehat u_0^{-k} e_k(x)$, where $u_0$ is the initial value of the backward component of any solution to the MFG system \eqref{eq:MFG:system} that minimizes 
${\mathcal J}_{\textrm{\rm det}}$ (with $m_0$ as initial condition). 
Proposition \ref{thm:solvability:MFG} says that there exists a constant $C$, independent of $N$, such that $\| v^N \|_2 \leq C$. 
As a consequence, if we let 
\begin{equation}
\label{eq:proof:thm:4:3:2}
\alpha^N_t(x) := \sum_{k \in F_N} \widehat{\alpha}_t^k e_{-k}(x),
\end{equation}
then we can easily replace $\alpha_t$ by $\alpha_t^N$ in 
\eqref{eq:proof:thm:4:3:3:b:b:b}.
Indeed, from the assumption 
$\| \nabla_x \alpha_t \|_\infty \leq c$, 
we deduce (see for instance 
\eqref{eq:compact:Fourier}) that the sequence 
$(\vert \widehat{\alpha}_t^k \vert)_{k \in {\mathbb Z}^d}$ is square-integrable, uniformly in 
$\alpha_t$ satisfying 
$\| \nabla_x \alpha_t \|_\infty \leq c$.
\vskip 5pt

\textit{Second Step.}
Using the same notation as in the first step, 
we now consider the cost 
\begin{equation*}
{\mathcal J}^{0,t}_{\textrm{\rm det}}(\alpha) := \int_0^t \biggl( F(m_s) +   \int_{{\mathbb T}^d} L\bigl(x,\alpha_s(x) \bigr) m_s(dx) 
\biggr) ds + V(t,m_t).
\end{equation*}
By 
recalling 
\eqref{eq:proof:thm:4:3:3:b:b:b} together 
with the fact that 
$\| m_t - m_0 \|_2 \leq C t^{1/2}$ and by using the 
quadratic growth of $L$,  
we obtain 
\begin{equation*}
\begin{split}
&\biggl\vert  {\mathcal J}_{\textrm{\rm det}}^{0,t}(\alpha) - \biggl( V(0,m_0) + t F(m_0)
+
 t \partial_t V(0,m_0)
 - t 2 \pi^2 \sum_{k \in F_N} \partial_{\widehat{m}^k} V(0,m_0) \vert k\vert^2  \widehat{m}^k_0
 \\
&\hspace{15pt}
+
 \int_0^t
 \biggl[ \int_{{\mathbb T}^d}  v^N(x) \cdot \alpha^N_s(x) m_0(x) dx    
+
 \int_{{\mathbb T}^d} L\bigl(x,  \alpha_s(x) \bigr) m_0(x) dx 
 \biggr] ds
\biggr) \biggr\vert \leq 
 C \eta_N t + o_N(t).
\end{split}
\end{equation*}
We now use the local Lipschitz property of $L$ with respect to the control parameter.
And, once again, we 
use the fact that 
the sequence 
$(\vert \widehat{\alpha}_t^k \vert)_{k \in {\mathbb Z}^d}$ is square-integrable, uniformly in 
$\alpha_t$ satisfying 
$\| \nabla_x \alpha_t \|_\infty \leq c$. We get
\begin{equation*}
\begin{split}
&\biggl\vert  {\mathcal J}_{\textrm{\rm det}}^{0,t}(\alpha) - \biggl( V(0,m_0) + t F(m_0)
+
 t \partial_t V(0,m_0)
 - t 2 \pi^2 \sum_{k \in F_N} \partial_{\widehat{m}^k} V(0,m_0) \vert k\vert^2  \widehat{m}^k_0
 \\
&\hspace{15pt}
+
 \int_0^t 
 \biggl[ \int_{{\mathbb T}^d}  v^N(x) \cdot \alpha^N_s(x) m_0(x) dx  
+
 \int_{{\mathbb T}^d} L\bigl(x,  \alpha^N_s(x) \bigr) m_0(x) dx 
\biggr] ds \biggr) \biggr\vert \leq 
 C \eta_N t + o_N(t).
\end{split}
\end{equation*}
\vspace{5pt}

\textit{Third step.}
We now make use of the dynamic programming principle 
\eqref{eq:DPP},
which says that 
\begin{equation*}
V(0,m_0) = {\mathcal J}_{\textrm{\rm det}}^{0,t}(\alpha),
\end{equation*}
if $\alpha$ is an optimal strategy of the MFCP. 
We thus consider a solution to the MFG system \eqref{eq:MFG:system} that minimizes the MFCP (see Proposition \ref{prop:Briani}). We call it $(m_t,u_t)_{0 \leq t \leq T}$ (with 
$m_0$ as initial condition) and we choose $\alpha_t(x) = - \partial_p H(x,\nabla_x u_t(x))$.
By Proposition 
\ref{thm:solvability:MFG}, 
$\sup_{t \in [0,T]} 
\| \nabla \alpha_t  \|_\infty \leq c'$ for a constant $c'$ only depending on 
the various data in the assumption stated in Subsection 
\ref{subse:MFG}.
Invoking Lemma 
\ref{lem:superjet} again, using the same notations as in
\eqref{eq:proof:thm:4:3:3}
and 
\eqref{eq:proof:thm:4:3:2}, 
using the Lipschitz property of $\partial_p H$ with respect to the variable
$p$ and 
the bound for $\| \nabla_x^2 u_0 \|_\infty$ provided by Proposition 
\ref{thm:solvability:MFG}, we then have
\begin{equation*}
\begin{split}
&\biggl( \int_{{\mathbb T}^d}
\bigl\vert \alpha_0^N(x) 
+ \partial_p H\bigl( x,v^N(x) \bigr) \bigr\vert^2 dx
\biggr)^{1/2}
\\
&\leq 
\biggl( \int_{{\mathbb T}^d}
\bigl\vert \alpha_0^N(x) 
- \alpha_0(x)\bigr\vert^2
dx \biggr)^{1/2} 
+
\biggl( 
\int_{{\mathbb T}^d}
\Bigl\vert
 \partial_p H\bigl( x, \nabla_x u_0(x) \bigr)
-
 \partial_p H\bigl( x,v^N(x) \bigr) \Bigr\vert^2 dx
\biggr)^{1/2}
\\
&\leq 
\biggl( \int_{{\mathbb T}^d}
\bigl\vert \alpha_0^N(x) 
- \alpha_0(x)\bigr\vert^2
dx \biggr)^{1/2} 
+
C \biggl( 
\int_{{\mathbb T}^d}
\bigl\vert  \nabla_x u_0(x)  - v^N(x)  \bigr\vert^2 dx
\biggr)^{1/2}
\\
&\leq C \eta_N. 
\end{split}
\end{equation*}
Using time-continuity of $\nabla_x u_t$ (see again Proposition \ref{thm:solvability:MFG}), invoking once again the local-Lipschitz property of $L$ and inserting 
\eqref{eq:formula:L:H}
into the conclusion of the second step, we obtain
\begin{equation*}
\begin{split}
&\biggl\vert     t F(m_0)
+
 t \partial_t V(0,m_0)
 - t 2 \pi^2 \sum_{k \in F_N} \partial_{\widehat{m}^k} V(0,m_0) \vert k\vert^2  \widehat{m}^k_0
-
\int_0^t \int_{{\mathbb T}^d} H \bigl( x, v^N(x) \bigr) m_0(x) dx  ds 
  \biggr\vert 
  \\
  &\leq 
 C \eta_N t + o_N(t),
\end{split}
\end{equation*}
which yields, by dividing by $t$, 
\eqref{eq:HJB:gen}
at point $(0,m_0)$. 
Obviously, we can proceed in a similar manner when the initial time is some $t \in (0,T]$
(it is easily checked that the sequences $(\epsilon_N)_{N \geq 1}$
and 
$(\eta_N)_{N \geq 1}$ do not depend on the choice of the initial condition, as we prescribed in the first step of the proof). 
\vskip 5pt

\textit{Fourth Step.}
We now prove the second part of the statement. 
The proof relies on 
Remark
\ref{rem:4.5}. 
Combined with 
Theorem 
\ref{prop:rademacher}, it says 
that, for any 
$t \in [0,T]$ (replacing $0$ by $t$ in 
Remark 
\ref{rem:4.5})
and, for almost every 
$m \in \PP$, 
the optimal trajectories issued from 
$(t,m)$ satisfy the MFG system 
\eqref{eq:MFG:system}
with a prescribed initial value for the 
backward equation (since the Fourier coefficients are uniquely determined). 
By \cite[Proposition 2.1]{BrianiCardaliaguet}, 
uniqueness follows at such points $(t,m)$.
\end{proof}

\subsection{Proof of the auxiliary lemmas}

\begin{proof}[Proof of Lemma  \ref{lem:superjet}]
{Following the statement, we fix $N \geq 1$
and we 
take a point $m \in {\mathcal P}_N$} such that the restriction of $V(0,\cdot)$ to ${\mathcal O}_N$ is differentiable
at $(\widehat{m}^k)_{k \in F_N^+}$. 
Moreover, we call $(m_t,u_t)_{0 \leq t  \leq T} $ a solution to the MFG system 
that minimizes 
${\mathcal J}_{\textrm{\rm det}}$.

The proof then relies on the principle of super-jets. We know that for another point  $m' \in {\mathcal P}_N$, 
\begin{equation*}
\begin{split}
V(0,m') &\leq \int_0^T \Bigl( F(m_t') +  \int_{{\mathbb R}^d} L \Bigl( x, - \partial_p H\bigl(x, \nabla_x u_t(x) \bigr) \Bigr)  dm_t'(x) \Bigr) dt+ 
 G(m_T'),
\end{split}
\end{equation*}
for $(m_t')_{0 \leq t \leq T}$ the solution of the Fokker-Planck equation
\begin{equation}
\label{eq:FP:1}
\partial_t m_t'(x) - \textrm{div} \Bigl( \partial_p H\bigl( x, \nabla_x u_t(x) \bigr) \, m_t'(x) \Bigr) - \frac12 \Delta m_t'(x) =0,
\end{equation}
with $m_0'=m'$ as initial condition. 

In turn,
\begin{equation*}
\begin{split}
V(0,m') - V(0,m) &\leq \int_0^T \Bigl( F(m_t')  - F(m_t) +  \int_{{\mathbb R}^d} L \Bigl( x, - \partial_p H\bigl(x, \nabla_x u_t(x) 
\bigr) \Bigr) d(m_t' - m_t)(x) \Bigr) dt
\\
&\hspace{15pt} + 
 G(m_T') - G(m_T),
 \end{split}
 \end{equation*}
 And then, 
 thanks to the fact that $F$ and $G$ have Lipschitz continuous derivatives, 
\begin{equation*}
\begin{split}
 & V(0,m') - V(0,m) 
 \\
 &\leq \int_0^T \biggl\{ \int_{{\mathbb T}^d} \frac{\delta F}{\delta m}(m_t)(y) d (m_t'- m_t)(y) + 
  \int_{{\mathbb R}^d}  
   L \Bigl( x, - \partial_p H\bigl(x, \nabla u_t(x) 
\bigr) \Bigr) 
    d(m_t' - m_t)(x) \biggr\} dt
 \\
&\hspace{15pt}  + \int_0^T \frac{\delta G}{\delta m}(m_T)(y) d \bigl( m_T' - m_T\bigr)(y) + C \sup_{0 \leq t \leq T} \int_{{\mathbb T}^d} 
 \vert m_t'(y) - m_t(y) \vert^2 dy.
\end{split}
\end{equation*}
Using the stability properties of 
the linear equation
\eqref{eq:FP:1}, the last term is less than $C 
\int_{{\mathbb T}^d} 
 \vert m'(y) - m(y) \vert^2 dy$ (for a new value of $C$), 
where we notice that $m$ and $m'$ have a finite number of non-zero Fourier coefficients and thus have a smooth density. 
 
We now replace $[{\delta G}/{\delta m}](m_T)(y) $ by $g(y,m_T)=u_T(y)$
and  $[\delta F/\delta m](m_t)(y) $ by $f(y,m_t)$
and we use the MFG system \eqref{eq:MFG:system} 
and the probabilistic representation \eqref{eq:u_t:probabilistic:representation}.
We get
\begin{equation}
\label{eq:proof:lemma:4:4:1}
\begin{split}
  V(0,m') - V(0,m) &\leq 
\int_{{\mathbb T}^d} 
u_0(y) d \bigl( m' - m \bigr)(y) +
C 
\int_{{\mathbb T}^d} 
 \vert m_0'(y) - m_0(y) \vert^2 dy.
 \end{split}
 \end{equation}
We deduce that
for
$(\hat{r}^k)_{k \in F_N^+} \in {\mathbb C}^{\vert F_N^+ \vert}$
\begin{equation*}
\sum_{k \in F_N \setminus \{0\}}
\partial_{\widehat{m}_k} V(m) \widehat r^k
\leq  
\sum_{k \in F_N \setminus \{0\}}
\widehat u_0^{-k} \widehat r^k.
\end{equation*}
(Above, it is implicitly understood that $\widehat{r}^k = - \overline{\widehat{r}^k}$.)
  And we should have equality by changing $r$ into $-r$. This is sufficient to identify the real and imaginary parts of 
  $\partial_{\widehat{m}_k} V(m)$ and $\widehat r^{-k}$ (consistently with the definition 
    \eqref{eq:gradient:complex:000}). This completes the proof. 
\end{proof}

\begin{proof}[Proof of Lemma \ref{lem:small:time:expansion}.]
We observe that $\widehat \mu_t^0=1$, hence 
$\mu_t$ has mass 1 for any $t \in [0,T]$. However,
$\mu_t$ may become negative valued for $t$ away from $0$. 
In order to prove that $\mu_t$ stays positive for $t$ close to $0$, 
 recall 
the equation  
\eqref{eq:control:FKP} satisfied by 
$(m_t)_{0 \leq t \leq T}$. 
Since the initial condition is Lipschitz continuous and bounded (as it is a density)
and since the drift $\alpha$ therein is also Lipschitz continuous in space, 
we can invoke standard estimates for parabolic equations to deduce that
the function $(t,x) \mapsto m_t(x)$ is $1/2$-H\"older continuous in time and Lipschitz continuous in space 
on $[0,T] \times {\mathbb T}^d$, 
for some H\"older and Lipschitz constants only depending on $c$. 
In particular, we can find $\epsilon_N(c) >0$ only depending on $c$ and $N$, such that 
\begin{equation*}
\sum_{k \in F_N} \bigl\vert \widehat m_t^k  - \widehat m^k_0 \bigr\vert \leq \frac{c}2,
\end{equation*} 
if $t \in [0,\epsilon_N(c)]$. And then, replacing 
$\widehat{m}_t^k$ by $\widehat{\mu}_t^k$ and recalling that 
$\widehat{m}_0^k=0$ if $k \not \in F_N$, 
\begin{equation*}
\sup_{x \in {\mathbb T}^d} 
\bigl\vert \mu_t(x)  - m_0(x) \bigr\vert \leq \frac{c}2,
\end{equation*} 
which shows that 
$\mu_t$ is positive
(and thus is a probability measure)
if $t \in [0,\epsilon_N(c)]$ (since $m_0$ is lower bounded by $c$).

Next, we consider a test function $h$ that is 1-Lipschitz continuous. Then, 
for $t \in [0,\epsilon_N(c)]$, 
\begin{equation*}
\begin{split}
&\int_{{\mathbb T}^d} h(x) d \bigl( m_t- \mu_t \bigr)(x) 
= \int_{{\mathbb T}^d} \bigl( h(x) - h * f_N(x) \bigr) d \bigl( m_t- \mu_t \bigr)(x),
\end{split}
\end{equation*}
where we used the fact that $\widehat{h*f_N}^k =0$ if $k \not \in F_N$
and 
$\widehat{m}_t^k = 
\widehat{\mu}_t^k$ if 
$k \in F_N$. 
Now, since $h$ is $1$-Lipschitz, there exists a sequence $(\eta_N)_{N \geq 1}$ converging to $0$ (and independent of $h$ in the class of $1$-Lipschitz continuous functions) such that 
$\| h - h*f_N\|_\infty \leq \eta_N$. We deduce that 
\begin{equation}
\label{eq:Fourier:truncation:mt:1}
d_{W_1} \bigl( m_t, \mu_t \bigr)
\leq \eta_N \biggl( \int_{{\mathbb T}^d} \bigl\vert m_t(x) - \mu_t(x)\bigr\vert^2 dx \biggr)^{1/2}
= \eta_N \biggl( \sum_{k \not \in F_N} \bigl\vert \widehat{m}_t^k \bigr|^2 
\biggr)^{1/2}. 
\end{equation}
Now, 
we compute the dynamics of $\widehat{m}_t^k$ for $k \not \in F_N$. 
By integrating 
\eqref{eq:control:FKP}
with respect to $e_k$, we get
\begin{equation*}
\begin{split}
\frac{d}{dt} \widehat{m}_t^k &= -2 \pi^2 \vert k \vert^2 \widehat{m}_t^k  -\widehat{ \textrm{\rm div} \bigl( \alpha_t m_t \bigr)}^k.
\end{split}
\end{equation*}  
Letting for simplicity $\beta_t =  \textrm{\rm div} \bigl( \alpha_t m_t \bigr)$, we obtain, for $k \not \in F_N$ (using the fact that $\widehat{m}_0^k=0$),
\begin{equation*}
\vert \widehat{m}_t^k \vert 
\leq \int_0^t \vert \widehat{\beta}^k_s\vert ds. 
\end{equation*}
Recalling that  $\alpha_t$ and $m_t$ are Lipschitz continuous in space (uniformly in $t$),
we deduce that there exists a constant $C(c)$ such that  
\begin{equation*}
\sum_{k \in {\mathbb Z}^d}  \vert \widehat{\beta}_t^k\vert^2 \leq C(c). 
\end{equation*}
Inserting in 
\eqref{eq:Fourier:truncation:mt:1}, 
we get $t \eta_N \sqrt{C(c)}$ as upper bound for 
$d_{W_1} ( m_t, \mu_t )$.  
\end{proof}

\begin{proof}[Proof of Lemma \ref{le:d-2:V}.]
We first prove that $V$ is Lipschitz in space with respect to $d_{-2}$. Without any loss of generality, we can prove the Lipschitz property at time $t=0$. 

By 
\eqref{eq:proof:lemma:4:4:1}
in the proof of Lemma \ref{lem:superjet}, we know that, for any two $m,m' \in {\mathcal P}_N$, for some $N \geq 1$,
\begin{equation*}
\begin{split}
  V(0,m') - V(0,m) &\leq 
\int_{{\mathbb T}^d} 
u_0(y) d \bigl( m' - m \bigr)(y) +
C 
\int_{{\mathbb T}^d} 
 \vert m'(y) - m(y) \vert^2 dy,
 \end{split}
 \end{equation*}
 where $(m_t,u_t)_{0 \leq t \leq T}$ is a solution of the MFG system 
 \eqref{eq:MFG:system} that minimizes ${\mathcal J}_{\textrm{\rm det}}$.
 By Proposition 
\ref{prop:Briani}, the function $u_0$ is ${\mathcal C}^2$ and we have a bound for $\| \nabla^2 u_0\|_{\infty}$
{that only depends on the data in the assumption stated in 
Subsection \ref{subse:MFG}}. Therefore,
 \begin{equation*}
  V(0,m') - V(0,m) \leq 
C d_{-2}(m',m) + 
C 
\int_{{\mathbb T}^d} 
 \vert m'(y) - m(y) \vert^2 dy,
 \end{equation*}
with $C$ being independent of $m$ and $m'$. 

By choosing iteratively $(m,m')$ as $(m_i,m_{i+1})$ with 
$m_i = ( 1 - i/n \bigr) m +(i/n) m'$, 
for $i \in \{0,\cdots,n-1\}$ and for an integer $n \geq 1$, 
we get
\begin{equation*}
\begin{split}
  V(0,m') - V(0,m) &\leq \sum_{i=1}^n \bigl[ V(0,m_{i+1}) - V(0,m_i)\bigr]
  \\
  &\leq \sum_{i=1}^n 
C d_{-2}(m_{i+1},m_i) + 
C 
\sum_{i=1}^n \int_{{\mathbb T}^d} 
 \vert m_{i+1}(y) - m_i(y) \vert^2 dy
 \\
 &= C d_{-2}(m',m) + \frac{C}{n}  \int_{{\mathbb T}^d} 
 \vert m'(y) - m(y) \vert^2 dy.
\end{split}
\end{equation*}
Letting $n$ tend to $\infty$, we get 
\begin{equation*}
  V(0,m') - V(0,m) \leq C d_{-2}(m',m).
\end{equation*}
The above is true when $m$ and $m'$ are in ${\mathcal P}_N$. For any two general $m$ and $m'$, we can apply the above to 
$m*f_N$ and $m'*f_N$. We then observe that 
\begin{equation*}
\begin{split}
d_{-2}\bigl( m'*f_N,m*f_N\bigr)
&= \sup_{ \phi \in {\mathcal C}^2 : \| \nabla^2 \phi \|_\infty \leq 1}
\int_{{\mathbb T}^d}
\phi(x) \bigl( m'*f_N - m*f_N\bigr)(x) dx
\\
&= \sup_{ \phi \in {\mathcal C}^2 : \| \nabla^2 \phi \|_\infty \leq 1}
\int_{{\mathbb T}^d}
\bigl( \phi * f_N\bigr) (x) d \bigl( m'- m \bigr)(x)
\\
&\leq d_{-2}(m',m). 
\end{split}
\end{equation*}
In the end, we have
\begin{equation*}
  V(0,m'*f_N) - V(0,m*f_N) \leq C d_{-2}(m',m).
  \end{equation*}
  By Proposition 
  \ref{prop:semi-concavity},  $V$ is continuous for the $1$-Wasserstein distance. We can easily pass to the limit in the left-hand side and then get that $V(0,\cdot)$ is $C$-Lipschitz continuous for $d_{-2}$. The same argument can be applied at any time 
  $t \in (0,T]$. 
 
It now remains to prove that $V$ is Lipschitz continuous in time. In order to do so, we use 
the dynamic programming principle as follows. 
Using the $d_{-2}$-Lipschitz property in space, 
we know that, for any $t\in (0,T]$
and an optimal control $(\alpha_s)_{0\leq s \leq t}$ (regarded as a bounded time-space feedback function), 
\begin{equation*}
\bigl\vert V(t,m_t) - V(t,m_0)  \vert \leq C d_{-2}(m_t,m_0)
\end{equation*}
where $(m_s)_{0 \leq s \leq t}$ solves 
\eqref{eq:control:FKP}
with $m_0$ as initial condition.

Now, using the form of the Fokker-Planck equation, we get 
\begin{equation*}
d_{-2}(m_t,m_0)
=\sup_{ \phi \in {\mathcal C}^2 : \| \nabla^2 \phi \|_\infty \leq 1}
\int_{{\mathbb T}^d} 
\phi(x) d \bigl( m_t - m_0 \bigr)(x) 
\leq C t,
\end{equation*}
for a constant $C$ that depends on $\| \alpha \|_\infty$, but,  by 
Proposition 
\ref{prop:Briani}, the latter can be assumed to be bounded 
independently of the initial condition $m_0$. 

It then remains to see from the dynamic programming principle 
 that (for a possibly new value of $C$)
\begin{equation*}
\bigl\vert V(t,m_t) - V(0,m_0) \bigr\vert \leq C t.
\end{equation*}
This completes the proof. 
\end{proof}

\subsection{Mollification of $d_{-2}$-Lipschitz continuous and semi-concave generalized solutions}
The aim of this subsection is to show that, provided that they are 
$d_{-2}$-Lipschitz continuous and semi-concave (in space),
generalized solutions can be mollified into functions that `nearly' solve the original HJB 
equation on the whole space (and not only almost everywhere). This is a key result to obtain uniqueness of 
$d_{-2}$-Lipschitz continuous and semi-concave generalized solutions. 

For 
$W$ a generalized solution of the HJB equation and for the same parameters $N$, $\varepsilon$ and $
\rho$ as in Definition \ref{def:mollification}, we call
the corresponding mollification
$W^{N,\varepsilon,\rho}$ (which implicitly depends on $\rho$, as emphasized in Definition 
\ref{def:mollification}). 
%
We start with the following observation: 
\begin{prop}
\label{prop:4:7}
If $W : [0,T] \times {\mathcal P}({\mathbb T}^d) \rightarrow {\mathbb R}$ is  
$d_{-2}$-Lipschitz continuous in space, uniformly in time, then
there exists a constant $C$ such that, for any $N \geq 1$,  
\begin{equation*}
\begin{split}
&\sup_{x \in {\mathbb T}^d} 
\biggl\vert 
\partial_{x} 
\biggl( 
\sum_{k \in F_N \setminus \{0\}} 
\partial_{\widehat{m}^k} W(t,m)
e_k(x)
\biggr) 
\biggr\vert \leq C,
\\
&\sup_{x \in {\mathbb T}^d} 
\biggl\vert 
\partial_{xx}^2 
\biggl[
\biggl( 
\sum_{k \in F_N \setminus \{0\}} 
\partial_{\widehat{m}^k} W(t,m)
e_k 
\biggr) * f_N \biggr] (x)
\biggr\vert \leq C,
\end{split}
\end{equation*}
at any point $(t,m) \in [0,T] \times {\mathcal P}_N$ at which the restriction $\widetilde W_N(t,\cdot)$ of $W(t,\cdot)$
to ${\mathcal O}_N$ is differentiable. 
Moreover, at those points, 
\begin{equation*}
\sum_{k \in F_N} \vert k\vert^4 \vert \partial_{\widehat{m}^k} W(t,m) \vert^2 \leq C.
\end{equation*}
\end{prop}

\begin{proof}
We recall that 
${\mathcal O}_N$ is an open subset of ${\mathbb R}^{2 \vert F_N^+\vert}$. 
Therefore, 
for a point $(t,m) \in [0,T] \times {\mathcal P}_N$
of differentiability of 
$\widetilde W_N$, we can choose 
$\varepsilon \in (0,1)$ small enough such that, for any 
$(\widehat{r}^k)_{k \in F_N \setminus \{0\}}$ with 
$ \widehat{r}^{-k}=\overline{\widehat{r}^k}$ 
and $\sum_{k \in F_N \setminus \{0\}} \vert \widehat{r}^k \vert^2 \leq 1$,
$(\widehat{m}^k + \varepsilon \widehat{r}^k)_{k \in F_N \setminus \{0\}} \in {\mathcal P}_N$. 
By the $d_{-2}$-Lipschitz property of $W$, we obtain
\begin{equation*}
\frac{d}{d\varepsilon}_{\vert \varepsilon=0}\widetilde W_N \Bigl(t, \bigl(\widehat{m}^k + \varepsilon \widehat{r}^k\bigr)_{k \in F_N \setminus \{0\}}
\Bigr) \leq C \sup_{\| \nabla^2 \phi \|_\infty \leq 1}
\sum_{k \in F_N \setminus \{0\}} \widehat{r}^k  \widehat{\phi}^{-k}. 
\end{equation*}
for a constant $C$ only depending on the $d_{-2}$-Lipschitz constant of 
$W$ in space and whose value is allowed to vary from line to line. 
And then,
\begin{equation*}
\sum_{k \in F_N \setminus \{0\}} \widehat{r}^k \partial_{\widehat{m}^k} W(t,m) 
 \leq C \sup_{\| \nabla^2 \phi \|_\infty \leq 1}
\sum_{k \in F_N \setminus \{0\}} \widehat{r}^k {\widehat{\phi}^{-k}}. 
\end{equation*}
Obviously, 
we can relax the condition 
$\sum_{k \in F_N \setminus \{0\}} \vert \widehat{r}^k \vert^2 \leq 1$
in the above inequality (as it is linear). In other words,
the above holds true for any 
$(\widehat{r}^k)_{k \in F_N \setminus \{0\}}$
with 
$\widehat{r}^{-k}=\overline{\widehat{r}^k}$ for $k \in F_N \setminus \{0\}$. 
In particular, we can replace 
$(\widehat{r}^k)_{k \in F_N \setminus \{0\}}$
by 
$( k_q  k_{q'} \widehat{r}^k)_{k \in F_N \setminus \{0\}}$, for some $q,q' \in \{1,\cdots,d\}$, from which we get 
\begin{equation}
\label{eq:en:plus:Fourier:1}
\sum_{k \in F_N \setminus \{0\}} k_q k_{q'} \widehat{r}^k \partial_{\widehat{m}^k} W(t,m) 
 \leq C \sup_{\| \nabla^2 \phi \|_\infty \leq 1}
\sum_{k \in F_N \setminus \{0\}} \widehat{r}^k \bigl( k_q k_{q'} \widehat{\phi}^{-k} \bigr). 
\end{equation}
The argument appearing in the supremum in the right-hand side can be rewritten (up to a multiplicative constant) 
$\int_{{\mathbb T}^d} r(x) \partial^2_{x_q x_{q'}} \phi(x) dx$ with 
$r(x) = \sum_{k \in F_N \setminus \{0\}} \widehat{r}^k e_{-k}(x)$.  
Since the test function $\phi$ satisfies 
$\| \nabla^2 \phi \|_\infty \leq 1$, we get in the end
\begin{equation}
\label{eq:en:plus:Fourier:2}
\sum_{k \in F_N \setminus \{0\}} k_q k_{q'} \widehat{r}^k \partial_{\widehat{m}^k} W(t,m) 
 \leq C  
\int_{{\mathbb T}^d} \vert r(x) \vert dx.
\end{equation}
This yields
\begin{equation}
\label{eq:en:plus:Fourier}
\int_{{\mathbb T}^d} 
r(x) 
\partial_{x_q x_{q'}}^2 
\biggl( \sum_{k \in F_N \setminus \{0\}}
\partial_{\widehat{m}^k} 
W(t,m) e_{k}(x)
\biggr)  
dx
\leq C  
\int_{{\mathbb T}^d} \vert r(x) \vert dx.
\end{equation}
We now use 
\eqref{eq:en:plus:Fourier}
in order to derive 
the last 
two claims in the statement. 
In order to prove the second claim, we choose 
$r(x)=(s*f_N)(x)$ in 
\eqref{eq:en:plus:Fourier}, 
for an arbitrary function $s : {\mathbb T}^d \rightarrow {\mathbb R}$ satisfying 
$\int_{{\mathbb T}^d} \vert s(x) \vert dx < \infty$. Since $f_N$ is even, we get 
\begin{equation*}
\begin{split}
&\int_{{\mathbb T}^d} 
s(x) 
\partial_{x_q x_{q'}}^2 
\biggl( \sum_{k \in F_N \setminus \{0\}}
\partial_{\widehat{m}^k} 
W(t,m) e_{k} * f_N(x)
\biggr)  
dx
\\
&= 
\int_{{\mathbb T}^d} 
r(x) 
\partial_{x_q x_{q'}}^2 
\biggl( \sum_{k \in F_N \setminus \{0\}}
\partial_{\widehat{m}^k} 
W(t,m) e_{k}(x)
\biggr)  
dx
\leq C  
\int_{{\mathbb T}^d} \vert r(x) \vert dx
\leq C \int_{{\mathbb T}^d} \vert s(x) \vert dx,
\end{split}
\end{equation*}
which is exactly the second claim. 

As for the last claim in the statement, it also follows 
from \eqref{eq:en:plus:Fourier}, 
by choosing therein
$r(x) = \partial_{x_q x_{q'}}^2 
( \sum_{k \in F_N \setminus \{0\}}
\partial_{\widehat{m}^k} 
W(t,m) e_{k}(x)
)$.

It then remains to derive the first claim in the statement. To do so, we rewrite 
\eqref{eq:en:plus:Fourier:1}--\eqref{eq:en:plus:Fourier:2}
in the following manner. 
Instead of using 
$( k_q  k_{q'} \widehat{r}^k)_{k \in F_N \setminus \{0\}}$
in 
\eqref{eq:en:plus:Fourier:1}, we just use 
$( k_q   \widehat{r}^k)_{k \in F_N \setminus \{0\}}$, for some $q \in \{1,\cdots,d\}$.
Then, for 
$r(x) = \sum_{k \in F_N \setminus \{0\}} \widehat{r}^k e_{-k}(x)$, 
we reformulate
\eqref{eq:en:plus:Fourier:2}
as
\begin{equation*}
\int_{{\mathbb T}^d} r(x) \partial_{x_q} 
\biggl( \sum_{k \in F_N \setminus \{0\}}
\partial_{\widehat{m}^k} 
W(t,m) e_{k}(x)
\biggr) dx  
\leq C 
\sup_{\| \nabla^2 \phi \|_\infty \leq 1}
\int_{{\mathbb T}^d} 
r(x) \partial_{x_q} \phi(x) dx.
\end{equation*}
If we now take a square integrable function $s : {\mathbb T}^d \rightarrow {\mathbb R}$ and 
apply the above inequality to $r(x)=s_N(x)$, with $s_N(x) = \sum_{k \in F_N} 
\widehat{s}^k e_{-k}(x)$ being the (cubic) Fourier sum of order $N$ at $x$, then 
\begin{equation*}
\begin{split}
\int_{{\mathbb T}^d} s(x) \partial_{x_q} 
\biggl( \sum_{k \in F_N \setminus \{0\}}
\partial_{\widehat{m}^k} 
W(t,m) e_{k}(x)
\biggr)  dx
&=
\int_{{\mathbb T}^d} s_N(x) \partial_{x_q} 
\biggl( \sum_{k \in F_N \setminus \{0\}}
\partial_{\widehat{m}^k} 
W(t,m) e_{k}(x)
\biggr) dx
\\
&\leq C 
\sup_{\| \nabla^2 \phi \|_\infty \leq 1}
\int_{{\mathbb T}^d} 
s_N(x) \partial_{x_q} \phi(x) dx
\\
&= 
\sup_{\| \nabla^2 \phi \|_\infty \leq 1}
\int_{{\mathbb T}^d} 
s(x) \Bigl[ \partial_{x_q} \phi\bigr]_N(x) dx,
\end{split}
\end{equation*}
 with 
$[ \partial_{x_q} \phi]_N(x) = \sum_{k \in F_N} 
\widehat{\partial_{x_q} \phi}^k e_{-k}(x)$. 
Here, a $d$-dimensional version of Dini theorem 
(see \cite[Section 3]{Golubov}, 
\cite[Corollary 2.1.2]{Zhizhiashvili}
or 
\cite[Chapter 1, \&2.2]{Alimov})
asserts that 
$\| [ \partial_{x_q} \phi]_N - \partial_{x_q} \phi \|_\infty$ converges to $0$ as $N$
tends to $\infty$, uniformly in 
$\phi$ satisfying 
$\| \nabla^2 \phi \|_\infty \leq 1$. As a result, we write the above bound as
\begin{equation*}
\begin{split}
\int_{{\mathbb T}^d} s(x) \partial_{x_q} 
\biggl( \sum_{k \in F_N \setminus \{0\}}
\partial_{\widehat{m}^k} 
W(t,m) e_{k}(x)
\biggr)  dx
&\leq C 
\int_{{\mathbb T}^d} 
\vert s(x) \vert dx.
\end{split}
\end{equation*}
By an obvious density argument, we can easily pass from $s \in L^2({\mathbb T}^d)$ to 
$s \in L^1({\mathbb T}^d)$ in this inequality. This proves the first claim in the statement. 
\end{proof}

We then have our main mollification result:
\begin{prop}
\label{prop:generalized:solution}
Let $W$ be
a
generalized solution of the HJB equation 
that is
$d_{-2}$-Lipschitz continuous and semi-concave in space, uniformly in time. 
For any $c>1$, 
there exist two families 
$(\eta_{N,\varepsilon})_{N,\varepsilon}$ 
(of positive reals)
and 
$(R^{N,\varepsilon,\rho})_{N,\varepsilon,\rho}$ 
(of compactly supported bounded functions from $[0,T] \times {\mathbb C}^{ \vert F_N^+ \vert} $ into ${\mathbb R}$, with the support and the bound depending on 
$N$ and $\varepsilon$ and not on $\rho$ and $t$)
such that, 
for almost every $t \in [0,T]$, 
for any $m \in {\mathcal P}({\mathbb T}^d)$ with 
$\|\nabla m \|_\infty \leq c$ and $m \geq 1/c$, 
\begin{equation*}
\begin{split}
&\biggl\vert \partial_t W^{N,\varepsilon,\rho}(t,m) -   \int_{{\mathbb T}^d}
H \bigl( x, 
\partial_{\mu} W^{N,\varepsilon,\rho}(t,m)(x) \bigr) d \bigl(m*f_N\bigr)(x) 
\\
&\hspace{15pt}
+ \frac12 \int_{{\mathbb T}^d}
\textrm{\rm Trace} \Bigl[ \partial_x \partial_\mu W^{N,\varepsilon,\rho}(t,m)(x) \Bigr]
dm(x) 
 + F^{N,\varepsilon,\rho}(m) \biggr\vert \leq \eta_{N,\varepsilon} + R^{N,\varepsilon,\rho}\Bigl( t,(\widehat{m}^k)_{k \in F_N^+} \Bigr),
\end{split}
\end{equation*}
with
$\lim_{(N,\varepsilon) \rightarrow (\infty,0)} \eta_{N,\varepsilon} = 0$,
and, for any $t,N,\varepsilon$ fixed,
\begin{equation*}
\lim_{\rho \rightarrow \delta_0}
\int_{{\mathbb R}^{2 \vert F_N^+ \vert}} \Bigl\vert 
R^{N,\varepsilon,\rho}\Bigl(t, (\widehat{m}^k)_{k \in F_N^+} \Bigr) \Bigr\vert 
\bigotimes_{k \in F_N^+} 
d \Bigl( \Re[\widehat{m}^k],
\Im[\widehat{m}^k]
 \Bigr) = 0.
\end{equation*}
Above, the convergence $\rho \rightarrow \delta_0$ is understood as 
the weak convergence of 
$\rho$ to $\delta_0$, the Dirac point mass at $0$ (in ${\mathbb R}^2$).  
\end{prop}

The reader will notice the difference between 
Definition 
\ref{defn:HJB:gen}
and 
Proposition
\ref{prop:generalized:solution}: not only the main inequality
in 
Proposition \ref{prop:generalized:solution}
 is stated 
everywhere in space, but 
also
the 
derivatives 
are derivatives on the space of probability measures (as opposed to 
the derivatives in 
Definition 
\ref{defn:HJB:gen}, which are just finite dimensional). 
This is more consistent with the original formulation of the HJB equation, 
see \eqref{eq:HJB}. See also
Remark 
\ref{rem:formulation:P:as} for a related comment. 

\begin{proof}
Throughout the proof, 
the constant $c>1$ is kept fixed. 
We recall the
following
 notation from 
 Definition 
\ref{defn:HJB:gen}: 
$A_N(c)=[0,T] \times B_N(c)$ with $B_N(c):= {\mathcal P}_N \cap B(c)$ and $B(c):= \{ m :  \sup_{x \in \bT^d} \vert \nabla m(x)  \vert \leq c, \ \inf_{x \in \bT^d}  m(x)   \geq 1/c \}$.
\vspace{5pt}

\textit{First Step.}
From the Definition
\ref{defn:HJB:gen}
 of a generalized solution, 
 there exists a full subset 
 $\widetilde A_N(c)$ of $A_N(c)$ 
 (for the image of the Lebesgue measure by the mapping $(t,(\widehat{m}^k)_{k \in F_N}) \in [0,T] \times {\mathcal O}_N \mapsto (t,{\mathscr I}_N((\widehat{m}^k)_{k \in F_N}))$, see Proposition \ref{prop:3:8}) such that 
\begin{equation*}
\begin{split}
&\biggl\vert \partial_t W(t,m*f_N) - 
 \int_{{\mathbb T}^d} H \biggl( y, \i 2 \pi \sum_{k \in F_N} k \partial_{\widehat m^k} W(t,m * f_N) e_{k}(y) \biggr) d(m*f_N)(y)
 \\
 &\hspace{15pt} - \sum_{k \in F_N} 2 \pi^2 \vert k\vert^2 
\partial_{ \widehat{m}^{k} }
W(t,m*f_N) \widehat{m}^k \widehat{f_N}^k + F(m * f_N)\biggr\vert \leq \eta_N(c),
\end{split}
\end{equation*}
when $(t,m * f_N) \in \widetilde A_N(c)$. Above, we specified the dependence of the sequence $(\eta_N)_{N \geq 1}$ upon the constant $c$ by writing 
$(\eta_N(c))_{N \geq 1}$. 

Fix now 
$m \in B(c)$, that is $m \in  {\mathcal P}({\mathbb T}^d)$ with 
$\|\nabla m \|_\infty \leq c$ and $m\geq 1/c$. 
Fix also
$t \in [0,T]$ such that 
$\{ m' : (t,m') \in \widetilde A_N(2c)\}$ is a full subset 
of $({\mathcal P}_N \cap \{ m' :  \sup_{x \in \bT^d} \vert \nabla m'(x)  \vert \leq 2c, \ \inf_{x \in \bT^d}  m'(x)   \geq 1/(2c) \})$: by Fubini's theorem, the set of such $t$'s is a full subset of $[0,T]$.
By definition of 
$\delta_{N,\varepsilon}$ 
(which we merely denote by $\delta$ in the rest of the proof)
in 
Definition 
\ref{def:admiss:threshold:smoothing}, we observe that, 
for 
$r=(\widehat{r}^k)_{k \in F_N^+}$ such that $\vert \widehat{r}^k \vert \leq \delta$, 
the probability measure 
$m^{N,\varepsilon}(r)$ (as given by Definition 
\ref{def:mollification})
satisfies
$\| \nabla m^{N,\varepsilon}(r) \|_\infty \leq 2c$ and 
$m^{N,\varepsilon}(r)\geq 1/(2c)$ for 
$\varepsilon$ small enough (independently of $N$). 
Therefore, 
$(t,m^{N,\varepsilon}(r)*f_N)$ belongs to 
$A_N(2c)$. 
Since the image of the Lebesgue measure 
by the mapping 
$(\widehat{r}^k)_{k \in F_N^+} \mapsto 
m^{N,\varepsilon}(r)*f_N \in {\mathcal P}_N$ is obviously absolutely-continuous, we deduce that 
$(t,m^{N,\varepsilon}(r)*f_N) \in \widetilde{A}_N(2c)$, for almost every 
$(\widehat{r}^k)_{k \in F_N^+}$ satisfying 
$\vert \widehat{r}^k \vert \leq \delta$
for all $k \in F_N^+$. 
Hence, we get, for almost every 
$(\widehat{r}^k)_{k \in F_N^+}$,
\begin{align}
&\biggl\vert \partial_t W\bigl(t,m^{N,\varepsilon}(r)*f_N\bigr) - 
\int_{{\mathbb T}^d} H \biggl( y, \i 2 \pi \sum_{k \in F_N} k \partial_{\widehat m^k} W\bigl(t,m^{N,\varepsilon}(r) * f_N\bigr) e_{k}(y) \biggr) d\bigl(m^{N,\varepsilon}(r) *f_N\bigr)(y)
\nonumber
 \\
 &\hspace{5pt} - \sum_{k \in F_N} 2 \pi^2 \vert k\vert^2 
\partial_{ \widehat{m}^{k} }
W\bigl(t,m^{N,\varepsilon}(r)*f_N\bigr) \widehat{m^{N,\varepsilon}(r)}^k \widehat{f_N}^k + F\bigl(m^{N,\varepsilon}(r) * f_N\bigr)\biggr\vert \leq \eta_N(2c).
\label{eq:1st:step:theorem:4:3:proof}
\end{align}

\textit{Second Step.}
The conclusion
\eqref{eq:1st:step:theorem:4:3:proof}
 of the first step right above holds true for almost every $t \in [0,T]$, for any 
$m \in B(c)$, and for 
almost every 
$(\widehat{r}^k)_{k \in F_N^+}$ satisfying 
$\vert \widehat{r}^k \vert \leq \delta$
for all $k \in F_N^+$. 

Therefore, integrating 
\eqref{eq:1st:step:theorem:4:3:proof}
over $(\Re[\widehat r^k],\Im[\widehat r^k])_{k \in F_N^+}$ with respect to $\otimes_{k \in F_N^+} \rho$ and recalling 
Corollary 
\ref{cor:mollif:time-space} (for the existence of the time-derivative below), we obtain, for almost every $t \in [0,T]$
and for any $m \in B(c)$, 
\begin{equation}
\label{eq:prop:4:9:T1Nepsilonrho}
\Bigl\vert \partial_t W^{N,\varepsilon,\rho}(t,m) + F^{N,\varepsilon,\rho}(m) - T_1^{N,\varepsilon,\rho}(t,m) - T_2^{N,\varepsilon,\rho}(t,m) \Bigr\vert \leq \eta_N(2c),
\end{equation}
with 
\begin{equation*}
\begin{split}
&T_1^{N,\varepsilon,\rho}(t,m) 
\\
&:=  
 \int_{{\mathbb R}^{2 \vert F_N^+ \vert}}
 \biggl\{
  \int_{{\mathbb T}^d} H \biggl( y, \i 2 \pi  \sum_{k \in F_N} k \partial_{\widehat m^k} W\bigl(t,m^{N,\varepsilon}(r) * f_N\bigr) e_{k}(y) \biggr) d\bigl(m^{N,\varepsilon}(r) *f_N\bigr)(y) 
  \biggr\}
  \\
&\hspace{300pt} \times \prod_{j \in F_N^+} \rho\bigl(  \widehat{r}^j \bigr)
\bigotimes_{j \in F_N^+} 
  d \widehat{r}^j ,
\end{split}
\end{equation*}
and
\begin{equation*}
T_2^{N,\varepsilon,\rho}(t,m)
: = \sum_{k \in F_N} 2 \pi^2 \vert k\vert^2 \biggl( \int_{{\mathbb R}^{2 \vert F_N^+ \vert}} 
\partial_{ \widehat{m}^{k} }
W\bigl(t,m^{N,\varepsilon}(r)*f_N\bigr) \widehat{m^{N,\varepsilon}(r)}^k 
\prod_{j \in F_N^+} \rho\bigl(  \widehat{r}^j \bigr)
\bigotimes_{j \in F_N^+} 
  d \widehat{r}^j 
\biggr)
\widehat{f}_N^k.
\end{equation*}
We recall from 
Proposition 
\ref{prop:mollification}
(see also Corollary
\ref{cor:mollif:time-space}) 
that, 
for any $(t,m) \in [0,T] \times B_N(c)$,
\begin{equation}
\label{eq:Fourier:derivative:convolution:regularisation:0}
\begin{split}
&\frac{\delta W^{N,\varepsilon,\rho}}{\delta m}(t,m)(x)
\\
&= (1-\varepsilon) 
\int_{{\mathbb R}^{2 \vert F_N^+ \vert}}
\sum_{k \in F_N  \setminus \{0\}} \Bigl[ \Bigl(\partial_{\widehat{m}^k}  W \bigl(t,m^{N,\varepsilon}(r) * f_N\bigr) e_{k} \Bigr) * f_N \Bigr](x) 
\prod_{j \in F_N^+} \rho\bigl(  \widehat{r}^j \bigr)
\bigotimes_{j \in F_N^+} 
  d \widehat{r}^j 
\\
&= 
(1-\varepsilon) 
\int_{{\mathbb R}^{2 \vert F_N^+ \vert}}
\sum_{k \in F_N  \setminus \{0\}} 
\Bigl[
\partial_{\widehat{m}^k}  W \bigl(t,m^{N,\varepsilon}(r) * f_N\bigr) \hat{f}_N^k e_{k}(x) 
\Bigr]
\prod_{j \in F_N^+} \rho\bigl(  \widehat{r}^j \bigr)
\bigotimes_{j \in F_N^+} 
  d \widehat{r}^j ,
\end{split}
\end{equation}
where we used on the last line the fact that 
$e_k * f_N (x)= e_k(x) \widehat{f}_N^{-k} 
= 
e_k(x) \widehat{f}_N^{k}$.  We deduce that 
\begin{equation}
\label{eq:Fourier:derivative:convolution:regularisation}
\reallywidehat{\displaystyle \biggl\{\frac{\delta W^{N,\varepsilon,\rho}}{\delta m}(t,m)\biggr\}}^{-k} =
(1-\varepsilon) 
\hat{f}^{k}_N
\int_{{\mathbb R}^{2 \vert F_N^+ \vert}}
\partial_{\hat{m}^k}  W \bigl(t,m^{N,\varepsilon}(r) * f_N\bigr) 
\prod_{j \in F_N^+} \rho\bigl(  \widehat{r}^j \bigr)
\bigotimes_{j \in F_N^+} 
  d \widehat{r}^j ,
\end{equation}
for $k \in F_N$. 
\vspace{5pt}

\textit{Third Step.}
The main difficulty of the proof is to handle $T_1^{N,\varepsilon,\rho}(t,m)$.
The goal is to identify it with the second term in the main inequality of the statement. In order
to do so, 
we first observe that this latter term can be rewritten as 
 \begin{align}
&\int_{{\mathbb T}^d} H \Bigl( y, 
\partial_{\mu} W^{N,\varepsilon,\rho}(t,m)(y) 
 \Bigr) d\bigl( m * f_N \bigr) (y) 
 \nonumber
\\
&=\int_{{\mathbb T}^d} H \biggl( y,- \i 2 \pi
\sum_{k \in F_N} k 
\reallywidehat{\displaystyle \biggl\{\frac{\delta W^{N,\varepsilon,\rho}}{\delta m}(t,m)\biggr\}}^k e_{-k}(y) \biggr) d\bigl( m * f_N \bigr) (y) 
\label{eq:main:estimate:proof:mollification:HJB}
\\
&= 
\int_{{\mathbb T}^d}
H \biggl( y, 
 \i 2 \pi (1-\varepsilon) \sum_{k \in F_N}
\widehat{f}^{k}_N
 k 
\int_{{\mathbb R}^{2 \vert F_N^+ \vert}} 
  \partial_{\widehat{m}^k}  W \bigl(t,m^{N,\varepsilon}(r) * f_N\bigr) e_k(y)
  \prod_{j \in F_N^+} \rho\bigl(  \widehat{r}^j \bigr)
\bigotimes_{j \in F_N^+} 
  d \widehat{r}^j 
\biggr)
\nonumber
\\
&\hspace{350pt} \times d \bigl( m * f_N \bigr) (y),   \nonumber
  \end{align}
  which prompts us to let 
  \begin{equation}
  \label{eq:R:1:N,varepsilon,rho}
  \begin{split}
  &T_3^{N,\varepsilon,\rho}(t,m)
  :=
  \int_{{\mathbb T}^d}
H \biggl( y, 
\i 2 \pi  \sum_{k \in F_N}
 k 
\int_{{\mathbb R}^{2 \vert F_N^+ \vert}} 
  \partial_{\widehat{m}^k}  W \bigl(t,m^{N,\varepsilon}(r) * f_N\bigr) e_k(y)
  \prod_{j \in F_N^+} \rho\bigl(  \widehat{r}^j \bigr)
\bigotimes_{j \in F_N^+} 
  d \widehat{r}^j 
\biggr)
\nonumber
\\
&\hspace{350pt} \times d \bigl( m * f_N \bigr) (y).
  \end{split}
  \end{equation}
  (Pay attention that there is no $1-\varepsilon$ and no 
  $\widehat{f}_N^k$
  inside $H$ in the above right-hand side.)  
  By the local Lipschitz property of $H$, we can find a constant $C$, only depending on the parameters in 
  the assumptions, such that 
\begin{equation*}
\begin{split}
&\biggl\vert 
\int_{{\mathbb T}^d} H \Bigl( y, 
\partial_{\mu} W^{N,\varepsilon,\rho}(t,m)(y) 
 \Bigr) d\bigl( m * f_N \bigr) (y)
-
T_3^{N,\varepsilon,\rho}(t,m)
 \biggr\vert
\\
&\leq C  
\biggl[ \int_{{\mathbb T}^d} 
\biggl\{ \biggl\vert
\sum_{k \in F_N}
\widehat{f}^{k}_N
 k 
\int_{{\mathbb R}^{2 \vert F_N^+ \vert}} 
  \partial_{\widehat{m}^k}  W \bigl(t,m^{N,\varepsilon}(r) * f_N\bigr) e_k(y)
  \prod_{j \in F_N^+} \rho\bigl(  \widehat{r}^j \bigr)
\bigotimes_{j \in F_N^+} 
  d \widehat{r}^j 
\biggr\vert
\\
&\hspace{15pt} 
+ \biggl\vert
\sum_{k \in F_N}
 k 
\int_{{\mathbb R}^{2 \vert F_N^+ \vert}} 
  \partial_{\widehat{m}^k}  W \bigl(t,m^{N,\varepsilon}(r) * f_N\bigr) e_k(y)
  \prod_{j \in F_N^+} \rho\bigl(  \widehat{r}^j \bigr)
\bigotimes_{j \in F_N^+} 
  d \widehat{r}^j 
\biggr\vert \biggr\}^2 d \bigl( m * f_N \bigr)(y) \biggr]^{1/2}
\\
&\hspace{5pt} \times 
\biggl[ \int_{{\mathbb T}^d} 
 \biggl\vert
\sum_{k \in F_N}
\bigl( 
(1-\varepsilon) 
\widehat{f}^{k}_N
-1
\bigr)
 k 
\int_{{\mathbb R}^{2 \vert F_N^+ \vert}} 
  \partial_{\widehat{m}^k}  W \bigl(t,m^{N,\varepsilon}(r) * f_N\bigr) e_k(y)
  \prod_{j \in F_N^+} \rho\bigl(  \widehat{r}^j \bigr)
\bigotimes_{j \in F_N^+} 
  d \widehat{r}^j 
\biggr\vert^2 
\\
&\hspace{250pt} \times d \bigl( m * f_N \bigr)(y) \biggr]^{1/2}.
\end{split}
\end{equation*}
We now show that the term on the last two lines above tends to $0$
as $(N,\varepsilon)$ tends to $0$. The same strategy would permit to show that the term on
the first two lines of the right-hand side is bounded by a constant  $C(c)$ depending on
the parameters in the assumptions and on the value of $c$. 
In order to proceed, we observe that $m$ and $m*f_N$ are bounded by $c+1$ since $m \in B_N(c)$.  
 Therefore, allowing the constant $C$ to depend on $c$ and thus writing $C(c)$ instead of 
 $C$, we get
\begin{equation*}
\begin{split}
& \int_{{\mathbb T}^d} 
 \biggl\vert
\sum_{k \in F_N}
\bigl( 
(1-\varepsilon) 
\widehat{f}^{k}_N
-1
\bigr)
 k 
\int_{{\mathbb R}^{2 \vert F_N^+ \vert}} 
  \partial_{\widehat{m}^k}  W \bigl(t,m^{N,\varepsilon}(r) * f_N\bigr) e_k(y)
  \prod_{j \in F_N^+} \rho\bigl(  \widehat{r}^j \bigr)
\bigotimes_{j \in F_N^+} 
  d \widehat{r}^j 
\biggr\vert^2 
\\
&\hspace{250pt} \times d \bigl( m * f_N \bigr)(y)
\\
&\leq C(c) 
\sum_{k \in F_N}
\int_{{\mathbb R}^{2 \vert F_N^+ \vert}} 
\int_{{\mathbb T}^d} 
\vert k\vert^2
\bigl\vert
(1-\varepsilon) 
\widehat{f}^{k}_N
-1
\bigr\vert^2
\Bigl\vert
  \partial_{\widehat{m}^k}  W \bigl(t,m^{N,\varepsilon}(r) * f_N\bigr) 
  \Bigr\vert^2
  \prod_{j \in F_N^+} \rho\bigl(  \widehat{r}^j \bigr)
\bigotimes_{j \in F_N^+} 
  d \widehat{r}^j 
d y. 
\end{split}
\end{equation*}
We now use the following three facts: First,  $\vert \widehat{f}^{k}_N \vert \leq 1$; 
Second, $\widehat{f}^k_N$ tends to $1$ when $N$ tends to $\infty$ and $k$ is fixed; 
Third, 
Proposition \ref{prop:4:7}
gives a bound for 
the series $\sum_{k \in {\mathbb Z}^d} 
\vert k \vert^4
\vert
  \partial_{\widehat{m}^k}  W (t,m^{N,\varepsilon}(r) * f_N ) \vert^2$ (at points where the derivatives exist).
  Then, for a fixed integer $N_0$, we can separate the sum in the above right-hand side 
  depending whether $\vert k \vert \leq N_0$ or
  $\vert k \vert > N_0$. 
  For $\vert k \vert \leq N_0$, we can easily use the fact that 
  $(1-\varepsilon) \widehat{f}_N^k$ tends to $1$ as $(N,\varepsilon)$ tends to $(\infty,0)$. 
  For $\vert k \vert \geq N_0$, we can easily bound 
  $\vert k \vert^2$ by $\vert k\vert^4/N_0^2$. All in all, this proves that the right-hand side
  tends to $0$ as $(N,\varepsilon)$ tends to $(\infty,0)$. We finally have
 \begin{equation}
\label{defn:HJB:gen:2}
\begin{split}
&\biggl\vert \int_{{\mathbb T}^d} H \Bigl( y, 
\partial_{\mu} W^{N,\varepsilon,\rho}(t,m)(y) 
 \Bigr) d\bigl( m * f_N \bigr) (y) 
-
T_3^{N,\varepsilon,\rho}(t,m)
\biggr\vert
\leq \eta_{N,\varepsilon}(c),
\end{split}
\end{equation}
where $\eta_{N,\varepsilon}(c)$ is a generic notation for a sequence that may depend on $c$ (but not on 
$(t,m)$) and that tends to $0$ as
$(N,\varepsilon)$ tends to $(\infty,0)$. 
\vspace{5pt}

\textit{Fourth Step.}
We now let 
\begin{align}
&\psi^{N,\varepsilon,\rho}(t,m) 
\nonumber
\\
&:= 
\int_{{\mathbb T}^d}
\biggl\{
\int_{{\mathbb R}^{2 \vert F_N^+\vert}} 
H \biggl( y, \i 2 \pi \sum_{k \in F_N} k \partial_{\widehat m^k} W\bigl(t,m^{N,\varepsilon}(r) * f_N\bigr) e_{k}(y) \biggr)  
\prod_{j \in F_N^+} \rho\bigl(  \widehat{r}^j \bigr)
\bigotimes_{j \in F_N^+} 
  d \widehat{r}^j
  \label{eq:psi:N,rho,epsilon}
\\
&\hspace{15pt} - 
H \biggl( y, \i 2 \pi \int_{{{\mathbb R}^{2 \vert F_N^+\vert}} }   \sum_{k \in F_N} k \partial_{\widehat m^k} W\bigl(t,m^{N,\varepsilon}(r) * f_N\bigr) e_{k}(y) 
\prod_{j \in F_N^+} \rho\bigl(  \widehat{r}^j \bigr)
\bigotimes_{j \in F_N^+} 
  d \widehat{r}^j
 \biggr)  \biggr\} d \bigl( m * f_N\bigr)(y). \nonumber
\end{align}
Using the notation introduced in the third step, this is also equal to 
\begin{equation*}
\begin{split}
\psi^{N,\varepsilon,\rho}(t,m) 
&:= 
\int_{{\mathbb T}^d}
\biggl\{ \int_{{\mathbb R}^{2 \vert F_N^+\vert}} 
H \biggl( y, \i 2 \pi \sum_{k \in F_N} k \partial_{\widehat m^k} W\bigl(t,m^{N,\varepsilon}(r) * f_N\bigr) e_{k}(y) \biggr)  
\prod_{j \in F_N^+} \rho\bigl(  \widehat{r}^j \bigr)
\bigotimes_{j \in F_N^+} 
  d \widehat{r}^j
  \biggr\}
  \\
&\hspace{300pt} \times d \bigl( m * f_N\bigr)(y)
 \\
 &\hspace{15pt} - T_3^{N,\varepsilon,\rho}(t,m). 
\end{split}
\end{equation*}
By convexity of the Hamiltonian, 
$\psi^{N,\varepsilon,\rho}(t,m)$ is obviously non-negative. 
Moreover, in 
\eqref{eq:psi:N,rho,epsilon}, 
we can regard each $\partial_{\widehat{m}^k} W(t,m^{N,\varepsilon}(0)*f_N)$ 
as a (almost-everywhere well-defined) function of vectors $(\widehat{m}^k)_{k \in F_N^+} \in {\mathbb C}^{\vert F_N^+\vert} \simeq  {\mathbb R}^{2 \vert F_N^+\vert}$
such that 
$(\widehat{m}^k \widehat{f}_N^k)_{k \in F_N^+}
\in {\mathcal O}_N$: in order to clarify the notation, we write 
it in the form $\widetilde \beta^{N,\varepsilon,k}(t,(\widehat{m}^j)_{j \in F_N^+})$, with $\widetilde \beta^{N,\varepsilon,k}(t,\cdot)$ being seen as an element of $L^\infty({\mathbb R}^{2 \vert F_N^+\vert})$
that is equal to $0$ at 
vectors 
$(\widehat{m}^k)_{k \in F_N^+}$
such that 
$(\widehat{m}^k \widehat{f}_N^k)_{k \in F_N^+}
\not \in {\mathcal O}_N$.  
Then, we rewrite
\eqref{eq:psi:N,rho,epsilon}
in the form
\begin{align}
 \label{eq:main:bound:phi:N,varepsilon,rho}
&\psi^{N,\varepsilon,\rho}(t,m) 
\\
&=  \sum_{\ell \in F_N \setminus \{0\}} \int_{{\mathbb T}^d}
\biggl\{
\int_{{\mathbb R}^{2 \vert F_N^+\vert}} H \biggl( y, \i 2 \pi  \sum_{k \in F_N} k \widetilde \beta^{N,\varepsilon,k} \Bigl(t, \bigl( \widehat{m}^j + \widehat{r}^j \bigr)_{j \in F_N^+} \Bigr) e_{k}(y) \biggr) 
\prod_{j \in F_N^+} \rho\bigl(  \widehat{r}^j \bigr)
\bigotimes_{j \in F_N^+} 
  d \widehat{r}^j
  \nonumber
\\
&\hspace{5pt} - 
H \biggl( y, \i 2 \pi \int_{{\mathbb R}^{2 \vert F_N^+\vert}}   \sum_{k \in F_N} k \widetilde \beta^{N,\varepsilon,k} \Bigl(t, \bigl( \widehat{m}^j + \widehat{r}^j \bigr)_{j \in F_N^+} \Bigr) e_{k}(y) 
\prod_{j \in F_N^+} \rho\bigl(  \widehat{r}^j \bigr)
\bigotimes_{j \in F_N^+} 
  d \widehat{r}^j
 \biggr) \biggr\} \widehat{m}^{\ell} \widehat{f}^\ell_N e_{-\ell}(y) dy. 
\nonumber
\end{align}
Notice that we expanded the convolution product $m*f_N$ in Fourier coefficients. This is a way to emphasize the fact that 
$\psi^{N,\varepsilon,\rho}(t,\cdot)$ can be here regarded as a real-valued function on ${\mathbb R}^{2 \vert F_N^+ \vert}$. 
To make the identification clear, we can denote the right-hand side in the above equality 
by $\widetilde \psi^{N,\varepsilon,\rho}(t,(\widehat{m}^j)_{j \in F_N^+})$, with 
$\widetilde \psi^{N,\varepsilon,\rho}(t,\cdot) : {\mathbb R}^{2 \vert F_N^+ \vert} \rightarrow 
{\mathbb R}$. We then have 
$\psi^{N,\varepsilon,\rho}(t,m) = 
\widetilde \psi^{N,\varepsilon,\rho}(t,(\widehat{m}^j)_{j \in F_N^+})$. 
We now make the following observation. If $\rho(\widehat{r}^j) d \widehat{r}^j$ were formally set equal to $\delta_0(d \widehat{r}^j)$
and if 
$\beta^{N,\varepsilon,k}$ were continuous in the second argument,  
the term on the right-hand side 
 would be equal to 0, hence leaving us with 
\begin{equation*}
\widetilde \psi^{N,\varepsilon,\delta_0}\bigl(t,(\widehat m^j)_{j \in F_N^+} \bigr) =0,
\quad (\widehat m^j)_{j \in F_N^+} \in {\mathbb R}^{ 2 \vert F_N^+ \vert},
\end{equation*}
where the symbol $\delta_0$ in the left-hand side is for the Dirac mass at $0$ and is to indicate that $\rho \equiv \delta_0$. 
In turn, if 
$\widetilde \beta^{N,\varepsilon,k}$ were continuous in the second argument,  
the right-hand side in 
\eqref{eq:main:bound:phi:N,varepsilon,rho} would converge to $0$ 
as $\rho$ tends to $\delta_0$. 
Obviously, 
$\widetilde \beta^{N,\varepsilon,k}$ is not continuous in our setting. However, 
by
 standard mollification results in finite dimension, we can approximate (for any $t$) the functions 
 $(\widetilde \beta^{N,\varepsilon,k}(t,\cdot))_{k \in F_N^+}$ 
 by continuous functions 
 in $L^1({\mathbb R}^{2 \vert F_N^+\vert})$. 
For $(N,\varepsilon,k)$ and $t$ fixed, we can indeed 
find 
a collection of real-valued continuous functions
$(\check \beta^{N,\varepsilon,k,n})_{n \in {\mathbb N}}$
 on ${\mathbb R}^{2 \vert F_N^+\vert}$ such that 
\begin{equation*}
\lim_{n \rightarrow \infty} 
\int_{{\mathbb R}^{2 \vert F_N^+\vert}}
\Bigl\vert \widetilde \beta^{N,\varepsilon,k}\Bigl(t,
\bigl( \widehat{m}^j \bigr)_{j \in F_N^+} 
\Bigr)
-
\check \beta^{N,\varepsilon,k,n}\Bigl(
\bigl( \widehat{m}^j \bigr)_{j \in F_N^+} 
\Bigr)
\Bigr\vert \bigotimes_{j \in F_N^+} 
d \widehat{m}^j =0.
\end{equation*}
Since each 
$\widetilde \beta^{N,\varepsilon,k}(t,
\cdot)$
is bounded and zero outside $\{ (\widehat m^j)_{j \in F_N^+}
\in 
{\mathbb C}^{ \vert F_N^+ \vert} 
\simeq 
{\mathbb R}^{2 \vert F_N^+ \vert} 
: 
(\widehat m^j \widehat f_N^j)_{j \in F_N^+}
\in {\mathcal O}_N\}$, we can assume that the functions  
$(\check \beta^{N,\varepsilon,k,n})_{n \in {\mathbb N}}$
are uniformly bounded and are zero outside a common bounded subset of ${\mathbb R}^{2 \vert F_N^+\vert}$.
Denoting by 
$\check \psi^{N,\varepsilon,\rho,n}((\widehat{m}^j)_{j \in F_N^+})$ the right-hand side in 
\eqref{eq:main:bound:phi:N,varepsilon,rho}
when $\widetilde \beta^{N,\varepsilon,k}$ therein is replaced by 
 $\check \beta^{N,\varepsilon,k,n}$
 and using the local Lipschitz property of 
 $H$, we deduce that 
 \begin{equation*}
 \lim_{n \rightarrow \infty} \int_{{\mathbb R}^{2 \vert F_N^+ \vert}} \Bigl\vert 
\widetilde  \psi^{N,\varepsilon,\rho}\Bigl(t,\bigl( \widehat m^j\bigr)_{j \in F_N^+} \Bigr) 
 - 
 \check \psi^{N,\varepsilon,\rho,n}\Bigl( 
 \bigl( \widehat m^j\bigr)_{j \in F_N^+}
 \Bigr) 
 \biggr\vert \bigotimes_{j \in F_N^+} 
d \widehat{m}^j =0,
 \end{equation*}
 the convergence being uniform with respect to $\rho$ satisfying the prescription 
 of Definitions
 \ref{def:admiss:threshold:smoothing}
and
\ref{def:mollification}. 
 Using the fact that, for each fixed $n$, 
 $\check \psi^{N,\varepsilon,\rho,n}$
 converges to $0$ in sup norm as $\rho$ tends to $\delta_0$, we 
deduce that 
\begin{equation}
 \label{eq:prop:4:9:fifth:step:2}
\begin{split}
\lim_{\rho \rightarrow \delta_0} 
\int_{{\mathbb R}^{2 \vert F_N^+ \vert}}
\Bigl\vert \widetilde \psi^{N,\varepsilon,\rho}\Bigl(t,\bigl(\widehat{m}^j\bigr)_{j \in F_N^+}
\Bigr) \Bigr\vert \bigotimes_{j \in F_N^+} d \widehat{m}^j
=0.
\end{split}
\end{equation}
Of course, the rate in the above convergence depends on $N$
and $\varepsilon$. 
Importantly, 
for $N$ and $\varepsilon$ fixed,
the function 
$\psi^{N,\varepsilon,\rho}(t,\cdot)$
is bounded uniformly in $\rho$ and $t$: this follows from 
Proposition 
\ref{prop:4:7}.
In short, this provides the function $R^{N,\varepsilon,\rho}$ in the statement. 
\vspace{5pt}

\textit{Fifth Step.} 
By combining the third and fourth steps, 
we now rewrite 
\eqref{defn:HJB:gen:2}
in the following form: 
 \begin{equation}
 \label{eq:prop:4:9:fifth:step:1}
 \begin{split}
&\biggl\vert \int_{{\mathbb T}^d} H \Bigl( y, 
\partial_{\mu} W^{N,\varepsilon,\rho}(t,m)(y) 
 \Bigr) d\bigl( m * f_N \bigr) (y) + 
 \psi^{N,\varepsilon,\rho}(t,m) 
\\
&\hspace{15pt}
- 
\int_{{\mathbb R}^{2 \vert F_N^+\vert}} 
\biggl\{
\int_{{\mathbb T}^d}
H \biggl( y, \i 2 \pi \sum_{k \in F_N} k \partial_{\widehat m^k} W\bigl(t,m^{N,\varepsilon}(r) * f_N\bigr) e_{k}(y) \biggr)  
 d \bigl( m * f_N\bigr)(y)
\biggr\} 
\\
&\hspace{300pt} \times \prod_{j \in F_N^+} \rho\bigl(  \widehat{r}^j \bigr)
\bigotimes_{j \in F_N^+} 
  d \widehat{r}^j
 \biggr\vert
\\
&\leq \eta_{N,\varepsilon}(c). 
\end{split}
\end{equation}
Back to the second step, we can compare the term on the second line with $T_1^{N,\varepsilon,\rho}$ introduced in the second step. 
The main difference between the two comes from the 
measure that is used for integrating in $y$. Here, it is 
$m*f_N$, while it is 
$m^{N,\varepsilon}(r)*f_N$ in 
$T_1^{N,\varepsilon,\rho}$.
In order to pass from one term to the other, we recall from the choice of $\delta$
in 
Definition 
\ref{def:admiss:threshold:smoothing}
 that 
\begin{equation*}
\bigl\| m - m^{N,\varepsilon}(r) \bigr\|_{\infty} \leq C(c)  \varepsilon,
\end{equation*}
when 
 $r=(\widehat r^k)_{k \in F_N^+}$ is in the support of $\bigotimes_{k \in F_N^+} \rho$
and
for $C(c)$ depending on $c$. 
Therefore,
\begin{equation}
\label{eq:cv:proof:mollif:convolution}
\bigl\| m  * f_N- m^{N,\varepsilon}(r) *f_N \bigr\|_{\infty} \leq C(c)  \varepsilon.
\end{equation}
By using once again 
Proposition 
\ref{prop:4:7}
in order to bound the $L^\infty$-norm of the function 
$y \mapsto 
\i 2 \pi \sum_{k \in F_N} k \partial_{\widehat m^k} W(t,m^{N,\varepsilon}(r) * f_N) e_{k}(y)$
and 
 by invoking 
 the local Lipschitz property of $H$, 
we have
\begin{equation*}
\begin{split}
&\biggl\vert \int_{{\mathbb R}^{2 \vert F_N^+\vert}} 
\biggl\{
\int_{{\mathbb T}^d}
H \biggl( y, \i 2 \pi \sum_{k \in F_N} k \partial_{\widehat m^k} W\bigl(t,m^{N,\varepsilon}(r) * f_N\bigr) e_{k}(y) \biggr)  
 d \bigl( m * f_N\bigr)(y)
\biggr\} 
\\
&\hspace{300pt} \times \prod_{j \in F_N^+} \rho\bigl(  \widehat{r}^j \bigr)
\bigotimes_{j \in F_N^+} 
  d \widehat{r}^j
  \\
  &\hspace{15pt} 
  -\int_{{\mathbb R}^{2 \vert F_N^+\vert}} 
\biggl\{
\int_{{\mathbb T}^d}
H \biggl( y, \i 2 \pi \sum_{k \in F_N} k \partial_{\widehat m^k} W\bigl(t,m^{N,\varepsilon}(r) * f_N\bigr) e_{k}(y) \biggr)  
 d \bigl( m^{N,\varepsilon}(r) * f_N\bigr)(y)
\biggr\} 
\\
&\hspace{300pt} \times \prod_{j \in F_N^+} \rho\bigl(  \widehat{r}^j \bigr)
\bigotimes_{j \in F_N^+} 
  d \widehat{r}^j
  \biggr\vert
  \\ 
  &\leq C(c) \varepsilon. 
\end{split}
\end{equation*}
As we already mentioned,  
the second term in the left-hand side is nothing but
$T_1^{N,\varepsilon,\rho}(t,m)$. 
Therefore, 
\eqref{eq:prop:4:9:fifth:step:1}
yields
 \begin{equation}
\label{defn:HJB:gen:27}
\begin{split}
&\biggl\vert \int_{{\mathbb T}^d} H \Bigl( y, 
\partial_{\mu} W^{N,\varepsilon,\rho}(t,m)(y) 
 \Bigr) d\bigl( m^{N,\varepsilon}(r) * f_N \bigr) (y) 
 + 
 \psi^{N,\varepsilon,\rho}(t,m) 
-
T_1^{N,\varepsilon,\rho}(t,m)
\biggr\vert
\leq \eta_{N,\varepsilon}(c).
\end{split}
\end{equation}
\vspace{5pt}

\textit{Sixth Step.}
We now address the term $T_2^{N,\varepsilon,\rho}$ defined in the second step. 
We recall 
\begin{equation*}
T_2^{N,\varepsilon,\rho}(t,m)
 = \sum_{k \in F_N} 2 \pi^2 \vert k\vert^2  \widehat{f}^k_N \biggl( \int_{{{\mathbb R}^{2 \vert F_N^+\vert}}} 
\partial_{ \widehat{m}^{k} }
W\bigl(t,m^{N,\varepsilon}(r)*f_N\bigr) \widehat{m^{N,\varepsilon}(r)}^k 
\prod_{j \in F_N^+} \rho\bigl(  \widehat{r}^j \bigr)
\bigotimes_{j \in F_N^+} 
  d \widehat{r}^j 
\biggr).
\end{equation*}
Back to the analysis of 
$T[m,W]$ in 
Subsection 
\ref{subse:spatial:deri:interpretation}, we have
\begin{equation*}
\begin{split}
&\frac12 
\int_{{\mathbb T}^d} \Bigl[ 
\frac{\delta}{\delta m} W^{N,\varepsilon,\rho}(t,m)(y) \Bigr] \Delta m(y) dy
= 
 - \sum_{k \in F_N} 2 \pi^2 \vert k\vert^2 
\partial_{ \widehat{m}^{k} }
W^{N,\varepsilon,\rho}(t,m) \widehat{m}^{k}.
\end{split}
\end{equation*}
Using Equation \eqref{eq:Fourier:derivative:convolution:regularisation:0}, we  get
\begin{equation*}
\begin{split}
&\frac12 
\int_{{\mathbb T}^d} \Bigl[ 
\frac{\delta}{\delta m} W^{N,\varepsilon,\rho}(t,m)(y) \Bigr] \Delta m(y) dy
\\
&= 
 -( 1-\varepsilon ) \int_{{{\mathbb R}^{2 \vert F_N^+\vert}}} 
\biggl(  \sum_{k \in F_N} 2 \pi^2 \vert k \vert^2 
\partial_{\hat{m}_k}  W \bigl(t,m^{N,\varepsilon}(r) * f_N\bigr) 
\hat{f}^{k}_N 
\widehat{m}^{k} \biggr)
\prod_{j \in F_N^+} \rho\bigl(  \widehat{r}^j \bigr)
\bigotimes_{j \in F_N^+} 
  d \widehat{r}^j .
\end{split}
\end{equation*}
We then proceed as in the second step in order to remove the product $(1-\varepsilon) \widehat{f}^k_N$ from the summand in the right-hand side. 
In short, 
we observe that the 
family of Fourier coefficients
$(\widehat{m}^k)_k$ is square-integrable, uniformly in 
$m \in B(c)$ (see 
\eqref{eq:compact:Fourier}). 
Moreover,
by Proposition 
\ref{prop:4:7},
 the sum
$\sum_{k \in F_N} \vert k \vert^4 \vert \partial_{\hat{m}^k}  W (t,m^{N,\varepsilon}(r) * f_N) \vert^2$ is finite, uniformly in 
$N$,
$\varepsilon$, 
$m$ and $r$. 
Therefore, we deduce that 
(as before, for any  $t \in [0,T]$ and any $m  \in B(c)$)
\begin{equation*}
\begin{split}
&\biggl\vert 
\frac12 \int_{{\mathbb T}^d} \Bigl[ 
\frac{\delta}{\delta m} W^{N,\varepsilon,\rho}(t,m)(y) \Bigr] \Delta m(y) dy
\\
&\hspace{15pt} + (1-\varepsilon)
 \int_{{{\mathbb R}^{2 \vert F_N^+\vert}}} 
\biggl(  \sum_{k \in F_N} 2 \pi^2 \vert k \vert^2 
\partial_{\widehat{m}^k}  W \bigl(t,m^{N,\varepsilon}(r) * f_N\bigr) 
\widehat{m}^{k} \biggr)
\prod_{j \in F_N^+} \rho\bigl(  \widehat{r}^j \bigr)
\bigotimes_{j \in F_N^+} 
  d \widehat{r}^j
\biggr\vert \leq \eta_{N,\varepsilon}(c),
\end{split}
\end{equation*}
which gives
\begin{equation*}
\begin{split}
&\biggl\vert 
\frac12 \int_{{\mathbb T}^d} \Bigl[ 
\frac{\delta}{\delta m} W^{N,\varepsilon,\rho}(t,m)(y) \Bigr] \Delta m(y) dy
\\
&\hspace{15pt} + 
 \int_{{{\mathbb R}^{2 \vert F_N^+\vert}}} 
\biggl(  \sum_{k \in F_N} 2 \pi^2 \vert k \vert^2 
\partial_{\widehat{m}^k}  W \bigl(t,m^{N,\varepsilon}(r) * f_N\bigr) 
\widehat{m^{N,\varepsilon}(0)}^{k} \biggr)
\prod_{j \in F_N^+} \rho\bigl(  \widehat{r}^j \bigr)
\bigotimes_{j \in F_N^+} 
d \widehat{r}^j
\biggr\vert \leq \eta_{N,\varepsilon}(c).
\end{split}
\end{equation*}
Recall now that, from our choice of $\delta_{N,\varepsilon}$
in 
Definition 
\ref{def:admiss:threshold:smoothing},
\begin{equation*}
\sum_{k \in \vert F_N\vert}
\vert k \vert^2 \bigl\vert \widehat{m^{N,\varepsilon}(0)}^{k}  - 
\widehat{m^{N,\varepsilon}(r)}^{k} \bigr\vert \leq \varepsilon,
\end{equation*} 
when 
 $(\widehat r^k)_{k \in F_N^+}$ is in the support of $\bigotimes_{k \in F_N^+} \rho$. 
 We deduce that 
 \begin{equation*}
\begin{split}
&\biggl\vert 
\frac12 \int_{{\mathbb T}^d} \textrm{\rm Trace} \Bigl[ 
\partial_y \partial_\mu W^{N,\varepsilon,\rho}(t,m)(y) \Bigr]  m(y) dy
 + 
T_2^{N,\varepsilon,\rho}(t,m)
\biggr\vert
\\
&= \biggl\vert 
\frac12 \int_{{\mathbb T}^d} \Bigl[ 
\frac{\delta}{\delta m} W^{N,\varepsilon,\rho}(t,m)(y) \Bigr] \Delta m(y) dy
 + 
T_2^{N,\varepsilon,\rho}(t,m)
\biggr\vert \leq \eta_{N,\varepsilon}(c).
\end{split}
\end{equation*}
Together 
with 
\eqref{eq:prop:4:9:T1Nepsilonrho},
\eqref{eq:prop:4:9:fifth:step:2}
and
 \eqref{eq:prop:4:9:fifth:step:1}, this completes the proof. 
\end{proof}

\subsection{Uniqueness of Lipschitz, semi-concave admissible generalized solutions}

Here is now the main statement of this section: 
\begin{thm}
\label{thm:uniqueness:HJB}
There is a unique 
generalized solution to the HJB equation \eqref{eq:HJB} that is 
$d_{-2}$-Lipschitz and semi-concave in space uniformly in time.
\end{thm}

\begin{proof}
\textit{Main lines.}
We use the notation 
from Definition 
\ref{defn:HJB:gen}, namely 
$A_N(c)=[0,T] \times B_N(c)$ with $B_N(c)= {\mathcal P}_N \cap B(c)$ and 
$B(c):= \{ m :  \sup_{x \in \bT^d} \vert \nabla m(x)  \vert \leq c, \ \inf_{x \in \bT^d}  m(x)   \geq 1/c \}$, for $c \geq 1$. 

Given two solutions $V_1$ and $V_2$ of \eqref{eq:HJB}
that are $d_{-2}$-Lipschitz and semi-concave in space, 
uniformly in time, 
we let 
\begin{equation}
\label{eq:DHlambda}
\begin{split}
&D {\mathcal H}_{\lambda}^{N,\varepsilon,\rho}(t,m)(x) :=
\partial_p H 
\Bigl( x, \lambda 
 \partial_\mu V_1^{N,\varepsilon,\rho}(t,m)(x) 
 + (1- \lambda) 
 \partial_\mu V_2^{N,\varepsilon,\rho}(t,m)(x) 
 \Bigr),
\\
&D {\mathcal H}^{N,\varepsilon,\rho}(t,m)(x) :=
\int_0^1 
D {\mathcal H}_{\lambda}^{N,\varepsilon,\rho}(t,m)(x)
d \lambda. 
\end{split}
\end{equation}
For any $c \geq 1$, 
there exist two families $(\eta_{N,\varepsilon})_{N,\varepsilon}$ and 
$(R^{N,\varepsilon,\rho})_{N,\varepsilon,\rho}$ as in 
Proposition 
\ref{prop:generalized:solution} such that, for almost every $t \in [0,T]$, for any 
$m \in B(c)$,  
\begin{align}
&\biggl\vert \partial_t \bigl( V_1^{N,\varepsilon,\rho} - V_2^{N,\varepsilon,\rho} \bigr)(t,m) 
+ \frac12 \int_{{\mathbb T}^d} 
\textrm{\rm Trace} \Bigl[ \partial_x \partial_\mu \bigl( V_1^{N,\varepsilon,\rho} - V_2^{N,\varepsilon,\rho} \bigr)(t,m)(x) 
\Bigr] dm(x) \nonumber
\\
&\hspace{15pt} -  \int_{{\mathbb T}^d}
D {\mathcal H}^{N,\varepsilon,\rho}(t,m)(x)
  \cdot \partial_{\mu} 
\bigl( V_1^{N,\varepsilon,\rho} - V_2^{N,\varepsilon,\rho} \big)(t,m)(x) d \bigl( m*f_N\bigr)(x)
 \biggr\vert 
\label{eq:uniqueness:proof:difference:V1:V2}
\\
&\leq \eta_{N,\varepsilon} + 
R^{N,\varepsilon,\rho}
\Bigl( t, \bigl( \widehat{m}^k\bigr)_{k \in F_N^+} \Bigr). 
\nonumber
\end{align}
The very idea of the proof is to consider the characteristics associated with the differential operator appearing in the left-hand side. 
Strictly speaking, this should be a (non-linear) Fokker-Planck equation driven by the field 
$(t,m,x) \mapsto  D {\mathcal H}^{N,\varepsilon,\rho}(t,m)(x)$. 
However, this is not what
we use next. 
Because of our definition of a generalized solution, which is based on a truncation 
of the Fourier expansion of the measure argument $m$, we switch to `approximate characteristics', given by the Fokker-Planck equation:
\begin{equation}
\label{eq:uniqueness:MKV}
\partial_t m_t (x)- \textrm{\rm div} \Bigl(D {\mathcal H}^{N,\varepsilon,\rho}(t,m_t)(x) \bigl(m_t*f_N\bigr)(x)
\Bigr) - \frac12 \Delta m_t(x) =0,
\end{equation}
with an initial condition $m_0$ such that $\widehat{m}_0^k =0$ for $\vert k \vert \geq N$, for some integer $N$,
and $m_0\geq 1/c$ for some $c\geq 1$.  
Importantly, the Fokker-Planck equation has a solution that is 
a probability measure 
for $N$
large enough 
(whilst the total mass is obviously preserved by the dynamics, the sign may not be). 
Even more, 
{we can find a constant $c' \geq 1$ (depending on $c$) such that, for $N$ large enough, 
the solution 
$(m_t)_{0 \leq t \leq T}$
is lower bounded by $1/c'$ and its gradient (in space) is bounded by $c'$, see 
Theorem \ref{thm:FPK:mollified:proof}.}
The thrust of 
\eqref{eq:uniqueness:MKV}
 is that the 
equation 
satisfied by the Fourier coefficients $(\widehat{m}_t^k)_{k \in F_N^+}$
is closed (meaning that it does not depend on the higher Fourier modes  
$(\widehat{m}_t^k)_{k \not \in F_N^+}$): intuitively, 
equation \eqref{eq:uniqueness:MKV}
is `almost' a finite dimensional dimensional ordinary differential equation. 


Once 
\eqref{eq:uniqueness:MKV}
has been solved, we want to plug 
$(m_t )_{0 \leq t \leq T}$ into 
\eqref{eq:uniqueness:proof:difference:V1:V2}
and then expand
$((V_1^{N,\varepsilon,\rho}-V_2^{N,\varepsilon,\rho})(t,m_t))_{0 \leq t \leq T}$. 
The very main point of the proof is then to show that, whenever 
$N$ and $\varepsilon$ are fixed, 
the remainder
\begin{equation}
\label{eq:uniqueness:proof:difference:V1:V2:2}
\Bigl( R^{N,\varepsilon,\rho}\bigl(t,( \widehat m_t^j)_{j \in F_N^+}\bigr) \Bigr)_{0 \leq t \leq T}
\end{equation}
tends to $0$ when $\rho$ tends to the Dirac mass $\delta_0$ (compare if needed with the statement of 
Proposition 
\ref{prop:generalized:solution}). 
This is however a challenging fact because 
$(m_t)_{0 \leq t \leq T}$
itself 
depends on $\rho$
(although the notation does not make it clear). 
And, even more, the reader must recall that, although 
$R^{N,\varepsilon,\rho}$ is known to tend to $0$ as $\rho$ tends to $\delta_0$ (when 
$N$ and $\varepsilon$ are fixed), the convergence just holds true in $L^1$ (once again, 
we refer to the statement of  
Proposition 
\ref{prop:generalized:solution}). 
Some work is thus needed to adapt the result of 
Proposition 
\ref{prop:generalized:solution}
to 
\eqref{eq:uniqueness:proof:difference:V1:V2:2}. 

Roughly speaking, the idea is to prove that the flow generated by the Fokker-Planck equation 
\eqref{eq:uniqueness:MKV}
cannot accumulate mass: If, instead of $m_0$,
 the initial condition
of 
\eqref{eq:uniqueness:MKV} 
is taken as a small random perturbation of $m_0$ in $B_N(c)$,
with the perturbation having a bounded density with respect to the Lebesgue measure
on ${\mathbb R}^{2 \vert F_{N}^+ \vert}$, then the law of 
$(\widehat{m}_t^k )_{k \in F_N^+}$
with respect to the Lebesgue measure on ${\mathbb R}^{2 \vert F_{N}^+\vert}$
also has a bounded density, the bound depending on $N$ and 
$\varepsilon$, but not on $t$ and $\rho$.
This claim is 
stated in the form of Lemma 
\ref{lem:flow} below, which we use in the rest of this proof and which we prove separately. 
\vskip 5pt

\textit{Application.}
Following our plan, we take an initial condition 
$m_0 \in B_N(c)$. 
For 
$(\widehat{r}^k)_{k \in F_N^+} \in {\mathbb R}^{2 \vert F_N^+ \vert}$, we let
$m_0(r)$ be the element of ${\mathcal P}_N$ defined by:
\begin{equation*}
\widehat{m_0(r)}^k := \left\{
\begin{array}{ll}
\widehat{m}_0^k + \widehat{r}^k, \quad &k \in F_N \setminus \{0\},
\\
0, \quad &k \not \in F_N,
\end{array} 
\right.
\end{equation*} 
with the usual convention that $\widehat{r}^{-k} = \overline{\widehat{r}^k}$. 
If we choose $(\widehat{r}^k)_{k \in F_N \setminus \{0\}}$
such that 
\begin{equation}
\label{eq:proof:uniqueness:condition:rk}
\sum_{k \in F_N \setminus \{0\}} \vert k \vert \vert \widehat{r}^k \vert < \frac1{2c}, 
\end{equation}
then $m_0(r) \in B_N(2c)$ and we can still solve the Fokker-Planck equation 
\eqref{eq:uniqueness:MKV}  
with $m_0(r)$ as initial condition 
provided that $N\geq N_c$ for some threshold $N_c$ only depending on $c$, on the regularity properties of $H$ in the variables $x$ and $p$ and on the $d_{-2}$-Lipschitz constant of $V_1$ and $V_2$, see Theorem 
\ref{thm:FPK:mollified:proof}
in the appendix. We call the solution 
$(m_t(r))_{0 \leq t \leq T}$: 
Theorem 
\ref{thm:FPK:mollified:proof}
asserts that $(m_t(r))_{0 \leq t \leq T}$ forms a flow of probability measures. 
Even more, 
Theorem 
\ref{thm:FPK:mollified:proof}
also says that
there exists $c' \geq 1$, independent of $N$, $\varepsilon$ and $\rho$ (but depending on $c$), such that 
 each 
 $m_t(r)$ belongs to $B(c')$. 

We now expand $(
( V_1^{N,\varepsilon,\rho} - V_2^{N,\varepsilon,\rho} )
(t,m_t(r)))_{0 \leq t \leq T}$
by means of the chain rule for 
real-valued functions defined on 
$[0,T] \times {\mathcal P}({\mathbb T}^d)$. 
Although this chain rule is well documented (see for instance \cite[Chapter 5]{CarmonaDelarue_book_I}), it here requires 
a modicum of care since the time derivative
of 
$(V_1^{N,\varepsilon,\rho}-V_2^{N,\varepsilon,\rho})$
 is just known to exist for almost every $t$, for all $m$, see
 Corollary 
\ref{cor:mollif:time-space}. 
The proof is as follows. 
If $(V_1^{N,\varepsilon,\rho} - V_2^{N,\varepsilon,\rho})$ were smooth, then
\cite[Chapter 5]{CarmonaDelarue_book_I}
would yield
\begin{equation}
\label{eq:ito:chain:rule:V1-V2}
\begin{split}
&\bigl[ V_1^{N,\varepsilon,\rho}-V_2^{N,\varepsilon,\rho} \bigr] \bigl(T,m_T(r)\bigr)
=
\bigl[ V_1^{N,\varepsilon,\rho}-V_2^{N,\varepsilon,\rho} \bigr] \bigl(0,m_0(r)\bigr)
\\
&\hspace{15pt} + \int_0^T  \biggl\{ \partial_t 
\bigl[ V_1^{N,\varepsilon,\rho}-V_2^{N,\varepsilon,\rho} \bigr] 
\bigl(t,m_t(r)   \bigr)
\\
&\hspace{15pt}- \int_{{\mathbb T}^d}  D {\mathcal H}^{N,\varepsilon,\rho}\bigl(t,m_t(r)\bigr)(x) \cdot \partial_{\mu} 
\bigl[ V_1^{N,\varepsilon,\rho} - V_2^{N,\varepsilon,\rho} \bigr] \bigl(t,m_t(r)   \bigr)(x) d\bigl[m_t(r) \bigr](x)
\\
&\hspace{15pt} + \frac12 \int_{{\mathbb T}^d} 
\textrm{\rm Trace} \Bigl[ \partial_x \partial_\mu \bigl[ V_1^{N,\varepsilon,\rho} - V_2^{N,\varepsilon,\rho} \bigr]\bigl(t,m_t(r)  \bigr)(x) 
\Bigr] d[m_t(r) ](x) \biggr\} dt.
\end{split}
\end{equation}
Under our assumption, we can apply 
the above chain rule
 to
 a mollified version of $(V_1^{N,\varepsilon,\rho}-
 V_2^{N,\varepsilon,\rho})$, 
 for instance
  $\int_{\mathbb R} [V_1^{N,\varepsilon,\rho} - V_2^{N,\varepsilon,\rho}](t-s,m)  \varsigma(s) ds$, 
  for a smooth compactly supported density 
  $\varsigma$ on ${\mathbb R}$. 
The most difficult term to handle in the resulting form of 
\eqref{eq:ito:chain:rule:V1-V2} is 
\begin{equation*}
 \int_0^T \int_{\mathbb R}  \partial_t 
\bigl[ V_1^{N,\varepsilon,\rho}-V_2^{N,\varepsilon,\rho} \bigr] 
\bigl(t-s,m_t(r)   \bigr) \varsigma(s) ds dt.
\end{equation*}
By
 Corollary 
\ref{cor:mollif:time-space}, for almost every 
$t \in [0,T]$,  the function 
$(t,m) \mapsto 
\partial_t 
[ V_1^{N,\varepsilon,\rho}-V_2^{N,\varepsilon,\rho} ] 
(t,m)$ is continuous in $m$, with a modulus of continuity independent of $t$.  Therefore we can easily approximate the argument 
$m_t(r)$ appearing in the integrand
by 
$\sum_{i=1}^N {\mathbf 1}_{[t_{i-1},t_i)}(t) m_{t_i}(r)$
for a finite subdivision 
$t_0=0<t_1<\cdots <t_n=T$
of $[0,T]$ with an arbitrarily small step. Using this approximation, it is pretty easy to prove that 
\begin{equation*}
 \int_0^T \int_{\mathbb R}  \partial_t 
\bigl[ V_1^{N,\varepsilon,\rho}-V_2^{N,\varepsilon,\rho} \bigr] 
\bigl(t-s,m_t(r)   \bigr) \varsigma(s) ds dt
\rightarrow 
 \int_0^T   \partial_t 
\bigl[ V_1^{N,\varepsilon,\rho}-V_2^{N,\varepsilon,\rho} \bigr] 
\bigl(t,m_t(r)   \bigr) \varsigma(s)  dt,
\end{equation*}
as $\varsigma$ tends to $\delta_0$ in the weak sense. 
As a corollary, we deduce that 
\eqref{eq:ito:chain:rule:V1-V2}
holds true under the standing assumption. 
Since $V_1^{N,\varepsilon,\rho}$ and $V_2^{N,\varepsilon,\rho}$ coincide at time 
$T$, we deduce from 
\eqref{eq:uniqueness:proof:difference:V1:V2} that
\begin{equation*}
\begin{split}
&\Bigl\vert \bigl( V_1^{N,\varepsilon,\rho} - V_2^{N,\varepsilon,\rho} \bigr)\bigl(0,m_0(r)\bigr) \Bigr\vert
\leq \eta_{N,\varepsilon} + \int_0^T R^{N,\varepsilon,\rho} \Bigl(t,\bigl(\reallywidehat{\bigl\{ m_t(r) \bigr\}}^k\bigr)_{k \in F_N^+} \Bigr) dt.
\end{split}
\end{equation*}

For $\rho_0$ a kernel similar to $\rho$, with a radius (compare if needed with 
Definition 
\ref{def:admiss:threshold:smoothing}, the choice of the radius being here adapted to the constraint 
\eqref{eq:proof:uniqueness:condition:rk}) less than $1/(2c \sqrt{d N} \vert F_N\vert)$, we obtain
\begin{equation*}
\begin{split}
&\int_{{\mathbb R}^{2\vert F_N^+ \vert}} 
\Bigl\vert \bigl( V_1^{N,\varepsilon,\rho} - V_2^{N,\varepsilon,\rho} \bigr)\bigl(0,m_0(r)\bigr) \Bigr\vert
\prod_{j \in F_N^+} \rho_0 \bigl(  \widehat{r}^j \bigr)
\bigotimes_{j \in F_N^+} 
  d \widehat{r}^j 
\\
&\hspace{15pt} \leq \eta_{N,\varepsilon} + \int_0^T \biggl[\int_{{\mathbb R}^{2\vert F_N^+\vert}}  R^{N,\varepsilon,\rho}\Bigl(t,\bigl(\reallywidehat{\bigl\{ m_t(r) \bigr\}}^k\bigr)_{k \in F_N^+} \Bigr)
\prod_{j \in F_N^+} \rho_0\bigl(  \widehat{r}^j \bigr)
\bigotimes_{j \in F_N^+} 
  d \widehat{r}^j 
\biggr]
 dt.
 \end{split}
 \end{equation*}
Here is now the 
place where we invoke Lemma 
\ref{lem:flow} (whose statement is given next): For any $t \in [0,T]$, we 
perform the change of variable  
$y=m_t(r)$, i.e. 
$\widehat{y}^k = \widehat{m_t(r)}^k$, for 
$k \in F_N^+$. As a corollary of 
Lemma \ref{lem:flow}, we deduce that there exists a constant $C_{N,\rho_0}$, depending on $N$ and 
$\rho_0$, but independent of $\varepsilon$ and $\rho$,
 such that 
\begin{equation}
\label{eq:proof:uniqueness:HJB:for:master:0}
\begin{split}
&\int_{{\mathbb R}^{2\vert F_N^+ \vert}} 
\Bigl\vert \bigl( V_1^{N,\varepsilon,\rho} - V_2^{N,\varepsilon,\rho} \bigr)\bigl(0,m_0(r)\bigr) \Bigr\vert
\prod_{j \in F_N^+} \rho_0 \bigl(  \widehat{r}^j \bigr)
\bigotimes_{j \in F_N^+} 
  d \widehat{r}^j 
\\
&\hspace{15pt} \leq \eta_{N,\varepsilon} + C_{N,\rho_0} 
\int_0^T \biggl[\int_{{\mathbb R}^{2\vert F_N^+\vert}}  R^{N,\varepsilon,\rho}\bigl(t,(\widehat y^k)_{k \in F_N^+} \bigr)
\bigotimes_{j \in F_N^+} 
  d \widehat{y}^j  
\biggr]
 dt.
 \end{split}
 \end{equation}

From the standing assumption, 
we know 
that 
$V_1$ 
and 
$V_2$
are Lipschitz continuous with respect to $d_{-2}$ and thus with respect to $d_{W_1}$. 
Therefore, 
by a straightforward adaptation of the proof of Lemma \ref{lem:convergence:mollif}, 
$V_1^{N,\varepsilon,\rho}(0,\cdot)$ (resp. $V_2^{N,\varepsilon,\rho}(0,\cdot)$)
converges uniformly to 
$V_1(0,\cdot)$ (resp. $V_2(0,\cdot)$) as 
$(N,\varepsilon)$ tends to $(\infty,0)$. 
Hence, 
\begin{equation*}
\begin{split}
&\lim_{(N,\varepsilon) \rightarrow (\infty,0)}
\biggl\vert 
\int_{{\mathbb R}^{2 \vert F_N^+\vert}} 
\Bigl\vert \bigl( V_1^{N,\varepsilon,\rho} - V_2^{N,\varepsilon,\rho} \bigr)\bigl(0,m_0(r)\bigr) \Bigr\vert
\prod_{j \in F_N^+} 
\rho_0(\widehat{r}^j) \otimes_{j \in F_N^+} d \widehat{r}^j
\\
&-
\int_{{\mathbb R}^{2 \vert F_N^+\vert}} \vert (V_1-V_2)(0,m_0(r))\vert \prod_{j \in F_N^+} 
\rho_0(\widehat{r}^j) \otimes_{j \in F_N^+} d \widehat{r}^j
\biggr\vert 
=0.
\end{split}
\end{equation*}
Therefore, 
\eqref{eq:proof:uniqueness:HJB:for:master:0}
can be rewritten in the following form: 
\begin{equation*}
\begin{split}
&\int_{{\mathbb R}^{2\vert F_N^+ \vert}} 
\Bigl\vert \bigl( V_1  - V_2  \bigr)\bigl(0,m_0(r)\bigr) \Bigr\vert
\prod_{j \in F_N^+} \rho_0 \bigl(  \widehat{r}^j \bigr)
\bigotimes_{j \in F_N^+} 
  d \widehat{r}^j 
\\
&\hspace{15pt} \leq \eta_{N,\varepsilon} + C_{N,\rho_0} 
\int_0^T \biggl[\int_{{\mathbb R}^{2\vert F_N^+\vert}}  R^{N,\varepsilon,\rho}\bigl(t,(\widehat y^k)_{k \in F_N^+} \bigr)
\bigotimes_{j \in F_N^+} 
  d \widehat{y}^j  
\biggr]
 dt,
 \end{split}
 \end{equation*}
for a possibly new value of the family $(\eta_{N,\varepsilon})_{N,\varepsilon}$, which still satisfies 
$\lim_{(N,\varepsilon) \rightarrow (\infty,0)} \eta_{N,\varepsilon} = 0$, and then, 
using the form of $\rho_0$ and 
following the proof of Lemma 
\ref{lem:convergence:mollif}, 
\begin{equation}
\label{eq:proof:uniqueness:HJB:for:master}
\begin{split}
&\Bigl\vert \bigl( V_1  - V_2  \bigr)\bigl(0,m_0\bigr) \Bigr\vert
 \leq \eta_{N,\varepsilon} + C_{N,\rho_0} 
\int_0^T \biggl[\int_{{\mathbb R}^{2\vert F_N^+\vert}}  R^{N,\varepsilon,\rho}\bigl(t,(\widehat y^k)_{k \in F_N^+} \bigr)
\bigotimes_{j \in F_N^+} 
  d \widehat{y}^j  
\biggr]
 dt.
 \end{split}
 \end{equation}
We recall that 
\eqref{eq:proof:uniqueness:HJB:for:master}
holds true for 
$m_0 \in B_N(c)$. Next, we can take 
$m_0 \in B_{N_0}(c)$, for some fixed $N_0$ and then choose $N \geq N_0$, in which case 
we indeed have $m_0 \in B_N(c)$.  
Now, 
for a given threshold $\epsilon >0$ (the reader should distinguish from $\varepsilon$), we can choose $N$ and $\varepsilon$ such that 
$\eta_{N,\varepsilon} < \epsilon/2$ and then, by the second identity in 
Proposition 
\ref{prop:generalized:solution}, 
$\delta$ (encoding the support of $\rho$) small enough such that the second term in the right-hand side is also less than $\epsilon/2$. 
Therefore, 
the left-hand side of 
\eqref{eq:proof:uniqueness:HJB:for:master}
can be made smaller than $\epsilon$. 
Since $\epsilon$ is arbitrary,
we easily deduce that 
$V_1(0,m_0)$ and $V_2(0,m_0)$ must coincide. 
This holds for any $m_0 \in B_{N_0}(c)$. By an obvious density argument, 
we obtain uniqueness. 
\end{proof}

\begin{lem}
\label{lem:flow}
For a real $c \geq 1$, recall the notation $B_N(c)= {\mathcal P}_N \cap \{ m :  \sup_{x \in \bT^d} \vert \nabla m(x)  \vert \leq c, \ \inf_{x \in \bT^d}  m(x)   \geq 1/c \}$
and let
 $C_N(c) := \{ 
(\widehat z^k)_{k \in F_N^+} \in {\mathbb R}^{2 \vert F_N^+ \vert} : 
\sum_{k \in F_N \setminus \{0\}} \vert k \vert \vert \widehat{z}^k \vert < 1/(2c) \}$, 
with the usual convention that $\widehat{z}^{-k} = \overline{\widehat{z}^k}$.  

For an element $m_0 \in B_N(c)$ 
and for 
$r=(\widehat{r}^k)_{k \in F_N^+} \in C_N(c)$, define $m_0(r)$ the element of $B_N(2c)$
by
\begin{equation*}
\widehat{m_0(r)}^k := \left\{
\begin{array}{ll}
\widehat{m}_0^k + \widehat{r}^k, \quad &k \in F_N \setminus \{0\},
\\
0, \quad &k \not \in F_N,
\end{array} 
\right.
\end{equation*} 
and call $(m_t(r))_{0 \leq t \leq T}$ the solution of the non-linear Fokker-Planck equation 
\eqref{eq:uniqueness:MKV}
with $m_0(r)$ as initial condition, for the same parameters $N$, $\varepsilon$ and $\rho$ as therein.

Then, we can find a constant $N_c$, only depending on $c$, on the strict convexity 
property of $H$ in the variable $p$ and on the $d_{-2}$-Lipschitz and semi-concave properties of $V_1$
and
$V_2$
in 
\eqref{eq:uniqueness:MKV},
and a constant $K$, independent of $\varepsilon$ and $\rho$ such that, if $N \geq N_c$, 
for any 
$t \in [0,T]$, 
the law of the mapping
\begin{equation*}
r=\bigl( \widehat{r}^k \bigr)_{k \in F_N^+} \in C_N(c)
\mapsto 
\Bigl( \widehat{m_t(r)}^k \Bigr)_{k \in F_N^+} \in {\mathbb R}^{2 \vert F_N^+ \vert}
\end{equation*}
has a density, bounded by $K$. 
\end{lem}

\begin{proof}
The strategy of proof is quite clear but the implementation is demanding: 
The whole point is to address the system of Fourier coefficients 
satisfied by the solution of the Fokker-Planck equation 
\eqref{eq:uniqueness:MKV}
and to regard it as (a finite-dimensional) system of ODEs. 
The solutions generates a flow of diffeomorphisms 
and the key step is to provide a lower bound for the Jacobian that is independent of 
$t$, $\rho$ and $\varepsilon$. This is the point where semi-concavity comes in explicitly. 

Throughout the proof, we use the notation 
$V_{2-\lambda}^{N,\varepsilon,\rho}:=
 \lambda V_1^{N,\varepsilon,\rho} + (1-\lambda) V_2^{N,\varepsilon,\rho}$, and similarly without the superscript $N,\varepsilon,\rho$. 
 \vskip 5pt

\textit{First Step.}
For an initial condition $m_0 \in B_N(c)$ and for 
$k \in F_N \setminus \{0\}$, 
we have
\begin{equation}
\label{eq:linear:ode:1}
\begin{split}
\frac{d}{dt} \widehat{m}_t^k = - 2 \pi^2 \vert k \vert^2 \widehat{m}_t^k - 
\i 2  \pi \int_{{\mathbb T}^d} k  \cdot  D {\mathcal H}^{N,\varepsilon,\rho}\bigl(t,m_t\bigr)(x) e_k(x) \bigl( m_t * f_N \bigr) (x) dx.
\end{split}
\end{equation}
From 
Proposition 
\ref{prop:mollification}, we now recall that, for every $t \in [0,T]$, for any 
$x \in {\mathbb T}^d$, any $m \in {\mathcal P}({\mathbb T}^d)$ and any $\lambda \in [0,1]$, 
\begin{equation*}
\begin{split}
&\frac{\delta }{\delta m}
V_{2-\lambda}^{N,\varepsilon,\rho}(t,m)(x)
\\
& = 
(1-\varepsilon) 
\int_{{\mathbb R}^{2 \vert F_N^+\vert}}
\biggl\{
\sum_{k \in F_N  \setminus \{0\}} 
\Bigl[
\partial_{\widehat{m}^k}  V_{2-\lambda} 
 \bigl(t,m^{N,\varepsilon}(r) * f_N\bigr) \widehat{f}_N^k e_k(x) 
\Bigr]
\biggl\}
\prod_{j \in F_N^+} \rho \bigl(  \widehat{r}^j \bigr)
\bigotimes_{j \in F_N^+} 
  d \widehat{r}^j.
\end{split}
\end{equation*} 
Therefore, 
\begin{equation*}
\begin{split}
&\partial_\mu V_{2-\lambda}^{N,\varepsilon,\rho} (t,m)(x)
\\
&= \i 2 \pi 
(1-\varepsilon) 
\int_{{\mathbb R}^{2 \vert F_N^+\vert}}
\biggl\{
\sum_{k \in F_N  \setminus \{0\}} 
k
\Bigl[
\partial_{\widehat{m}^k}  V_{2-\lambda} 
\bigl(t,m^{N,\varepsilon}(r) * f_N\bigr) \widehat{f}_N^k e_{k}(x) 
\Bigr]
\biggr\}
\prod_{j \in F_N^+} \rho \bigl(  \widehat{r}^j \bigr)
\bigotimes_{j \in F_N^+} 
  d \widehat{r}^j .
\end{split}
\end{equation*}
In turn, recalling the notation 
\eqref{eq:DHlambda}, 
the last term in
\eqref{eq:linear:ode:1}
becomes:
\begin{equation*}
\begin{split}
&\i 2  \pi  \int_{{\mathbb T}^d} k \cdot  D {\mathcal H}^{N,\varepsilon,\rho} (t,m)(x) e_k(x) \bigl( m* f_N \bigr)(x) dx
\\
&= \i 2 \pi 
\int_0^1 \int_{{\mathbb T}^d} k \cdot
 \partial_p H \biggl( x, \i 2 \pi 
   (1-\varepsilon) 
   \\
&\hspace{30pt} \times   \int_{{\mathbb R}^{2 \vert F_N^+\vert}}
\sum_{\ell \in F_N\setminus \{0\}}
  \ell
 \Bigl[ \partial_{\widehat{m}^\ell}   V_{2-\lambda} 
\bigl(t,m^{N,\varepsilon}(r) * f_N \bigr)
\widehat{f}^\ell_N
 e_{\ell}(x) \Bigr]
 \prod_{j \in F_N^+} \rho \bigl(  \widehat{r}^j \bigr)
\bigotimes_{j \in F_N^+} 
  d \widehat{r}^j 
 \biggr)
\\
&\hspace{200pt} \times  e_k(x) \bigl( m * f_N \bigr)(x)  dx \, d \lambda.
\end{split}
\end{equation*}
We recall that
%
we can interpret 
$ \partial_{\widehat{m}^\ell} W(t,m^{N,\varepsilon}(r) * f_N )$ as the complex number
(see 
\eqref{eq:gradient:complex:000}):
\begin{equation*}
\begin{split}
\partial_{\widehat{m}^\ell} W\bigl(t,m^{N,\varepsilon}(r) * f_N\bigr)
&= \frac12
\partial_{\Re[\widehat{m}^\ell]} W\bigl(t,m^{N,\varepsilon}(r) * f_N\bigr)
- \frac{\i}2 
\partial_{\Im[\widehat{m}^\ell]} W\bigl(t,m^{N,\varepsilon}(r) * f_N\bigr)
\\
&= \overline{\partial_{\widehat{m}^{-\ell}} W\bigl(t,m^{N,\varepsilon}(r) * f_N\bigr)}.
\end{split}
\end{equation*}
Recalling that $\widehat{f}_N^\ell$ is real, this allows us to rewrite 
\eqref{eq:linear:ode:1} in the form: 
\begin{equation}
\begin{split}
&\frac{d}{dt} \widehat{m}_t^k
= - 2 \pi^2 \vert k \vert^2 \widehat{m}_t^k 
\\
&\hspace{5pt} 
-  \i 2 \pi 
\int_0^1 \int_{{\mathbb T}^d} k \cdot
 \partial_p H \biggl( x, - 4 \pi 
   (1-\varepsilon) 
   \\
&\hspace{30pt} \times   \int_{{\mathbb R}^{2 \vert F_N^+\vert}}
\sum_{\ell \in F_N^+}
  \ell
\Im \Bigl[ \partial_{\widehat{m}^\ell}  V_{2-\lambda}
\bigl(t,m_t^{N,\varepsilon}(r) * f_N \bigr)
 e_{\ell}(x) \Bigr]
 \widehat{f}^\ell_N
 \prod_{j \in F_N^+} \rho \bigl(  \widehat{r}^j \bigr)
\bigotimes_{j \in F_N^+} 
  d \widehat{r}^j 
 \biggr)
\\
&\hspace{200pt} \times  e_k(x) \bigl( m_t * f_N \bigr)(x) \, dx \, d \lambda.
\end{split}
\end{equation}
By integration by parts, this 
can be reformulated as
\begin{align}
&\frac{d}{dt} \widehat{m}_t^k
= - 2 \pi^2 \vert k \vert^2 \widehat{m}_t^k  
\nonumber
\\
&-  \i 2 \pi 
\int_0^1 \int_{{\mathbb T}^d} k \cdot
 \partial_p H \biggl( x,  2 \pi 
 \int_{{\mathbb R}^{2 \vert F_N^+\vert}}
V_{2-\lambda} 
\bigl(t,m_t^{N,\varepsilon}(r) * f_N \bigr)
 \nonumber
     \\
&\hspace{30pt} \times   
\sum_{\ell \in F_N^+}
\biggl\{
  \ell
  \Bigl[
  \Im \bigl[
e_\ell(x)
\bigr] \partial_{\Re[\widehat{r}^\ell]} 
\rho\bigl(\widehat{r}^{\ell}
\bigr)
-
\Re \bigl[
e_\ell(x)
\bigr] \partial_{\Im[\widehat{r}^\ell]} 
\rho\bigl(\widehat{r}^{\ell}
\bigr)
\bigr]
\Bigr] \prod_{j \in F_N^+ : j \not = \ell} \rho \bigl(  \widehat{r}^j \bigr)
\biggr\}
\bigotimes_{j \in F_N^+ } 
  d \widehat{r}^j 
 \biggr)
 \nonumber
\\
&\hspace{200pt} \times  e_k(x) \bigl( m_t * f_N \bigr)(x) \, dx \, d \lambda.
\label{eq:linear:ode:3}
  \\
&= 
 - 2 \pi^2 \vert k \vert^2 \widehat{m}_t^k  
 \nonumber
\\
&\hspace{5pt} 
- 2 \pi  
\int_0^1
\int_{{\mathbb T}^d} k \cdot
 \partial_p H \biggl( x, - 2 \pi 
 \int_{{\mathbb R}^{2 \vert F_N^+\vert}}
 V_{2-\lambda} \bigl(t,m_t^{N,\varepsilon}(r) * f_N \bigr)
 \nonumber
\\
&\hspace{30pt}
\times 
\sum_{\ell \in F_N^+}
\biggl\{ \ell
\Bigl[
\Bigl( 
\Re \bigl[
\i e_\ell(x)
\bigr], 
\Im \bigl[
\i e_\ell(x)
\bigr]
\Bigr)
\cdot
\nabla_{\Re[\widehat{r}^\ell],\Im[\widehat{r}^\ell]}
\Bigl( \prod_{j \in F_N^+ }
  \rho\bigl( \widehat{r}^j \bigr) 
  \Bigr)
 \Bigr]  
 \biggr\}
 \bigotimes_{j \in F_N^+} d \widehat{r}^j   \biggr)   \nonumber
\\
&\hspace{200pt} \times \i  e_k(x) \bigl( m_t * f_N \bigr)(x) \, dx \, d \lambda.
\nonumber
  \end{align}
Notice that we passed the $\i$ in front of the second line to $\i e_k(x)$ inside the integral on the penultimate  line. 
Also, the symbol $\cdot$ on the same penultimate line is used to denote the inner product in ${\mathbb R}^2$. 
%
\vspace{5pt}

\textit{Second Step.}
We now compute the derivative of the flow 
with respect to the real and imaginary parts of the 
Fourier coefficients of the initial condition.
In order to avoid to repeat the computations twice, once for the derivative with respect to the real part and  once for the derivative
with respect to the imaginary part, 
we introduce the generic notation $\fD[z]$ for a complex number $z$, which may stand for 
$\Re[z]$ 
or 
$\Im[z]$.

Now, it makes sense to compute the derivative of the flow 
$\partial_{\fD[\widehat{m}_0^{k'}]} \widehat{m}_t^k$, for 
$k,k' \in F_N^+$ and $\fD=\Re,\Im$.
By representing $m*f_N$ in 
the last line of 
\eqref{eq:linear:ode:3}
in the form 
\begin{equation*}
m * f_N (x) = 1 + 2 \sum_{\ell' \in F_N^+} \Re \Bigl[ \widehat{m}^{\ell'} \widehat{f}^{\ell'}_N e_{-\ell'}(x) \Bigr] , 
\end{equation*}
we get, using 
both 
\eqref{eq:linear:ode:1} 
and 
\eqref{eq:linear:ode:3},
 \begin{align*}
&\frac{d}{dt} \Bigl( \partial_{\fD[\widehat{m}_0^{k'}]} \widehat{m}_t^k \Bigr) 
= - 2 \pi^2 \vert k \vert^2 \partial_{\fD[\widehat{m}_0^{k'}]} \widehat{m}_t^k
\\
&\hspace{2pt} 
+ 4  \pi^2 (1-\varepsilon) \int_0^1 \int_{{\mathbb R}^{2 \vert F_N^+ \vert}}
d \lambda \biggl(
 \bigotimes_{j \in F_N^+} d \widehat{r}^j
 \biggr)
 \\
&\hspace{7pt} \times 
\biggl[
\sum_{\ell,\ell' \in F_N^+}
k \cdot \bigl( \partial^2_{pp} H(\lambda,t) \ell \bigr) \, 
\partial_{\fD[\widehat{m}^{\ell'}]}   V_{2-\lambda}
\bigl(t,m_t^{N,\varepsilon}(r) * f_N \bigr)
\hat{f}^{\ell'}_N
\partial_{\fD[\widehat{m}_0^{k'}]} \widehat{m}_t^{\ell'} 
\\
&\hspace{15pt} \times 
\int_{{\mathbb T}^d}
\Bigl[
\Bigl( 
\Re \bigl[
\i e_\ell(x)
\bigr], 
\Im \bigl[
\i e_\ell(x)
\bigr]
\Bigr)
\cdot
\nabla_{\Re[\widehat{r}^\ell],\Im[\widehat{r}^\ell]}
\Bigl( \prod_{j \in F_N^+ }
  \rho\bigl( \widehat{r}^j \bigr) 
  \Bigr)
 \Bigr] \i e_k(x) \bigl( m_t * f_N \bigr)(x) dx
 \biggr] 
\\
&\hspace{2pt} 
 -4 \pi  
\int_{{\mathbb T}^d} k \cdot
D {\mathcal H}^{N,\varepsilon,\rho}(t,m_t)(x) 
%
%
%
%
\i  e_k(x)  \sum_{\ell' \in F_N^+} \Re \Bigl[ \partial_{\fD[\widehat{m}_0^{k'}]} \widehat{m}_t^{\ell'} \widehat{f}^{\ell'}_N e_{-\ell'}(x) \Bigr]  \, dx,
\nonumber
\end{align*}
with 
\begin{equation*}
\begin{split}
\partial^2_{pp} H(\lambda,t) &=
\partial^2_{pp} H 
\biggl( x, - 2 \pi 
\int_{{\mathbb R}^{2 \vert F_N^+\vert}}
V_{2-\lambda} \bigl(t,m_t^{N,\varepsilon}(r) * f_N \bigr)
 \nonumber
\\
&\hspace{15pt}
\times 
\sum_{\ell \in F_N^+}
\biggl\{ \ell
\Bigl[
\Bigl( 
\Re \bigl[
\i e_\ell(x)
\bigr], 
\Im \bigl[
\i e_\ell(x)
\bigr]
\Bigr)
\cdot
\nabla_{\Re[\widehat{r}^\ell],\Im[\widehat{r}^\ell]}
\Bigl( \prod_{j \in F_N^+ }
  \rho\bigl( \widehat{r}^j \bigr) 
  \Bigr)
 \Bigr]  
 \biggr\}
 \bigotimes_{j \in F_N^+} d \widehat{r}^j   \biggr).
\end{split}
\end{equation*}
Expanding 
$
 \partial_{\fD[\widehat{m}_0^{k'}]} \widehat{m}_t^k
$
as
$( \partial_{\fD[\widehat{m}_0^{k'}]} \Re[\widehat{m}_t^k],\partial_{\fD[\widehat{m}_0^{k'}]} \Im[\widehat{m}_t^k])$, 
we get a linear system of the form 
\begin{equation*}
\begin{split}
\frac{d}{dt} \Bigl( \partial_{\fD[\widehat{m}_0^{k'}]} \widetilde{\fD}[\widehat{m}_t^k] \Bigr) 
&=  - 2 \pi^2 \vert k \vert^2 \partial_{\fD[\widehat{m}_0^{k'}]}  \widetilde{\fD}\bigl[\widehat{m}_t^k\bigr]
\\
&\hspace{2pt} 
+
\sum_{\ell \in F_N^+}
\Bigl[ A_{k,\ell}^{\widetilde{\fD},\Re}(t)
 \partial_{\fD[\widehat{m}_0^{k'}]} \Re[\widehat{m}_t^\ell]
 +
 A_{k,\ell}^{\widetilde{\fD},\Im}(t)
 \partial_{\fD[\widehat{m}_0^{k'}]} \Im[\widehat{m}_t^\ell]
\Bigr]
\\
&\hspace{15pt} + 
  \sum_{\ell \in F_N^+}
\Bigl[ B_{k,\ell}^{\widetilde{\fD},\Re}(t)
 \partial_{\fD[\widehat{m}_0^{k'}]} \Re[\widehat{m}_t^\ell]
 +
B_{k,\ell}^{\widetilde{\fD},\Im}(t)
 \partial_{\fD[\widehat{m}_0^{k'}]} \Im[\widehat{m}_t^\ell]
\Bigr],
\end{split}
\end{equation*}
with
\begin{equation*}
\begin{split}
&A_{k,\ell}^{\widetilde{\fD},\fD}(t)
= \epsilon_{\widetilde{\fD},{\fD}}  4 \pi^2 (1-\varepsilon)
\int_0^1
 \int_{{\mathbb R}^{2 \vert F_N^+ \vert}}
d\lambda 
\biggl(
 \bigotimes_{j \in F_N^+} d \widehat{r}^j
 \biggr)
\\
&\hspace{7pt} \times 
\biggl[
\sum_{ \ell' \in F_N^+}
k \cdot \bigl( \partial^2_{pp} H(\lambda,t) \ell' \bigr) \, 
\partial_{\fD[\widehat{m}^{\ell}]}  V_{2-\lambda} 
\bigl(t,m_t^{N,\varepsilon}(r) * f_N \bigr)
\hat{f}^{\ell}_N
\\
&\hspace{15pt} \times 
\int_{{\mathbb T}^d}
\Bigl[
\Bigl( 
\Re \bigl[
\i e_{\ell'}(x)
\bigr], 
\Im \bigl[
\i e_{\ell'}(x)
\bigr]
\Bigr)
\cdot
\nabla_{\Re[\widehat{r}^{\ell'}],\Im[\widehat{r}^{\ell'}]}
\Bigl( \prod_{j \in F_N^+ }
  \rho\bigl( \widehat{r}^j \bigr) 
  \Bigr)
 \Bigr] \widetilde{\fD} \bigl[\i e_k(x) \bigr] \bigl( m_t * f_N \bigr)(x) dx
 \biggr],
\end{split}
\end{equation*}
where
$\epsilon_{\widetilde{\fD},{\fD}} =1$ except if 
$(\widetilde{\fD},\fD) = (\Re,\Im)$ in which case 
$\epsilon_{\Re,\Im} =-1$ (those choices following from the standard relationships 
$\Re[zz']=\Re[z]\Re[z']-\Im[z]\Im[z']$ 
and 
$\Im[zz']=\Re[z]\Im[z']+\Im[z]\Re[z']$, for $z,z' \in {\mathbb C}$, 
which explains why only 
$\epsilon_{\Re,\Im}$ is here equal to $-1$), and  
with
\begin{equation*}
\begin{split}
  &B_{k,\ell}^{\widetilde{\fD},\fD}(t) =
-   \epsilon_{\fD} 4 \pi  
\int_{{\mathbb T}^d} k \cdot
D {\mathcal H}^{N,\varepsilon,\rho}(t,m_t)(x) 
\widetilde{\fD} \bigl[ \i e_k(x) \bigr] 
  \fD \Bigl[  \widehat{f}^{\ell'}_N e_{-\ell'}(x) \Bigr]
  \, dx,
\end{split}
\end{equation*}
where $\epsilon_{\fD}=1$ if $\fD=\Re$ and $-1$ if $\fD=\Im$.

We now recall the following standard fact from the theory of ODEs. If we let
\begin{equation*}
J_t = \textrm{\rm det} \biggl( \Bigl( 
 \partial_{\fD[\widehat{m}_0^{k'}]} \widetilde{\fD}[\widehat{m}_t^k]
\Bigr)_{k,k' \in F_N^+, \fD,\widetilde{\fD} = \Re,\Im} 
\biggr),
\end{equation*}
then
\begin{equation}
\label{eq:equation:jacobien}
\frac{d}{dt} J_t = 
J_t \sum_{k \in F_N^+}
\Bigl( -4 \pi \vert k\vert^2 + A_{k,k}^{\Re,\Re} (t) + A_{k,k}^{\Im,\Im}(t)
+
B_{k,k}^{\Re,\Re}(t) + B_{k,k}^{\Im,\Im}(t)
\Bigr).
\end{equation}

\textit{Third Step.} The goal is now to prove that the sum in 
the right-hand side of 
\eqref{eq:equation:jacobien} is lower bounded, by a constant that is independent of 
$\varepsilon$
and
$\rho$.  
The easiest term to handle is the sum 
$\sum_{k \in F_N^+}
[B_{k,k}^{\Re,\Re}(t) + B_{k,k}^{\Im,\Im}(t)
]$. 
Indeed, in the definition of 
$B_{k,\ell}^{\widetilde{\fD},\fD}(t)$, 
$D {\mathcal H}^{N,\varepsilon,\rho}$ is bounded independently 
of $(N,\varepsilon,\rho)$, 
since 
$V_1$ and $V_2$ are 
$d_{-2}$-Lipschitz continuous, see
\eqref{eq:DHlambda} and Proposition 
\ref{prop:4:7}. 

The term $\sum_{k \in F_N^+}
[A_{k,k}^{\Re,\Re}(t) + A_{k,k}^{\Im,\Im}(t)
]$ is much more difficult to handle. We get 
\begin{equation*}
\begin{split}
&A_{k,k}^{{\fD},\fD}(t)
=  4 \pi^2 (1-\varepsilon) \int_0^1  \int_{{\mathbb R}^{2 \vert F_N^+ \vert}}
d \lambda
\biggl(
 \bigotimes_{j \in F_N^+} d \widehat{r}^j
 \biggr)
 \\
 &\hspace{7pt} \times 
\biggl[
\sum_{ \ell' \in F_N^+}
k \cdot \bigl( \partial^2_{pp} H(\lambda,t) \ell' \bigr) \, 
\partial_{\fD[\widehat{m}^{k}]}   V_{2-\lambda} 
\bigl(t,m_t^{N,\varepsilon}(r) * f_N \bigr)
\hat{f}^{k}_N
\\
&\hspace{15pt} \times 
\int_{{\mathbb T}^d}
\Bigl[
\Bigl( 
\Re \bigl[
\i e_{\ell'}(x)
\bigr], 
\Im \bigl[
\i e_{\ell'}(x)
\bigr]
\Bigr)
\cdot
\nabla_{\Re[\widehat{r}^{\ell'}],\Im[\widehat{r}^{\ell'}]}
\Bigl( \prod_{j \in F_N^+ }
  \rho\bigl( \widehat{r}^j \bigr) 
  \Bigr)
 \Bigr]  {\fD} \bigl[\i e_k(x) \bigr] \bigl( m_t * f_N \bigr)(x) dx
 \biggr],
\end{split}
\end{equation*}
which, by integration by parts, can be rewritten as
\begin{equation*}
\begin{split}
&A_{k,k}^{{\fD},\fD}(t)
= -
  4 \pi^2  \int_0^1 \int_{{\mathbb R}^{2 \vert F_N^+ \vert}}
d \lambda \biggl(
 \bigotimes_{j \in F_N^+} d \widehat{r}^j
 \biggr)
 \\
 &\hspace{7pt} \times 
\biggl[
\sum_{ \ell' \in F_N^+}
k \cdot \bigl( \partial^2_{pp} H(\lambda,t) \ell' \bigr) \, 
  V_{2-\lambda} 
\bigl(t,m_t^{N,\varepsilon}(r) * f_N \bigr)
\\
&\hspace{15pt} \times 
\int_{{\mathbb T}^d}
\sum_{\widetilde{\fD} = \Re,\Im}
\Bigl[
{\fD} \bigl[\i e_k(x) \bigr]
\widetilde{\fD}
\bigl[
\i e_{\ell'}(x)
\bigr]
\partial^2_{\fD[\widehat{r}^{k}],\widetilde{\fD}[\widehat{r}^{\ell'}]}
\Bigl( \prod_{j \in F_N^+ }
  \rho\bigl( \widehat{r}^j \bigr) 
  \Bigr)
 \Bigr]   \bigl( m_t * f_N \bigr)(x) dx
 \biggr],
\end{split}
\end{equation*}
and then, 
\begin{equation*}
\begin{split}
&\sum_{k \in F_N^+,\fD=\Re,\Im}
A_{k,k}^{{\fD},\fD}(t)
= -
  4 \pi^2  \int_0^1  \int_{{\mathbb R}^{2 \vert F_N^+ \vert}}
  d \lambda 
\biggl(
 \bigotimes_{j \in F_N^+} d \widehat{r}^j
 \biggr)
\biggl[  V_{2-\lambda}
\bigl(t,m_t^{N,\varepsilon}(r) * f_N \bigr)
\\
&\hspace{50pt} \times 
\int_{{\mathbb T}^d}
\sum_{\widetilde{\fD} = \Re,\Im}
\sum_{k, \ell \in F_N^+}
\Bigl[
\Bigl( {\fD} \bigl[\i e_k(x) \bigr]
k \Bigr)
\cdot 
\Bigl( \widetilde{\fD} \bigl[\i e_\ell(x) \bigr]
\partial^2_{pp} H(\lambda,t)
\ell \Bigr)
\partial^2_{\fD[\widehat{r}^{k}],\widetilde{\fD}[\widehat{r}^{\ell}]}
\Bigl( \prod_{j \in F_N^+ }
  \rho\bigl( \widehat{r}^j \bigr) 
  \Bigr)
 \Bigr] 
 \\
&\hspace{50pt} \times   \bigl( m_t * f_N \bigr)(x) dx
 \biggr].
\end{split}
\end{equation*}
We can now apply 
Corollary 
\ref{cor:weak:semi-concavity} since $V_{2-\lambda}$ is displacement semi-concave. Using the fact that $H$ is convex with respect to $p$, we deduce that 
\begin{equation*}
\sum_{k \in F_N^+,\fD=\Re,\Im}
A_{k,k}^{{\fD},\fD}(t) \geq - C_N,
\quad t \in [0,T], 
\end{equation*}
for a constant $C_N$ depending on $N$, but independent of $\rho$ and $\varepsilon$, from which we deduce that 
$J_t
 \geq - C_N$, for any $t \in [0,T]$. 
 The result follows from a mere change of variable at any time $t$, using the lower bound for the Jacobian $J_t$. 
\end{proof}

\section{Application to Mean Field Games}
\label{sec:MFG}

Following our agenda, we now address weak solutions to the conservative form of the master equation
\eqref{master:cons}.

\subsection{Notion of weak solution}

We start with the following definition:

\begin{defn}
\label{def:master:equation:gen}
We call a weak solution 
to the master equation a bounded measurable function 
${\mathcal Z} : [0,T] \times {\mathcal P}({\mathbb T}^d) \times {\mathbb T}^d \rightarrow {\mathbb R}$, 
such that 
the system of Fourier coefficients
\begin{equation}
\label{eq:master:Fourier:coefficients}
\widehat{\mathcal Z}^k : [0,T] \times {\mathcal P}({\mathbb T}^d) \ni (t,m) \mapsto \reallywidehat{\Bigl\{ {\mathcal Z}(t,m,\cdot) \Bigr\}}^{-k},
\quad
k \in {\mathbb Z}^d, 
\end{equation}
satisfies
$\widehat{\mathcal Z}^0 \equiv 0$ together with 
the following two properties: 

$(i)$ 
There exists a sequence of functions 
$(\zeta_N : [0,T] \times {\mathcal P}_N  
\rightarrow {\mathbb R})_{N \geq 1}$, 
such that, for any $N \geq 1$, 
for $\fD=\Re,\Im$, 
the following equation holds true in the distributional sense on $[0,T] \times {\mathcal P}_N$: 
\begin{equation}
\label{eq:master:weak}
\begin{split}
&\epsilon_{\fD} \partial_t \fD \bigl[ \widehat{\mathcal Z}^k(t,m)  \bigr] 
- \frac12 \partial_{\fD[\widehat{m}^k]} 
\biggl\{ \int_{{\mathbb T}^d} H \biggl(y, \i 2 \pi \sum_{j \in F_N} 
j \widehat{\mathcal Z}^j(t,m) e_j(y) \biggr)  d m (y) 
\biggr\}
\\
&\hspace{15pt}
- \frac12 \sum_{j \in F_N} 2 \pi^2 \vert j \vert^2 \partial_{\fD[\widehat{m}^k]} \Bigl( \widehat{\mathcal Z}^j(t,m) 
\widehat{m}^j 
\Bigr) 
+ \fD \biggl[\reallywidehat{\displaystyle \frac{\delta F}{\delta m}(m)(\cdot)}^{k} \biggr]
 +  \partial_{\fD[\widehat{m}^k]} \bigl( \zeta_N(t,m)\bigr) = 0,
\end{split}
\end{equation}
for any $k \in F_N^+$, 
where
$\epsilon_{\Re}=1$ and 
$\epsilon_{\Im}=-1$, and
the following limit holds true   
for any $c>1$: 
\begin{equation}
\label{eq:master:weak:00}
\lim_{N \rightarrow \infty} 
\sup_{t \in [0,T]} 
\sup_{m \in  B_N(c) }
\bigl\vert 
\zeta_N(t,m) 
\bigr\vert =0,
\end{equation}
with the
same notation 
as in Definition 
\ref{defn:HJB:gen}
for $B_N(c)$, i.e.
$B_N(c)= {\mathcal P}_N \cap \{ m :  \sup_{x \in \bT^d} \vert \nabla m(x)  \vert \leq c, \ \inf_{x \in \bT^d}  m(x)   \geq 1/c \}$.

$(ii)$ For any $N \geq 1$ and $k \in F_N^+$, 
for any smooth and compactly supported function $\varphi : {\mathcal O}_N \rightarrow {\mathbb R}$, 
the function 
\begin{equation*}
\begin{split}
t \in [0,T] \mapsto 
 \int_{{\mathcal O}_N} 
\widehat{\mathcal Z}^k(t,m) 
\varphi \Bigl( (\widehat{m}^j)_{j \in F_N^+} \Bigr) 
\bigotimes_{j \in F_N^+} 
 d \Bigl( \Re\bigl[ \widehat{m}^j \bigr],\Im\bigl[\widehat{m}^j \bigr]
 \Bigr)
 \end{split}
 \end{equation*}
 is continuous 
and matches, at $t=T$,
\begin{equation*}
\int_{{\mathcal O}_N} 
\reallywidehat{ \displaystyle \frac{\partial G}{\partial m}(m)(\cdot)
   }^{-k}
\varphi \Bigl( (\widehat{m}^j)_{j \in F_N^+} \Bigr) 
\bigotimes_{j \in F_N^+} 
 d \Bigl( \Re\bigl[ \widehat{m}^j \bigr],\Im\bigl[\widehat{m}^j \bigr]
 \Bigr),
 \end{equation*}
 where 
 $m$ in the latter two integrands is implicitly understood as 
 ${\mathscr I}_N
( (\widehat{m}^j)_{j \in F_N^+} )$
with ${\mathscr I}_N$ as in \eqref{eq:I_N}.
%
%
\end{defn}

The reader will find a clear explanation of the 
choice 
$\widehat{\mathcal Z}^0 \equiv 0$
in Proposition 
\ref{prop:value:is:gen:master:equation}.
In brief, when ${\mathcal Z}$ is the $m$-derivative of a potential $Z : [0,T] \times {\mathcal P}({\mathbb T}^d) \rightarrow {\mathbb R}$, 
the condition 
$\widehat{Z}^0 \equiv 0$
 is consistent with the centering condition 
\eqref{eq:centring}.

Following 
\eqref{eq:prop:3:7:identification}--\eqref{eq:gradient:complex:000}--\eqref{eq:gradient:complex}, 
\eqref{eq:master:weak}
can be rewritten in the following compact form:
\begin{equation}
\label{eq:master:weak:bb}
\begin{split}
&\partial_t \widehat{\mathcal Z}^k(t,m)   
-  \partial_{\widehat{m}^k}
\biggl\{ 
\int_{{\mathbb T}^d}H \biggl(y, \i 2 \pi \sum_{j \in F_N} 
j \widehat{\mathcal Z}^j(t,m) e_j(y) \biggr)  d m (y) 
\biggr\}
\\
&\hspace{15pt}
- \sum_{j \in F_N} 2 \pi^2 \vert j \vert^2 \partial_{\widehat{m}^k} \Bigl( \widehat{\mathcal Z}^j(t,m) 
\widehat{m}^j 
\Bigr) 
+  \reallywidehat{\displaystyle \frac{\delta F}{\delta m}(m)(\cdot)}^{-k} 
 +  2 \partial_{\widehat{m}^k} \bigl(  \zeta_N(t,m)\bigr) = 0,
\end{split}
\end{equation}
for any $k \in F_N \setminus \{0\}$. 
It indeed suffices to take the operator $\epsilon_{\fD} \fD[\cdot]$ in the above equation
and to recall that $\epsilon_{\fD} \fD[\partial_{\widehat{m}^k}]=\tfrac12 
\partial_{\fD[\widehat{m}^k]}$.


We then have:
\begin{prop}
\label{prop:value:is:gen:master:equation}
The value function $V$, as defined in \eqref{eq:MKV:V}, 
induces a weak solution to the master equation, in the sense that there exists a bounded measurable function 
${\mathcal V} : [0,T] \times {\mathcal P}({\mathbb T}^d) \times {\mathbb T}^d \rightarrow {\mathbb R}$ satisfying
Definition 
\ref{def:master:equation:gen} together with the following two items: 
\begin{enumerate}
\item 
 for any integer  $N \geq 1$, the Fourier coefficients 
$(\widehat{\mathcal V}^k)_{k \in F_N}$ of ${\mathcal V}$, seen as complex-valued functions defined on 
$[0,T] \times {\mathcal P}({\mathbb T}^d)$ with the same form as in
\eqref{eq:master:Fourier:coefficients},
coincide
everywhere in time
and 
${\mathbb P}_N$ almost everywhere in space 
with 
the Fourier coefficients of the (space)-derivatives of $V$ on 
$[0,T] \times {\mathcal P}_N$, 
namely
\begin{equation}
\label{eq:prop5.2:statement}
\begin{split}
&\widehat{\mathcal V}^k(t,m) =  
\reallywidehat{\displaystyle \frac{\delta V}{\delta m}(t,m)(\cdot)}^{-k} \biggl( = 
\partial_{\widehat{m}^k} V(t,m) \biggr), \quad k \in F_N \setminus \{0\}, 
\\
&\widehat{\mathcal V}^0(t,m) =  
\reallywidehat{\displaystyle \frac{\delta V}{\delta m}(t,m)(\cdot)}^{0} =0, 
\end{split}
\end{equation}
for any $t \in [0,T]$ and for ${\mathbb P}_N$ almost every $m \in {\mathcal P} _N$  (or, equivalently, for 
almost every $m \in {\mathcal P}_N$ when 
${\mathcal P}_N$
is equipped with the image of the Lebesgue measure on ${\mathcal O}_N$ by the canonical mapping 
${\mathscr I}_N$ in 
\eqref{eq:I_N});
\item 
 the coefficients 
$(\widehat{\mathcal V}^k)_{k \in {\mathbb Z}^d}$ of ${\mathcal V}$
coincide 
everywhere in $t$ and 
${\mathbb P}$
almost everywhere 
in $m$
with 
the Fourier coefficients of the (space)-derivatives of $V$ on 
$[0,T] \times \PP$,
as defined by 
Theorem  
\ref{prop:rademacher}, 
namely
\eqref{eq:prop5.2:statement}
holds
for every $t \in [0,T]$ , for ${\mathbb P}$ almost every $m \in \PP$ (with 
${\mathbb P}$ the probability measure defined in the statement of 
Theorem 
\ref{thm:probability:probability}) and for any $k \in {\mathbb Z}^d \setminus \{0\}$;
\item 
for any 
$k \in {\mathbb Z}^d$, 
 for any $t \in [0,T]$,
for any integer $N_0 \geq 1$ and 
for any bounded (measurable) 
function $\varphi$ defined on ${\mathcal O}_{N_0}$, 
\begin{equation*}
\begin{split}
&\lim_{N \rightarrow \infty} 
\int_{{\mathcal P}({\mathbb T}^d)} \widehat{\mathcal V}^k(t,m*f_N) 
\varphi \Bigl( (\widehat{m}^j)_{j \in F_{N_0}^+} \Bigr)   d {\mathbb P}(m)
 =  
\int_{{\mathcal P}({\mathbb T}^d)} \widehat{\mathcal V}^k(t,m) 
\varphi \Bigl( (\widehat{m}^j)_{j \in F_{N_0}^+} \Bigr)   d {\mathbb P}(m).
\end{split}
\end{equation*}
\end{enumerate}
\end{prop}
In fact,
the reader will notice from Theorem 
\ref{prop:rademacher}
that items (2) and (3) are redundant, as the combination of (1) and (3) implies (2).

\begin{proof}
The proof relies on the following main observation: for a given integer $N_0 \geq 1$, 
${\mathcal P}_{N_0}$ is of zero measure under 
${\mathbb P}_N$  (or equivalently when  
it is regarded as a subset of 
${\mathcal P}_N$ equipped with the image of the Lebesgue measure on 
${\mathcal O}_N$ by the canonical mapping 
${\mathscr I}_N$ in 
\eqref{eq:I_N}); this follows from the fact that the elements of 
${\mathcal P}_{N_0}$ are required to have zero Fourier coefficients of modes $k$ with $\vert k\vert \geq N_0$. 
Therefore, we can define ${\mathcal V}$ inductively, by first assigning values on $[0,T] \times {\mathcal P}_1 \times {\mathbb T}^d$, and then on $[0,T] \times ({\mathcal P}_2 \setminus {\mathcal P}_1) \times {\mathbb T}^d$, so on and so forth...  
To make it clear, there is no difficulty for defining ${\mathcal V}$ such that, 
for every $N \geq 1$, 
for every $t \in [0,T]$, 
for almost every $m$ in  
${\mathcal P}_N$
and
for any $x \in {\mathbb T}^d$, 
\begin{equation}
\label{eq:def:Z:V}
{\mathcal V}(t,m,x) = \sum_{k \in F_N \setminus \{0\}} 
\biggl\{
\reallywidehat{\displaystyle \frac{\delta V}{\delta m}(t,m)(\cdot)}^{k}
e_{-k}(x)
\biggr\}
=
\sum_{k \in F_N \setminus \{0\}} 
\partial_{\widehat{m}^k} V(t,m)
e_{k}(x).
\end{equation}
We then let
$\widehat{\mathcal V}^k(t,m): = 
\reallywidehat{\displaystyle {\mathcal V}(t,m)(\cdot)}^{-k} ( = \partial_{\widehat{m}^k} 
V(t,m))$
and then
\begin{equation*}
\begin{split}
2 \zeta_N(t,m) &:= 
\partial_t V(t,m) - 
 \int_{{\mathbb T}^d}
 H \biggl( y, \i 2 \pi  \sum_{k \in F_N} k \widehat{\mathcal V}^k(t,m) e_{k}(y) \biggr)  d m(y)
 \\
&\hspace{15pt} - \sum_{k \in F_N} 2 \pi^2 \vert k\vert^2 
\widehat{\mathcal V}^k(t,m) \widehat{m}^{k}  + F(m)
\\
&=
\partial_t V(t,m) - 
 \int_{{\mathbb T}^d}
 H \biggl( y, \i 2 \pi  \sum_{k \in F_N} k \partial_{\widehat m^k} V(t,m) e_{k}(y) \biggr)  d m(y)
 \\
&\hspace{15pt} - \sum_{k \in F_N} 2 \pi^2 \vert k\vert^2 
\partial_{ \widehat{m}^{k} }
V(t,m) \widehat{m}^{k}  + F(m),
\end{split}
\end{equation*}
which satisfies 
\eqref{eq:master:weak:00} thanks to 
\eqref{eq:HJB:gen}.

Multiplying, for each value of 
$k \in F_N^+$, 
both sides of the definition of 
$\zeta_N(t,m)$ above
by $(1/2) \partial_{\fD[\widehat{m}^k]} \varphi(t,(\widehat{m}^j)_{j \in F_N^+})$
for an arbitrary smooth test function on $[0,T] \times {\mathcal O}_N$
with a support included in $(0,T) \times {\mathcal O}_N$, we
easily get 
\eqref{eq:master:weak} (in a distributional sense). Indeed, it suffices to observe that 
\begin{equation*}
\begin{split}
&\frac{1}{2} \int_0^T \int_{{\mathcal O}_N}
\partial_t V\bigl(t,m \bigr) 
 \partial_{\fD[\widehat{m}^k]} 
 \varphi \Bigl( t, (\widehat{m}^j)_{j \in F_N^+} \Bigr) 
\bigotimes_{j \in F_N^+} 
 d \Bigl( \Re\bigl[ \widehat{m}^j \bigr],\Im\bigl[\widehat{m}^j \bigr]
 \Bigr)
 \\
 &= 
\frac{1}{2} \int_0^T \int_{{\mathcal O}_N}
\partial_{\fD[\widehat{m}^k]}  V\bigl(t, m \bigr) 
\partial_t \varphi \Bigl( t, (\widehat{m}^j)_{j \in F_N^+} \Bigr) 
\bigotimes_{j \in F_N^+} 
 d \Bigl( \Re\bigl[ \widehat{m}^j \bigr],\Im\bigl[\widehat{m}^j \bigr]
 \Bigr)
 \\
 &= \epsilon_{\fD} \int_0^T \int_{{\mathcal O}_N}
 \fD \Bigl[  {\mathcal V}^k \bigl(t,m \bigr) 
 \Bigr]
\partial_t \varphi \Bigl( t, (\widehat{m}^j)_{j \in F_N^+} \Bigr) 
\bigotimes_{j \in F_N^+} 
 d \Bigl( \Re\bigl[ \widehat{m}^j \bigr],\Im\bigl[\widehat{m}^j \bigr]
 \Bigr),
 \end{split}
\end{equation*}
with 
$m$ in the above three lines being understood 
as $m={\mathscr I}_N( (\widehat{m}^j)_{j \in F_N^+})$
and 
with the last line following from 
\eqref{eq:gradient:complex:000}. Similarly, 
\begin{equation*}
\begin{split}
&\frac12 \int_0^T \int_{{\mathcal O}_N}
F \Bigl((\widehat{m}^j)_{j \in F_N^+} \Bigr) 
 \partial_{\fD[\widehat{m}^k]} 
 \varphi \Bigl( t, (\widehat{m}^j)_{j \in F_N^+} \Bigr) 
\bigotimes_{j \in F_N^+} 
 d \Bigl( \Re\bigl[ \widehat{m}^j \bigr],\Im\bigl[\widehat{m}^j \bigr]
 \Bigr)
 \\
 &= - \frac12 \int_0^T \int_{{\mathcal O}_N}
\partial_{\fD[\widehat{m}^k]}  F\Bigl((\widehat{m}^j)_{j \in F_N^+} \Bigr) 
 \varphi \Bigl( t, (\widehat{m}^j)_{j \in F_N^+} \Bigr) 
\bigotimes_{j \in F_N^+} 
 d \Bigl( \Re\bigl[ \widehat{m}^j \bigr],\Im\bigl[\widehat{m}^j \bigr]
 \Bigr)
 \\
 &=- \int_0^T \int_{{\mathcal O}_N}
  \fD \biggl[\reallywidehat{\displaystyle \frac{\delta F}{\delta m}(m,\cdot)}^{k} \biggr]
  \varphi \Bigl( t, (\widehat{m}^j)_{j \in F_N^+} \Bigr) 
\bigotimes_{j \in F_N^+} 
 d \Bigl( \Re\bigl[ \widehat{m}^j \bigr],\Im\bigl[\widehat{m}^j \bigr]
 \Bigr).
 \end{split}
\end{equation*}
This proves $(i)$ in Definition  
\ref{def:master:equation:gen}. As for $(ii)$ in the same definition, 
it follows from an obvious integration by parts.

We prove that the values of ${\mathcal V}$ that are hence defined are bounded by a common constant.
In order to do so, we use the fact that $V$ is Lipschitz continuous (in space) 
for the $d_{-2}$-distance and thus for the $d_{\textrm{\rm TV}}$-distance. Indeed, 
Proposition 
\ref{prop:4:7} says that there exists a constant $C$ such that, for any $N \geq 1$, 
\begin{equation*}
\textrm{\rm essup}_{(t,m) \in [0,T] \times {\mathcal P}_N}
\sup_{x \in {\mathbb T}^d} 
\biggl\vert \partial_x \biggl( \sum_{k \in F_N \setminus \{0\}} \partial_{\widehat{m}^k} 
V(t,m) e_k(x) \biggr) \biggr\vert \leq C.
\end{equation*}
By 
\eqref{eq:def:Z:V},
this proves that
\begin{equation}
\label{eq:proof:master:weak:0:4:bb}
\sup_{N \geq 1} 
\textrm{\rm essup}_{(t,m) \in [0,T] \times {\mathcal P}_N}
 \textrm{\rm essup}_{x \in {\mathbb T}^d} \bigl\vert 
\partial_x {\mathcal V}(t,m,x) \bigr\vert \leq C. 
\end{equation}
Since ${\mathcal V}(t,m,\cdot)$ has zero mean, this yields
\begin{equation}
\label{eq:proof:master:weak:0:4}
\sup_{N \geq 1} 
\textrm{\rm essup}_{(t,m) \in [0,T] \times {\mathcal P}_N}
 \textrm{\rm essup}_{x \in {\mathbb T}^d} \bigl\vert 
{\mathcal V}(t,m,x) \bigr\vert \leq C. 
\end{equation}

It then remains to extend 
${\mathcal V}$ to the whole 
$[0,T] \times {\mathcal P}({\mathbb T}^d) \times {\mathbb T}^d$
by taking a weak-star limit, in the $\sigma^*(L^\infty;L^1)$ sense, of the bounded sequence 
$([0,T] \times {\mathcal P}({\mathbb T}^d) \times {\mathbb T}^d \ni (t,m,x) \mapsto {\mathcal V}(t,m*f_N,x))_{N \geq 1}$. 
Weak limits are defined almost everywhere under $\textrm{\rm Leb}_{1} \otimes {\mathbb P} 
\otimes 
\textrm{\rm Leb}_{d}$. Since $\cup_{N \geq 1} {\mathcal P}_N$ is of zero measure under 
${\mathbb P}$ (the argument is given next), we can easily  modify
any weak limit such that it coincides with 
${\mathcal V}$ itself 
  on 
$[0,T] \times (\cup_{N \geq 1} {\mathcal P}_N) \times {\mathbb T}^d$: the weak limit then provides an extension of ${\mathcal V}$ to the whole 
$[0,T] \times {\mathcal P}({\mathbb T}^d) \times {\mathbb T}^d$, which shows (1). 
The main difficulty is to prove that there is in fact only one limit point.
In order to do so, we proceed as follows. 
For any subsequence $(N_n)_{n \geq 1}$ for which 
the sequence 
$([0,T] \times {\mathcal P}({\mathbb T}^d) \times {\mathbb T}^d \ni (t,m,x) \mapsto {\mathcal V}(t,m*f_{N_n},x))_{n \geq 1}$
is converging in the weak sense, 
we can apply 
Lemma
\ref{lem:W:mathcalW:Nn}
below 
with ${\mathcal W}(m*f_N,x)={\mathcal V}(t,m*f_N,x)$ and $W^{(N)}(m)=V(t,m)$
and then invoke Lemma 
\ref{lem:W:mathcalW:Nn:2}
with ${\mathcal W}_1$ and ${\mathcal W}_2$ being two possible
limiting points 
of 
$({\mathcal P}({\mathbb T}^d) \times {\mathbb T}^d \ni (t,m,x) \mapsto {\mathcal V}(t,m*f_{N},x))_{N \geq 1}$
and with $W_1=W_2=V$. 
The key point in this respect is that  
the sequence
$([0,T] \times {\mathcal P}({\mathbb T}^d) \ni (t,m) \mapsto V(t,m*f_{N}))_{N \geq 1}$
converges to 
the function 
$[0,T] \times {\mathcal P}({\mathbb T}^d) \ni (t,m) \mapsto V(t,m)$. 
Therefore, 
Lemma \ref{lem:W:mathcalW:Nn:2} says that the sequence 
$([0,T] \times {\mathcal P}({\mathbb T}^d) \times {\mathbb T}^d \ni (t,m,x) \mapsto {\mathcal V}(t,m*f_{N},x))_{N \geq 1}$
has a unique limit point (in the weak sense)
and is thus convergent (in the weak sense). 
This proves (2) in the statement.

It then remains to prove that 
each ${\mathcal P}_N$ is of zero measure under ${\mathbb P}$. 
It suffices to 
 recall from 
Theorem 
\ref{thm:probability:probability} 
that, for $N' > N$, 
\begin{equation*}
\begin{split}
{\mathbb P} \bigl( {\mathcal P}_N \bigr) 
&= 
\bigl( {\mathbf 1}_{{\mathcal P}_{N'}} \cdot 
{\mathbb P} \bigr) \Bigl( {\mathcal P}_N \Bigr) 
 = 
\bigl( {\mathbf 1}_{{\mathcal P}_{N'}} \cdot 
{\mathbb P} \bigr) \Bigl(
 \bigcap_{k \in F_{N'}^+ \setminus F_N^+} 
\bigl\{ m : [\pi_{N'}^{(2)}(m)]^k = 0 \bigr\} \Bigr) =0.
\end{split}
\end{equation*}
This completes the proof. 
\end{proof}

Importantly, we stress (once again) the fact that the function ${\mathcal V}$ defined in the above statement 
inherits the $d_{-2}$-Lipschitz property and the semi-concavity of $V$. 
By Proposition 
\ref{prop:4:7} and by 
\eqref{eq:def:Z:V}, 
we have (see 
\eqref{eq:proof:master:weak:0:4:bb}
for the first line below) 
\begin{equation*}
\begin{split}
&\sup_{N \geq 1} \textrm{\rm essup}_{(t,m) \in [0,T] \times {\mathcal P}_N}
\sup_{x \in {\mathbb T}^d} 
\bigl\vert 
\partial_{x} {\mathcal V}(t,m,x)
\bigr\vert < \infty,
\\
&\sup_{N \geq 1}
\textrm{\rm essup}_{(t,m) \in [0,T] \times {\mathcal P}_N}
\sum_{k \in F_N} \vert k\vert^4 \vert \widehat{\mathcal V}^k(t,m) \vert^2 < \infty.
\end{split}
\end{equation*}
Moreover,
we recall from Corollary 
\ref{cor:weak:semi-concavity}
(see in particular \eqref{eq:corollary:weak:semi-concavity:2})
that, with 
$\rho$ as in Definition 
\ref{def:mollification},  
for any $N \geq 1$, for any $(t,m) \in [0,T] \times {\mathcal P}({\mathbb T}^d)$
with $m$ having a strictly positive density, 
for any collection of complex numbers $(z^k)_{k \in F_N^+}$ and any non-negative symmetric matrix $S$, 
\begin{equation}
\label{eq:weak:one-sided:1}
\begin{split}
&\int_{{\mathbb R}^{2 \vert F_N^+\vert}} 
\biggl\{
V\Bigl(t,m^{N,\varepsilon}(r) * f_N\Bigr)
\\
&\hspace{45pt} \times 
\sum_{\widetilde{\fD} = \Re,\Im}
\sum_{k, \ell \in F_N^+}
\biggl[
\Bigl( {\fD} \bigl[ z^k  \bigr]
k \Bigr)
\cdot 
\Bigl( \widetilde{\fD} \bigl[ z^\ell\bigr]
S \ell \Bigr)
\partial_{\fD[\widehat{r}^{k}],\widetilde{\fD}[\widehat{r}^{\ell}]}^2
\Bigl( \prod_{j \in F_N^+ }
  \rho\bigl(\Re\bigl[ \widehat{r}^j\bigr],\Im\bigl[ \widehat{r}^j\bigr] \bigr) 
  \Bigr)
 \biggr]
 \biggr\}
 \\
&\hspace{30pt} \bigotimes_{j \in F_N^+}
 d \Bigl( \Re\bigl[ \widehat{r}^j\bigr],\Im \bigl[ \widehat{r}^j  \bigr]
 \Bigr) 
 \\
&\leq C
\vert S \vert 
\sum_{q=1}^d \biggl(
\sum_{k \in F_N^+} \vert k_q \vert \, \vert z^k
\vert \biggr)^2.
\end{split}
\end{equation}
with $C$ depending on $c$ such that $m \geq 1/c$. 
By integration by parts,
this can be rewritten
\begin{equation}
\label{eq:weak:one-sided:2}
\begin{split}
&- \int_{{\mathbb R}^{2 \vert F_N^+\vert}} 
\biggl\{
\sum_{\widetilde{\fD} = \Re,\Im}
\sum_{k, \ell \in F_N^+}
\biggl[
\partial_{\fD[\widehat{r}^k]}
V\Bigl(t,m^{N,\varepsilon}(r) * f_N\Bigr)
 \widehat{f}_N^k 
\\
&\hspace{45pt} \times 
\biggl[
\Bigl( {\fD} \bigl[ z^k  \bigr]
k \Bigr)
\cdot 
\Bigl( \widetilde{\fD} \bigl[ z^\ell\bigr]
S \ell \Bigr)
\partial_{\widetilde{\fD}[\widehat{r}^{\ell}]}
\Bigl( \prod_{j \in F_N^+ }
  \rho\bigl(\Re\bigl[ \widehat{r}^j\bigr],\Im\bigl[ \widehat{r}^j\bigr] \bigr) 
  \Bigr)
 \biggr]
 \biggr\}
 \bigotimes_{j \in F_N^+}
 d \Bigl( \Re\bigl[ \widehat{r}^j\bigr],\Im \bigl[ \widehat{r}^j  \bigr]
 \Bigr) 
 \\
&\leq C\vert S \vert 
\sum_{q=1}^d  \biggl(
\sum_{k \in F_N^+} \vert k_q \vert \, \vert z^k
\vert \biggr)^2,
\end{split}
\end{equation}
and then
\begin{equation}
\label{eq:weak:one-sided:3}
\begin{split}
&- \int_{{\mathbb R}^{2 \vert F_N^+\vert}} 
\biggl\{
\sum_{\widetilde{\fD} = \Re,\Im}
\sum_{k, \ell \in F_N^+}
\biggl[
\fD \Bigl[ 
\widehat{\mathcal V}^k \Bigl(t,m^{N,\varepsilon}(r) * f_N\Bigr)
\Bigr] 
 \widehat{f}_N^k 
\\
&\hspace{40pt} \times 
\biggl[
\Bigl( {\fD} \bigl[ z^k  \bigr]
k \Bigr)
\cdot 
\Bigl( \widetilde{\fD} \bigl[ z^\ell\bigr]
S \ell \Bigr)
\partial_{\widetilde{\fD}[\widehat{r}^{\ell}]}
\Bigl( \prod_{j \in F_N^+ }
  \rho\bigl(\Re\bigl[ \widehat{r}^j\bigr],\Im\bigl[ \widehat{r}^j\bigr] \bigr) 
  \Bigr)
 \biggr]
 \biggr\}
 \bigotimes_{j \in F_N^+}
 d \Bigl( \Re\bigl[ \widehat{r}^j\bigr],\Im \bigl[ \widehat{r}^j  \bigr]
 \Bigr) 
 \\
&\leq C
\vert S \vert 
\sum_{q=1}^d 
\biggl(
\sum_{k \in F_N^+} \vert k_q \vert \, \vert z^k
\vert \biggr)^2,
\end{split}
\end{equation}
which prompts us to introduce the following definition:
\begin{defn}
\label{def:weak:one-sided:lip}
A bounded measurable function 
${\mathcal Z} : [0,T] \times {\mathcal P}({\mathbb T}^d) \times {\mathbb T}^d \rightarrow {\mathbb R}$
is said to be one-sided Lipschitz in the weak sense 
if, for any $c>1$, 
 there exists a constant $C$ such that, for any $N \geq 1$, 
 any $\varepsilon \in (0,1)$
 and any  
$\rho$ as in Definition 
\ref{def:mollification},  
it holds, for any $(t,m) \in [0,T] \times {\mathcal P}({\mathbb T}^d)$
with $m$ having a density lower bounded by $1/c$, 
and 
for any collection of complex numbers $(z^k)_{k \in F_N^+}$ and any symmetric matrix $S$, 
\begin{equation*}
\begin{split}
&- \int_{{\mathbb R}^{2 \vert F_N^+\vert}} 
\biggl\{
\sum_{\widetilde{\fD} = \Re,\Im}
\sum_{k, \ell \in F_N^+}
\biggl[
\fD \Bigl[ 
\widehat{\mathcal Z}^k \Bigl(t,m^{N,\varepsilon}(r) * f_N\Bigr)
\Bigr] 
 \widehat{f}_N^k 
\\
&\hspace{45pt} \times 
\biggl[
\Bigl( {\fD} \bigl[ z^k  \bigr]
k \Bigr)
\cdot 
\Bigl( \widetilde{\fD} \bigl[ z^\ell\bigr]
S \ell \Bigr)
\partial_{\widetilde{\fD}[\widehat{r}^{\ell}]}
\Bigl( \prod_{j \in F_N^+ }
  \rho\bigl(\Re\bigl[ \widehat{r}^j\bigr],\Im\bigl[ \widehat{r}^j\bigr] \bigr) 
  \Bigr)
 \biggr]
 \biggr\}
 \bigotimes_{j \in F_N^+}
 d \Bigl( \Re\bigl[ \widehat{r}^j\bigr],\Im \bigl[ \widehat{r}^j  \bigr]
 \Bigr) 
 \\
&\leq C
\vert S \vert 
\sum_{q=1}^d 
\biggl(
\sum_{k \in F_N^+} \vert k_q \vert \, \vert z^k
\vert \biggr)^2,
\end{split}
\end{equation*}
with $\widehat{\mathcal Z}^k(t,m) = \reallywidehat{\displaystyle 
{\mathcal Z}(t,m,\cdot)}^{-k}$. 
\end{defn}

\subsection{Uniqueness of weak solutions}

Here is now the main theorem:
\begin{thm}
\label{main:thm:MFG}
Let ${\mathbb P}$ be the probability measure on ${\mathcal P}({\mathbb T}^d)$ introduced in the statement of Theorem 
\ref{thm:probability:probability}. Then,
uniqueness holds everywhere on $[0,T]$, 
${\mathbb P}$-almost everywhere 
on ${\mathcal P}({\mathbb T}^d)$ and 
everywhere on ${\mathbb T}^d$
within the class of 
bounded measurable functions 
\begin{equation*}
{\mathcal Z} : [0,T] \times {\mathcal P}({\mathbb T}^d) \times {\mathbb T}^d \rightarrow {\mathbb R},
\end{equation*}
satisfying (with the same notation as in 
\eqref{eq:master:Fourier:coefficients})
\begin{enumerate}
\item 
$\displaystyle 
\sup_{N \geq 1} 
\textrm{\rm essup}_{(t,m) \in [0,T] \times {\mathcal P}_N}
\biggl[ 
\sup_{x \in {\mathbb T}^d} 
\Bigl\vert 
\sum_{k \in F_N} 
k \widehat{\mathcal{Z}}^k(t,m) e_k(x)
\Bigr\vert
+
\sum_{k \in F_N} \vert k\vert^4 \vert \widehat{\mathcal Z}^k(t,m) \vert^2 \biggr] < \infty;$
\item
${\mathcal Z}$ is one-sided Lipschitz in the weak sense, see
Definition 
\ref{def:weak:one-sided:lip};
\item
for any 
$k \in {\mathbb Z}^d$, 
for any $t \in [0,T]$, 
for any bounded (measurable) 
function $\varphi$ defined on ${\mathcal O}_{N_0}$, for some $N_0 \geq 1$,
\begin{equation*}
\lim_{N \rightarrow \infty} 
\int_{{\mathcal P}({\mathbb T}^d)} \widehat{\mathcal  Z}^k(t,m*f_N) 
\varphi \bigl( (\widehat{m}^j)_{j \in F_{N_0}^+} \bigr) d {\mathbb P}(m)
=  
\int_{{\mathcal P}({\mathbb T}^d)} \widehat{\mathcal  Z}^k(t,m) 
\varphi \bigl( (\widehat{m}^j)_{j \in F_{N_0}^+} \bigr) d {\mathbb P}(m).
\end{equation*}
\end{enumerate}
In particular, 
for every $t \in [0,T]$, for 
${\mathbb P}$ almost every $m \in \PP$ and 
for every $x \in {\mathbb T}^d$, 
${\mathcal Z}(t,x,m)$ is 
equal to 
${\mathcal V}(t,x,m)$,
with 
${\mathcal V}$ being as in the statement of 
Proposition \ref{prop:value:is:gen:master:equation}.
\end{thm}
The following remarks are in order.

\begin{rem}
\label{rem:reg:Z}
Item \emph{(1)} in the statement of Theorem 
\ref{main:thm:MFG}
says that 
$$\sup_{N \geq 1} 
\textrm{\rm essup}_{(t,m) \in [0,T] \times {\mathcal P}_N}
\sup_{x \in {\mathbb T}^d} 
\Bigl\vert 
\partial_{x} 
\Bigl( 
\sum_{k \in F_N} 
 \widehat{\mathcal{Z}}^k(t,m) e_k(x)
 \Bigr) 
\Bigr\vert < \infty.$$ 
Thus,
for any $N \geq 1$ and for almost every $(t,m) \in [0,T] \times {\mathcal P}_N$, 
 the function 
$x \in {\mathbb T}^d \mapsto
\sum_{k \in F_N}  
\mathcal{Z}^k(t,m) e_k(x)$
is Lipschitz 
continuous,
with a Lipschitz constant that is independent of $N$, $t$ and $m$. 
\end{rem}

\begin{rem}
It is worth noticing that the  formulation of uniqueness depends on the choice of the measure ${\mathbb P}$ in the statement of Theorem 
\ref{thm:probability:probability}. 
This is an intriguing observation since our choice for the measure ${\mathbb P}$
is somewhat arbitrary, 
recall the definition
\eqref{eq:Gamma_N} of the approximating measure ${\mathbb P}_N$. 
Alternatively, in order to formulate uniqueness 
without imposing a specific choice for 
${\mathbb P}$, 
we could also wonder about identifying 
$\widehat{\mathcal Z}^k(t,m)$ for $m \in {\mathcal P}_N$, for a fixed $N$, but this looks 
rather challenging. 
In fact, with our approach and with item $(iii)$ in 
Theorem \ref{thm:probability:probability}, we 
identify
\begin{equation*}
\int_{{\mathcal P}({\mathbb T}^d)} \widehat{\mathcal  Z}^k(t,m*f_N) 
\varphi \bigl( (\widehat{m}^j)_{j \in F_{N_0}^+} \bigr) d {\mathbb P}_N(m)
\end{equation*}
up to a remainder that tends to $0$ as $N$ tends to $\infty$.
 This remainder depends on 
$\varphi$ but can be made uniform when $\varphi$ is taken in a compact family of continuous 
functions. 
Interestingly, the above integral can be reformulated as an integral on 
${\mathcal O}_N$ equipped with the Lebesgue measure, 
which makes the interpretation easier.  
\end{rem}
\begin{proof}
\textit{First Step.}
We start with an arbitrary weak solution, as given in
Definition \ref{def:master:equation:gen}. 
We observe 
from \eqref{eq:master:weak}
that, for any $N \geq 1$, 
the functions $(\Re[\widehat{\mathcal Z}^k],\Im[\widehat{\mathcal Z}^k])_{k \in F_N^+}$ satisfy a system of 
hyperbolic 
equations of the form 
\begin{equation*}
\begin{split}
&\epsilon_{\fD} \partial_t \fD \Bigl[ \widehat{\mathcal Z}^k(t,m)  
\Bigr]  
+ \partial_{\fD[\widehat{m}^k]} H_N \Bigl(t,m,\Bigl(\Re\bigl[\widehat{\mathcal Z}^j(t,m) 
\bigr], 
\Im\bigl[\widehat{\mathcal Z}^j(t,m) 
\bigr]
\Bigr)_{j \in F_N^+} \Bigr)
=0,
\end{split}
\end{equation*}
for $\fD \in \{\Re,\Im\}$, $k \in F_N^+$ and $(t,m) \in [0,T] \times {\mathcal P}_N$, 
with ${\mathcal P}_N$ being identified with  
${\mathcal O}_N$ and with 
$\epsilon_{\Re}=1$ and $\epsilon_{\Im}=-1$. 

Following the proof of Theorem 6.6 in \cite{Cecchin:Delarue:CPDE} (which is itself adapted from 
\cite{kruzkov}), 
we deduce that, for almost every $t \in [0,T]$, we can find a function
$U^{(N)}(t,\cdot) : {\mathcal P}_N \rightarrow {\mathbb R}$ such that, almost everywhere on ${\mathcal P}_N$,   
\begin{equation*}
\epsilon_{\fD} \fD \Bigl[ \widehat {\mathcal Z}^k(t,m) \Bigr] =
\partial_{\fD[\widehat{m}^k]} U^{(N)}(t,m), \quad k \in F_N^+, 
\end{equation*}
which can be easily reformulated by means of 
\eqref{eq:gradient:complex:000} as 
\begin{equation*}
 \reallywidehat{\Bigl\{ {\mathcal Z}(t,m,\cdot) \Bigr\}}^{-k}=
\widehat {\mathcal Z}^k(t,m)
= \partial_{\widehat{m}^k} 
U^{(N)}(t,m).
\end{equation*}
Then, the function $U^N$ can be easily assumed to measurable in time (assuming for instance that 
$t \mapsto U^N(t,m_0)$ is measurable for a fixed $m_0 \in {\mathcal O}_N$). 
Following the same mollification procedure as in the second and third steps of 
the proof of \cite[Theorem 6.6]{Cecchin:Delarue:CPDE} and expanding the form of $H_N$, we can even assume that 
$U^{(N)}$ is Lipschitz in time and space on $[0,T] \times {\mathcal P}_N$ and satisfies, almost everywhere, the following equation
\begin{equation*}
\begin{split}
&\partial_t U^{(N)}(t,m) - 
 \int_{{\mathbb T}^d} H \biggl(y,  \i 2 \pi   \sum_{k \in F_N} k \partial_{\widehat m_k} U^{(N)}(t,m ) e_{k}(y) \biggr) d m(y)
\\
&\hspace{15pt} - \sum_{k \in F_N} 2 \pi^2 \vert k\vert^2 
\partial_{ \widehat{m}^{k} }
U^{(N)}(t,m) \widehat{m}^{k}  + F(m)
+ 2 \zeta_N(t,m) = 0,
\end{split}
\end{equation*}
with $U^{(N)}(T,m)=G(m)$ as boundary condition. 
Intuitively, 
the above equation is obtained by taking the anti-derivative in \eqref{eq:master:weak:bb}.
\vskip 5pt

\textit{Second Step.}
For the same value of $N$, 
and for 
$\varepsilon$ and $\rho$ and with the same notation as in Definition \ref{def:mollification}, 
we now consider the function 
\begin{equation*}
\widetilde{U}^{(N,\varepsilon,\rho)} : m \in {\mathcal P}_N \mapsto 
\int_{{\mathbb R}^{2 \vert F_N^+\vert}}
U^{(N)}\bigl(t,m^{N,\varepsilon}(r) \bigr) \prod_{k \in F_N^+} \rho
\bigl( \widehat{r}^k \bigr) 
\bigotimes_{k \in F_N^+} 
d \widehat{r}^k.
\end{equation*}
Pay attention that 
$\widetilde{U}^{(N,\varepsilon,\rho)}$
is not the same as 
$(U^{(N)})^{N,\varepsilon,\rho}$ in Definition 
\ref{def:mollification}
and in 
Corollary \ref{cor:mollif:time-space}, since the 
measure argument in $U^{(N)}$ is not convoluted by $f_N$. 

For a given $t \in [0,T]$, $\widetilde{U}^{(N,\varepsilon,\rho)}$ is continuously differentiable with respect to 
$(\widehat{m}^k)_{k \in F_N^+} \in {\mathcal O}_N$. Moreover, for almost every $t \in [0,T]$, for any two $m_1,m_2 \in {\mathcal P}_N$, we have
\begin{equation*}
\begin{split}
&U^{(N,\varepsilon,\rho)}\bigl(t,m_2 \bigr)
-U^{(N,\varepsilon,\rho)}\bigl(t,m_1 \bigr)
\\
&=
\int_0^1 
\biggl[ \sum_{k \in F_N} 
\partial_{\widehat{m}^k} 
U^{(N,\varepsilon,\rho)}\bigl( t , \lambda m_2 + (1-\lambda) m_1 
\bigr) \bigl(  \widehat{m}_2^k - \widehat{m}_1^k \bigr)
\biggr] d\lambda
\\
&= (1- \varepsilon) \int_0^1 
\biggl\{
\int_{{\mathbb R}^{2 \vert F_N^+ \vert}} \biggl[ 
\sum_{k \in F_N}  
\partial_{\widehat{m}^k} U^{(N)} \Bigl( t, \lambda m_2^{N,\varepsilon}(r) + 
(1- \lambda) m_1^{N,\varepsilon}(r) \Bigr)  \bigl( \widehat{m}^k_2- \widehat{m}_1^k \bigr) \biggr] 
\\
&\hspace{200pt} \times \prod_{j \in F_N^+} \rho
\bigl( \widehat{r}^j \bigr) 
\bigotimes_{j \in F_N^+} 
d \widehat{r}^j
\biggr\} d\lambda 
\\
&= (1- \varepsilon) \int_0^1 
\biggl\{
\int_{{\mathbb R}^{2 \vert F_N^+ \vert}} \biggl[  \int_{{\mathbb T}^d} 
\Bigl[ \sum_{k \in F_N \setminus \{0\}}
\widehat{{\mathcal Z}}^k
\Bigl( t, \lambda m_2^{N,\varepsilon}(r) + 
(1- \lambda) m_1^{N,\varepsilon}(r) \Bigr) e_k(x) \Bigr] d\bigl( m_2- m_1\bigr)(x) \biggr] 
\\
&\hspace{200pt} \times \prod_{j \in F_N^+} \rho
\bigl( \widehat{r}^j \bigr) 
\bigotimes_{j \in F_N^+} 
d \widehat{r}^j
\biggr\} d\lambda . 
\end{split}
\end{equation*}
By means of Remark 
\ref{rem:reg:Z}, 
we know that, 
for almost every $(t,m) \in [0,T] \times {\mathcal P}_N$, 
 the function $x \in {\mathbb T}^d \mapsto 
\sum_{k \in F_N \setminus \{0\}}
\widehat{{\mathcal Z}}^k
(t,m) e_k(x)$ is Lipschitz continuous (in $x$), 
uniformly in $(t,m)$ and in $N \geq 1$. 
This shows that 
the Lipschitz constant of $\widetilde U^{(N,\varepsilon,\rho)}(t,\cdot)$ in $m \in {\mathcal P}_N$ with respect to 
the distance $d_{W_1}$
is uniform with respect to $N \geq 1$ and almost every $t \in [0,T]$. 
Letting $\rho$ converge to the Dirac mass $\delta_0$ and then 
$\varepsilon$ tend to $0$, we deduce that 
the same is true for $U^{(N)}$: its Lipschitz constant in $m \in {\mathcal P}_N$
with respect to $d_{W_1}$ is uniform in $N \geq 1$ and $t \in [0,T]$ (since 
$U^{(N)}$ is already known to be continuous, it is quite straightforward to pass from 
almost every $t \in [0,T]$ to any $t \in [0,T]$). 

Moreover, 
by
using the growth properties of the Hamiltonian and by using 
item (1) in the statement of 
Theorem 
\ref{main:thm:MFG} together with \eqref{eq:master:weak:00}, we deduce that 
the essential supremum norm of 
$\partial_t U^{(N)}$ 
is bounded, uniformly with respect to $t$ and $N$, on 
any 
subset of ${\mathcal P}_N$ of the 
form $\{ m  \in {\mathcal P}_N :  \| m \|_\infty \leq c\}$, for any $c > 1$. 
In particular, using the fact that 
$U^{(N)}(T,\cdot) = G(\cdot)$, we deduce that, for any $c>1$, 
\begin{equation*}
\sup_{t \in [0,T]}
\int_{{\mathcal P}({\mathbb T}^d)} 
\bigl\vert U^{(N)}(t,m) 
\bigr\vert
{\mathbf 1}_{\{ m  \in {\mathcal P}_N :  \| m \|_\infty \leq c\}}
d {\mathbb P}_N(m) 
\end{equation*}
is bounded by a constant independent of $N$. For $c >1$, the 
subset 
$\{ m  \in {\mathcal P}_N :  \| m \|_\infty \leq c\}$ is of positive measure under ${\mathbb P}_N$
(as it contains a neighborhood of the Lebesgue measure), 
from which we deduce that 
\begin{equation*}
\sup_{N \geq 1} \sup_{t \in [0,T]}
\inf_{m \in {\mathcal P}_N} 
\bigl\vert U^{(N)}(t,m) \bigr\vert < \infty.
\end{equation*}
Using the Lipschitz property of $U^{(N)}$ in space, we get 
\begin{equation}
\label{eq:proof:uniqueness:bound:UN}
\sup_{N \geq 1} \sup_{t \in [0,T]}
\sup_{m \in {\mathcal P}_N} 
\bigl\vert U^{(N)}(t,m)\bigr\vert < \infty.
\end{equation}

\textit{Third Step.} 
By assumption (1) in the statement, 
we observe that  
$U^{(N)}$ satisfies the two main conclusions of Proposition 
\ref{prop:4:7},
for a fixed value of $N$ therein and for almost every $(t,m) \in [0,T] \times {\mathcal P}_N$. 
Similarly, using the weak 
one-sided Lipschitz property of ${\mathcal Z}$ in 
Definition \ref{def:weak:one-sided:lip}
and reverting the computations 
in \eqref{eq:weak:one-sided:1}--\eqref{eq:weak:one-sided:2}--\eqref{eq:weak:one-sided:3}, 
we can easily deduce that  
$U^{(N)}$ satisfies the conclusions of Corollary
\ref{cor:weak:semi-concavity}, say in the form of 
\eqref{eq:corollary:weak:semi-concavity:2}, again for a fixed value of $N$.  
{Importantly, this suffices to
repeat the arguments underpinning the proof of 
Theorem 
\ref{thm:uniqueness:HJB}}. 
Notice indeed that, for a given value of $N$, the fact that 
$U^{(N)}$ is defined on $[0,T] \times {\mathcal P}_N$ suffices to 
give a meaning to $(U^{(N)})^{N,\varepsilon,\rho}$
in Definition 
\ref{def:mollification}
and then to follow the computations of
Proposition 
\ref{prop:generalized:solution}
and
Theorem \ref{thm:uniqueness:HJB}
when $N$ therein is fixed. 
The key fact
is that, by combining the assumption
\eqref{eq:master:weak:00} 
with 
Theorem \ref{thm:FPK:mollified:proof} (which permit to handle the remainder $\zeta_N$ along the 
characteristics 
\eqref{eq:uniqueness:MKV}), 
we still have 
\eqref{eq:proof:uniqueness:HJB:for:master}, namely
for an 
initial $m_0 \in B_N(c)= {\mathcal P}_N \cap \{ m :  \sup_{x \in \bT^d} \vert \nabla m(x)  \vert \leq c, \ \inf_{x \in \bT^d}  m(x)   \geq 1/c \}$ for a constant $c>1$, 
\begin{equation*}
\begin{split}
&\int_{{\mathbb R}^{2\vert F_N^+ \vert}} 
\Bigl\vert \bigl( V^{N,\varepsilon,\rho} - (U^{(N)})^{N,\varepsilon,\rho} \bigr)\bigl(0,m_0(r)\bigr) \Bigr\vert
\prod_{j \in F_N^+} \rho_0 \bigl(  \widehat{r}^j \bigr)
\bigotimes_{j \in F_N^+} 
  d \widehat{r}^j 
\\
&\hspace{15pt} \leq \eta_{N,\varepsilon} + C_{N,\rho_0} 
\int_0^T \biggl[\int_{{\mathbb R}^{2\vert F_N^+\vert}} R^{N,\varepsilon,\rho}\bigl(t,(\widehat y^k)_{k \in F_N^+} \bigr)
{\mathbf 1}_{\{y \in {\mathcal P}_N\}} \bigotimes_{j \in F_N^+} 
  d \widehat{y}^j  
\biggr]
 dt,
 \end{split}
 \end{equation*}
with 
$\lim_{(N,\varepsilon) \rightarrow (\infty,0)} \eta_{N,\varepsilon} = 0$, 
with 
$C_{N,\rho_0}$ depending on $N$,  
$\rho_0$ and $c$, but independent of $\varepsilon$ and $\rho$, 
and with the remainder 
$R^{N,\varepsilon,\rho}$ being
as in the statement of Proposition 
\ref{prop:generalized:solution}. 

 Letting $\rho$ tend to the Dirac mass at $0$ for a fixed value of $N$ and invoking 
 Proposition \ref{prop:generalized:solution}, we can get rid of the term containing 
 $R^{N,\varepsilon,\rho}$ in the above inequality. 
 Using the regularity of $V$ and $U^{(N)}$ with respect to 
 $d_{W_1}$, we deduce that there exists a sequence 
 $(\eta_N')_{N \geq 1}$, converging to $0$, such that 
 $\vert V(0,m_0) - U^{(N)}(0,m_0) \vert \leq \eta_{N}'$ when $m_0 \in B_N(c)$. 
 Replacing $m_0$ by $m * f_N$,
 we deduce that 
 $(V(0,m*f_N) - U^{(N)}(0,m*f_N))_{N \geq 1}$ tends to $0$ for any $m \in \PP$ with a continuously differentiable strictly positive density. 
Recalling from $(i)$ in Theorem 
 \ref{thm:probability:probability}
 that, for ${\mathbb P}$ almost every 
 $m$, 
 $m$
has a continuously differentiable strictly positive density, we deduce that, 
for ${\mathbb P}$ almost every 
 $m$,
 $(V(0,m*f_N) - U^{(N)}(0,m*f_N))_{N \geq 1}$
 tends to $0$. Obviously, the same holds true for any initial time $t \in [0,T]$ instead of $0$. 
  
We eventually deduce that, for any $t \in [0,T]$, the sequence
 $(m \in {\mathcal P}({\mathbb T}^d) \mapsto 
 U^{(N)}(t,m*f_N))_{N \geq 1}$ 
 converges
 to $V(t,\cdot)$
  for the weak-star topology $\sigma^*(L^\infty( {\mathcal P}({\mathbb T}^d),{\mathbb P});L^1( {\mathcal P}({\mathbb T}^d),{\mathbb P}))$. 
The next step is to apply 
Lemma 
\ref{lem:W:mathcalW:Nn:2}
below
to $({\mathcal W}_1,W_1)=({\mathcal Z},V)$ and 
 $({\mathcal W}_2,W_2)=({\mathcal V},V)$, from which we deduce that 
 ${\mathcal Z}$ and ${\mathcal V}$ are equal.
 The key point is indeed to observe from assumption (3) in the statement that, for any $t \in [0,T]$, 
for any $k \in {\mathbb Z}^d \setminus \{0\}$,  
the sequence
$(m \mapsto \widehat{\mathcal Z}^k(t,m*f_{N}))_{N \geq 1}$ 
converges
to 
$m \mapsto \widehat{\mathcal Z}^k(t,m)$
 for the weak-star topology $\sigma^*(L^\infty({\mathcal P}({\mathbb T}^d),{\mathbb P});L^1({\mathcal P}({\mathbb T}^d),{\mathbb P}))$. In this respect, 
the form of assumption (3) is sufficient to identify uniquely  
 any weak limit: this follows from the 
 description of the Borel $\sigma$-algebra 
 on $\PP$ provided by Lemma 
  \ref{lem:Borel} together with a monotone class argument.
 We deduce that, for any $t \in [0,T]$, for ${\mathbb P}$ almost every 
 $m \in \PP$, the two functions 
 ${\mathcal V}(t,m,\cdot)$
 and 
 ${\mathcal Z}(t,m,\cdot)$
 coincide. 
\end{proof}

\subsection{Auxiliary lemmas}

We now provide two important auxiliary lemmas that we invoked in the proofs
of 
Proposition 
\ref{prop:value:is:gen:master:equation}
and
Theorem
\ref{main:thm:MFG}.

\begin{lem}
\label{lem:W:mathcalW:Nn}
Let ${\mathcal W} : {\mathcal P}({\mathbb T}^d) \times {\mathbb T}^d \rightarrow {\mathbb R}$
and 
$W : {\mathcal P}({\mathbb T}^d) \rightarrow {\mathbb R}$
 be bounded measurable functions such that
for a collection of bounded functions $(W^{(N)} :  {\mathcal P}_{N} \rightarrow {\mathbb R})_{N \geq 1}$, with each $W^{(N)}$ being Lipschitz continuous, the following three assumptions hold true: 
\vskip 5pt

$(i)$ For any $N \geq 1$, any $k \in F_{N}^+$ and almost every 
$m \in {\mathcal P}_{N}$, 
$\widehat{\mathcal W}^k(m) = \partial_{\widehat{m}_k} 
W^{(N)}(m)$, 
with $\widehat{\mathcal W}^k(m): = \reallywidehat{\displaystyle {\mathcal W}(m,\cdot)}^{-k}$; 
\vskip 5pt

$(ii)$ Up to a common subsequence,  
the functions 
$(m \mapsto \widehat{\mathcal W}^k(m*f_{N}))_{N \geq 1}$ 
converge,  
for each $k \in {\mathbb Z}^d \setminus \{0\}$,  
to 
$\widehat{\mathcal W}^k$
 for the weak-star topology $\sigma^*(L^\infty({\mathcal P}({\mathbb T}^d),
 {\mathbb P});L^1({\mathcal P}({\mathbb T}^d),
 {\mathbb P}))$;
\vskip 5pt

$(iii)$ 
Up to the same subsequence as in $(ii)$, 
the 
functions
$(m \mapsto
 W^{(N)}(m *f_N))_{N \geq 1}$
converge
to $W$
 for the same weak-star topology $\sigma^*(L^\infty({\mathcal P}({\mathbb T}^d),
 {\mathbb P});L^1({\mathcal P}({\mathbb T}^d),
 {\mathbb P}))$. 
\vskip 5pt

 Then, for any $\delta >0$ and $c>1$, there exists 
 an integer $N_{c,\delta}$ such that, for any $N_0 \geq N_{c,\delta}$, 
 any $k_0 \in F_{N_0}^+$,  any smooth function $\varphi : {\mathbb R}^{2 \vert F_{N_0}^+ \vert} \rightarrow {\mathbb R}$
 whose support is included in 
 $\{ (\widehat{m}^j)_{j \in F_{N_0}^+} \in {\mathbb C}^{\vert F_{N_0}^+\vert} 
 \simeq 
 {\mathbb R}^{2 \vert F_{N_0}^+\vert} 
 : 
1 + 2 \inf_{x \in {\mathbb T}^d}  \sum_{j \in F_{N_0}^+} \Re[ \widehat{m}^j e_{-j}(x) ]
\geq 1/c\}$, 
 \begin{equation*}
\begin{split} 
 &\biggl\vert 
 \int_{{\mathcal P}({\mathbb T}^d)}  W(m) \partial _{\widehat{m}^{k_0}} \varphi\Bigl( \bigl( \widehat{m}^k \bigr)_{k \in F_{N_0}^+} \Bigr)d {\mathbb P}(m)
\\
&\hspace{15pt} +    
 \int_{\mathcal P({\mathbb T}^d)}  \biggl\{ \widehat{\mathcal W}^{k_0}( m )
  - 
2 W (m  ) 
 \vert k_0 \vert^{2pd} 
 \widehat{m}^{k_0} \biggr\}
\varphi\Bigl( \bigl( \widehat{m}^k \bigr)_{k \in F_{N_0}^+} \Bigr)
d {\mathbb P}(m) 
 \biggr\vert \leq
 C \delta,
 \end{split}
 \end{equation*} 
 for a constant $C$ independent of $N_0$ and 
 depending on $\varphi$ only through 
 $\| \varphi \|_\infty$ and 
 $
 \|\partial_{\widehat{m}^{k_0}} \varphi \|_\infty$. 
\end{lem}

Notice that the role of $N_0$ in the statement is to
adjust the size of the support of the test function 
$\varphi$.

\begin{proof}
Throughout
the proof, we fix $\delta >0$ and $c>1$ and then choose 
$N_0 \geq 1$, $\vartheta : [0,T] \rightarrow {\mathbb R}$
 and $\varphi : {\mathbb R}^{2 \vert F_{N_0}^+ \vert} \rightarrow {\mathbb R}$ as in the statement. 
Without any loss of generality, we also do as if the subsequences in assumptions $(ii)$ and 
$(iii)$ were the full sequence $(N)_{N \geq 1}$ itself.  
 
By item $(iii)$ in the statement, we get 
 \begin{equation*}
 \begin{split}
  \lim_{N \rightarrow \infty}
 \int_{{\mathcal P}({\mathbb T}^d)} W^{(N)}\bigl(m*f_{N}\bigr) 
  \varphi\Bigl( \bigl( \widehat{m}^k \bigr)_{k \in F_{N_0}^+} \Bigr) d {\mathbb P}(m)
 & =
\int_{{\mathcal P}({\mathbb T}^d)} W(m)  \varphi\Bigl( \bigl( \widehat{m}^k \bigr)_{k \in F_{N_0}^+} \Bigr) d {\mathbb P}(m).
\end{split}
 \end{equation*}
Using the fact that the functions 
$(W^{(N)})_{N \geq 1}$ are uniformly bounded
to get the first equality and
invoking 
 $(iii)$ in 
the statement of 
Theorem 
\ref{thm:probability:probability}, the above limit can be reformulated as
\begin{equation}
\label{eq:lem:5:7:step:1}
 \begin{split}
& \lim_{N \rightarrow \infty}
 \int_{\PP} W^{(N)}\bigl(m*f_{N}\bigr) 
  \varphi\Bigl( \bigl( \widehat{m}^k \bigr)_{k \in F_{N_0}^+} \Bigr) d {\mathbb P}_{N}(m) 
 =
\int_{{\mathcal P}({\mathbb T}^d)} W(m)    \varphi\Bigl( \bigl( \widehat{m}^k \bigr)_{k \in F_{N_0}^+} \Bigr) d {\mathbb P}(m).
\end{split}
 \end{equation}
Here, we can easily replace 
$\varphi$ by its partial derivative 
$\partial_{\widehat{m}^{k_0}} \varphi$, 
for a given $k_0 \in F_{N_0}^+$,
 \begin{equation*}
 \begin{split}
& \lim_{N \rightarrow \infty} 
 \int_{\PP}  W^{(N)}\bigl(m*f_{N}\bigr)   \partial _{\widehat{m}^{k_0}} \varphi\Bigl( \bigl( \widehat{m}^k \bigr)_{k \in F_{N_0}^+} \Bigr)d {\mathbb P}_{N}(m) 
 \\
 &\hspace{15pt} = 
 \int_{{\mathcal P}({\mathbb T}^d)} W(m)    \partial _{\widehat{m}^{k_0}} \varphi\Bigl( \bigl( \widehat{m}^k \bigr)_{k \in F_{N_0}^+} \Bigr) d {\mathbb P}(m).
 \end{split}
 \end{equation*}  
 We now recall the form of 
 ${\mathbb P}_{N}$ in 
 \eqref{eq:Gamma_N}--\eqref{eq:I_N} in order to study the left-hand side. 
 We have
  \begin{equation*}
 \begin{split}
&  \int_{\PP}  W^{(N)}\bigl(m*f_{N}\bigr)  \partial _{\widehat{m}^{k_0}} \varphi\Bigl( \bigl( \widehat{m}^k \bigr)_{k \in F_{N_0}^+} \Bigr)d {\mathbb P}_{N}(m) 
\\
&= \frac1{Z_{N}} 
 \int_{{\mathcal O}_{N}} W^{(N)} \bigl( m*f_{N}\bigr)  \partial _{\widehat{m}^{k_0}} \varphi\Bigl( \bigl( \widehat{m}^k \bigr)_{k \in F_{N_0}^+} \Bigr)  
  \exp \Bigl( - \sum_{k \in F_{N}^+} \vert k\vert^{2pd} { \vert \widehat{m}^k \vert^2} \Bigr)
\bigotimes_{k \in F_{N}^+}
d \widehat{m}^k,
 \end{split}
 \end{equation*}  
with $p \geq 5$
fixed 
(see \eqref{eq:Gamma_N} for the original occurrence of $p$)
and with $m$ in the second line being obviously identified with 
${\mathscr I}_{N}((\widehat{m}^k)_{k \in F_{N}^+})$. 

Fix now $N \geq 1$. By Lemmas
\ref{lem:25:second} 
and
\ref{lem:27} in the appendix, we can choose $N_0$
large enough such that, 
if
$N \geq N_0$ and the support of $\varphi$ is included in ${\mathscr I}_{N}^{-1}(A_{N_0})$ (with $A_{N_0}$ as in the statement of 
Lemma \ref{lem:25:second}),
then
\begin{equation}
\label{eq:pass:from:N0:N} 
{\mathbb P}_{N}
\biggl( 
\Bigl\{ 
(\widehat{m}^k)_{k \in F_{N_0}^+} \in \textrm{\rm support}(\varphi)
\Bigr\} 
\cap 
\biggl\{
\sum_{k \in F_{N}^+ \setminus F_{N_0}^+}
 \vert \widehat{m}^k \vert 
  \geq  \frac{a_0}{\vert k \vert^{5d/2}} \biggr\} \biggr) 
\leq \delta,
\end{equation}
for $a_0$ fixed as in the statement of Lemma 
\ref{lem:25:second} (in particular, $a_0$ only depends on the dimension). By the same statement, the simple fact that 
$${\mathscr I}_{N_0}((\widehat{m}^k)_{k \in F_{N_0}^+}) \in A_{N_0}
\quad \text{and} 
\quad \sum_{k \in F_{N}^+ \setminus F_{N_0}^+}
 \vert \widehat{m}^k \vert 
  < \frac{a_0}{\vert k \vert^{5d/2}}$$
  implies 
  that $(\widehat{m}^k)_{k \in F_{N}^+} \in {\mathcal O}_{N}$. 
  For such a choice, we get that 
  \begin{equation*}
\begin{split} 
&\biggl\vert \int_{\PP}  W^{(N)}\bigl(m*f_{N}\bigr)  \partial _{\widehat{m}^{k_0}} \varphi\Bigl( \bigl( \widehat{m}^k \bigr)_{k \in F_{N_0}^+} \Bigr)d {\mathbb P}_{N}(m)  
\\
&\hspace{15pt} - \frac1{Z_{N}}    \int_{{\mathbb R}^{2 \vert F_{N}^+ \vert}} W^{(N)}\bigl( m*f_{N}\bigr) \partial _{\widehat{m}^{k_0}} \varphi\Bigl( \bigl( \widehat{m}^k \bigr)_{k \in F_{N_0}^+} \Bigr) 
\\
&\hspace{50pt} \times 
{\mathbf 1}_{\{
\sum_{k \in F_{N}^+ \setminus F_{N_0}^+}
 \vert \widehat{m}^k \vert 
  < a_0/\vert k \vert^{5d/2}\}}
\exp \Bigl( - \sum_{k \in F_{N}^+} \vert k\vert^{2pd} { \vert \widehat{m}^k \vert^2} \Bigr)
\bigotimes_{k \in F_{N}^+}
d \widehat{m}^k \biggr\vert \leq C \delta,  
 \end{split}
 \end{equation*}  
 with $C$ depending on 
 $\varphi$ through 
 $\|\partial_{\widehat{m}^{k_0}} \varphi \|_\infty$, but being independent of $N$.
   
 Then, we can make an integration by parts in the second term in the left-hand side. 
This leads to two terms: 
$(a)$ 
one derivative is acting on   $W^{(N)}$ and leads to 
$\widehat{\mathcal W}^{k_0}$ (using item $(i)$ in the statement); $(b)$ the other derivative is acting on the density. 
Therefore, we obtain 
 \begin{equation*}
 \begin{split}
&\biggl\vert   \int_{\PP}  W^{(N)}\bigl(m*f_{N}\bigr)  \partial _{\widehat{m}^{k_0}} \varphi\Bigl( \bigl( \widehat{m}^k \bigr)_{k \in F_{N_0}^+} \Bigr)d {\mathbb P}_{N}(m)
\\
&\hspace{15pt} + \frac1{Z_{N}}   
 \int_{{\mathbb R}^{2 \vert F_{N}^+ \vert}}  \biggl\{ \partial_{\widehat{m}^{k_0}} W^{(N)}\bigl( m*f_{N}\bigr)
 \widehat{f}_{N}^{k_0}  - 
2 W^{(N)}\bigl( m*f_{N}\bigr) 
  \vert k_0 \vert^{2pd} 
 \widehat{m}^{k_0} \biggr\}
\\
&\hspace{30pt} \times  
{\mathbf 1}_{\{
\sum_{k \in F_{N}^+ \setminus F_{N_0}^+}
 \vert \widehat{m}^k \vert 
  < a_0/\vert k \vert^{5d/2}\}}
    \varphi\Bigl( \bigl( \widehat{m}^k \bigr)_{k \in F_{N_0}^+} \Bigr)  \exp \Bigl( - \sum_{k \in F_{N}^+} \vert k\vert^{2pd} { \vert \widehat{m}^k \vert^2} \Bigr)
\bigotimes_{k \in F_{N}^+}
d \widehat{m}^k
 \biggr\vert 
 \\
 &\leq C \delta,
 \end{split}
 \end{equation*}
and then,
using once 
again 
\eqref{eq:pass:from:N0:N} together with the 
bound 
$\sup_{N \geq 1}
 \textrm{\rm essup}_{m \in {\mathcal P}_N}
\vert 
\partial_{\widehat{m}^{k_0}} W^{(N)}(m) \vert 
= 
\sup_{N \geq 1}
 \textrm{\rm essup}_{m \in {\mathcal P}_N}
\vert 
\widehat{\mathcal W}^{k_0}(m)
 \vert < \infty$, 
 \begin{equation*}
 \begin{split}
&\biggl\vert    \int_{\PP} W^{(N)}\bigl(m*f_{N}\bigr)  \partial _{\widehat{m}^{k_0}} \varphi\Bigl( \bigl( \widehat{m}^k \bigr)_{k \in F_{N_0}^+} \Bigr)d {\mathbb P}_{N}(m)  
\\
&\hspace{15pt} +    
 \int_{{\mathcal P}({\mathbb T}^d)}  \biggl\{ \partial_{\widehat{m}^{k_0}} W^{(N)}\Bigl( m*f_{N}\Bigr)
 \widehat{f}_{N}^{k_0}  - 
2 W^{(N)}\bigl( m*f_{N}\bigr) 
  \vert k_0 \vert^{2pd} 
 \widehat{m}^{k_0} \biggr\}
       \varphi\Bigl( \bigl( \widehat{m}^k \bigr)_{k \in F_{N_0}^+} \Bigr)
d {\mathbb P}_{N}(m)  
  \biggr\vert
  \\
  &\leq C \delta, 
 \end{split}
 \end{equation*}
 with $C$ now depending on $\varphi$ through $\| \varphi \|_\infty$
 and 
 $\|\partial_{\widehat{m}^{k_0}} \varphi \|_\infty$. 
Recalling again that 
$  \partial_{\widehat{m}^{k_0}} W^{(N)}(m*f_{N}) = 
\widehat{\mathcal W}^{k_0}(m*f_{N})$ 
almost everywhere, we can use assumption $(ii)$ in the statement in order to pass to the limit in the term 
$\partial_{\widehat{m}^{k_0}} W^{(N)}( m*f_{N})$. 
Recalling 
\eqref{eq:lem:5:7:step:1}
and letting $N$ tend to $\infty$, we eventually get 
 \begin{equation*}
\begin{split} 
 &\biggl\vert 
 \int_{{\mathcal P}({\mathbb T}^d)}  W(m)  \partial _{\widehat{m}^{k_0}} \varphi\Bigl( \bigl( \widehat{m}^k \bigr)_{k \in F_{N_0}^+} \Bigr)d {\mathbb P}(m)
\\
&\hspace{15pt} +    
 \int_{\mathcal P({\mathbb T}^d)}  \biggl\{ \widehat{\mathcal W}^{k_0}( m )
  - 
2 W (m  ) 
  \vert k_0 \vert^{2pd} 
 \widehat{m}^{k_0} \biggr\}
   \varphi\Bigl( \bigl( \widehat{m}^k \bigr)_{k \in F_{N_0}^+} \Bigr)
d {\mathbb P}(m) 
 \biggr\vert \leq
 C \delta,
 \end{split}
 \end{equation*} 
 which completes the proof. 
\end{proof}

\begin{lem}
\label{lem:W:mathcalW:Nn:2}
Let 
$({\mathcal W}_1,W_1)$ and 
$({\mathcal W}_2,W_2)$ be two pairs of functions 
satisfying the assumptions of Lemma 
\ref{lem:W:mathcalW:Nn}. 
If $W_1$ and $W_2$ are equal almost everywhere under ${\mathbb P}$, then 
${\mathcal W}_1$ and ${\mathcal W}_2$
are also equal 
almost everywhere under ${\mathbb P}$.
\end{lem}

\begin{proof}
We apply 
Lemma \ref{lem:W:mathcalW:Nn} to $({\mathcal W}_i,W_i)$, for $i=1,2$, and then use the fact that $W_1=W_2$. 
We deduce that, 
for $\delta>0$ and $c>1$, there exists 
an integer $N_{c,\delta}$ such that, for $N_0 \geq N_{c,\delta}$, 
for
any $k_0 \in F_{N_0}^+$
 and
 any smooth function $\varphi : {\mathbb R}^{2 \vert F_{N_0}^+ \vert} \rightarrow {\mathbb R}$
 whose support is included in 
 $\{ (\widehat{m}^j)_{j \in F_{N_0}^+} \in {\mathbb C}^{\vert F_{N_0}^+\vert} 
 \simeq {\mathbb R}^{2 \vert F_{N_0}^+ \vert} : 
1 + 2 \inf_{x \in {\mathbb T}^d}   \sum_{j \in F_{N_0}^+} \Re [\widehat{m}^j e_{-j}(x)] 
\geq 1/c\}$, 
 \begin{equation}
 \label{eq:W1-W2}
 \biggl\vert 
 \int_{{\mathcal P}({\mathbb T}^d)} \Bigl(\widehat{\mathcal W}_1^{k_0}(m) - \widehat{\mathcal W}_2^{k_0}(m)\Bigr)   \varphi\Bigl( \bigl( \widehat{m}^k \bigr)_{k \in F_{N_0}^+} \Bigr) d {\mathbb P}(m)   \biggr\vert \leq  C \delta,
 \end{equation} 
 for a constant $C$ only depending on $\varphi$ through $\| \varphi \|_\infty$
 and $\| \partial_{\widehat{m}^{k_0}} \varphi \|_\infty$. 
By an obvious regularization procedure, the same holds if $\varphi$ is merely bounded 
and  
Lipschitz continuous, in which case the constant $C$ depends on $\varphi$ through 
the bound and the Lipschitz constant of $\varphi$.

For $\vartheta : {\mathbb R} \rightarrow [0,1]$, 
a $c$-Lipschitz function 
that is equal to $0$ on $]-\infty,1/c[$, let 
\begin{equation*}
\Theta
\Bigl( \bigl( \widehat{m}^k \bigr)_{k \in F_{N_0}^+} \Bigr)
=
\vartheta \Bigl(
 1 + 2
 \inf_{x \in {\mathbb T}^d} 
 \sum_{k \in F_{N_0}^+} 
\Re \bigl[ \widehat{m}^k e_{-k}(x)\bigr]
\Bigr), \quad 
\bigl( \widehat{m}^k \bigr)_{k \in F_{N_0}^+} \in {\mathbb C}^{\vert F_{N_0}^+ \vert}
\simeq 
{\mathbb R}^{2 \vert F_{N_0}^+ \vert}.
\end{equation*}
Observe that
the support of $\Theta$ is included in ${\mathcal O}_{N_0}$.
Moreover, 
for $( \widehat{m}^k)_{k \in F_{N_0}^+} \in {\mathbb C}^{\vert F_{N_0}^+ \vert}
\simeq {\mathbb R}^{2 \vert F_{N_0}^+\vert}$
and 
$\widehat{r}^{k_0} \in {\mathbb C}$, with $k_0 \in F_{N_0}^+$, we have 
\begin{equation*}
\begin{split}
&\Bigl\vert 
\Theta
\Bigl( \bigl( \widehat{m}^k \bigr)_{k \in F_{N_0}^+} + {\mathbf 1_{\{k=k_0\}}} \widehat{r}^{k_0}  \Bigr)
-
\Theta
\Bigl( \bigl( \widehat{m}^k \bigr)_{k \in F_{N_0}^+} \Bigr)
\Bigr\vert
\\
&\leq 2 c \sup_{x \in {\mathbb T}^d} 
\biggl\vert 
\sum_{k \in F_{N_0}^+} 
\widehat{m}^k e_{-k}(x)
-
\sum_{k \in F_{N_0}^+} \bigl( 
\widehat{m}^k + {\mathbf 1}_{\{k=k_0\}} \widehat{r}^{k_0} 
\bigr) 
e_{-k}(x)
\biggr\vert\leq 2c  \bigl\vert \widehat{r}^{k_0} \bigr\vert,
\end{split}
\end{equation*}
hence proving that $\Theta$ is $2 c$-Lipschitz
with respect to $\widehat{m}^{k_0}$. 
In particular, for a continuously differentiable test function $\phi$ on ${\mathcal P}({\mathbb T}^d)$, 
we deduce from 
\eqref{eq:W1-W2}
 that, for $N_0$ large enough, 
\begin{equation*}
\begin{split}
&\biggl\vert 
 \int_{{\mathcal P}({\mathbb T}^d)} \Bigl(\widehat{\mathcal W}_1^{k_0}(m) - \widehat{\mathcal W}_2^{k_0}(m)\Bigr)   \phi\bigl( m* f_{N_0} \bigr) 
 \vartheta\Bigl(
 1 + 
2 \inf_{x \in {\mathbb T}^d} 
\sum_{k \in F_{N_0}^+} 
\Re\bigl[\widehat{m}^k e_{-k}(x)\bigr] 
\Bigr)
 d {\mathbb P}(m)  \biggr\vert 
 \leq  C(c) \delta,
\end{split}
\end{equation*}
with $C(c)$ depending on $c$, but being independent of $N_0$ and $\delta$. 
Above, $\phi(m*f_{N_0})$ is regarded as a smooth function on ${\mathcal O}_{N_0}$. 

As before, $m*f_{N_0}$ converges to $m$ for the Wasserstein distance
as $N_0$ tends to $\infty$. 
Moreover, Lemma 
\ref{lem:density:under:P}
below says that, 
${\mathbb P}$-almost everywhere, the series $\sum_{k \in {\mathbb Z}^d} \vert \widehat{m}^k \vert$
is absolutely convergent hence proving that 
$ 1 + 
 \inf_{x \in {\mathbb T}^d} 
2\sum_{k \in F_{N_0}^+} 
\Re[\widehat{m}^k e_{-k}(x)]$
converges to $\inf_{x \in {\mathbb T}^d} m(x)$
with $m$ being identified with its density. 
Therefore, we can let $N_0$ tend to $\infty$. We thus get 
\begin{equation*}
\begin{split}
 \int_{{\mathcal P}({\mathbb T}^d)} \Bigl(\widehat{\mathcal W}_1^{k_0}(m) - \widehat{\mathcal W}_2^{k_0}(m)\Bigr)
   \phi(m)
 \vartheta \Bigl(
 \inf_{x \in {\mathbb T}^d} 
m(x)
\Bigr)
 d {\mathbb P}(m)  
=0.
\end{split}
\end{equation*}
There is no difficulty for sending $c$ to $+\infty$ in the definition of $\vartheta$
since ${\mathbb P}$ almost every $m \in \PP$ has a (strictly) positive density.
We obtain
\begin{equation*}
 \int_{{\mathcal P}({\mathbb T}^d)} \Bigl(\widehat{\mathcal W}_1^{k_0}(m) - \widehat{\mathcal W}_2^{k_0}(m)\Bigr)  \phi(m)
 d {\mathbb P}(m)  
=0,
\end{equation*}
for any 
continuously differentiable test function $\phi$ on ${\mathcal P}({\mathbb T}^d)$. 
By the approximation procedure 
defined in Proposition 
\ref{prop:mollification}, 
the same holds true if  
$\phi$ is merely continuous on ${\mathcal P}({\mathbb T}^d)$. 
And, then 
the same is true 
if  
$\phi$ is bounded and measurable on ${\mathcal P}({\mathbb T}^d)$, see
\cite{Wisniewski}.
This suffices to identify 
$\widehat{\mathcal W}_1^{k_0}$ and $\widehat{\mathcal W}_2^{k_0}$. 
\end{proof}

\subsection{Classical solutions to the master equation as weak solutions}
We now prove the following result, which legitimates Definition 
\ref{def:master:equation:gen}: 

\begin{prop}
\label{prop:verification:classical:conservative}
If $U$ is a classical solution to the master equation, 
with $\partial_t U$, $\partial_x U$, $\partial^2_{xx} U$, $\partial_\mu U$, $\partial_y \partial_{\mu} U$ 
being continuous
(with $y$
denoting the last argument in the derivative $\partial_\mu U(t,x,m)(y)$), then the centered version 
\begin{equation*}
\tilde U(t,x,m) := U(t,x,m) - \int_{{\mathbb T}^d} U(t,y,m) dy, 
\quad t \in [0,T], \ x \in {\mathbb T}^d, \ m \in \PP, 
\end{equation*}
satisfies the 
conservative version \eqref{master:centred} of the master equation. 
{In particular, $\tilde{U}$ is a weak solution to the master equation, in the sense of Definition \ref{def:master:equation:gen}.}
\end{prop}

\begin{proof}
For any starting point $(t,m) \in [0,T] \times {\mathcal P}({\mathbb T}^d)$, we can construct 
a solution to the system 
\eqref{eq:MFG:system} by solving 
the Fokker-Planck equation
\begin{equation*}
\partial_t m_t(x) 
- \textrm{\rm div}_x \Bigl( \partial_p H\bigl(x,\nabla_x U(t,x,m_t)\bigr) m_t(x) \Bigr)
- \frac12 \Delta_x m_t(x) = 0,
\end{equation*}
and then, by letting, $u_t(x)= U(t,x,m_t)$. 
In fact, this solution can be proven to be unique, see for instance
\cite[Proposition 5.106]{CarmonaDelarue_book_I}. 

By Proposition 
	\ref{prop:Briani}, the forward component of this unique solution, i.e. $(m_t)_{0 \leq t \leq T}$, must coincide with the (hence unique)
	optimal path of the MFCP, when initialized from $(t,m)$. 
	In turn, 
Lemma
\ref{lem:superjet}
	implies, for any $N \geq 1$, 
	 \begin{equation*}
\partial_{\widehat{m}^k} V(t,m) = 
\int_{{\mathbb T}^d} U(t,x,m) e^{-2 \i \pi k \cdot x} dx, \quad k \in F_N \setminus \{0\}, 
\end{equation*}
for almost every $(t,m) \in [0,T] \times {\mathcal P}_N$ under the probability measure $\textrm{\rm Leb}_1 \otimes {\mathbb P}_N$.
By continuity of $U$ in the right-hand side, 
we deduce that $V$ is differentiable with respect to $\widehat{m}^k$ on the entire $[0,T] \times 
{\mathcal P}_N$. 
We deduce from Schwarz' theorem that, for 
every $(t,m) \in [0,T] \times {\mathcal P}_N$
and
for
any 
$k,\ell \in F_N \setminus \{0\}$,  
 \begin{equation*}
\int_{{\mathbb T}^d} \partial_{\widehat{m}^\ell} U(t,x,m) e^{-2 \i \pi k \cdot x} dx = 
\int_{{\mathbb T}^d} \partial_{\widehat{m}^k} U(t,x,m) e^{-2 \i \pi \ell \cdot x} dx,
\quad k, \ \ell \in F_N \setminus \{0\}, 
\end{equation*}
which rewrites
 \begin{equation*}
\int_{{\mathbb T}^d} \int_{{\mathbb T}^d}  
\frac{\delta U}{\delta m}(t,x,m)(y) e^{-2 \i \pi k \cdot x}e^{-2 \i \pi \ell \cdot y} dx dy = 
\int_{{\mathbb T}^d}\int_{{\mathbb T}^d} \frac{\delta U}{\delta m}(t,x,m)(y)
e^{-2 \i \pi k \cdot y}
 e^{-2 \i \pi \ell \cdot x} dx, 
\end{equation*}
for 
$k, \ell \in F_N \setminus \{0\}$.
By invoking the continuity of $U$
and by approximating 
any measure $m \in \PP$ by the sequence 
$(m*f_N)_{N \geq 1}$, 
the above holds true for any $(t,m) \in [0,T] \times \PP$ and any 
$k,\ell \in {\mathbb Z}^d \setminus \{0\}$. 
This shows
\begin{equation*}
\frac{\delta U}{\delta m}(t,x,m)(y)
-
\frac{\delta U}{\delta m}(t,y,m)(x)
= 
 \psi(y)-\phi(x),
 \quad x,y \in {\mathbb T}^d, 
\end{equation*}
for two real-valued functions $\phi$ and $\psi$ 
defined on ${\mathbb T}^d$ (both depending on $m$). 
Integrating the above identity in $y$ and then in $x$ with respect to the Lebesgue measure on ${\mathbb T}^d$, we get
from 
\eqref{eq:centring}:
\begin{equation*}
\begin{split}
&\int_{{\mathbb T}^d} \frac{\delta U}{\delta m}(t,y,m)(x) dy = \phi(x)
-
\int_{{\mathbb T}^d} \psi(y) dy, 
\\
&\int_{{\mathbb T}^d} 
\frac{\delta U}{\delta m}(t,x,m)(y)
dx
= 
\psi(y)-
\int_{{\mathbb T}^d} \phi(x) dx, 
\end{split}
\end{equation*}
for
$x,y \in {\mathbb T}^d$.
Integrating again in $x$ the first line, we get 
\begin{equation*}
\int_{{\mathbb T}^d} \psi(y) dy
= 
\int_{{\mathbb T}^d} \phi(x) dx, 
\end{equation*}
and then
\begin{equation*}
\frac{\delta U}{\delta m}(t,x,m)(y)
-
\frac{\delta U}{\delta m}(t,y,m)(x)
= 
\int_{{\mathbb T}^d} 
\frac{\delta U}{\delta m}(t,x,m)(y)
dx 
-
\int_{{\mathbb T}^d} 
\frac{\delta U}{\delta m}(t,y,m)(x)
dy,
\end{equation*}
for
$x,y \in {\mathbb T}^d$.
We then have 
\begin{equation*}
\begin{split}
&\frac{\delta}{\delta m} \bigg\{  \int_{{\mathbb T}^d} 
H \bigl( y,  \nabla_y  U(t,y,m) \bigr) d m(y) \biggr\}(x)
\\
&=
H \bigl( x,  \nabla_y  U(t,x,m) \bigr) 
-
 \int_{{\mathbb T}^d} 
H \bigl( y,  \nabla_y  U(t,y,m) \bigr) dy
\\
&\hspace{15pt} +
 \int_{{\mathbb T}^d} 
\partial_p H \bigl( y,  \nabla_y  U(t,y,m) \bigr) 
\cdot 
\partial_y \frac{\delta U}{\delta m}(t,y,m)(x)
d m(y) 
\\
&=
H \bigl( x,  \nabla_y  U(t,x,m) \bigr) 
-
 \int_{{\mathbb T}^d} 
H \bigl( y,  \nabla_y   U(t,y,m) \bigr) d y
\\
&\hspace{15pt} +
 \int_{{\mathbb T}^d} 
\partial_p H \bigl( y,  \nabla_y  U(t,y,m) \bigr) 
\cdot 
\partial_y \frac{\delta U}{\delta m}(t,x,m)(y)
d m(y) 
\\
&\hspace{15pt} - 
\int_{{\mathbb T}^d}  \int_{{\mathbb T}^d} 
\partial_p H \bigl( y,  \nabla_y  U(t,y,m) \bigr) 
\cdot 
\partial_y \frac{\delta U}{\delta m}(t,x,m)(y)
d x \, d m(y), 
\end{split}
\end{equation*}
and similarly, 
\begin{equation*}
\begin{split}
& \frac{\delta}{\delta m} \bigg\{  \int_{{\mathbb T}^d} 
\text{Tr} \Bigl[ 
\partial^2_y   U(t,y,m) \Bigr] d m(y)   \bigg\}(x)
\\
&=
\text{Tr} \Bigl[ 
\partial^2_x  U(t,x,m) \Bigr]
-
 \int_{{\mathbb T}^d} 
\text{Tr} \Bigl[ 
\partial^2_y  U(t,y,m) \Bigr] d y 
\\
&\hspace{15pt} + 
  \int_{{\mathbb T}^d} 
\text{Tr} \Bigl[ 
\partial^2_y   \frac{\delta U}{\delta m}(t,y,m)(x) \Bigr] d m(y)  
\\
&= \text{Tr} \Bigl[ 
\partial^2_x   U(t,x,m) \Bigr]
-
 \int_{{\mathbb T}^d} 
\text{Tr} \Bigl[ 
\partial^2_y   U(t,y,m) \Bigr] d y
\\
&\hspace{15pt} + 
  \int_{{\mathbb T}^d} 
\text{Tr} \Bigl[ 
\partial^2_y   \frac{\delta U}{\delta m}(t,x,m)(y) \Bigr] d m(y)  
-
  \int_{{\mathbb T}^d} 
  \int_{{\mathbb T}^d} 
  \text{Tr} \Bigl[ 
\partial^2_y   \frac{\delta U}{\delta m}(t,x,m)(y) \Bigr] d x \,  d m(y).
\end{split}
\end{equation*}
The conclusion easily follows: it suffices to
integrate 
the equation
\eqref{master}
in 
$x$ with respect to 
the Lebesgue measure, to make the difference between 
\eqref{master}
and the hence integrated version of 
\eqref{master} 
and then to insert the latter  two identities, noticing that 
$\nabla_x U=\nabla_x \tilde U$.   

{The last claim then follows since $\tilde{U}$ solves \eqref{eq:master:weak} in the classical sense. Indeed, we can argue as in Subsection \ref{subse:spatial:deri:interpretation}, the computations being here legitimate because the solution is smooth. }
\end{proof}

\section{appendix}
\label{se:6}

\subsection{Construction of a probability measure on ${\mathcal P}({\mathbb T}^d)$ satisfying 
Theorem \ref{thm:probability:probability}.}
\label{sec:6}

This (long) subsection is dedicated to the proof of 
Theorem 
\ref{thm:probability:probability}. 
This comes as a by-product of a generic construction, which makes use (in a quite systematic manner) 
of the notations introduced in Subsection 
\ref{subse:PP} (we invite the reader to have a new look at them before she/he enters the details of the proof below). 
In particular, we recall 
\eqref{eq:Gamma_N}, 
\eqref{eq:I_N}
and the subsequent notation
${\mathbb P}_N= \Gamma_N \circ {\mathscr I}_N^{-1}$. 
Quite often in the analysis, we also make use of the normalization constant 
\begin{equation}
\label{eq:cN} c_N := 
\int_{{\mathbb R}^{2\vert F_N^+\vert  }}
\exp \biggl( - \sum_{k \in F_{N}^+ }\vert k\vert^{2 p d} { \vert \widehat{m}^k \vert^2} \biggr)
\bigotimes_{j \in F_{N}^+ } 
d \widehat{m}^j.
\end{equation}
(Notice that $c_N$ is different from $Z_N$ in \eqref{eq:Gamma_N}.)
Moreover, we also introduce  a sequence 
$(\xi_k)_{k \in {\mathbb N}^d \setminus \{0\}}$
of independent two-dimensional Gaussian random variables on an auxiliary probability space with 
${\mathbf P}$ as probability measure, such that each 
$\xi_k$ has $I_2$
 as covariance matrix (with $I_2$ being the identity matrix of dimension 2).

\subsubsection{Properties of the sequence 
$({\mathbb P}_N)_{N \geq 1}$}
In this paragraph, we obtain a series of lemmas on the properties of
the measures 
$({\mathbb P}_N)_{N \geq 1}$. We start with the following lemma: 
\begin{lem}
\label{lem:24}
For any two integers $N_0 \leq N$  
and for any $[0,1]$-valued 
measurable 
function 
$\varphi$ defined on ${\mathbb R}^{2 \vert F_{N_0}^+\vert}$ that is equal to zero outside ${\mathcal O}_{N_0}$,
\begin{equation*}
\begin{split}
&\int_{{\mathcal P}({\mathbb T}^d)} \varphi\Bigl( (\widehat{m}^k)_{k \in F_{N_0}^+ } \Bigr) d {\mathbb P}_N(m) 
\\
&\hspace{15pt} \leq \frac{c_N}{c_{N_0}} \frac1{Z_{N}}
\int_{{\mathcal O}_{N_0}} \varphi \Bigl( (\widehat{m}^k)_{k \in F_{N_0}^+ } \Bigr) 
\exp \biggl( - \sum_{k \in F_{N_0}^+} \vert k\vert^{2 p d} { \vert \widehat{m}^k \vert^2} \biggr)
\bigotimes_{j \in  F_{N_0}^+} 
d \widehat{m}^j.
\end{split}
\end{equation*}
Moreover, for any $\varepsilon \in (0,1)$,
if 
$\varphi$ is null outside the points
$(\widehat{m}^k)_{k \in F_{N_0}^+}$ such that 
\begin{equation}
\label{eq:compact:support}
1+
2 \inf_{x \in {\mathbb T}^d} 
\sum_{k \in F_{N_0}^+} 
2 \Re \bigl[ \widehat{m}^k e_{-k}(x) \bigr] \geq \varepsilon, 
\end{equation}
then
\begin{equation*}
\begin{split}
&\int_{{\mathcal P}({\mathbb T}^d)} \varphi\Bigl( (\widehat{m}^k)_{k \in F_{N_0}^+ } \Bigr) d {\mathbb P}_N(m)
\\
&\hspace{15pt} \geq \frac{c_N}{c_{N_0}}  \frac{c(\varepsilon,N_0)}{Z_{N}}
\int_{{\mathcal O}_{N_0}} \varphi \Bigl( (\widehat{m}^k)_{k \in F_{N_0}^+ } \Bigr) 
\exp \biggl( - \sum_{k \in F_{N_0}^+} \vert k\vert^{2 p d} { \vert \widehat{m}^k \vert^2} \biggr)
\bigotimes_{j \in  F_{N_0}^+} 
d \widehat{m}^j,
\end{split}
\end{equation*}
with 
\begin{equation}
\label{eq:compact:support:2}
c(\varepsilon,N_0) := 1 - \frac{C(p)}{\varepsilon N_0^{(p-5/2)d}}, 
\end{equation}
for $C(p)$ a constant only depending on $p$ and $d$. 
\end{lem}
\begin{proof}
\textit{First Step.} 
Using the same notations as in the statement and recalling that $\varphi$ is zero outside 
${\mathcal O}_{N_0}$, we have
\begin{align}
&\int_{{\mathcal P}({\mathbb T}^d)} \varphi \Bigl( (\widehat{m}^k)_{k \in F_{N_0}^+ } \Bigr)d {\mathbb P}_N(m) 
\nonumber
\\
&= \frac1{Z_N}
\int_{{\mathcal O}_N} \varphi \Bigl( (\widehat{m}^k  )_{k \in F_{N_0}^+ } \Bigr) 
\exp \biggl( - \sum_{k \in F_N^+} \vert k\vert^{2 p d} { \vert \widehat{m}^k \vert^2} \biggr)
\bigotimes_{k \in F_N^+} 
d \widehat{m}^k
\nonumber
\\
&= \frac1{Z_{N}}
\int_{{\mathcal O}_{N_0}} \varphi \Bigl( (\widehat{m}^k  )_{k \in F_{N_0}^+ } \Bigr) 
\exp \biggl( - \sum_{k \in F_{N_0}^+} \vert k\vert^{2 p d} { \vert \widehat{m}^k \vert^2} \biggr)
\label{eq:c_d:ON:0}
\\
&\hspace{5pt} \times   
\int_{{\mathbb R}^{2(\vert F_N^+\vert - \vert F_{N_0}^+\vert )}}
\biggl[ 
{\mathbf 1}_{{\mathcal O}_N}\Bigl( (\widehat{m}^k)_{k \in F_{N}^+ } \Bigr) 
\exp \biggl( - \sum_{k \in F_{N}^+ \setminus F_{N_0}^+}\vert k\vert^{2 p d} { \vert \widehat{m}^k \vert^2} \biggr)
\bigotimes_{j \in F_{N}^+ \setminus F_{N_0}^+} 
d \widehat{m}^j
\biggr]
\bigotimes_{j \in  F_{N_0}^+} 
d \widehat{m}^j. \nonumber
\end{align}
Obviously, we can bound the indicator function in the last line by 1 and then get $c_N/c_{N_0}$ 
(see \eqref{eq:cN} for the definition) as bound for the whole term on the last line. 
This proves the first inequality in the statement. 
\vskip 4pt

\textit{Second Step.} 
Take now $\varphi$ as in the statement and $(\widehat m^k)_{k \in F_{N_0}^+}$ satisfying 
\eqref{eq:compact:support}. 
If we choose another collection $(\widehat{m}^k)_{k \in F_N  \setminus F_{N_0}}$
(with the usual requirement that $\widehat{m}^{-k} = \overline{\widehat{m}^k}$) such that 
\begin{equation*}
\sum_{k \in F_N \setminus F_{N_0}} \vert \widehat{m}^k \vert < \frac12 \varepsilon,
\end{equation*}
then 
\begin{equation*}
1 + \inf_{x \in {\mathbb T}^d} \sum_{k \in F_{N} \setminus \{0\}} 
\widehat{m}_k e_{-k}(x)=
1 + 2 \inf_{x \in {\mathbb T}^d} \sum_{k \in F_{N} \setminus \{0\}} 
\Re \bigl[ \widehat{m}_k e_{-k}(x) \bigr] \geq \frac12 \varepsilon,
\end{equation*}
and accordingly the left hand side can be regarded as a probability measure  in 
${\mathcal P}_N$. For instance, so is the case if 
$\vert \widehat{m}^k \vert \leq \kappa_{d}^{-1} \varepsilon
\vert k\vert^{-3d/2}$, for $\kappa_d= 2 \sum_{j \in {\mathbb Z}^d \setminus \{0\}} \vert j \vert^{-3d/2}$. 
Therefore,
for
 $m \in {\mathcal P}({\mathbb T}^d)$ such that 
$(\widehat{m}^k )_{k \in F_{N_0}^+}$
belongs to the support of $\varphi$, 
\begin{equation}
\label{eq:c_d:ON}
\begin{split}
&\int_{{\mathbb R}^{2(\vert F_N^+\vert - \vert F_{N_0}^+\vert) }}
\prod_{k \in F_{N}^+ \setminus F_{N_0}^+} 
{\mathbf 1}_{\{\vert \widehat{m}^k \vert < \kappa_d^{-1} \varepsilon \vert k\vert^{-3d/2}\}}  
\exp \biggl( - \sum_{k \in F_{N}^+ \setminus F_{N_0}^+}\vert k\vert^{2 p d} { \vert \widehat{m}^k \vert^2} \biggr)
\bigotimes_{j \in F_{N}^+ \setminus F_{N_0}^+} 
d \widehat{m}^j
\\
&\leq \int_{{\mathbb R}^{2(\vert F_N^+\vert - \vert F_{N_0}^+\vert )}}
{\mathbf 1}_{{\mathcal O}_N}\Bigl( (\widehat{m}^k)_{k \in F_{N}^+ } \Bigr) 
\exp \biggl( - \sum_{k \in F_{N}^+ \setminus F_{N_0}^+}\vert k\vert^{2 p d} { \vert \widehat{m}^k \vert^2} \biggr)
\bigotimes_{j \in F_{N}^+ \setminus F_{N_0}^+} 
d \widehat{m}^j
\\
&\leq 
\int_{{\mathbb R}^{2(\vert F_N^+\vert - \vert F_{N_0}^+\vert )}}
\exp \biggl( - \sum_{k \in F_{N}^+ \setminus F_{N_0}^+}\vert k\vert^{2 p d} { \vert \widehat{m}^k \vert^2} \biggr)
\bigotimes_{j \in F_{N}^+ \setminus F_{N_0}^+} 
d \widehat{m}^j. 
\end{split}
\end{equation}
Multiplying both sides by $c_{N_0}/c_{N}$, we get
\begin{equation}
\label{eq:c_d:ON:2}
\begin{split}
& \prod_{k \in F_{N}^+ \setminus F_{N_0}^+} {\mathbf P} \biggl( \Bigl\{ \frac1{\sqrt{2} \vert k\vert^{pd}} \vert \xi_k\vert < 
\frac{\varepsilon}{\kappa_d \vert k \vert^{3d/2}} \Bigr\} \biggr)
\\
&\leq 
\frac{c_{N_0}}{c_N} 
 \int_{{\mathbb R}^{2(\vert F_N^+\vert - \vert F_{N_0}^+\vert )}}
{\mathbf 1}_{{\mathcal O}_N}\Bigl( (\widehat{m}^k)_{k \in F_{N}^+ } \Bigr) 
\exp \biggl( - \sum_{k \in F_{N}^+ \setminus F_{N_0}^+}\vert k\vert^{2 p d} { \vert \widehat{m}^k \vert^2} \biggr)
\bigotimes_{j \in F_{N}^+ \setminus F_{N_0}^+} 
d \widehat{m}^j
\\
&\leq 1,
%
\end{split}
\end{equation}
where $(\xi_k)_{k \in F_{N}^+ \setminus F_{N_0}^+}$ is defined right below \eqref{eq:cN}.

By Markov inequality, for $p > 3/2$,   
\begin{equation}
\label{eq:Markov}
\begin{split}
\prod_{k \in F_{N}^+ \setminus F_{N_0}^+} {\mathbf P} \biggl( \Bigl\{ 
\frac1{\sqrt{2} \vert k\vert^{pd}} \vert \xi_k\vert < 
\frac{\varepsilon}{\kappa_d \vert k \vert^{3d/2}}
 \Bigr\}
 \biggr)
&\geq  1 - \sum_{k \in F_{N}^+ \setminus F_{N_0}^+} {\mathbf P} \biggl( \Bigl\{ 
  \vert \xi_k\vert \geq 
\sqrt{2}
\frac{\varepsilon}{\kappa_d} \vert k \vert^{(p-3/2)d}
\Bigr\} \biggr) 
\\
&\geq 1 - C(p) \sum_{k \in F_N^+ \setminus F_{N_0}^+} \frac{1}{\varepsilon \vert k\vert^{(p-3/2)d}} 
\\
&\geq 1 - \frac{C(p)}{\varepsilon N_0^{(p-5/2)d}}, 
\end{split}
\end{equation}
where the value of the constant $C(p)$ is allowed to change from line to line (as long as it only depends on 
$p$ and $d$). 
The difference on the last line coincides with 
$c(\varepsilon,N_0)$ in 
\eqref{eq:compact:support:2}. 

Back to 
\eqref{eq:c_d:ON:2}, we deduce 
\begin{equation*}
\begin{split}
&c(\varepsilon,N_0) 
\leq \frac{c_{N_0}}{c_{N}} \int_{{\mathbb R}^{2 (\vert F_N^+\vert - \vert F_{N_0}^+\vert) }}
{\mathbf 1}_{{\mathcal O}_N}\Bigl( (\widehat{m}^k)_{k \in F_{N}^+ } \Bigr) 
\exp \biggl( - \sum_{k \in F_{N}^+ \setminus F_{N_0}^+}\vert k\vert^{2 p d} { \vert \widehat{m}^k \vert^2} \biggr)
\bigotimes_{j \in F_{N}^+ \setminus F_{N_0}^+} 
d \widehat{m}^j.
\end{split}
\end{equation*}
Back to 
\eqref{eq:c_d:ON:0}, we get
the second inequality in the statement (and this even if 
$c(\varepsilon,N_0) $ is negative). 
\end{proof}

We get the following lemma:

\begin{lem}
\label{lem:25:new}
With the same notations as above, 
\begin{equation*}
\lim_{N_0 \rightarrow \infty} \sup_{N \geq N_0} 
\Bigl\vert \frac{Z_{N_0}}{Z_N} \frac{c_N}{c_{N_0}} - 1 \Bigr\vert =0,
\end{equation*}
and
\begin{equation*}
\lim_{N_0 \rightarrow \infty}
\inf_{N \geq N_0} {\mathbb P}_N \biggl( \biggl\{ 
1 +\inf_{x \in {\mathbb T}^d} \sum_{k \in F_{N_0} \setminus \{0\}} 
\widehat{m}^k e_{-k}(x) \geq (2N_0)^{-3d/2} 
\biggr\} \biggr) = 1,
\end{equation*}
with the usual convention that 
$\widehat{m}^{-k} = \overline{\widehat{m}^k}$ for $k \in F_{N_0}^+$. 
\end{lem}

\begin{proof}
\textit{First Step.}
We choose
$\varphi$ in Lemma \ref{lem:24} as $\varphi_0$, the indicator function of the set 
of Fourier coefficients 
$(\widehat{m}^k)_{k \in F_{N_0}^+ }$ such that 
\begin{equation}
\label{eq:support:varphi}
1 + \inf_{x \in {\mathbb T}^d}
\sum_{k \in F_{N_0} \setminus \{0\}}
\widehat{m}^k e_{-k}(x) > \varepsilon,
\end{equation}
for some $\varepsilon \in (0,1)$ as in the statement of Lemma \ref{lem:24} and 
with the same convention 
as in the statement of 
Lemma
\ref{lem:25:new}
that 
$\widehat{m}^{-k} = \overline{\widehat{m}^k}$. 
Then, 
we have
\begin{equation*}
\begin{split}
&\frac1{Z_N}
\frac{c_N}{c_{N_0}}
\int_{{\mathcal O}_{N_0}} \varphi_0 \Bigl( (\widehat{m}^k)_{k \in F_{N_0}^+ } \Bigr) 
\exp \biggl( - \sum_{k \in F_{N_0}^+} \vert k\vert^{2 p d} { \vert \widehat{m}^k \vert^2} \biggr)
\bigotimes_{j \in  F_{N_0}^+} 
d \widehat{m}^j
\\
&=
\frac1{Z_N}
\frac{c_N}{c_{N_0}}
\int_{{\mathbb R}^{2\vert F^+_{N_0}\vert}} \varphi_0 \Bigl( (\widehat{m}^k)_{k \in F_{N_0}^+ } \Bigr) 
\exp \biggl( - \sum_{k \in F_{N_0}^+} \vert k\vert^{2 p d} { \vert \widehat{m}^k \vert^2} \biggr)
\bigotimes_{j \in  F_{N_0}^+} 
d \widehat{m}^j.
\end{split}
\end{equation*}
We then make the change of variable
\begin{equation*}
\widehat{y}^k = \frac1{1-\varepsilon} \widehat{m}^k. 
\end{equation*}
Then, the condition 
\eqref{eq:support:varphi}
merely says that 
$(\widehat{y}^k)_{k \in F_{N_0}^+}$
belongs to ${\mathcal O}_{N_0}$ if 
$(\widehat{m}^k)_{k \in F_{N_0}^+}$ belongs to the support of $\varphi_0$. 
Therefore, 
\begin{equation}
\label{eq:lem:25:new:1}
\begin{split}
&\frac1{Z_N}
\frac{c_N}{c_{N_0}}
\int_{{\mathcal O}_{N_0}} \varphi_0 \Bigl( (\widehat{m}^k)_{k \in F_{N_0}^+ } \Bigr) 
\exp \biggl( - \sum_{k \in F_{N_0}^+} \vert k\vert^{2 p d} { \vert \widehat{m}^k \vert^2} \biggr)
\bigotimes_{j \in  F_{N_0}^+} 
d \widehat{m}^j
\\
&=\frac{(1-\varepsilon)^{2 \vert F_{N_0}^+ \vert}}{Z_N}
\frac{c_N}{c_{N_0}}
\int_{{\mathcal O}_{N_0}}  
\exp \biggl( - (1-\varepsilon)^2 \sum_{k \in F_{N_0}^+} \vert k\vert^{2 p d} { \vert \widehat{y}^k \vert^2} \biggr)
\bigotimes_{j \in  F_{N_0}^+} 
d \widehat{y}^j
\\
&\geq (1-\varepsilon)^{2 \vert F^+_{N_0}\vert} \frac{Z_{N_0}}{Z_N} \frac{c_N}{c_{N_0}}.
\end{split}
\end{equation}
By the 
second inequality in the statement of Lemma \ref{lem:24}, we deduce that 
\begin{equation*}
c(\varepsilon,N_0)
(1-\varepsilon)^{2 \vert F^+_{N_0}\vert }  \frac{Z_{N_0}}{Z_{N}} \frac{c_N}{c_{N_0}} \leq 1.
\end{equation*}
Moreover, choosing $\varphi=1$ in 
the first inequality in the statement of Lemma \ref{lem:24}, we also have
\begin{equation*}
\frac{Z_{N_0}}{Z_N} \frac{c_N}{c_{N_0}} \geq 1. 
\end{equation*}

\textit{Second Step.}
We notice from Proposition 
\ref{prop:3} that $2 \vert F_{N_0}^+ \vert = D_{N_0} \leq 2 (2 N_0)^d$. So, 
if we choose 
$\varepsilon= N_0^{-3d/2}$, then
the shape of 
$c(\varepsilon,N_0)$ in 
Lemma \ref{lem:24}
together with the fact that $p \geq 5$
yield
\begin{equation*}
c(\varepsilon,N_0)
(1-\varepsilon)^{2 \vert F^+_{N_0}\vert}
\geq \Bigl( 1 - \frac{ C(p)}{N_0^{d}} \Bigr)
\exp \Bigl(2  (2N_0)^d \ln( 1 - (N_0)^{-3d/2}) \Bigr).
\end{equation*}
Obviously the right-hand side tends to $1$ as $N_0$ tends to $\infty$. 
Combining with the conclusion of the first step, we get 
\begin{equation*}
\lim_{N_0 \rightarrow \infty} \sup_{N \geq N_0} 
\Bigl\vert \frac{Z_{N_0}}{Z_N} \frac{c_N}{c_{N_0}} - 1 \Bigr\vert =0,
\end{equation*}
which is the first claim in the statement. 
\vskip 5pt

\textit{Third Step.}
By the second inequality in the statement of Lemma \ref{lem:24}, 
\begin{equation*}
\begin{split}
&\int_{{\mathcal P}({\mathbb T}^d)} \varphi_0 \Bigl( (\widehat{m}^k)_{k \in F_{N_0}^+ } \Bigr) d {\mathbb P}_N(m)
\\
&\hspace{15pt} \geq \frac{c_N}{c_{N_0}}  \frac{c(\varepsilon,N_0)}{Z_{N}}
\int_{{\mathcal O}_{N_0}} \varphi_0 \Bigl( (\widehat{m}^k)_{k \in F_{N_0}^+ } \Bigr) 
\exp \biggl( - \sum_{k \in F_{N_0}^+} \vert k\vert^{2 p d} { \vert \widehat{m}^k \vert^2} \biggr)
\bigotimes_{j \in  F_{N_0}^+} 
d \widehat{m}^j,
\end{split}
\end{equation*}
By 
\eqref{eq:lem:25:new:1} (with $\varepsilon = N_0^{-3d/2}$), we get 
\begin{equation*}
\begin{split}
{\mathbb P}_N \biggl( \biggl\{ 
1 +\inf_{x \in {\mathbb T}^d} \sum_{k \in F_{N_0} \setminus \{0\}} 
\widehat{m}^k e_{-k}(x) > N_0^{-3d/2} 
\biggr\} \biggr)  
&= \int_{{\mathcal P}({\mathbb T}^d)} \varphi_0\Bigl( (\widehat{m}^k)_{k \in F_{N_0}^+ } \Bigr) d {\mathbb P}_N(m)
\\
& \geq   c(\varepsilon,N_0) (1-\varepsilon)^{2 \vert F^+_{N_0}\vert} \frac{Z_{N_0}}{Z_N} \frac{c_N}{c_{N_0}},
\end{split}
\end{equation*}
By the second step, we get  that the infimum over $N \geq N_0$ of the right-hand side
tends to $1$ as $N_0$ tends to $\infty$. 
This completes the proof. 
%
\end{proof}

Here is an application, which is very useful throughout the paper. 

\begin{lem}
\label{lem:25:second}
Let $(\psi_{N_0})_{N_0 \geq 1}$ be a sequence of $[-1,1]$-valued Borel measurable functions defined on 
${\mathcal P}({\mathbb T}^d)$ with the following two properties:
\begin{enumerate}
\item $\psi_{N_0}$ is measurable with respect to the $\sigma$-algebra generated by the mappings 
$(m \in {\mathcal P}({\mathbb T}^d) \mapsto \widehat{m}^k)_{k \in F_{N_0}^+}$;
\item $\psi_{N_0}$ is null outside
\begin{equation*}
A_{N_0} := \biggl\{ m \in {\mathcal P}({\mathbb T}^d) : 
1+ \inf_{x \in {\mathbb T}^d} \sum_{k \in F_{N_0} \setminus \{0\}}
\widehat{m}^k e_{-k}(x) \geq ( 2N_0)^{-3d/2}
\biggr\}.
\end{equation*}
\end{enumerate}
Then, there exists $a_0>0$, only depending on $d$, such that, for any $a \in (0,a_0]$,
\begin{equation*}
\begin{split}
& A_{N_0} \cap 
\bigcap_{ k  \in F_N^+ \setminus F_{N_0}^+} 
 \Bigl\{  m \in {\mathcal P}({\mathbb T}^d) :
   \vert \widehat{m}^k \vert 
  < \frac{a}{\vert k \vert^{5d/2}} 
  \Bigr\}
  \\
  &\hspace{15pt} \subset
 \biggl\{  m \in {\mathcal P}({\mathbb T}^d) :
1+ \inf_{x \in {\mathbb T}^d} \sum_{k \in F_{N} \setminus \{0\}}
\widehat{m}^k e_{-k}(x) \geq 2^{-(3d/2+1)} N_0^{-3d/2}
  \biggr\},  
  \end{split}
\end{equation*}
and
\begin{equation*}
\begin{split}
&\lim_{N_0 \rightarrow \infty} \sup_{N \geq N_0} \biggl\vert 
\biggl(
 \int_{{\mathcal P}({\mathbb T}^d)}
\psi_{N_0}(m) 
d {\mathbb P}_{N_0}(m) 
\biggr)^{-1}
\int_{{\mathcal P}({\mathbb T}^d)}
\biggl[
\psi_{N_0}(m) 
\prod_{ k  \in F_N^+ \setminus F_{N_0}^+} 
{\mathbf 1}_{\{  
   \vert \widehat{m}^k \vert 
  < a \vert k \vert^{-5d/2} 
  \Bigr\}}
  \biggr]
d {\mathbb P}_N(m) 
  - 
1 
\biggr\vert
\\
&=0,
\end{split}
\end{equation*}
with the 
numerator in the first line being necessarily equal  $0$ when 
the denominator is $0$ and the convention that 
the ratio is then understood as $1$. 
\end{lem}

\begin{rem}
\label{rem:extension:lem:25:second}
In fact, the first claim can be reformulated in the following broader sense. 
If $(\widehat m^k)_{k \in F_N^+}$ satisfies, for $N \geq N_0$,  
$$
1+ \inf_{x \in {\mathbb T}^d} \sum_{k \in F_{N_0} \setminus \{0\}}
\widehat{m}^k e_{-k}(x) \geq ( 2N_0)^{-3d/2},
\quad 
\textrm{and}
\quad 
\max_{ k  \in F_N^+ \setminus F_{N_0}^+} 
   \vert \widehat{m}^k \vert 
  < \frac{a}{\vert k \vert^{5d/2}},
$$ 
for $a \in (0,a_0]$, 
then
$$
1+ \inf_{x \in {\mathbb T}^d} \sum_{k \in F_{N} \setminus \{0\}}
\widehat{m}^k e_{-k}(x) \geq 2^{-(3d/2+1)} N_0^{-3d/2}.
$$
In words, there is no need to assume \emph{a priori} that $(\widehat{m}^k)_{k \in F_N^+}$ are the Fourier coefficients of a  probability measure. The conditions on the Fourier coefficients suffice to prove 
it \emph{a posteriori}. 
\end{rem} 

\begin{proof}
We first observe that, for 
$$m
\in A_{N_0} \cap 
\bigcap_{ k  \in F_N^+ \setminus F_{N_0}^+} 
 \Bigl\{  
   \vert \widehat{m}^k \vert 
  < \frac{a}{\vert k \vert^{5d/2}} 
  \Bigr\},$$
 it holds
 \begin{equation*}
 \begin{split}
 1+ \inf_{x \in {\mathbb T}^d} \sum_{k \in F_{N} \setminus \{0\}}
\widehat{m}^k e_{-k}(x) 
&\geq ( 2N_0)^{-3d/2} - \sum_{k \in F_N \setminus F_{N_0}} \frac{a}{\vert k \vert^{5d/2}} 
\\
&\geq ( 2N_0)^{-3d/2} - c_d a (N_0)^{-3d/2},
\end{split}
 \end{equation*}
 for a constant $c_d$ only depending on $d$. 
Therefore, for $c_d a \leq 2^{-(3d/2+1)}$, 
 \begin{equation*}
 \begin{split}
 1+ \inf_{x \in {\mathbb T}^d} \sum_{k \in F_{N} \setminus \{0\}}
\widehat{m}^k e_{-k}(x) 
&\geq  2^{-(3d/2+1)} N_0^{-3d/2}.
\end{split}
 \end{equation*}
We deduce that, for such an $a$ and for $\psi_{N_0}$ as in the statement, 
\begin{equation*}
\begin{split}
&\int_{{\mathcal P}({\mathbb T}^d)}
\biggl[
\psi_{N_0}(m) 
\prod_{ k  \in F_N^+ \setminus F_{N_0}^+} 
{\mathbf 1}_{\{  
   \vert \widehat{m}^k \vert 
  < a \vert k \vert^{-5d/2} 
  \}}
  \biggr]
d {\mathbb P}_N(m) 
\\
&= \frac1{Z_N}
\int_{{\mathbb R}^{2\vert F_N^+\vert}}
\biggl[ \varphi_{N_0}\Bigl( \bigl(\widehat{m}^k \bigr)_{k \in F_{N_0}^+} \Bigr)  
\prod_{k  \in F_N^+ \setminus F_{N_0}^+} 
{\mathbf 1}_{\{    \vert \widehat{m}^k \vert 
  < a \vert k \vert^{-5d/2} \}} 
  \exp \biggl( - \sum_{k \in F_{N}^+} \vert k\vert^{2 p d} { \vert \widehat{m}^k \vert^2} \biggr)
\biggr] 
\bigotimes_{j \in  F_{N}^+}  d \widehat{m}^j ,
\end{split}
\end{equation*}
where 
$\varphi_{N_0} = \psi_{N_0} \circ {\mathscr I}_{N_0}$, see
the notation in 
\eqref{eq:I_N}.
%
 We rewrite the above equality as
 \begin{equation}
 \label{eq:lem:25:second:1}
\begin{split}
&\int_{{\mathcal P}({\mathbb T}^d)}
\biggl[
\psi_{N_0}(m) 
\prod_{ k  \in F_N^+ \setminus F_{N_0}^+} 
{\mathbf 1}_{\{  
   \vert \widehat{m}^k \vert 
  < a \vert k \vert^{-5d/2} 
  \}}
  \biggr]
d {\mathbb P}_N(m) 
\\
&= \frac1{Z_N}
\int_{{\mathcal O}_{N_0}}
 \varphi_{{N_0}}\Bigl( \bigl(\widehat{m}^k \bigr)_{k \in F_{N_0}^+} \Bigr)  
 \exp \biggl( - \sum_{k \in F_{N_0}^+} \vert k\vert^{2 p d} { \vert \widehat{m}^k \vert^2} \biggr)
\bigotimes_{j \in  F_{N_0}^+}  
d \widehat{m}^j 
\\
&\hspace{15pt} \times 
\int_{{\mathbb R}^{2 (\vert F^+_N \vert -\vert F^+_{N_0}\vert)}}
\prod_{k  \in F_N^+ \setminus F_{N_0}^+} 
{\mathbf 1}_{\{    \vert \widehat{m}^k \vert 
  < a \vert k \vert^{-5d/2} \}} 
 \exp \biggl( - \sum_{k \in F_{N}^+ \setminus F_{N_0}^+} \vert k\vert^{2 p d} { \vert \widehat{m}^k \vert^2} \biggr)
\bigotimes_{j \in  F_{N}^+ \setminus F_{N_0}^+} 
d \widehat{m}^j
\\
&= \frac{Z_{N_0}}{Z_N}  \frac{c_N}{c_{N_0}}
\biggl( \int_{{\mathcal P}({\mathbb T}^d)}
\psi_{N_0}(m) 
d {\mathbb P}_{N_0}(m) 
\biggr)
\prod_{k \in F_N^+ \setminus F_{N_0}^+}
{\mathbf P}
\Bigl( 
\Bigl\{
\vert \xi_k \vert \leq 
\sqrt{2} a \vert k \vert^{(p-5/2)d} 
\Bigr\}
\Bigr).
\end{split}
\end{equation}
Repeating 
\eqref{eq:Markov}, 
\begin{equation*}
\begin{split}
\prod_{k \in F_N^+ \setminus F_{N_0}^+}
{\mathbf P}
\Bigl( 
\Bigl\{
\vert \xi_k \vert \leq 
\sqrt{2} a \vert k \vert^{(p-5/2)d} 
\Bigr\}
\Bigr)
\geq 1 -  \frac{C(p)}{a N_0^{(p-7/2)d}},
\end{split}
\end{equation*}
for a constant $C(p)$ only depending on 
$p$ and $d$. 
Therefore, by 
Lemma 
\ref{lem:25:new}, we have
\begin{equation*}
\lim_{N_0 \rightarrow \infty} 
\sup_{N \geq N_0} 
\biggl\vert 
\frac{Z_{N_0}}{Z_N}  \frac{c_N}{c_{N_0}}
\prod_{k \in F_N^+ \setminus F_{N_0}^+}
{\mathbf P}
\Bigl( 
\Bigl\{
\vert \xi_k \vert \leq 
\sqrt{2} a \vert k \vert^{(p-5/2)d} 
\Bigr\}
\Bigr)
-
1 \biggr\vert =0.
\end{equation*}
Inserting the latter into 
\eqref{eq:lem:25:second:1}, we get the conclusion. 
\end{proof}

We deduce the following lemma: 

\begin{lem}
\label{lem:27}
Let $a_0>0$
be as in the statement of Lemma 
\ref{lem:25:second}. Then, 
for any $\delta >0$, there exists 
$N_1 \geq 1$ such that, 
for any $N \geq N_0 \geq N_1$, 
\begin{equation}
\begin{split}
&{\mathbb P}_N \Bigl( \bigl\{
m > 
 2^{-(3d/2+1)} N_0^{-3d/2}
\bigr\}
\Bigr)
  \geq
{\mathbb P}_N
\biggl( A_{N_0} \cap 
\bigcap_{ k  \in F_N^+ \setminus F_{N_0}^+} 
 \Bigl\{  m \in {\mathcal P}({\mathbb T}^d) :
   \vert \widehat{m}^k \vert 
  < \frac{a_0}{\vert k \vert^{5d/2}} 
  \Bigr\}
\biggr)
\geq 1 - \delta,
\end{split}
\end{equation} 
with $A_{N_0}$
being 
as in 
Lemma 
\ref{lem:25:second}.
\end{lem}

\begin{proof}
For $\delta$ as in the statement, we know 
from Lemma 
\ref{lem:25:second} 
(with $\psi_{N_0} = {\mathbf 1}_{A_{N_0}}$)
that we can find 
$a=a_0>0$, only depending on $d$, 
such that, for any  
$N_0 \geq 1$, 
\begin{equation*}
\begin{split}
\sup_{N \geq N_0} \biggl\vert &{\mathbb P}_N \biggl( A_{N_0} \cap 
\bigcap_{ k  \in F_N^+ \setminus F_{N_0}^+} 
 \Bigl\{  
   \vert \widehat{m}^k \vert 
  < \frac{a}{\vert k \vert^{5d/2}} 
  \Bigr\}
\biggr)
- 
{\mathbb P}_{N_0} \bigl( A_{N_0} \bigr) 
\biggr\vert
 \leq \frac{\delta}2.
\end{split}
\end{equation*}
By Lemma 
\ref{lem:25:second}, 
for any $N \geq N_0$,
\begin{equation*}
\begin{split}
& A_{N_0} \cap 
\bigcap_{ k  \in F_N^+ \setminus F_{N_0}^+} 
 \Bigl\{  m \in {\mathcal P}({\mathbb T}^d) :
   \vert \widehat{m}^k \vert 
  < \frac{a}{\vert k \vert^{5d/2}} 
  \Bigr\}
  \\
  &\hspace{15pt} \subset
 \biggl\{  m \in {\mathcal P}({\mathbb T}^d) :
1+ \inf_{x \in {\mathbb T}^d} \sum_{k \in F_{N_0} \setminus \{0\}}
\widehat{m}^k e_{-k}(x) \geq 2^{-(3d/2+1)} N_0^{-3d/2}
  \biggr\}.
\end{split}
\end{equation*}
Therefore, for any $N \geq N_0$, 
\begin{equation}
\label{eq:lem:27:1}
\begin{split}
&{\mathbb P}_N 
\biggl( 
 \biggl\{ 
1+ \inf_{x \in {\mathbb T}^d} \sum_{k \in F_{N_0} \setminus \{0\}}
\widehat{m}^k e_{-k}(x) \geq 2^{-(3d/2+1)} N_0^{-3d/2}
  \biggr\}
  \biggr)
  \\
&\geq 
{\mathbb P}_N
\biggl( A_{N_0} \cap 
\bigcap_{ k  \in F_N^+ \setminus F_{N_0}^+} 
 \Bigl\{  m \in {\mathcal P}({\mathbb T}^d) :
   \vert \widehat{m}^k \vert 
  < \frac{a}{\vert k \vert^{5d/2}} 
  \Bigr\}
\biggr)
\\
  &\geq 
  {\mathbb P}_{N_0} \bigl( A_{N_0} \bigr) 
- \tfrac12 \delta.
\end{split}
\end{equation}
By Lemma 
\ref{lem:25:new}, we know that, for $N_0$ large enough, 
$  {\mathbb P}_{N_0} \bigl( A_{N_0} \bigr) \geq 1 - \delta /2$ and, then, 
the right-hand side is greater than $1-\delta$. 
This completes the proof. 
\end{proof} 

\begin{lem}
\label{lem:25:three}
With the same notations as before, 
\begin{equation*}
\lim_{N_0 \rightarrow \infty} 
\sup_{\varphi : {\mathbb R}^{D_{N_0}} \rightarrow [-1,1]}
\sup_{N \geq N_0}
\biggl\vert \int_{{\mathcal P}({\mathbb T}^d)}
\varphi\Bigl( (\widehat{m}^k)_{k \in F_{N_0}^+ } \Bigr) d\bigl( {\mathbb P}_N - {\mathbb P}_{N_0} \bigr)(m)
\biggr\vert
=0,
\end{equation*}
where, for any $N_0 \geq 1$, $\varphi$ in the argument of the supremum is required to be measurable, and with the already used notation 
$D_{N_0}= 2 \vert F_{N_0}^+\vert$. 
\end{lem}

\begin{proof}
We fix $\delta>0$. 
By Lemma 
\ref{lem:27}, 
we can choose $a_0$ as in Lemma \ref{lem:25:second}, $N_0$ large enough
and $\varepsilon \in (0,1)$ small enough 
such that, for all $N \geq N_0$, 
\begin{equation*}
{\mathbb P}_N \Bigl( \bigl\{
m > \varepsilon
\bigr\}
\Bigr)
 \geq
{\mathbb P}_N
\biggl( A_{N_0} \cap 
\bigcap_{ k  \in F_N^+ \setminus F_{N_0}^+} 
 \Bigl\{  m \in {\mathcal P}({\mathbb T}^d) :
   \vert \widehat{m}^k \vert 
  < \frac{a_0}{\vert k \vert^{5d/2}} 
  \Bigr\}
\biggr)
\geq 1 - \delta.
\end{equation*}
For a function $\varphi$ as in the statement, we then have
\begin{equation}
\label{eq:proof:TV:2ndstep}
\begin{split}
&\sup_{N \geq N_0} \biggl\vert \int_{{\mathcal P}({\mathbb T}^d)}
\varphi\Bigl((\widehat{m}^k)_{k \in F_{N_0}^+}\Bigr) 
d{\mathbb P}_N(m)
\\
&\hspace{30pt} -
\int_{{\mathcal P}({\mathbb T}^d)}
\varphi\Bigl((\widehat{m}^k)_{k \in F_{N_0}^+}\Bigr) 
\biggl( 
{\mathbf 1}_{A_{N_0}}(m)
\prod_{k \in F_N^+ \setminus F_{N_0}^+}
{\mathbf 1}_{\{
\vert \widehat{m}^k \vert < a_0 \vert k\vert^{-5d/2}
 \}}
 \biggr)
d{\mathbb P}_N(m)
\biggr\vert \leq  \delta. 
\end{split}
\end{equation}
We now 
apply Lemma 
\ref{lem:25:second}
to the function 
$m \mapsto 
\varphi((\widehat{m}^k)_{k \in F_{N_0}^+}) 
{\mathbf 1}_{A_{N_0}}(m)$.
For $N_0$ large enough, we get 
\begin{equation*}
\begin{split}
&\sup_{N \geq N_0} \biggl\vert \int_{{\mathcal P}({\mathbb T}^d)}
\varphi\Bigl((\widehat{m}^k)_{k \in F_{N_0}^+}\Bigr) 
d{\mathbb P}_N(m)
 -
\int_{{\mathcal P}({\mathbb T}^d)}
\varphi\Bigl((\widehat{m}^k)_{k \in F_{N_0}^+}\Bigr) 
d{\mathbb P}_{N_0}(m)
\biggr\vert \leq 3 \delta. 
\end{split}
\end{equation*}
\end{proof}

\subsubsection{Convergence of the sequence $({\mathbb P}_N)_{N \geq 1}$}

Using the properties proven in the previous paragraph, 
we now have all the ingredients to 
address the limiting points of the sequence 
$({\mathbb P}_N)_{N \geq 1}$.

The following lemma proves 
the second claim in item $(ii)$ of 
Theorem \ref{thm:probability:probability}.
\begin{lem}
\label{lem:28:0}
Let ${\mathbb P}$ 
be a weak limit of $({\mathbb P}_N)_{N \geq 1}$ on ${\mathcal P}({\mathbb T}^d)$.
For an integer $N_0 \geq 1$, 
define, 
on ${\mathcal P}({\mathbb T}^d)$, 
 the (sub-probability) measure
${\mathbb Q}$ by  
\begin{equation*}
\frac{d{\mathbb Q}}{d {\mathbb P}}(m) 
:=
{\mathbf 1}_{{\mathcal P}_{N_0}}(m)
={\mathbf 1}_{{\mathcal O}_{N_0}}\Bigl(
\bigl( \widehat{m}^k \bigr)_{k \in F_{N_0}^+} \Bigr), \quad m \in {\mathcal P}({\mathbb T}^d).
\end{equation*}
Then, the image of ${\mathbb Q}$ by the mapping 
$$\pi^{(2)}_{N_0} : m \in {\mathcal P}({\mathbb T}^d) \mapsto 
\bigl( \widehat{m}^k   \bigr)_{k \in F_{N_0}^+} \in {\mathbb R}^{2 \vert F_{N_0}^+\vert} $$
is 
supported by 
${\mathcal O}_{N_0}$
(i.e., 
${\mathbb Q} \circ ( \pi^{(2)}_{N_0})^{-1}
({\mathcal O}_{N_0}^{\complement})=0$)
and is 
absolutely continuous with respect to the Lebesgue measure. 
Precisely, 
for any Borel subset $B$ of 
 ${\mathbb R}^{2 \vert F^+_{N_0}\vert}$ that is included
 in ${\mathcal O}_{N_0}$,
 \begin{equation*}
 \begin{split}
&{\mathbb Q} \circ \bigl( \pi^{(2)}_{N_0}\bigr)^{-1}(B)
\\
&=\int_{{\mathcal P}({\mathbb T}^d)} {\mathbf 1}_B \Bigl( (\widehat{m}^k)_{k \in F_{N_0}^+ } \Bigr) d {\mathbb P}(m)
\\
&\leq \Bigl( \sup_{N \geq N_0} \frac{Z_{N_0} c_N}{Z_N c_{N_0}} \Bigr) \frac1{Z_{N_0}}
\int_{{\mathcal O}_{N_0}} {\mathbf 1}_B \Bigl( (\widehat{m}^k)_{k \in F_{N_0}^+ } \Bigr) 
\exp \biggl( - \sum_{k \in F_{N_0}^+} \vert k\vert^{2 p d} { \vert \widehat{m}^k \vert^2} \biggr)
\bigotimes_{j \in  F_{N_0}^+} 
d \widehat{m}^j,
\end{split} 
 \end{equation*}
where we recall from 
Lemma 
\ref{lem:25:new} that the first factor in the above right-hand side is finite. 
\end{lem}

\begin{proof}
Let $E$ be an open subset of ${\mathcal O}_{N_0}$ (equivalently, $E$ is an open subset of 
${\mathbb R}^{2\vert F^+_{N_0}\vert }$ included in ${\mathcal O}_{N_0}$). 
We observe that $\{ m \in {\mathcal P}({\mathbb T}^d) : 
(\widehat{m}^k)_{k \in F_{N_0}^+ } \in E\}$ is an open subset 
of ${\mathcal P}({\mathbb T}^d)$.

Up to a subsequence, we can assume that 
$({\mathbb P}_N)_{N \geq 1}$ 
weakly 
converges to 
${\mathbb P}$.
Then, 
by Portmanteau theorem,
\begin{equation*}
{\mathbb P}\Bigl( \Bigl\{ (\widehat{m}^k)_{k \in F_{N_0}^+ } \in E \Bigr\}
\Bigr)
\leq
\liminf_{N \rightarrow \infty} {\mathbb P}_N\Bigl( \Bigl\{ (\widehat{m}^k)_{k \in F_{N_0}^+ } \in E \Bigr\}
\Bigr). 
\end{equation*}
Now, by 
Lemma \ref{lem:24}
(with $\varphi = {\mathbf 1}_E$), letting $N$ tend to $\infty$ therein, we get
\begin{equation}
\label{eq:lem:30:1}
\begin{split}
&\int_{{\mathcal P}({\mathbb T}^d)} {\mathbf 1}_E \Bigl( (\widehat{m}^k)_{k \in F_{N_0}^+ } \Bigr) d {\mathbb P}(m) 
\\
&\hspace{15pt} \leq  
\Bigl( \sup_{N \geq N_0} \frac{Z_{N_0} c_N}{Z_N c_{N_0}} \Bigr)
\frac1{Z_{N_0}}
\int_{{\mathcal O}_{N_0}} {\mathbf 1}_E \Bigl( (\widehat{m}^k)_{k \in F_{N_0}^+ } \Bigr) 
\exp \biggl( - \sum_{k \in F_{N_0}^+} \vert k\vert^{2 p d} { \vert \widehat{m}^k \vert^2} \biggr)
\bigotimes_{j \in  F_{N_0}^+} 
d \widehat{m}^j,
\end{split}
\end{equation}
where we recall from 
Lemma 
\ref{lem:25:new} that the first factor in the above right-hand side is finite. 
By definition of ${\mathbb Q}$, the left-hand side can be rewritten as
${\mathbb Q} \circ (\pi^{(2)}_{N_0})^{-1}(E)$.

Take now a Borel subset $B$ of 
 ${\mathbb R}^{2 \vert F^+_{N_0}\vert}$ that is included
 in ${\mathcal O}_{N_0}$. By outer-regularity of 
 the Lebesgue measure, we can find, for any $\delta >0$,  an open subset $E$ of ${\mathcal O}_{N_0}$, containing $B$, such that 
\begin{equation*}
\begin{split}
&\frac1{Z_{N_0}}
\int_{{\mathcal O}_{N_0}} {\mathbf 1}_E \Bigl( (\widehat{m}^k)_{k \in F_{N_0}^+ } \Bigr) 
\exp \biggl( - \sum_{k \in F_{N_0}^+} \vert k\vert^{2 p d} { \vert \widehat{m}^k \vert^2} \biggr)
\bigotimes_{j \in  F_{N_0}^+} 
d \widehat{m}^j
\\
&\leq 
\frac1{Z_{N_0}}
\int_{{\mathcal O}_{N_0}} {\mathbf 1}_B \Bigl( (\widehat{m}^k)_{k \in F_{N_0}^+ } \Bigr) 
\exp \biggl( - \sum_{k \in F_{N_0}^+} \vert k\vert^{2 p d} { \vert \widehat{m}^k \vert^2} \biggr)
\bigotimes_{j \in  F_{N_0}^+} 
d \widehat{m}^j + \delta.
\end{split}
\end{equation*}
By \eqref{eq:lem:30:1},
 \begin{equation*}
 \begin{split}
&\int_{{\mathcal P}({\mathbb T}^d)} {\mathbf 1}_E \Bigl( (\widehat{m}^k)_{k \in F_{N_0}^+ } \Bigr) d {\mathbb P}(m) 
\\
&\leq 
\Bigl( \sup_{N \geq N_0} \frac{Z_{N_0} c_N}{Z_N c_{N_0}} \Bigr)
\biggl[
\frac1{Z_{N_0}}
\int_{{\mathcal O}_{N_0}} {\mathbf 1}_B \Bigl( (\widehat{m}^k)_{k \in F_{N_0}^+ } \Bigr) 
\exp \biggl( - \sum_{k \in F_{N_0}^+} \vert k\vert^{2 p d} { \vert \widehat{m}^k \vert^2} \biggr)
\bigotimes_{j \in  F_{N_0}^+} 
d \widehat{m}^j + \delta
\biggr].
\end{split}
\end{equation*}
The left-hand side writes
${\mathbb Q} \circ (\pi^{(2)}_{N_0})^{-1}(E)$.
This yields 
 \begin{equation*}
 \begin{split}
&{\mathbb Q} \circ \bigl(\pi^{(2)}_{N_0} \bigr)^{-1}(B)
\\
&\leq
{\mathbb Q} \circ \bigl(\pi^{(2)}_{N_0} \bigr)^{-1}(E)
\\
&\leq 
\Bigl( \sup_{N \geq N_0} \frac{Z_{N_0} c_N}{Z_N c_{N_0}} \Bigr)
\biggl[
\frac1{Z_{N_0}}
\int_{{\mathcal O}_{N_0}} {\mathbf 1}_B \Bigl( (\widehat{m}^k)_{k \in F_{N_0}^+ } \Bigr) 
\exp \biggl( - \sum_{k \in F_{N_0}^+} \vert k\vert^{2 p d} { \vert \widehat{m}^k \vert^2} \biggr)
\bigotimes_{j \in  F_{N_0}^+} 
d \widehat{m}^j + \delta \biggr].
\end{split}
\end{equation*}
Since this is for any $\delta >0$, 
we get the result. 
\end{proof}

\begin{lem}
\label{lem:AN0}
Let ${\mathbb P}$ 
be a weak limit of $({\mathbb P}_N)_{N \geq 1}$ on ${\mathcal P}({\mathbb T}^d)$.
Then
\begin{equation*}
\lim_{N_0 \rightarrow \infty} 
{\mathbb P}(A_{N_0}) = 1, 
\end{equation*}
with $A_{N_0}$
being 
as in 
Lemma 
\ref{lem:25:second}.
\end{lem}

\begin{proof}
By the same argument as in
the proof of 
Lemma \ref{lem:28:0}, $A_{N_0}$ is a closed subset of ${\mathcal P}({\mathbb T}^d)$. 
Therefore, 
by Portmanteau
theorem (assuming without any loss of generality that 
$({\mathbb P}_N)_{N \geq 1}$ converges to ${\mathbb P}$ in the weak sense),
\begin{equation*}
{\mathbb P}(A_{N_0}) \geq \limsup_{N \rightarrow \infty} 
{\mathbb P}_N(A_{N_0}).
\end{equation*}
The conclusion follows from Lemma 
\ref{lem:27}. 
\end{proof}

We now prove item $(iii)$ in the statement of Theorem 
\ref{thm:probability:probability}.

\begin{lem}
\label{lem:30}
Let ${\mathbb P}$ 
be a weak limit of $({\mathbb P}_N)_{N \geq 1}$ on ${\mathcal P}({\mathbb T}^d)$.
Then, 
\begin{equation*}
\lim_{N_0 \rightarrow \infty} 
\sup_{\varphi : {\mathbb R}^{{D_{N_0}}} \rightarrow [-1,1]}
\biggl\vert \int_{{\mathcal P}({\mathbb T}^d)}
\varphi\Bigl( (\widehat{m}^k)_{k \in F_{N_0}^+ } \Bigr) d\bigl( {\mathbb P} - {\mathbb P}_{N_0} \bigr)(m)
\biggr\vert
=0,
\end{equation*}
where, for any $N_0 \geq 1$, $\varphi$ in the argument of the supremum is required to be measurable, 
 and with the already used notation 
$D_{N_0}= 2 \vert F_{N_0}^+\vert$.   
\end{lem}

\begin{proof}
\textit{First Step.}
Introduce the sequence
\begin{equation*}
\eta_{N_0}:= 
\sup_{\varphi : {\mathbb R}^{D_{N_0}} \rightarrow [-1,1]}
\sup_{N \geq N_0}
\biggl\vert \int_{{\mathcal P}({\mathbb T}^d)}
\varphi\Bigl( (\widehat{m}^k)_{k \in F_{N_0}^+ } \Bigr) d\bigl( {\mathbb P}_N - {\mathbb P}_{N_0} \bigr)(m)
\biggr\vert, \quad N_0 \geq 1.
\end{equation*}
By 
Lemma \ref{lem:25:three}, 
the sequence $(\eta_{N_0})_{N_0 \geq 1}$ converges to $0$.

Take now $\varphi$ 
a continuous function from ${\mathbb R}^{D_{N_0}}$ into $[-1,1]$.
Letting $N$ tend to $\infty$, 
we get
\begin{equation*}
\begin{split}
&\biggl\vert \int_{{\mathcal P}({\mathbb T}^d)}
\varphi\Bigl((\widehat{m}^k)_{k \in F_{N_0}^+}\Bigr) 
d\bigl( {\mathbb P}  - {\mathbb P}_{N_0}\bigr) (m)
\biggr\vert \leq \eta_{N_0}. 
\end{split}
\end{equation*}
\vspace{4pt}

\textit{Second Step.}
Assume now that $\varphi$ is merely measurable and $[-1,1]$-valued. 
Consider also another sequence 
$(\delta_{N_0})_{N_0 \geq 1}$ converging to $0$.
Then, 
by Lusin's theorem, 
we can find, for each $N_0 \geq 1$, 
 a
 continuous function $\widetilde \varphi$ on 
 ${\mathbb R}^{D_{N_0}}$ 
 with values in $[-1,1]$ such that 
 \begin{align}
\label{eq:lem:30:2}
 & \Bigl( \sup_{N \geq N_0} \frac{Z_{N_0} c_N}{Z_N c_{N_0}} \Bigr)
 {\mathbb P}_{N_0}\Bigl( \Bigl\{ \varphi 
\Bigl( (\widehat{m}^k)_{k \in F_{N_0}^+ } \Bigr)  
 \not = 
\widetilde \varphi 
  \Bigl( (\widehat{m}^k)_{k \in F_{N_0}^+ } \Bigr) \Bigr\} \Bigr) 
\\
&= \Bigl( \sup_{N \geq N_0} \frac{Z_{N_0} c_N}{Z_N c_{N_0}} \Bigr) \frac1{Z_{N_0}}
\int_{{\mathcal O}_{N_0}} {\mathbf 1}_{\{ \varphi \not = \widetilde{\varphi}\} }  \Bigl( (\widehat{m}^k)_{k \in F_{N_0}^+ } \Bigr) 
\exp \biggl( - \sum_{k \in F_{N_0}^+} \vert k\vert^{2 p d} { \vert \widehat{m}^k \vert^2} \biggr)
\bigotimes_{j \in  F_{N_0}^+} 
d \widehat{m}^j 
\leq \delta_{N_0}. 
\nonumber
\end{align}
By Lemma 
\ref{lem:28:0},
 \begin{equation}
\label{eq:lem:30:3}
 \begin{split}
&\int_{{\mathcal P}({\mathbb T}^d)} {\mathbf 1}_{\{\varphi \not = \widetilde \varphi\}} \Bigl( (\widehat{m}^k)_{k \in F_{N_0}^+ } \Bigr) 
{\mathbf 1}_{{\mathcal O}_{N_0}} \Bigl( (\widehat{m}^k)_{k \in F_{N_0}^+ } \Bigr) 
d {\mathbb P}(m)
\leq \delta_{N_0}. 
\end{split}
\end{equation}
Therefore, 
by 
\eqref{eq:lem:30:2}
and 
\eqref{eq:lem:30:3},
\begin{equation*}
\begin{split}
&\biggl\vert \int_{{\mathcal P}({\mathbb T}^d)}
\varphi\Bigl((\widehat{m}^k)_{k \in F_{N_0}^+}\Bigr) 
d\bigl( {\mathbb P}  - {\mathbb P}_{N_0}\bigr) (m)
-
 \int_{{\mathcal P}({\mathbb T}^d)}
\widetilde \varphi \Bigl((\widehat{m}^k)_{k \in F_{N_0}^+}\Bigr) 
d\bigl( {\mathbb P}  - {\mathbb P}_{N_0}\bigr) (m)
\biggr\vert 
\\
&\leq
\biggl\vert \int_{{\mathcal P}({\mathbb T}^d)}
\Bigl(\varphi
-
\widetilde \varphi
\Bigr) \Bigl((\widehat{m}^k)_{k \in F_{N_0}^+}\Bigr) 
d {\mathbb P}   (m)
\biggr\vert 
+
\biggl\vert \int_{{\mathcal P}({\mathbb T}^d)}
\Bigl(\varphi
-
\widetilde \varphi
\Bigr) \Bigl((\widehat{m}^k)_{k \in F_{N_0}^+}\Bigr) 
d {\mathbb P}_{N_0}   (m)
\biggr\vert 
\\
&\leq
\biggl\vert \int_{{\mathcal P}({\mathbb T}^d)}
\Bigl(\varphi
-
\widetilde \varphi
\Bigr) \Bigl((\widehat{m}^k)_{k \in F_{N_0}^+}\Bigr) 
{\mathbf 1}_{A_{N_0}}(m)
d {\mathbb P}   (m)
\biggr\vert 
+
\biggl\vert \int_{{\mathcal P}({\mathbb T}^d)}
\Bigl(\varphi
-
\widetilde \varphi
\Bigr) \Bigl((\widehat{m}^k)_{k \in F_{N_0}^+}\Bigr) 
{\mathbf 1}_{A_{N_0}^{\complement}}(m)
d {\mathbb P}   (m)
\biggr\vert 
\\
&\hspace{15pt} +
\biggl\vert \int_{{\mathcal P}({\mathbb T}^d)}
\Bigl(\varphi
-
\widetilde \varphi
\Bigr) \Bigl((\widehat{m}^k)_{k \in F_{N_0}^+}\Bigr) 
d {\mathbb P}_{N_0}   (m)
\biggr\vert 
\\
&\leq 4 \delta_{N_0} + 2{\mathbb P} \Bigl( A_{N_0}^{\complement} \Bigr),
\end{split}
\end{equation*}
where we used 
the obvious implication $m \in A_{N_0} \Rightarrow (\widehat{m}^k)_{k \in F_{N_0}^+} 
\in {\mathcal O}_{N_0}$ (see Lemma
\ref{lem:25:second}).
By Lemma 
\ref{lem:AN0}, 
we can modify our choice of $(\delta_{N_0})_{N_0 \geq 1}$ 
such that 
the right-hand side is less than $5 \delta_{N_0}$. 
\vspace{5pt}

\textit{Third Step.} 
By the first and second steps, we get
\begin{equation*}
\begin{split}
&\biggl\vert \int_{{\mathcal P}({\mathbb T}^d)}
\varphi\Bigl((\widehat{m}^k)_{k \in F_{N_0}^+}\Bigr) 
d\bigl( {\mathbb P}  - {\mathbb P}_{N_0}\bigr) (m)
\biggr\vert \leq \eta_{N_0} + 5 \delta_{N_0}. 
\end{split}
\end{equation*}
The right-hand side tends to $0$ as $N_0$ tends to $\infty$, uniformly with respect to $\varphi$. 
This completes the proof. 
\end{proof}

As a corollary, we deduce 
\begin{lem}
\label{lem:uniqueness}
The sequence $({\mathbb P}_N)_{N \geq 1}$ is weakly converging. 
\end{lem}

\begin{proof}
Take two weak limits ${\mathbb P}$ and 
${\mathbb P}'$ of the sequence
$({\mathbb P}_N)_{N \geq 1}$. 
By Lemma 
\ref{lem:30},
\begin{equation}
\label{eq:uniqueness:P:1}
\lim_{N_0 \rightarrow \infty} 
\sup_{\varphi : {\mathbb R}^{{D_{N_0}}} \rightarrow [-1,1]}
\biggl\vert \int_{{\mathcal P}({\mathbb T}^d)}
\varphi\Bigl( (\widehat{m}^k)_{k \in F_{N_0}^+ } \Bigr) d\bigl( {\mathbb P} - {\mathbb P}' \bigr)(m)
\biggr\vert
=0,
\end{equation}
where, for any $N_0 \geq 1$, $\varphi$ in the argument of the supremum is required to be measurable.  

Therefore, ${\mathbb P}$ and ${\mathbb P}'$ coincide on the $\sigma$-algebra generated by 
the mappings 
$(m \mapsto 
\widehat{m}^k)_{k \in F_{N_0}^+ }$, for any $N_0 \geq 1$. 
By Lemma 
 \ref{lem:Borel} together with a standard monotone class argument, we deduce that 
${\mathbb P}$ and ${\mathbb P}'$ coincide. 
\end{proof}

\subsubsection{Completion of the proof of 
Theorem 
\ref{thm:probability:probability}}

We now prove the first claim in item $(ii)$ of Theorem 
\ref{thm:probability:probability}.

\begin{lem}
\label{lem:28}
Let ${\mathbb P}$ 
be the weak limit of $({\mathbb P}_N)_{N \geq 1}$ on ${\mathcal P}({\mathbb T}^d)$.
Then, 
for any integer $N_0 \geq 1$, 
the image of ${\mathbb P}$ by the projection mapping 
\begin{equation*}
\pi_{N_0}^{(1)} : m 
\in {\mathcal P}({\mathbb T}^d)
\mapsto 
\bigl( \widehat{m}^k \widehat{f}_{N_0}^k \bigr)_{k \in F_{N_0}^+} 
\in 
{\mathcal O}_{N_0} 
\end{equation*}
is absolutely continuous with respect to the Lebesgue measure. 
\end{lem}

\begin{proof}
For $N_0 \geq 1$, 
we consider a 
$[0,1]$-valued measurable function 
$\varphi$ defined on 
${\mathbb R}^{2 \vert F^+_{N_0}\vert}$ that is equal to $0$ outside 
${\mathcal O}_{N_0}$. 

Following 
\eqref{eq:c_d:ON:0}, we get, for any $N \geq N_0$, 
\begin{equation*}
\begin{split}
&\int_{{\mathcal P}({\mathbb T}^d)} \varphi\Bigl( (\widehat{m}^k \widehat{f}^k_{N_0})_{k \in F_{N_0}^+ } \Bigr) d {\mathbb P}_N(m) 
\\
&= \frac1{Z_N} 
\int_{{\mathcal O}_N}
\varphi\Bigl( (\widehat{m}^k \widehat{f}^k_{N_0})_{k \in F_{N_0}^+ } \Bigr)
\exp \biggl( - \sum_{k \in F_N^+} \vert k\vert^{2 p d} { \vert \widehat{m}^k \vert^2} \biggr)
\bigotimes_{k \in F_N^+} 
d \widehat{m}^k
\\
&= \frac1{Z_{N}}
\int_{{\mathbb R}^{2  \vert F_{N_0}^+\vert } } \varphi \Bigl( (\widehat{m}^k \widehat{f}^k_{N_0})_{k \in F_{N_0}^+ }  \Bigr) 
\exp \biggl( - \sum_{k \in F_{N_0}^+} \vert k\vert^{2 p d} { \vert \widehat{m}^k \vert^2} \biggr)
\\
&\hspace{5pt} \times   
\int_{{\mathbb R}^{2(\vert F_N^+\vert - \vert F_{N_0}^+\vert )}}
\biggl[ 
{\mathbf 1}_{{\mathcal O}_N}\Bigl( (\widehat{m}^k)_{k \in F_{N}^+ } \Bigr) 
\exp \biggl( - \sum_{k \in F_{N}^+ \setminus F_{N_0}^+}\vert k\vert^{2 p d} { \vert \widehat{m}^k \vert^2} \biggr)
\bigotimes_{j \in F_{N}^+ \setminus F_{N_0}^+} 
d \widehat{m}^j
\biggr]
\bigotimes_{j \in  F_{N_0}^+} 
d \widehat{m}^j.
\end{split}
\end{equation*}
And then, 
\begin{equation*}
\begin{split}
&\int_{{\mathcal P}({\mathbb T}^d)} \varphi\Bigl( (\widehat{m}^k \widehat{f}^k_{N_0})_{k \in F_{N_0}^+ } \Bigr) d {\mathbb P}_N(m) 
\\
&\leq  
 \frac1{Z_{N}} \frac{c_N}{c_{N_0}}
\int_{{\mathbb R}^{2  \vert F_{N_0}^+\vert } } \varphi \Bigl( (\widehat{m}^k \widehat{f}^k_{N_0})_{k \in F_{N_0}^+ }  \Bigr) 
\exp \biggl( - \sum_{k \in F_{N_0}^+} \vert k\vert^{2 p d} { \vert \widehat{m}^k \vert^2} \biggr) \bigotimes_{j \in  F_{N_0}^+} 
d \widehat{m}^j.
\end{split}
\end{equation*}
We make, in the integral on the second line, the following change of variable: 
\begin{equation*}
\widehat{y}^k = 
\widehat{m}^k \widehat{f}^k_{N_0}, \quad k \in F_{N_0}^+.
\end{equation*}
Then, since $\vert \widehat{f}^k_{N_0}\vert \leq 1$ for each $k \in F_{N_0}^+$, we get
\begin{equation*}
\begin{split}
&\int_{{\mathcal P}({\mathbb T}^d)} \varphi\Bigl( (\widehat{m}^k \widehat{f}^k_{N_0})_{k \in F_{N_0}^+ } \Bigr) d {\mathbb P}_N(m) 
\\
&\leq  
 \frac1{Z_{N}} \frac{c_N}{c_{N_0}} 
\biggl(  \prod_{k \in F_{N_0}^+} 
\bigl\vert \widehat{f}^k_{N_0} \vert 
\biggr)^{-1}
\int_{{\mathbb R}^{2  \vert F_{N_0}^+\vert } } \varphi \Bigl( (\widehat{y}^k )_{k \in F_{N_0}^+ }  \Bigr) 
\exp \biggl( - \sum_{k \in F_{N_0}^+} 
\frac{\vert k\vert^{2 p d} { \vert \widehat{y}^k \vert^2} }{
\vert \widehat{f}^k_{N_0} \vert^{2} 
}
\biggr) \bigotimes_{j \in  F_{N_0}^+} 
d \widehat{y}^j
\\
&\leq
\frac1{Z_{N}} \frac{c_N}{c_{N_0}} 
\biggl(  \prod_{k \in F_{N_0}^+} 
\bigl\vert \widehat{f}^k_{N_0} \vert 
\biggr)^{-1}
\int_{{\mathbb R}^{2  \vert F_{N_0}^+\vert } } \varphi \Bigl( (\widehat{y}^k )_{k \in F_{N_0}^+ }  \Bigr) 
\exp \biggl( - \sum_{k \in F_{N_0}^+} 
 \vert k\vert^{2 p d} { \vert \widehat{y}^k \vert^2} 
\biggr) \bigotimes_{j \in  F_{N_0}^+} 
d \widehat{y}^j. 
\end{split}
\end{equation*}
Since $\varphi$ is zero outside ${\mathcal O}_{N_0}$, the above integral reduces to an integral over
${\mathcal O}_{N_0}$. 
Therefore, 
\begin{equation*}
\begin{split}
&\int_{{\mathcal P}({\mathbb T}^d)} \varphi\Bigl( (\widehat{m}^k \widehat{f}^k_{N_0})_{k \in F_{N_0}^+ } \Bigr) d {\mathbb P}_N(m) 
\\
&\leq  
\frac1{Z_{N}} \frac{c_N}{c_{N_0}} 
\biggl(  \prod_{k \in F_{N_0}^+} 
\bigl\vert \widehat{f}^k_{N_0} \vert 
\biggr)^{-1}
\int_{{\mathcal O}_{N_0} } \varphi \Bigl( (\widehat{y}^k )_{k \in F_{N_0}^+ }  \Bigr) 
\exp \biggl( - \sum_{k \in F_{N_0}^+} 
 \vert k\vert^{2 p d} 
\biggr) \bigotimes_{j \in  F_{N_0}^+} 
d \widehat{y}^j.
\end{split}
\end{equation*}
By Lemma 
\ref{lem:25:three}, we can let $N$ tend to $\infty$ in the left-hand side. 
We
get 
\begin{equation*}
\begin{split}
&\int_{{\mathcal P}({\mathbb T}^d)} \varphi\Bigl( (\widehat{m}^k \widehat{f}^k_{N_0})_{k \in F_{N_0}^+ } \Bigr) d {\mathbb P}(m) 
\\
&\leq  \Bigl( \sup_{N \geq N_0} \frac{Z_{N_0} c_N}{Z_N c_{N_0}} \Bigr)
\frac1{Z_{N_0}} 
\biggl(  \prod_{k \in F_{N_0}^+} 
\bigl\vert \widehat{f}^k_{N_0} \vert 
\biggr)^{-1}
\int_{{\mathcal O}_{N_0} } \varphi \Bigl( (\widehat{y}^k )_{k \in F_{N_0}^+ }  \Bigr) 
\exp \biggl( - \sum_{k \in F_{N_0}^+} 
 \vert k\vert^{2 p d} 
\biggr) \bigotimes_{j \in  F_{N_0}^+} 
d \widehat{y}^j,
\end{split}
\end{equation*}
where we recall from 
Lemma 
\ref{lem:25:new} that the first factor in the above right-hand side is finite. 
This completes the proof. 
\end{proof}

We now prove item $(i)$ in the statement of Theorem 
\ref{thm:probability:probability}.

\begin{lem}
\label{lem:33}
Let ${\mathbb P}$ 
be the weak limit of $({\mathbb P}_N)_{N \geq 1}$ on ${\mathcal P}({\mathbb T}^d)$.
Then, 
${\mathbb P}$ 
 has a full support on ${\mathcal P}({\mathbb T}^d)$ equipped with the $1$-Wasserstein distance.
\end{lem}

\begin{proof}
\textit{First Step.}
We consider 
a measure $m_0 \in {\mathcal P}({\mathbb T}^d)$ and a real $\varepsilon \in (0,1)$. We have to prove 
that ${\mathbb P}(B(m_0,\varepsilon))>0$ with 
$B(m_0,\varepsilon)$ the ball of 
${\mathcal P}({\mathbb T}^d)$ of
center 
$m_0$ and of radius $\varepsilon$ for the 1-Wasserstein distance. 

By Lemma 
\ref{lem:weak:convergence}, we notice that, for $N_0$ large enough, 
\begin{equation*}
m_0*f_{N_0}=
\sum_{k \in F_{N_0}} \widehat{m}_0^k \widehat{f}_{N_0}^k e_{-k} \in B\bigl(m_0,\tfrac14 \varepsilon\bigr), 
\end{equation*}
And then,
using the same constant $c_0$ as in the notation 
introduced in Section \ref{se:introduction}, we get 
\begin{equation}
\label{eq:lem:33:1}
\frac{\varepsilon}{4c_0} + 
\bigl( 1 - \frac{\varepsilon}{4c_0} \bigr) 
m_0 * {f}_{N_0} 
=
\frac{\varepsilon}{4c_0} + 
\bigl( 1 - \frac{\varepsilon}{4c_0} \bigr) 
\sum_{k \in F_{N_0}} \widehat{m}_0^k \widehat{f}_{N_0}^k e_{-k}
 \in B(m_0,\tfrac34 \varepsilon).  
\end{equation}
Recall that there exists a constant $c_d$, only depending on $d$, such that, for any $a>0$, 
\begin{equation*}
\sum_{k \not \in F_{N_0}} 
\frac{a}{\vert k \vert^{5d/2}} 
\leq \frac{c_d a}{{N_0}^{3d/2}}.
\end{equation*}
Therefore, for  
$a$ fixed, for 
$c_d a/ N_0^{3d/2} < \varepsilon/(8c_0)$
and
for any 
collection $(\widehat{m}^k)_{k \in {\mathbb Z}^d}$, with
$\widehat{m}^{0}=1$, 
$\widehat{m}^{-k} = \overline{\widehat{m}^k}$ for $k \not =0$, and 
\begin{equation*}
\begin{cases}
&\displaystyle \sum_{k \in F_{N_0} \setminus \{0\}} \Bigl\vert \widehat{m}^k - 
\bigl( 1 - \frac{\varepsilon}{4c_0}  \bigr) 
\widehat{m}_0^k \widehat{f}_{N_0}^k
\Bigr\vert \leq \frac{\varepsilon}{8c_0},  
\\
&\displaystyle \forall k \not \in F_{N_0}, \quad 
\vert \widehat{m}^k
 \vert \leq 
a \vert k \vert^{-5d/2},
\end{cases}  
\end{equation*}
it holds that 
\begin{equation*}
\begin{split}
&\sup_{x \in {\mathbb T}^d} 
\Bigl\vert
\sum_{k \in F_{N_0}}
\widehat{m}^k 
e_{-k}(x)
 - 
\Bigl(
\frac{\varepsilon}{4c_0}  + 
\bigl( 1 - \frac{\varepsilon}{4c_0}  \bigr) 
m_0 * {f}_{N_0}(x)
\Bigr)
\Bigr\vert \leq \frac{\varepsilon}{8c_0},
\\
&\sup_{x \in {\mathbb T}^d} 
\Bigl\vert m(x) - 
\Bigl(
\frac{\varepsilon}{4c_0} + 
\bigl( 1 - \frac{\varepsilon}{4c_0} \bigr) 
m_0 * {f}_{N_0}(x)
\Bigr)
\Bigr\vert < \frac{\varepsilon}{4c_0},
\end{split}
\end{equation*}
with
\begin{equation*}
m = \sum_{k \in {\mathbb Z}^d} \widehat{m}^k   e_{-k}.
\end{equation*} 
Therefore, 
from 
\eqref{eq:lem:33:1},
$m  \in B(m_0,\varepsilon)$
and
\begin{equation}
\label{eq:lem:33:4}
\inf_{x \in {\mathbb T}^d} \biggl( 
\sum_{k \in F_{N_0}}
\widehat{m}^k 
e_{-k}(x)
 \biggr) \geq \frac{\varepsilon}{4c_0} + 
\bigl( 1 - \frac{\varepsilon}{4c_0} \bigr) 
\inf_{x \in {\mathbb T}^d}
m_0 * {f}_{N_0}(x) - \frac{\varepsilon}{8c_0}  \geq \frac{\varepsilon}{8c_0} 
> 
c_d a  N_0^{-3d/2}.
\end{equation}
From the fact that 
$m  \in B(m_0,\varepsilon)$, 
we deduce that, 
in order to complete the proof, it suffices to choose $a$ such that, for some $N_0$ large enough, 
\begin{equation}
\label{eq:lem:33:2}
{\mathbb P} 
\biggl( 
\biggl\{ 
 \sum_{k \in F_{N_0} \setminus \{0\}} \bigl\vert \widehat{m}^k - 
\bigl( 1 - \frac{\varepsilon}{4c_0} \bigr) 
\widehat{m}_0^k \widehat{f}_{N_0}^k
\bigr\vert \leq \frac{\varepsilon}{8c_0} 
\biggr\} 
\cap   
\biggl\{ \forall k \not \in F_{N_0}, \ 
\vert \widehat{m}^k
 \vert \leq 
a \vert k \vert^{-5d/2}
\biggr\}
\biggr) > 0. 
\end{equation}
\vskip 3pt

\textit{Second Step.}
We claim that in order to prove \eqref{eq:lem:33:2}, it suffices to prove that there exists $c>0$ such that, for some $N_0$ large enough and
for $N \geq N_0$, 
\begin{equation}
\label{eq:lem:33:3}
{\mathbb P}_N 
\biggl( 
\biggl\{ 
 \sum_{k \in F_{N_0} \setminus \{0\}} \bigl\vert \widehat{m}^k - 
\bigl( 1 - \frac{\varepsilon}{4c_0}  \bigr) 
\widehat{m}_0^k \widehat{f}_{N_0}^k
\bigr\vert \leq \frac{\varepsilon}{8c_0} 
\biggr\} 
\cap   
\biggl\{ \forall k  \in F_N^+ \setminus F_{N_0}^+, \ 
\vert \widehat{m}^k
 \vert \leq 
a \vert k \vert^{-5d/2}
\biggr\}
\biggr) \geq c. 
\end{equation}
Assume indeed that the above holds true. Then, for a given $N_1 \geq N_0$ and for $N \geq N_1$, 
we have
\begin{equation*}
{\mathbb P}_N 
\biggl( 
\biggl\{ 
 \sum_{k \in F_{N_0} \setminus \{0\}} \bigl\vert \widehat{m}^k - 
\bigl( 1 - \frac{\varepsilon}{4c_0}  \bigr) 
\widehat{m}_0^k \widehat{f}_{N_0}^k
\bigr\vert \leq \frac{\varepsilon}{8c_0} 
\biggr\} 
\cap   
\biggl\{ \forall k  \in F_{N_1}^+ \setminus F_{N_0}^+, \ 
\vert \widehat{m}^k
 \vert \leq 
a \vert k \vert^{-5d/2}
\biggr\}
\biggr) \geq c. 
\end{equation*}
Then, by Portmanteau theorem, we can easily replace ${\mathbb P}_N$ by 
${\mathbb P}$ in the left-hand side. 
Letting $N_1$ tend to $\infty$, we then get 
\eqref{eq:lem:33:2}.

We now prove 
\eqref{eq:lem:33:3}.
To do so, we apply Lemma 
\ref{lem:25:second}. 
Assuming that $a$ right above is greater than $a_0$ in the statement of Lemma 
\ref{lem:25:second}, it suffices to show that 
\begin{equation*}
{\mathbb P}_N 
\biggl( 
\biggl\{ 
 \sum_{k \in F_{N_0} \setminus \{0\}} \bigl\vert \widehat{m}^k - 
\bigl( 1 - \frac{\varepsilon}{4c_0} \bigr) 
\widehat{m}_0^k \widehat{f}_{N_0}^k
\bigr\vert \leq \frac{\varepsilon}{8c_0}
\biggr\} 
\cap   
\biggl\{ \forall k  \in F_N^+ \setminus F_{N_0}^+, \ 
\vert \widehat{m}^k
 \vert \leq 
\frac12 a_0 \vert k \vert^{-5d/2}
\biggr\}
\biggr) \geq c. 
\end{equation*}
From 
\eqref{eq:lem:33:4}, we also notice that, for $a$ large enough (the threshold being independent of $N_0$ and only depending on $d$),
\begin{equation*}
\biggl\{ 
 \sum_{k \in F_{N_0} \setminus \{0\}} \bigl\vert \widehat{m}^k - 
\bigl( 1 - \frac{\varepsilon}{4c_0} \bigr) 
\widehat{m}_0^k \widehat{f}_{N_0}^k
\bigr\vert \leq \frac{\varepsilon}{8c_0}
\biggr\} \subset A_{N_0},
\end{equation*}
with $A_{N_0}$ as in the statement of Lemma \ref{lem:25:second}. 
Therefore, Lemma 
\ref{lem:25:second} says that, for some $N_0$ large enough, for any $N \geq N_0$, 
\begin{equation*}
\begin{split}
&{\mathbb P}_N 
\biggl( 
\biggl\{ 
 \sum_{k \in F_{N_0} \setminus \{0\}} \bigl\vert \widehat{m}^k - 
\bigl( 1 - \frac{\varepsilon}{4c_0} \bigr) 
\widehat{m}_0^k \widehat{f}_{N_0}^k
\bigr\vert \leq \frac{\varepsilon}{8c_0}
\biggr\} 
\cap   
\biggl\{ \forall k  \in F_N^+ \setminus F_{N_0}^+, \ 
\vert \widehat{m}^k
 \vert \leq 
\frac12 a_0 \vert k \vert^{-5d/2}
\biggr\}
\biggr) 
\\
&\geq \frac12 {\mathbb P}_{N_0} 
\biggl( 
\biggl\{ 
 \sum_{k \in F_{N_0} \setminus \{0\}} \bigl\vert \widehat{m}^k - 
\bigl( 1 - \frac{\varepsilon}{4c_0} \bigr) 
\widehat{m}_0^k \widehat{f}_{N_0}^k
\bigr\vert \leq \frac{\varepsilon}{8c_0}
\biggr\} 
\biggr). 
\end{split}
\end{equation*}

Obviously,  
$(( 1 -   \varepsilon/(4c_0)) 
\widehat{m}_0^k \widehat{f}_{N_0}^k)_{k \in F_{N_0}^+ \setminus \{0\}}$ is in 
${\mathcal O}_{N_0}$ since $1 + \sum_{k \in F_{N_0} \setminus \{0\}}
( 1 - {\varepsilon}/(4c_0)) 
\widehat{m}_0^k \widehat{f}_{N_0}^k \geq \varepsilon/(4c_0)$. 
Therefore, using the fact that the density of ${\Gamma}_{N_0}$ 
in 
\eqref{eq:Gamma_N}
is strictly positive on ${\mathcal O}_{N_0}$,
the last term in the above inequality is strictly positive. 
\end{proof}

It remains to see that
\begin{lem}
\label{lem:density:under:P}
Let ${\mathbb P}$ 
be the weak limit of $({\mathbb P}_N)_{N \geq 1}$ on ${\mathcal P}({\mathbb T}^d)$.
Then, for ${\mathbb P}$-almost every $m \in {\mathcal P}({\mathbb T}^d)$, 
$m$ has a strictly positive density
and $\sum_{k \in {\mathbb Z}^d} \vert k\vert \vert \widehat{m}^k \vert < \infty$. 
\end{lem}

\begin{proof}
From Lemma \ref{lem:27},
we recall that, with $a_0$ as therein, for any $\delta \in (0,1)$, for $N_0$ large enough
and for $N \geq N_1 \geq N_0$, 
\begin{equation*}
\begin{split}
&{\mathbb P}_N
\biggl( A_{N_0} \cap 
\bigcap_{ k  \in F_{N_1}^+ \setminus F_{N_0}^+} 
 \Bigl\{  m \in {\mathcal P}({\mathbb T}^d) :
   \vert \widehat{m}^k \vert 
  \leq \frac{a_0}{\vert k \vert^{5d/2}} 
  \Bigr\}
\biggr)
\geq 1 - \delta,
\end{split}
\end{equation*} 
where $A_{N_0}$
is as in the statement of 
Lemma 
\ref{lem:25:second}.
By arguing as in the proof of 
Lemma
\ref{lem:33}, we can apply Portmanteau theorem and replace ${\mathbb P}_N$ by 
${\mathbb P}$
(with $N_1$ being fixed in the intersection symbol).
Then, letting $N_1$ tend to $\infty$, we get 
\begin{equation*}
\begin{split}
&{\mathbb P}
\biggl( A_{N_0} \cap 
\bigcap_{ k  \in F_{\infty}^+ \setminus F_{N_0}^+} 
 \Bigl\{  m \in {\mathcal P}({\mathbb T}^d) :
   \vert \widehat{m}^k \vert 
  \leq \frac{a_0}{\vert k \vert^{5d/2}} 
  \Bigr\}
\biggr)
\geq 1 - \delta,
\end{split}
\end{equation*}
with $F_{\infty}^+ := \bigcup_{N \geq 1} F_N^+$.
In particular, since $\delta$ is arbitrary, 
\begin{equation*}
\begin{split}
&{\mathbb P}
\biggl( \bigcup_{N_0 \geq 1} \biggl( A_{N_0} \cap 
\bigcap_{ k  \in F_{\infty}^+ \setminus F_{N_0}^+} 
 \Bigl\{  m \in {\mathcal P}({\mathbb T}^d) :
   \vert \widehat{m}^k \vert 
  \leq \frac{a_0}{\vert k \vert^{5d/2}} 
  \Bigr\}
\biggr)\biggr)
=1.
\end{split}
\end{equation*}

 Take now $m$ in the event that appears in the left-hand side. We can find $N_0 \geq 1$
 such that $m$ belongs to $A_{N_0}$ and, for any $N \geq N_0$,
 \begin{equation*}
   \vert \widehat{m}^k \vert 
  \leq \frac{a_0}{\vert k \vert^{5d/2}}. 
 \end{equation*}
By Lemma
\ref{lem:25:second}, this implies
\begin{equation*}
1+ \inf_{x \in {\mathbb T}^d} \sum_{k \in F_{N} \setminus \{0\}}
\widehat{m}^k e_{-k}(x) \geq 2^{-(3d/2+1)} N_0^{-3d/2}
\end{equation*}
Given the decay of the Fourier modes of $m$, it is easy to deduce that 
$m$ has a continuously differentiable density (which we identify with $m$ itself) and that 
\begin{equation*}
\inf_{x \in {\mathbb T}^d}  m(x) \geq 2^{-(3d/2+1)} N_0^{-3d/2},
\end{equation*}
which completes the proof. 
\end{proof}

\subsection{Proof of Theorem 
\ref{prop:rademacher} (Rademacher's theorem on 
$({\mathcal P}({\mathbb T}^d),{\mathbb P})$)}
\label{subse:proof:rademacher}

\subsubsection*{First part: existence of the derivative}
We follow the main lines of the proof given in \cite[Theorem 10.6.4]{Bogachev} (with the little difference that the latter is 
given for functionals defined on a Banach space, which makes it slightly easier).

We start with the following definition. For an integer $N \geq 1$, we consider the mapping 
\begin{equation}
\label{eq:PNinfini}
{\mathscr L}_{N,\infty} : m \in {\mathcal P}({\mathbb T}^d) \mapsto 
\bigl( \widehat m^k \bigr)_{k \in F_\infty^+ \setminus F_N^+} \in \ell^2_{\mathbb C}( F_\infty^+ \setminus F_N^+ ):=\bigl\{(\widehat r^k)_{k \in F_\infty^+ \setminus F_N^+} : 
\widehat r^k \in {\mathbb C} \bigr\},
\end{equation}
with $F_\infty^+ = \{ k \in {\mathbb Z}^d : \sharp(k) >0\} = \bigcup_{n \geq 1} F_n^+$, 
and we denote by ${\textsf P}_{N,\infty}$ the image (on $\ell^2_{\mathbb C}(F_\infty^+ \setminus F_N^+)$) of the measure ${\mathbb P}$
by 
${\mathscr L}_{N,\infty}$ and, then
by 
$$\Bigl({\mathbb P}_{N \vert N,\infty}\bigl( \, \cdot \,  \vert (\widehat r^k)_{k \in F_\infty^+ \setminus F_N^+}\bigr)\Bigr)_{(\widehat r^k)_k \in \ell^2_{\mathbb C}(F_\infty^+ \setminus F_N^+)}$$
a regular conditional probability distribution of ${\mathbb P}$ given ${\mathscr L}_{N,\infty}$, namely, 
for any Borel subsets $A$ of ${\mathbb C}^{\vert F_N^+\vert} (\simeq {\mathbb R}^{2 \vert F_N^+\vert})$ and $B$ of 
$\ell^2_{\mathbb C}(F_{\infty}^+ \setminus F_N^+)$, it holds that 
\begin{equation}
\label{eq:rademacher:0}
\begin{split}
&{\mathbb P} 
\Bigl( \bigl\{ m \in {\mathcal P}({\mathbb T}^d) : 
(\widehat{m}^k)_{k \in F_N^+} 
\in A \ \textrm{\rm and} 
\
(\widehat{m}^k)_{k \in F_\infty^+ \setminus F_N^+} 
\in B
\bigr\}
\Bigr) 
\\
&= \int_{\ell^2_{\mathbb C}(F_\infty^+ \setminus F_N^+)}
\biggl[ 
{\mathbb P}_{N\vert(N,\infty)} 
\Bigl( 
\bigl\{ m \in {\mathcal P}({\mathbb T}^d) : 
(\widehat{m}^k)_{k \in F_N^+} 
\in A \bigr\} 
\vert 
(\widehat{r}^k)_{k \in F_\infty^+ \setminus F_N^+} 
\Bigr) 
\\
&\hspace{90pt}
\times  
{\mathbf 1}_B 
\Bigl(
(\widehat{r}^k)_{k \in F_\infty^+ \setminus F_N^+}
\Bigr) 
\biggr]
d {\textsf P}_{N,\infty} 
\Bigl(
(\widehat{r}^k)_{k \in F_\infty^+ \setminus F_N^+}
\Bigr).
\end{split}
\end{equation}
In particular, recalling 
from Lemma 
\ref{lem:density:under:P} that, 
for 
 ${\mathbb P}$
 almost every 
$m$, 
$ \sum_{j \in {\mathbb Z}^d} \vert \widehat m^j \vert < \infty$,
we deduce that, 
 for 
$\textsf{\sf P}_{N,\infty}$ almost every 
$(\widehat r^k)_{k \in F_\infty^+ \setminus F_N^+} $, 
$\sum_{j \in F_\infty^+ \setminus F_N^+} \vert \widehat{r}^j \vert < \infty$. 

We now consider the set 
\begin{equation*}
C_N := \bigl\{ m \in {\mathcal P}({\mathbb T}^d) :  \phi \ \textrm{\rm is differentiable at \textit{m} along the directions} \ (\fD[e_k])_{k \in F_N^+,\fD=\Re,\Im} \bigr\}. 
\end{equation*}
In clear, 
$m \in C_N$
if, 
for any $k \in F_N^+$,
for $\fD=\Re,\Im$, 
the limit 
\begin{equation*}
\lim_{\eta \rightarrow 0} \frac{1}{\eta} \Bigl( \phi\bigl( m + \eta  \fD[e_k]\bigr) 
 - \phi(m) \Bigr)
{\mathbf 1}_{ {\mathcal P}({\mathbb T}^d)}\bigl(m + \eta   \fD[e_k]\bigr)
\end{equation*}
exists, which proves that 
$C_N$ is a Borel subset of 
${\mathcal P}({\mathbb T}^d)$. 
We now address, 
for a given 
$(\widehat r^k)_{k \in F_\infty^+ \setminus F_N^+} $ in $\ell^2_{\mathbb C}(F_\infty^+ \setminus F_N^+)$, 
the probability 
${\mathbb P}_{N\vert(N,\infty)} 
( C_N \vert 
(\widehat r^k)_{k \in F_\infty^+ \setminus F_N^+} 
)$. By definition of a regular conditional probability distribution, we have, for 
$\textsf{\sf P}_{N,\infty}$ almost every 
$(\widehat r^k)_{k \in F_\infty^+ \setminus F_N^+} $, 
\begin{equation*}
\begin{split}
&{\mathbb P}_{N\vert(N,\infty)} 
\bigl( C_N \vert 
(\widehat r^k)_{k \in F_\infty^+ \setminus F_N^+} 
\bigr)
\\
&= {\mathbb P}_{N\vert(N,\infty)} 
\Bigl(
\bigcap_{k \in F_N^+}
\bigcap_{\fD=\Re,\Im}
\Bigl\{ m : 
\lim_{\eta \rightarrow 0} \frac{1}{\eta} \Bigl( \phi( m + \eta \fD[  e_k]) 
 - \phi(m) \Bigr)
{\mathbf 1}_{ {\mathcal P}({\mathbb T}^d)}(m + \eta \fD[ e_k])
\ \textrm{\rm exists}
\Bigr\}
\\
&\hspace{75pt} \cap 
\bigcap_{k \in F_\infty^+ \setminus F_N^+}
\{ \widehat{m}^k =  \widehat{r}^k
\}
\,
\vert 
\, 
(\widehat r^k)_{k \in F_\infty^+ \setminus F_N^+} 
\Bigr).
\end{split}
\end{equation*}
Obviously, on the event 
$\bigcap_{k \in F_\infty^+ \setminus F_N^+}
\{ \widehat{m}^k =  \widehat{r}^k
\}$,
we have, for any $k \in F_N \setminus \{0\}$
and for $\fD=\Re,\Im$, 
\begin{equation}
\label{eq:rademacher:1}
\begin{split}
& \frac1{\eta} \Bigl( \phi(m + \eta \fD[e_k]) - \phi(m) \Bigr) 
{\mathbf 1}_{ {\mathcal P}({\mathbb T}^d)}(m + \eta \fD[e_k]) 
 \\
&= \frac{1}{\eta} \biggl\{ \phi\Bigl( 1 + \sum_{j \in F_N^+ }  2 \Re [\widehat m^j e_{-j}] + \eta   \fD[e_k]
+\sum_{j \in F_\infty^+ \setminus F_N^+}  2 \Re [ \widehat r^j e_{-j}]
 \Bigr) 
 \\
&\hspace{30pt} - \phi
\Bigl( 1 +  \sum_{j \in F_N^+ } 2 \Re[ \widehat m^j e_{-j}]
+ \sum_{j  \in F_\infty^+ \setminus F_N^+} 2 \Re[ \widehat r^j e_{-j}]
\Bigr)
  \biggr\}
  \\
  &\hspace{15pt} \times
{\mathbf 1}_{ {\mathcal P}({\mathbb T}^d)}\Bigl(
1+
 \sum_{j \in F_N^+ }  2 \Re [\widehat m^j e_{-j}] + \eta   \fD[e_k]
+\sum_{j \in F_\infty^+ \setminus F_N^+}  2 \Re [ \widehat r^j e_{-j}]
 \Bigr).
\end{split}
\end{equation}
For $(\widehat r^j)_{j \in F_\infty^+ \setminus F_N^+}$ in $\ell^2_{\mathbb C}(F_\infty^+ \setminus F_N^+)$ such that $\sum_{j \in F_\infty^+ \setminus F_N^+} \vert \widehat{r}^j \vert < \infty$, we let 
\begin{equation*}
\begin{split}
&{\mathcal O}_N\bigl( 
(\widehat r^j)_{j \in F_\infty^+ \setminus F_N^+} \bigr) 
\\
&:= 
\biggl\{ 
(\widehat{m}^j)_{j \in F_N^+}
\in {\mathbb C}^{\vert F_N^+\vert} :  
 \inf_{x \in {\mathbb T}^d} 
\Bigl[ 1 + 
 \sum_{j \in F_N^+ } 2 \Re[ \widehat m^j e_{-j}(x)]
+ \sum_{j  \in F_\infty^+ \setminus F_N^+} 2 \Re[ \widehat r^j e_{-j}(x)]
\Bigr]
>0
\biggr\}.
\end{split}
\end{equation*}
It is easy to see that 
${\mathcal O}_N( 
(\widehat r^j)_{j \in F_\infty^+ \setminus F_N^+} )$
is an open convex
subset 
of ${\mathbb R}^{2 \vert F_N^+ \vert}$ (regarding the complex coordinates as pairs of reals). 
Its closure 
writes
\begin{equation*}
\begin{split}
&\overline{{\mathcal O}_N}\bigl( 
(\widehat r^j)_{j \in F_\infty^+ \setminus F_N^+} \bigr)
\\ 
&:= 
\biggl\{ 
(\widehat{m}^j)_{j \in F_N^+}
\in {\mathbb C}^{\vert F_N^+\vert} :  
 \inf_{x \in {\mathbb T}^d} 
\Bigl[ 1 + 
 \sum_{j \in F_N^+ } 2 \Re[ \widehat m^j e_{-j}(x)]
+ \sum_{j  \in F_\infty^+ \setminus F_N^+} 2 \Re[ \widehat r^j e_{-j}(x)]
\Bigr]
\geq 0
\biggr\}.
\end{split}
\end{equation*}
We can consider 
$p^N$ the orthogonal projection from 
${\mathbb R}^{2\vert F_N^+\vert}$ onto 
$\overline{{\mathcal O}_N}( 
(\widehat r^j)_{j \in F_\infty^+ \setminus F_N^+} )$ (for simplicity, we do not specify the dependence upon 
$(\widehat r^j)_{j \in F_\infty^+ \setminus F_N^+}$ in the notation $p^N$). 
We then let 
\begin{equation*}
\begin{split}
&\phi^N\Bigl( (\widehat{m}^j)_{j \in F_N^+} \vert 
(\widehat r^j)_{j \in F_\infty^+ \setminus F_N^+} 
\Bigr)
\\
&:=\phi
\biggl( 1 + \sum_{j \in F_N^+ } 2 \Re \Bigl\{   \Bigl[ p^N 
\Bigl(
(\widehat{m}^k)_{k \in F_N \setminus \{0\}}
\Bigr)
\Bigr]_j
e_{-j}\Bigr\}
+ \sum_{  j \in F_\infty^+ \setminus F_N^+} 2 \Re [ \widehat r^j e_{-j}]
\Bigr\} \biggr).
\end{split}
\end{equation*}
Obviously, $\phi^N$ is a Lipschitz function on 
${\mathbb R}^{2\vert F_N^+\vert}$. As such, it is differentiable almost everywhere:
we call ${\mathcal D}[\phi^N]$
the set of differentiability points. 
Back to 
\eqref{eq:rademacher:1}, we have, on the event 
$\bigcap_{k \in F_\infty^+ \setminus F_N^+}
\{ \widehat{m}^k =  \widehat{r}^k
\}$,
the identity
\begin{equation*}
\begin{split}
& \frac1{\eta} \Bigl( \phi(m + \eta \fD[e_k]) - \phi(m) \Bigr) 
{\mathbf 1}_{ {\mathcal P}({\mathbb T}^d)}(m + \eta \fD[e_k]) 
 \\
&= \frac{1}{\eta} \Bigl[ \phi^N\Bigl( (\widehat{m}^j + \eta z_{\fD} {\mathbf 1}_{\{k=j\}})_{j \in F_N^+} \vert 
(\widehat r^j)_{j \in F_\infty^+ \setminus F_N^+} 
\Bigr)
 - \phi^N 
 \Bigl( (\widehat{m}^j)_{j \in F_N^+} \vert 
(\widehat r^j)_{j \in F_\infty^+ \setminus F_N^+} 
\Bigr)
 \Bigr] 
\\
&\hspace{15pt} \times 
{\mathbf 1}_{ {\mathcal P}({\mathbb T}^d)}\Bigl(
1 +  \sum_{j \in F_N^+} 2 \Re[\widehat m^j e_{-j}] + \eta\fD[   e_k]
+ \sum_{j \in F_\infty^+ \setminus F_N^+} 2 \Re[\widehat r^j e_{-j}]
 \Bigr),
\end{split}
\end{equation*}
with $z_{\fD}=1$ if $\fD=\Re$ and $z_{\fD}=\i$ if $\fD=\Im$. 
In particular, if
$(\widehat{m}^j + \eta z_{\fD} {\mathbf 1}_{\{k=j\}})_{j \in F_N^+}$
belongs to 
${\mathcal O}_N( 
(\widehat r^j)_{j \in F_\infty^+ \setminus F_N^+} )$ and to ${\mathcal D}[\phi^N]$, then 
the above right-hand side has a limit as $\eta$ tends to $0$. 
Therefore, 
for $\textsf{P}_{N,\infty}$-a.e. $(\widehat r^k)_{k \in F_\infty^+ \setminus F_N^+} $, 
\begin{equation*}
\begin{split}
&{\mathbb P}_{N\vert(N,\infty)} 
\bigl( C_N \vert 
(\widehat r^k)_{k \in F_\infty^+ \setminus F_N^+} 
\bigr)
\\
&\geq {\mathbb P}_{N\vert(N,\infty)} 
\Bigl(
{\mathcal D}[\phi^N] \cap 
\Bigl\{
(\widehat{m}^j)_{j \in F_N \setminus \{0\}}
\in 
{\mathcal O}_N( 
(\widehat r^k)_{k \in F_\infty^+ \setminus F_N^+} )
\Bigr\}
 \cap 
\bigcap_{\vert k \vert \geq N}
\{ \widehat{m}^k = \widehat{r}^k
\}
\,
\vert 
\, 
(\widehat r^k)_{k \in F_\infty^+ \setminus F_N^+} 
\Bigr).
\end{split}
\end{equation*}
We now prove the following two things:
for $\textsf{P}_{N,\infty}$-a.e. $(\widehat r^k)_{k \in F_\infty^+ \setminus F_N^+} $, 
\begin{equation}
\label{eq:rademacher:2}
\begin{split}
&{\mathbb P}_{N\vert(N,\infty)} 
\Bigl(
{\mathcal D}[\phi^N] 
\,
\vert 
\, 
(\widehat r^k)_{k \in F_\infty^+ \setminus F_N^+} 
\Bigr) = 1,
\\
&{\mathbb P}_{N\vert(N,\infty)} 
\Bigl(
\Bigl\{
(\widehat{m}^j)_{j \in F_N^+}
\in 
{\mathcal O}_N( 
(\widehat r^k)_{k \in F_\infty^+ \setminus F_N^+} )
\Bigr\}
 \cap 
\bigcap_{k \in F_\infty^+ \setminus F_N^+}
\{ \widehat{m}^k = \widehat{r}^k
\}
\,
\vert 
\, 
(\widehat r^k)_{k \in F_\infty^+ \setminus F_N^+} 
\Bigr)=1.
\end{split}
\end{equation}
Thanks to \eqref{eq:rademacher:0}, this is sufficient to prove 
${\mathbb P}(C_N)=1$

We start with the proof of the first claim 
in \eqref{eq:rademacher:2}. 
We invoke $(ii)$ in Theorem 
\ref{thm:probability:probability}.
Moreover, we introduce the function 
${\mathscr R}^N : (\widehat \varrho^j)_{j \in F_N^+} \in {\mathbb C}^{\vert F_N^+ \vert} \simeq 
 {\mathbb R}^{2 \vert F_N^+ \vert} 
\mapsto 
(\widehat \varrho^j \widehat f_N^j)_{j \in F_N^+}$. 
Since 
$\widehat f_N^j$ is a (strictly) positive real for each $j \in F_N^+$, 
${\mathscr R}^N$ can be identified with a linear mapping
from $ {\mathbb R}^{2 \vert F_N^+ \vert} $
into itself 
 with a (strictly) positive determinant. In particular, 
$(\widehat m^j)_{j \in F_N^+}$ belongs to 
${\mathcal D}[\phi^N]$ if and only if 
${\mathscr R}^N( (\widehat m^j)_{j \in F_N^+})$ belongs to 
${\mathscr R}^N({\mathcal D}[\phi^N])$
and the complementary of the latter has a zero Lebesgue measure. Here now comes 
Theorem 
\ref{thm:probability:probability}: 
for $m \in {\mathcal P}({\mathbb T}^d)$, 
${\mathscr R}^N( (\widehat m^j)_{j \in F_N^+})$
rewrites 
$\pi_N^{(1)}(m)$, 
and the image of 
the probability measure 
${\mathbb P}$ by $\pi_N^{(1)}$ is absolutely continuous with respect to the Lebesgue measure on ${\mathbb R}^{2 \vert F_N^+ \vert}$, from which we deduce that 
\begin{equation*}
{\mathbb P} \Bigl( 
\Bigl\{ 
m :
(\widehat m^j)_{j \in F_N^+} \in 
 {\mathcal D}[\phi^N]
 \Bigr\} 
 \Bigr) 
 =
{\mathbb P} \Bigl( 
\Bigl\{ 
m :
\pi_N^{(1)}(m) \in 
{\mathscr R}^N( {\mathcal D}[\phi^N])
 \Bigr\} 
 \Bigr) = 1. 
 \end{equation*}
 In turn, we get 
 that the first line in 
 \eqref{eq:rademacher:2}
 holds true for
 $\textsf{P}_{N,\infty}$ almost every $(\widehat r^k)_{k \in F_\infty^+ \setminus F_N^+}$.
 
 We proceed similarly for the second claim in 
 \eqref{eq:rademacher:1}. 
 We know that 
 \begin{equation*}
 {\mathbb P}
 \Bigl( \Bigl\{ 
 m : 
  \inf_{x \in {\mathbb T}^d} m(x) >0 
 \Bigr\}
 \Bigr) = 1.
 \end{equation*}
 Therefore, 
for $\textsf{P}_{N,\infty}$ almost every $(\widehat r^k)_{k \in F_\infty^+ \setminus F_N^+} $, 
\begin{equation*}
\begin{split}
&{\mathbb P}_{N\vert(N,\infty)} 
\Bigl(
\Bigl\{ m : 
 \inf_{x \in {\mathbb T}^d} 
\Bigl[ 1 + 
 \sum_{j \in F_N^+ } 2 \Re[ \widehat m^j e_{-j}(x)]
+ \sum_{j  \in F_\infty^+ \setminus F_N^+} 2 \Re[ \widehat r^j e_{-j}(x)]
\Bigr]
>0
\Bigr\}
\\
&\hspace{150pt} \cap 
\bigcap_{k \in F_\infty^+ \setminus F_N^+}
\{ \widehat{m}^k = \widehat{r}^k
\}
\,
\vert 
\, 
(\widehat r^k)_{k \in F_\infty^+ \setminus F_N^+} 
\Bigr) = 1,
\end{split}
\end{equation*}
from which we easily deduce that the second line in 
 \eqref{eq:rademacher:2}
 is indeed satisfied. 
 
 Notice from  
 \eqref{eq:rademacher:1}
 that, for any $k_0 \in {\mathbb Z}^d$ and any $N_0 \geq \vert k_0 \vert$, the partial derivative $\partial_{\widehat{m}^{k_0}} \phi(m)$, which exists 
 for
 ${\mathbb P}$ almost every 
 $m \in {\mathcal P}({\mathbb T}^d)$, 
 coincides with
 the finite dimensional derivative 
 $\partial_{\widehat{m}^{k_0}} \phi(1 +  \sum_{j \in F_{N_0}^+} 2 \Re[\widehat{m}^j e_{-j} ]+ 
 \sum_{j \in F_\infty^+ \setminus F_{N_0}^+} 2 \Re[\widehat{r}^j e_{-j} ])$
 on the event $\cap_{j \in F_\infty^+ \setminus F_{N_0}^+}
 \{ \widehat{m}^j =  \widehat{r}^j \}$.

\subsubsection*{Second part: identification of the derivative with a weak limit}
We now turn to the proof of 
\eqref{weak:convergence:derivatives}. 
Generally speaking, 
it relies on Lemmas
\ref{lem:W:mathcalW:Nn}
and 
\ref{lem:W:mathcalW:Nn:2}. 
It also requires a preliminary result: in order to proceed, 
we need to identify, for a given integer $N_0 \geq 1$, the conditional measure 
$\textsf{\sf P}_{N_0,\infty}$ 
introduced in the first part, see 
\eqref{eq:PNinfini}. Here, we do so, but 
on the set 
\begin{equation*}
B_{N_0,\infty} := \Bigl\{ m \in {\mathcal P}({\mathbb T}^d) : 
\forall  k \in F_\infty^+ \setminus F_{N_0}^+, 
\ 
\vert \widehat{m}^k \vert \leq \frac{a_0}{2\vert k \vert^{5d/2}}
\Bigr\}, 
\end{equation*}
with 
$a_0$ as in the statement of 
Lemma 
\ref{lem:25:second}. 
The complete identification is given (and proven) in 
Lemma 
\ref{lem:integrationbyparts:1}
below. 
In this second part,
we take it 
for granted and explain how 
it applies to 
the current problem. 

Following
the same notations as in
Lemma 
\ref{lem:W:mathcalW:Nn}, 
we consider a smooth function $\vartheta : {\mathbb R}^{2 \vert F_{N_0}^+ \vert} \rightarrow {\mathbb R}$
 whose support is included in 
$$\biggl\{ (\widehat{m}^j)_{j \in F_{N_0}^+} \in {\mathbb C}^{\vert F_{N_0}^+\vert} 
 \simeq {\mathbb R}^{2 \vert F_{N_0}^+\vert} : 
1 + 2 \inf_{x \in {\mathbb T}^d}  \sum_{j \in F_{N_0}^+} \Re[ \widehat{m}^j e_{-j}(x) ]
\geq \frac1{c} \biggr\},$$ 
for some $c>1$. 
In this framework, 
Lemma 
\ref{lem:integrationbyparts:1}
says that 
 \begin{equation}
 \label{eq:integration:by:parts:0:1}
\begin{split} 
 & \int_{{\mathcal P}({\mathbb T}^d)}  \phi(m)  \partial_{\widehat{m}^{k_0}} \vartheta\Bigl( \bigl( \widehat{m}^k \bigr)_{k \in F_{N_0}^+} \Bigr)
{\mathbf 1}_{A_{N_0} \cap B_{N_0,\infty}}(m)  
d {\mathbb P}(m) 
 = \frac1{{\mathbb P}(A_{N_0})}
  \int_{{\mathcal P}({\mathbb T}^d)}  \Psi(m) 
 {\mathbf 1}_{A_{N_0}}(m) 
d {\mathbb P}(m),
 \end{split}
 \end{equation} 
where $A_{N_0}$ is defined 
 as in Lemma 
 \ref{lem:25:second}
 and has a (strictly) positive probability for $N_0$ large enough, see
 Lemma \ref{lem:AN0}, and where
\begin{align}
\Psi(m) 
&= 
{\mathbf 1}_{B_{N_0,\infty}}(m) 
 \int_{{\mathcal O}_{N_0}} 
\biggl[ \phi
\Bigl( 
{\mathscr I}_{N_0} 
\bigl( (\widehat{r}^k - \widehat{m}^k)_{k \in F_{N_0}^+}
\bigr)
+ m
\Bigr)
{\mathbf 1}_{A_{N_0}}
\bigl( (\widehat{r}^k)_{k \in F_{N_0}^+}
\bigr) \nonumber
\\
&\hspace{150pt} \times  
\partial_{\widehat{m}^{k_0}} \vartheta
\Bigl( \bigl( \widehat{r}^k \bigr)_{k \in F_{N_0}^+} \Bigr)
\biggr]
d \Gamma_{N_0}\Bigl( 
\bigl( \widehat{r}^k\bigr)_{k \in F_{N_0}^+} 
\Bigr) \nonumber
\\
&= 
{\mathbf 1}_{B_{N_0,\infty}}(m) 
 \int_{{\mathcal O}_{N_0}} 
\biggl[ \phi
\Bigl( 
{\mathscr I}_{N_0} 
\bigl( (\widehat{r}^k
 - \widehat{m}^k)_{k \in F_{N_0}^+}
\bigr)
+ m
\Bigr)
\label{eq:Psi:t,m}
\\
&\hspace{150pt} \times  
\partial_{\widehat{m}^{k_0}} \vartheta
\Bigl( \bigl( \widehat{r}^k \bigr)_{k \in F_{N_0}^+} \Bigr)
\biggr]
d \Gamma_{N_0}\Bigl( 
\bigl( \widehat{r}^k\bigr)_{k \in F_{N_0}^+} 
\Bigr) + \epsilon_{N_0}, \nonumber
\end{align}
where $\epsilon_{N_0}$ tends to $0$ as $N_0$ tends to $\infty$, uniformly with respect to 
$\| \phi \|_\infty$ and $\| \partial_{\widehat{m}^{k_0}} \vartheta  \|_\infty$. 
Importantly, 
this identity requires some care:
\vskip 4pt

\textit{(a)} First, 
the notation 
${\mathscr I}_{N_0} 
( (\widehat{r}^k
 - \widehat{m}^k)_{k \in F_{N_0}^+}
)$ is rather abusive since 
$ (\widehat{r}^k
 - \widehat{m}^k)_{k \in F_{N_0}^+}$ may not belong to 
 ${\mathcal O}_{N_0}$. However, 
 the definition of 
${\mathscr I}_{N_0}$ easily extends to the entire 
${\mathbb R}^{2 \vert F_{N_0}^+ \vert}$, see \eqref{eq:I_N}. 
\vskip 4pt

\textit{(b)}
 Second, 
 the notation 
 ${\mathbf 1}_{A_{N_0}}
( (\widehat{r}^k)_{k \in F_{N_0}^+}
)$ is also abusive. Instead, 
we should write 
 ${\mathbf 1}_{A_{N_0}} \circ {\mathscr I}_{N_0}
( (\widehat{r}^k)_{k \in F_{N_0}^+}
)$, but this would be obviously heavier.
\vskip 4pt

\textit{(c)}
Third, 
we recall that 
the support of $\vartheta$
 is included
 in $A_{N_0}$ (up 
 to the embedding 
${\mathscr I}_{N_0}$ and for $N_0$ large enough). 
\vskip 4pt

As a result of the latter observation, 
we notice that, 
whenever 
$( \widehat{r}^k )_{k \in F_{N_0}^+}$ belongs to the support of 
$\vartheta$ and 
$( \widehat{m}^k )_{k \in F_\infty^+ \setminus F_{N_0}^+}$
satisfies 
$\vert \widehat{m}^k \vert \leq (a_0/2) \vert k \vert^{-5d/2}$ 
(which is the case if 
the 
$( \widehat{m}^k )_{k \in F_\infty^+ \setminus F_{N_0}^+}$'s are the Fourier coefficients of some $m \in B_{N_0,\infty}$), the distribution 
\begin{equation*}
1 
+ 2 \sum_{k \in F_{N_0}^+ }
\Re \bigl( \widehat{r}^k e_{-k} \bigr)
+ 2 \sum_{k \in F_\infty^+ \setminus F_{N_0}^+} 
\Re \bigl( \widehat{m}^k e_{-k} \bigr)
\end{equation*}
is a probability measure, see Lemma  
 \ref{lem:25:second} again. Obviously, this probability measure identifies  with 
${\mathscr I}_{N_0} 
( (\widehat{r}^k
 - \widehat{m}^k)_{k \in F_{N_0^+}})
+ m $, which is 
 the argument of $\phi$ 
in 
\eqref{eq:Psi:t,m}, 
when the 
$( \widehat{m}^k )_{k \in F_\infty^+ \setminus F_{N_0}^+}$'s are the Fourier coefficients of some $m \in B_{N_0,\infty}$. In particular, 
the integrand 
in \eqref{eq:Psi:t,m}
is always well-defined under the sole assumption that 
$m \in B_{N_0,\infty}$. 
Equivalently, we can regard $\Psi$ as a function of 
$(\widehat{m}^k)_{k \in F_\infty^+ \setminus F_{N_0}^+}$ provided that 
$\vert \widehat{m}^k \vert \leq (a_0/2) \vert k \vert^{-5d /2}$ for any 
$k \in F_{\infty}^+ \setminus F_{N_0}^+$. 
In order to make this clear, we let 
\begin{equation}
\begin{split}
\widetilde{\Psi}\Bigl( 
(\widehat{m}^k)_{k \in F_\infty^+ \setminus F_{N_0}^+} \Bigr)
&:= 
 \int_{{\mathcal O}_{N_0}} 
\biggl[ \phi
\Bigl( 
1 
+ 2 \sum_{k \in F_{N_0}^+ }
\Re \bigl( \widehat{r}^k e_{-k} \bigr)
+ 2 \sum_{k \in F_\infty^+ \setminus F_{N_0}^+} 
\Re \bigl( \widehat{m}^k e_{-k} \bigr)
\Bigr)
\label{eq:tildePsi:t,m}
\\
&\hspace{125pt} \times  
\partial_{\widehat{m}^{k_0}} \vartheta
\Bigl( \bigl( \widehat{r}^k \bigr)_{k \in F_{N_0}^+} \Bigr)
\biggr]
d \Gamma_{N_0}\Bigl( 
\bigl( \widehat{r}^k\bigr)_{k \in F_{N_0}^+} 
\Bigr),
\end{split}
\end{equation}
for 
$( \widehat m^k)_{k \in F_\infty^+ \setminus F_N^+} \in \ell^2_{\mathbb C}( F_\infty^+ \setminus F_N^+ )$ such that 
$\vert \widehat{m}^k \vert \leq (a_0/2) \vert k \vert^{-5d /2}$ for any 
$k \in F_{\infty}^+ \setminus F_{N_0}^+$. 

And then, by combining 
 \eqref{eq:integration:by:parts:0:1}, 
 \eqref{eq:Psi:t,m}
and
\eqref{eq:tildePsi:t,m}, 
we obtain 
 \begin{equation}
 \label{eq:integration:by:parts:0:2}
\begin{split} 
 & \int_{{\mathcal P}({\mathbb T}^d)}  \phi(m)  \partial_{\widehat{m}^{k_0}} \vartheta\Bigl( \bigl( \widehat{m}^k \bigr)_{k \in F_{N_0}^+} \Bigr)
{\mathbf 1}_{A_{N_0}}(m) 
{\mathbf 1}_{B_{N_0,\infty}}(m)  
d {\mathbb P}(m) 
\\
&= \frac1{{\mathbb P}(A_{N_0})}
  \int_{{\mathcal P}({\mathbb T}^d)}  \Psi(m)  
d {\mathbb P}(m)   + \epsilon_{N_0}
\\
&= \frac1{{\mathbb P}(A_{N_0})}
  \int_{\ell^2_{\mathbb C}( F_\infty^+ \setminus F_N^+ )}  \widetilde{\Psi}\Bigl(
\bigl( \widehat{m}^k \bigr)_{k \in F_\infty^+ \setminus F_{N_0}^+}
\Bigr)  
d \textsf{\sf P}_{N_0,\infty}\Bigl(
\bigl( \widehat{m}^k \bigr)_{k \in F_\infty^+ \setminus F_{N_0}^+} \Bigr)  + \epsilon_{N_0},
 \end{split}
 \end{equation} 
where, as before, $\epsilon_{N_0}$ tends to $0$ as $N_0$ tends to $\infty$, uniformly in $\| \phi \|_\infty$ and $\| \partial_{\widehat{m}^{k_0}} \vartheta \|_\infty$. 

Now, by the first part of the proof, we know that, 
for $\textsf{\sf P}_{N_0,\infty}$-almost every 
$( \widehat{m}^k )_{k \in F_\infty^+ \setminus F_{N_0}^+}
\in 
\ell^2_{\mathbb C}( F_\infty^+ \setminus F_N^+ )$,  
the function 
$$( \widehat{r}^k )_{k \in F_{N_0}^+}
\in {\mathcal O}_{N_0} 
\mapsto 
\phi
\biggl( 
1 
+ 2 \sum_{k \in F_{N_0}^+ }
\Re ( \widehat{r}^k e_{-k} )
+ 2 \sum_{k \in F_\infty^+ \setminus F_{N_0}^+} 
\Re ( \widehat{m}^k e_{-k} )
\biggr)$$ 
is 
$\textrm{\rm Leb}_{N_0}$
almost everywhere differentiable along any direction 
$k \in F_{N_0}^+$. Recalling the form of $\Gamma_{N_0}$ in 
\eqref{eq:Gamma_N}, we can make an integration by parts in 
\eqref{eq:tildePsi:t,m}. We get: 
\begin{equation*}
\begin{split}
\widetilde{\Psi}\Bigl( 
(\widehat{m}^k)_{k \in F_\infty^+ \setminus F_{N_0}^+} \Bigr)
&= 
-  \int_{{\mathcal O}_{N_0}} 
\biggl[
\partial_{\widehat{m}^{k_0}}
\phi
\Bigl( 
1 
+ 2 \sum_{k \in F_{N_0}^+ }
\Re \bigl( \widehat{r}^k e_{-k} \bigr)
+ 2 \sum_{k \in F_\infty^+ \setminus F_{N_0}^+} 
\Re \bigl( \widehat{m}^k e_{-k} \bigr)
\Bigr)
\\
&\hspace{25pt}
-   2 
\vert k_0 \vert^{2pd} 
\widehat{m}^{k_0}
 \phi
\Bigl( 
1 
+ 2 \sum_{k \in F_{N_0}^+ }
\Re \bigl( \widehat{r}^k e_{-k} \bigr)
+ 2 \sum_{k \in F_\infty^+ \setminus F_{N_0}^+} 
\Re \bigl( \widehat{m}^k e_{-k} \bigr)
\Bigr)
\biggr] 
\\
&\hspace{100pt}
\times 
\vartheta 
\Bigl( \bigl( \widehat{r}^k \bigr)_{k \in F_{N_0}^+} \Bigr)
d \Gamma_{N_0}\Bigl( 
\bigl( \widehat{r}^k\bigr)_{k \in F_{N_0}^+} 
\Bigr).
\end{split}
\end{equation*}
Then, 
we can revert back the computations in 
\eqref{eq:Psi:t,m}
and write 
\begin{equation*}
\begin{split}
\Psi(m) 
&= 
- {\mathbf 1}_{B_{N_0,\infty}}(m) 
 \int_{{\mathcal O}_{N_0}} 
\biggl[ \partial_{\widehat{m}^{k_0}} \phi
\Bigl(  
{\mathscr I}_{N_0} 
\bigl( (\widehat{r}^k - \widehat{m}^k)_{k \in F_{N_0}^+}
\bigr)
+ m
\Bigr)
\\
&\hspace{100pt} 
-
2
\vert k_0 \vert^{2pd}
\widehat{m}^{k_0}
\phi
\Bigl(  
{\mathscr I}_{N_0} 
\bigl( (\widehat{r}^k
-
\widehat{m}^k)_{k \in F_{N_0}^+}
\bigr)
+ m
\Bigr)
\biggr]
\\
&\hspace{100pt} \times  
{\mathbf 1}_{A_{N_0}}
\bigl( (\widehat{r}^k)_{k \in F_{N_0^+}}
\bigr) 
  \vartheta 
\Bigl( \bigl( \widehat{r}^k \bigr)_{k \in F_{N_0}^+} \Bigr)
d \Gamma_{N_0}\Bigl( 
\bigl( \widehat{r}^k\bigr)_{k \in F_{N_0}^+} 
\Bigr)
  + \epsilon_{N_0}.
\end{split}
\end{equation*}
Back to 
 \eqref{eq:integration:by:parts:0:1}
 and 
 \eqref{eq:integration:by:parts:0:2}, we get 
  \begin{equation*}
\begin{split} 
 &  \int_{{\mathcal P}({\mathbb T}^d)}  \phi(m)   \partial_{\widehat{m}^{k_0}} \vartheta\Bigl( \bigl( \widehat{m}^k \bigr)_{k \in F_{N_0}^+} \Bigr)
{\mathbf 1}_{A_{N_0} \cap B_{N_0,\infty}}(m)  
d {\mathbb P}(m)  
\\
&= -
  \int_{{\mathcal P}({\mathbb T}^d)}  
\Bigl[ 
\partial_{\widehat{m}^{k_0}} \phi(m) 
- 
2 \vert k_0\vert^{2pd}
\widehat{m}^{k_0} 
  \phi(m)
\Bigr] 
   \vartheta\Bigl( \bigl( \widehat{m}^k \bigr)_{k \in F_{N_0}^+} \Bigr)  
  {\mathbf 1}_{A_{N_0} \cap B_{N_0,\infty}}(m)  
d {\mathbb P}(m) + \epsilon_{N_0}. 
 \end{split}
 \end{equation*} 
By Lemma 
\ref{lem:27}, we have
${\mathbb P}(A_{N_0} \cap B_{N_0,\infty}) \rightarrow 1$ 
as $N_0 \rightarrow \infty$, 
from which we deduce that 
 \begin{equation}
 \label{eq:integration:by:parts:0:3}
\begin{split} 
 &  \int_{{\mathcal P}({\mathbb T}^d)}  \phi(m)  \partial_{\widehat{m}^{k_0}} \vartheta\Bigl( \bigl( \widehat{m}^k \bigr)_{k \in F_{N_0}^+} \Bigr)
d {\mathbb P}(m) 
\\
&= -
  \int_{{\mathcal P}({\mathbb T}^d)}  
\Bigl[ 
\partial_{\widehat{m}^{k_0}} \phi(m) 
- 
2 \vert k_0\vert^{2pd}
\widehat{m}^{k_0} 
  \phi(m)
\Bigr] 
  \vartheta\Bigl( \bigl( \widehat{m}^k \bigr)_{k \in F_{N_0}^+} \Bigr)  
d {\mathbb P}(m)   + \epsilon_{N_0}, 
 \end{split}
 \end{equation} 
 where, as before, $\epsilon_{N_0}$ tends to $0$ as $N_0$ tends to $\infty$, uniformly in $\| \phi \|_\infty$, 
 $\|\vartheta \|_\infty$, $\| \partial_{\widehat{m}^{k_0}} \phi \|_\infty$ and $\| \partial_{\widehat{m}^{k_0}} \vartheta \|_\infty$. 

For each 
$k_0 \in {\mathbb Z}^d \setminus \{0\}$, we now call $\widehat{\Phi}^{k_0}(m)$ a weak limit of the sequence 
$(\partial_{\widehat{m}^{k_0}} \phi(m*f_N))_{N > \vert k_0\vert}$ with 
$\partial_{\widehat{m}^{k_0}} \phi$ being here understood, for each $N > \vert k_0 \vert$, 
as the almost everywhere (for the Lebesgue measure on ${\mathcal P}_N$) 
derivative of $\phi$ on ${\mathcal P}_N$. 
Then, 
by combining 
Lemma 
\ref{lem:W:mathcalW:Nn}
and 
\eqref{eq:integration:by:parts:0:3}, we reach the same conclusion as 
 \eqref{eq:W1-W2}
in the proof of 
Lemma 
\ref{lem:W:mathcalW:Nn:2}, but with 
$\widehat{\mathcal W}_1^{k_0}(m)$ 
therein being replaced by 
$\widehat{\Phi}^{k_0}(m)$
and 
 with 
$\widehat{\mathcal W}_2^{k_0}(m)$ 
being replaced by 
$\partial_{\widehat{m}^{k_0}} \phi(m)$ (as given by the conclusion of the first part). 
We hence conclude that,  
almost everywhere under ${\mathbb P}$,
$\widehat{\Phi}^{k_0}$
and 
$\partial_{\widehat{m}^{k_0}} \phi$
are equal, which proves that 
$(\partial_{\widehat{m}^{k_0}} \phi(m*f_N))_{N > \vert k_0\vert}$
converges (in the weak sense) to 
$\partial_{\widehat{m}^{k_0}} \phi$. \qed

\subsubsection*{Third part: auxiliary statement}

It now
remains to prove the following lemma, which we invoked in the derivation of the above second part. 

\begin{lem}
\label{lem:integrationbyparts:1}
For any bounded measurable function 
$\varphi : {\mathcal P}({\mathbb T}^d) \rightarrow {\mathbb R}$, 
\begin{equation*}
\begin{split}
&\int_{\PP} 
\varphi(m) 
{\mathbf 1}_{A_{N_0} \cap B_{N_0,\infty}}(m) 
d {\mathbb P}(m) 
 = 
\frac1{{\mathbb P}_{N_0}(A_{N_0})} 
\int_{\PP} 
\psi\bigl(m\bigr) {\mathbf 1}_{A_{N_0}}(m) 
d {\mathbb P}(m),
\end{split} 
\end{equation*}
with $A_{N_0}$ in the first line being as in Lemma 
\ref{lem:25:second}
and with 
 \begin{equation*}
 \begin{split}
& \psi(m)
 =
{\mathbf 1}_{B_{N_0,\infty}}(m) 
 \int_{{\mathcal O}_{N_0}} 
\biggl[ \varphi
\Bigl( 
{\mathscr I}_{N_0} 
\bigl( (\widehat{r}^k)_{k \in F_{N_0}^+}
\bigr)
- 
{\mathscr I}_{N_0} 
\bigl( (\widehat{m}^k)_{k \in F_{N_0}^+}
\bigr)
+ m
\Bigr)
{\mathbf 1}_{A_{N_0}}
\circ
{\mathscr I}_{N_0} 
\bigl( (\widehat{r}^k)_{k \in F_{N_0}^+}
\bigr)
\biggr]
\\
&\hspace{250pt} \times  
d \Gamma_{N_0}\Bigl( 
\bigl( \widehat{r}^k\bigr)_{k \in F_{N_0}^+} 
\Bigr),
\end{split}
\end{equation*}
and with the same notation 
as in \eqref{eq:I_N} for the map ${\mathscr I}_{N_0}$.
\end{lem}

\begin{proof}
Throughout the proof, we use the same abuse of notations as 
in the proof of 
Theorem 
\ref{prop:rademacher}: We 
write
\begin{equation*}
\begin{split}
&{\mathscr I}_{N_0} 
\bigl( (\widehat{r}^k
-
\widehat{m}^k)_{k \in F_{N_0}^+}
\bigr)
\quad \textrm{\rm for} 
\quad 
{\mathscr I}_{N_0} 
\bigl( (\widehat{r}^k)_{k \in F_{N_0}^+}
\bigr)
- 
{\mathscr I}_{N_0} 
\bigl( (\widehat{m}^k)_{k \in F_{N_0}^+}
\bigr),
\\
&
\vartheta
\bigl( (\widehat{r}^k)_{k \in F_{N}^+}
\bigr)
\quad \textrm{\rm for} 
\quad 
\vartheta
\circ
{\mathscr I}_{N} 
\bigl(
 (\widehat{r}^k)_{k \in F_{N}^+}
\bigr),
\end{split}
\end{equation*}
when $\vartheta$ is defined on 
$\PP$ and 
$ (\widehat{r}^k)_{k \in F_{N}^+} \in {\mathcal O}_N$. 
\vskip 4pt

\textit{First Step.}
We start with the following observation, very similar to
an argument used in the proof of Theorem 
\ref{prop:rademacher}.
If, for $N \geq N_0$ being fixed, 
we set 
\begin{equation*}
B_{N_0,N} := 
 \Bigl\{ m \in {\mathcal P}({\mathbb T}^d) : 
\forall  k \in F_N^+ \setminus F_{N_0}^+, \ 
\vert \widehat{m}^k \vert \leq \frac{a_0}{2\vert k \vert^{5d/2}}
\Bigr\}, 
\end{equation*}
then, for any bounded measurable function 
$\varphi : \PP \rightarrow {\mathbb R}$, we have 
\begin{align}
&\int_{\PP} 
\varphi\bigl(m\bigr) {\mathbf 1}_{A_{N_0} \cap B_{N_0,N}}(m) 
d {\mathbb P}_N(m)
\label{eq:proof:integration:by:parts:dim:infinie}
\\
&=
\int_{{\mathcal O}_N} 
  \varphi
\bigl( (\widehat{m}^k)_{k \in F_{N}^+}
\bigr)
{\mathbf 1}_{A_{N_0}}
\bigl( (\widehat{m}^k)_{k \in F_{N_0}^+}
\bigr)
{\mathbf 1}_{B_{N_0,N}}
\bigl( (\widehat{m}^k)_{k \in F_{N}^+}
\bigr)
d \Gamma_N\Bigl( 
\bigl( \widehat{m}^k\bigr)_{k \in F_N^+} 
\Bigr)
\nonumber
\\
&= 
\frac{Z_{N_0}}{Z_N}
\int_{{\mathbb R}^{2 \vert F_N^+ \setminus F_{N_0}^+ \vert}} 
\biggl[
\int_{{\mathcal O}_{N_0}} 
\varphi
\bigl( (\widehat{m}^k)_{k \in F_{N}^+}
\bigr)
{\mathbf 1}_{A_{N_0}}
\bigl( (\widehat{m}^k)_{k \in F_{N_0}^+}
\bigr)
d \Gamma_{N_0}\Bigl( 
\bigl( \widehat{m}^k\bigr)_{k \in F_{N_0}^+} 
\Bigr) 
\biggr]
\nonumber
\\
&\hspace{15pt} \times 
{\mathbf 1}_{B_{N_0,N}}
\biggl( (0)_{k \in F_{N_0}^+}, (\widehat{m}^k)_{k \in F_{N}^+ \setminus F_{N_0}^+}
\biggr)
\exp \biggl( - \sum_{k \in F_N^+ \setminus F_{N_0}^+}
\vert k \vert^{2pd} 
\vert \widehat{m}^k \vert^2 \biggr)  \bigotimes_{k \in F_N^+ \setminus F_{N_0}^+} 
d \widehat{m}^k, 
\nonumber
\end{align}
where we used Lemma 
\ref{lem:25:second}
in order to derive the last equality together with the fact that 
${\mathbf 1}_{B_{N_0,N}}
( (\widehat{m}^k)_{k \in F_{N}^+})=
{\mathbf 1}_{B_{N_0,N}}
( (0)_{k \in F_{N_0}^+}, (\widehat{m}^k)_{k \in F_{N}^+ \setminus F_{N_0}^+})$
only depends on 
$(\widehat{m}^k)_{k \in F_N^+ \setminus F_{N_0}^+}$. 

In particular, if we let 
\begin{equation}
\begin{split}
\label{eq:proof:integration:by:parts:dim:infinie:4}
&\psi(m) 
:=
\int_{{\mathcal O}_{N_0}} 
\varphi
\Bigl( 
(\widehat{r}^k)_{k \in F_{N_0}^+}, 
(\widehat{m}^k)_{k \in F_N^+ \setminus F_{N_0}^+}
\Bigr)
{\mathbf 1}_{A_{N_0}}
\bigl( (\widehat{r}^k)_{k \in F_{N_0}^+}
\bigr)
d \Gamma_{N_0}\Bigl( 
\bigl( \widehat{r}^k\bigr)_{k \in F_{N_0}^+} 
\Bigr),
\end{split}
\end{equation}
for 
$m \in B_{N_0,N}$
and $\psi(m):=0$ otherwise, 
then, by Fubini's theorem, $\psi$ is 
measurable with respect to the sigma-field generated by 
the 
mappings
$(m \mapsto \widehat{m}^k)_{k \in F_N^+ \setminus F_{N_0}^+}$
and 
\eqref{eq:proof:integration:by:parts:dim:infinie}
becomes
\begin{align}
&\int_{\PP} 
\varphi\bigl(m\bigr) {\mathbf 1}_{A_{N_0} \cap B_{N_0,N}}(m) 
d {\mathbb P}_N(m)
\nonumber
\\
&= 
\frac{Z_{N_0}}{Z_N}
\int_{{\mathbb R}^{2 \vert F_N^+ \setminus F_{N_0}^+ \vert}} 
\biggl[ 
\label{eq:proof:integration:by:parts:dim:infinie:2}
\Bigl( \psi {\mathbf 1}_{B_{N_0,N}}
\Bigr) 
\biggl( (0)_{k \in F_{N_0}^+}, (\widehat{m}^k)_{k \in F_{N}^+ \setminus F_{N_0}^+}
\biggr)
\\
&\hspace{150pt} \times 
\exp \biggl( - \sum_{k \in F_N^+ \setminus F_{N_0}^+}
\vert k \vert^{2pd} 
\vert \widehat{m}^k \vert^2 \biggr) \biggr] \bigotimes_{k \in F_N^+ \setminus F_{N_0}^+} 
d \widehat{m}^k. \nonumber
\end{align}
Notice that 
the indicator function 
${\mathbf 1}_{B_{N_0,N}}$ in the product 
$\psi {\mathbf 1}_{B_{N_0,N}}$ is redundant with the definition of 
$\psi$, but we keep it for clarity. 

Now, the equality \eqref{eq:proof:integration:by:parts:dim:infinie}, with $\varphi$ being replaced by $\psi$, leads to 
\begin{equation*}
\begin{split}
&\int_{\PP} 
\psi\bigl(m\bigr) {\mathbf 1}_{A_{N_0} \cap B_{N_0,N}}(m) 
d {\mathbb P}_N(m)
\\
&=
\frac{Z_{N_0}}{Z_N}
\int_{{\mathbb R}^{2 \vert F_N^+ \setminus F_{N_0}^+ \vert}} 
\biggl[
\int_{{\mathcal O}_{N_0}} 
{\mathbf 1}_{A_{N_0}}
\bigl( (\widehat{m}^k)_{k \in F_{N_0^+}}
\bigr)
d \Gamma_{N_0}\Bigl( 
\bigl( \widehat{m}^k\bigr)_{k \in F_{N_0}^+} 
\Bigr) 
\biggr]
\\
&\hspace{15pt} \times 
\Bigl( \psi
{\mathbf 1}_{B_{N_0,N}}
\Bigr) 
\Bigl( 
(0)_{k \in F_{N_0}^+},
(\widehat{m}^k)_{k \in F_{N}^+ \setminus F_{N_0}^+}
\Bigr)
\exp \biggl( - \sum_{k \in F_N^+ \setminus F_{N_0}^+}
\vert k \vert^{2pd} 
\vert \widehat{m}^k \vert^2 \biggr)  \bigotimes_{k \in F_N^+ \setminus F_{N_0}^+} 
d \widehat{m}^k
\\
&=  
{\mathbb P}_{N_0}(A_{N_0}) 
\frac{Z_{N_0}}{Z_N}
 \times 
\int_{{\mathbb R}^{2 \vert F_N^+ \setminus F_{N_0}^+ \vert}} 
\biggl[
\Bigl( \psi
{\mathbf 1}_{B_{N_0,N}}
\Bigr) 
\Bigl( 
(0)_{k \in F_{N_0}^+},
(\widehat{m}^k)_{k \in F_{N}^+ \setminus F_{N_0}^+}
\Bigr)
\\
&\hspace{150pt} \times \exp \biggl( - \sum_{k \in F_N^+ \setminus F_{N_0}^+}
\vert k \vert^{2pd} 
\vert \widehat{m}^k \vert^2 \biggr)  
\biggr] \bigotimes_{k \in F_N^+ \setminus F_{N_0}^+} 
d \widehat{m}^k,
\end{split}
\end{equation*}
and then,
\eqref{eq:proof:integration:by:parts:dim:infinie:2}
becomes
\begin{equation}
\label{eq:proof:integration:by:parts:dim:infinie:3}
\begin{split}
&\int_{\PP} 
\varphi\bigl(m\bigr) {\mathbf 1}_{A_{N_0} \cap B_{N_0,N}}(m) 
d {\mathbb P}_N(m)
=
\frac1{{\mathbb P}_{N_0}(A_{N_0})} 
\int_{\PP} 
\psi\bigl(m\bigr) {\mathbf 1}_{A_{N_0} \cap B_{N_0,N}}(m) 
d {\mathbb P}_N(m),
\end{split}
\end{equation}
where we recall from 
Lemma 
\ref{lem:27}, that 
for $N_0$ large enough, 
${\mathbb P}_{N_0}(A_{N_0})>0$. 
\vskip 5pt

\textit{Second Step.}
We now send $N \rightarrow \infty$ in 
\eqref{eq:proof:integration:by:parts:dim:infinie:3}.
In order to do so, 
we first assume that 
$\varphi$ is 
measurable with respect to $(\widehat{m}^k)_{k \in F_{N_1}^+}$
for $N_1 \geq N_0$. 
Then, for $N \geq N_1$, 
$\psi$ in 
\eqref{eq:proof:integration:by:parts:dim:infinie:4}
only depends, on the set $B_{N_0,N}$, on $m$ through 
$(\widehat{m}^k)_{k \in F_{N_1}^+ \setminus 
F_{N_0}^+}$ and is independent of $N$. 
Next, by using 
Lemma \ref{lem:30}, we get
\begin{equation*}
\begin{split}
&\int_{\PP} 
\varphi \bigl(m\bigr) {\mathbf 1}_{A_{N_0} \cap B_{N_0,N}}(m) 
d {\mathbb P}(m)
=
\frac1{{\mathbb P}_{N_0}(A_{N_0})} 
\int_{\PP} 
\psi \bigl(m\bigr) {\mathbf 1}_{A_{N_0} \cap B_{N_0,N}}(m) 
d {\mathbb P}(m) + \varepsilon_N,
\end{split} 
\end{equation*}
where $\lim_{N \rightarrow \infty} \varepsilon_N=0$. 
It then remains to observe that 
the subsets $(B_{N_0,N})_{N \geq N_0}$ are non-increasing and their intersection is precisely $B_{N_0,\infty}$. Then, 
splitting $\varphi$ into the difference $\varphi_+ - \varphi_-$, we can easily let $N$ tend to $\infty$ by regarding the left-hand side and the first term in the right-hand side as the masses of 
$B_{N_0,N}$ under finite positive measures. 
We get 
\begin{equation*}
\begin{split}
&\int_{\PP} 
\varphi\bigl(m\bigr) {\mathbf 1}_{A_{N_0}   \cap B_{N_0,\infty}}(m) 
d {\mathbb P}(m)
=
\frac1{{\mathbb P}_{N_0}(A_{N_0})} 
\int_{\PP} 
\psi\bigl(m\bigr) {\mathbf 1}_{A_{N_0} \cap B_{N_0,\infty}}(m) 
d {\mathbb P}(m),
\end{split} 
\end{equation*}
where, on the set 
$B_{N_0,\infty}$, we have 
 \begin{equation*}
 \begin{split}
& \psi(m)
 =
 \int_{{\mathcal O}_{N_0}} 
\biggl[ \varphi
\Bigl( 
{\mathscr I}_{N_0} 
\bigl( (\widehat{r}^k
- \widehat{m}^k)_{k \in F_{N_0^+}}
\bigr)
+ m
\Bigr)
{\mathbf 1}_{A_{N_0}}
\bigl( (\widehat{r}^k)_{k \in F_{N_0^+}}
\bigr)
\biggr]
d \Gamma_{N_0}\Bigl( 
\bigl( \widehat{r}^k\bigr)_{k \in F_{N_0}^+} 
\Bigr).
\end{split}
\end{equation*}
By monotone class theorem (together with Lemma  \ref{lem:Borel}), we easily extend the result to bounded and measurable functions
$\varphi$ on $\PP$. 

\end{proof}

\subsection{Mollification of McKean-Vlasov equations}

The following statement plays an important role in the proof of 
Theorem 
\ref{thm:uniqueness:HJB}. 

\begin{thm}
\label{thm:FPK:mollified:proof}
Let $W_1$ and $W_2$ be two Lipschitz continuous functions on $[0,T] \times {\mathcal P}({\mathbb T}^d)$, 
with ${\mathcal P}({\mathbb T}^d)$ being equipped with $d_{-2}$. Then, for a constant 
$c>1$, we can find 
an integer $N_c$ such that, for $N \geq N_c$, for $\varepsilon \in (0,1)$ and $\rho$ as in 
Definition 
\ref{def:admiss:threshold:smoothing}, the Fokker-Planck equation 
\begin{equation}
\label{eq:FPK:mollified:proof}
\partial_t m_t (x)- \textrm{\rm div} \Bigl(
D {\mathcal H}^{N,\varepsilon,\rho}(t,m_t)(x)  \bigl(m_t*f_N\bigr)(x)
\Bigr) - \frac12 \Delta m_t(x) =0,
\end{equation}
for $t \in [0,T]$, 
with 
\begin{equation*}
\begin{split}
&D {\mathcal H}_{\lambda}^{N,\varepsilon,\rho}(t,m)(x) :=
\partial_p H 
\Bigl( x, \lambda 
 \partial_\mu W_1^{N,\varepsilon,\rho}(t,m)(x) 
 + (1- \lambda) 
 \partial_\mu W_2^{N,\varepsilon,\rho}(t,m)(x) 
 \Bigr),
\\
&D {\mathcal H}^{N,\varepsilon,\rho}(t,m)(x) :=
\int_0^1 
D {\mathcal H}_{\lambda}^{N,\varepsilon,\rho}(t,m)(x)
d \lambda, 
\end{split}
\end{equation*}
and with 
$m_0 \in B_N(c)$ as initial condition, 
has a (unique) smooth solution with values in ${\mathcal P}({\mathbb T}^d)$. 
It satisfies 
\begin{equation*}
m_t \geq 1/c', \quad \| \nabla m_t \|_\infty \leq c', \quad t \in [0,T],
\end{equation*}
for $c' >1$ independent of $N$. 
\end{thm}

\begin{proof}
We first recall that the function $(t,m,x) \mapsto D {\mathcal H}^{N,\varepsilon,\rho}(t,m)(x)$ is smooth. It depends on 
$m$ through $(\widehat{m}^k)_{k \in F_N}$. We can extend the mapping 
to inputs $m$ which are general distributions. It suffices to project the Fourier coefficients on the closure of ${\mathcal O}_N$. 
In short, we call $\Pi_N$ the mapping that maps $m=(\widehat m^k)_{k \in {\mathbb Z}^d}$ onto
$\Pi_N(m)$ defined by 
\begin{equation*}
\widehat{\Pi_N(m)}^k =
\left\{
\begin{array}{ll}
1 &\quad \textrm{\rm if} \ k =0
\vspace{3pt}
\\
{\pi_N\bigl(({\widehat{m}^j})_{j \in F_N^+}\bigr)}^k 
&\quad \textrm{\rm if} \ k \in F_N^+
\vspace{3pt}
\\
\overline{\widehat{\Pi_N(m)}^k}
&\quad \textrm{\rm if} \ -k \in F_N^+
\vspace{3pt}
\\
0 &\quad \textrm{\rm if} \ k \not \in F_N
\end{array}
\right.,
\end{equation*}
where $\pi_N$ is the orthogonal projection from ${\mathbb R}^{2 \vert F_N^+\vert}
\simeq {\mathbb C}^{\vert F_N^+ \vert}$ onto 
the closure of 
${\mathcal O}_N$. 
We then consider 
$(t,m,x) \mapsto D {\mathcal H}^{N,\varepsilon,\rho}(t,\Pi_N(m))(x)$. It is Lipschitz 
with respect to the Fourier coefficients $(\widehat{m}^k)_{k \in F_N \setminus \{0\}}$.
There is no difficulty for proving that the solution is classical.  However, 
the solution to the hence extended version of 
\eqref{eq:FPK:mollified:proof} may not take values in ${\mathcal O}_N$. 

We now consider the local Fokker-Planck equation 
\begin{equation}
\label{eq:FPK:mollified:proof:2}
\partial_t \widetilde m_t (x)- \textrm{\rm div} \Bigl( D {\mathcal H}^{N,\varepsilon,\rho}(t,m_t)(x)  \widetilde m_t(x)
\Bigr) - \frac12 \Delta \widetilde m_t(x) =0,
\end{equation}
with 
$\widetilde m_0=m_0 \in B_N(c)$ as initial condition. It is important to notice that, in the non-linear dependence, the argument is 
$m_t$ and not $\widetilde m_t$. 
Since $W_1$ and $W_2$ are $d_{-2}$-Lipschitz continuous in the measure argument, the field $x \mapsto D {\mathcal H}^{N,\varepsilon,\rho}(t,m_t)(x) $
is bounded and Lipschitz continuous, with a bound and a Lipschitz constant that are independent of $N$, $\varepsilon$, $\rho$: it suffices to combine Corollary  
\ref{cor:mollif:time-space}
and
the second claim in 
Proposition 
\ref{prop:4:7}.
In turn, the solution to the Fokker-Planck equation satisfies
\begin{equation*}
\widetilde m_t \geq 1/c', \quad \| \nabla_x \widetilde m_t \|_\infty, \quad \vvvert \nabla_x \widetilde m_t \vvvert_{\gamma} \leq c',
\quad t \in [0,T],
\end{equation*}
for a constant $c'$ that depends on $c$ but that is independent of $N$, $\varepsilon$ and $\rho$, 
and for some $\gamma \in (0,1)$ that is also independent of $N$, $\varepsilon$ and $\rho$. 

We then compare $m$ and $\widetilde m$. 
Denoting by $p(t,x)$ the standard heat kernel on the torus, 
we have
\begin{equation*}
\begin{split}
m_t(x) - \widetilde{m}_t(x) 
&= - \int_0^t 
\int_{{\mathbb T}^d} 
\nabla_x p(t-s,x-y) 
D {\mathcal H}^{N,\varepsilon,\rho}(s,m_s)(y) 
\bigl( f_N*m(s,y) - \widetilde{m}(s,y)
\bigr) dy ds,
\end{split}
\end{equation*}
from which we deduce that 
\begin{equation*}
\begin{split}
\| m_t - \widetilde m_t \|_\infty
&\leq C \int_0^t 
\int_{{\mathbb T}^d} 
\frac{\| m_s - \widetilde m_s \|_\infty}{\sqrt{t-s}}
ds + \eta_N,
\quad t \in [0,T], 
\end{split}
\end{equation*}
for a sequence $(\eta_N)_{N \geq 1}$ that tends to $0$ as $N$ tends to $\infty$ and that only depends on $m_0$ through the value of $c$. 
It is standard to deduce that 
\begin{equation*}
\| m_t - \widetilde m_t \|_\infty
\leq C \eta_N,
\end{equation*}
which proves that $m_t$ is positive for $N$ large enough. Since $\int_{{\mathbb T}^d} m_t(x) dx=1$, it is a probability measure. 

Using the regularity in $y$ of 
$D {\mathcal H}^{N,\varepsilon,\rho}(s,y)$ together with the H\"older continuity of $\nabla_x \widetilde m_t$, we can proceed in the same way for 
$\|\nabla_x m_t - \nabla_x \widetilde{m}_t\|_\infty$. 
\end{proof}

\bibliographystyle{abbrv}
\bibliography{references2}

\begin{thebibliography}{10}

\bibitem{Ahuja}
S.~Ahuja.
\newblock Wellposedness of mean field games with common noise under a weak
  monotonicity condition.
\newblock {\em SIAM J. Control Optim.}, 54(1):30--48, 2016.

\bibitem{Alimov}
S.~A. Alimov, R.~R. Ashurov, and A.~K. Pulatov.
\newblock Multiple {F}ourier series and {F}ourier integrals [ {MR}1027847
  (91b:42022)].
\newblock In {\em Commutative harmonic analysis, {IV}}, volume~42 of {\em
  Encyclopaedia Math. Sci.}, pages 1--95. Springer, Berlin, 1992.

\bibitem{AmbroseMesazros}
D.~M. Ambrose and A.~R. Mészáros.
\newblock Well-posedness of mean field games master equations involving
  non-separable local hamiltonians.
\newblock {\em arXiv}, https://arxiv.org/abs/2105.03926, 2021.

\bibitem{AmbrosioGangbo}
L.~Ambrosio and W.~Gangbo.
\newblock Hamiltonian {ODE}s in the {W}asserstein space of probability
  measures.
\newblock {\em Comm. Pure Appl. Math.}, 61(1):18--53, 2008.

\bibitem{AGS}
L.~Ambrosio, N.~Gigli, and G.~Savar\'{e}.
\newblock {\em Gradient flows in metric spaces and in the space of probability
  measures}.
\newblock Lectures in Mathematics ETH Z\"{u}rich. Birkh\"{a}user Verlag, Basel,
  second edition, 2008.

\bibitem{Bandini}
E.~Bandini, A.~Cosso, M.~Fuhrman, and H.~Pham.
\newblock Randomized filtering and {B}ellman equation in {W}asserstein space
  for partial observation control problem.
\newblock {\em Stochastic Process. Appl.}, 129(2):674--711, 2019.

\bibitem{bayraktar2018randomized}
E.~Bayraktar, A.~Cosso, and H.~Pham.
\newblock Randomized dynamic programming principle and feynman-kac
  representation for optimal control of mckean-vlasov dynamics.
\newblock {\em Transactions of the American Mathematical Society},
  370(3):2115--2160, 2018.

\bibitem{BayZhang}
E.~Bayraktar and X.~Zhang.
\newblock On non-uniqueness in mean field games.
\newblock {\em Proc. Amer. Math. Soc.}, 148(9):4091--4106, 2020.

\bibitem{BayZhang-corr}
E.~Bayraktar and X.~Zhang.
\newblock Corrigendum to ``{O}n non-uniqueness in mean field games''.
\newblock {\em Proc. Amer. Math. Soc.}, 149(3):1359--1360, 2021.

\bibitem{ben-fre-yam2015}
A.~Bensoussan, J.~Frehse, and S.~C.~P. Yam.
\newblock The master equation in mean field theory.
\newblock {\em J. Math. Pures Appl.}, 103(6):1441--1474, 2015.

\bibitem{Bertucci}
C.~Bertucci.
\newblock Monotone solutions for mean field games master equations : continuous
  state space and common noise.
\newblock {\em arXiv}, https://arxiv.org/abs/2107.09531, 2021.

\bibitem{Bertucci:finite}
C.~Bertucci.
\newblock Monotone solutions for mean field games master equations: finite
  state space and optimal stopping.
\newblock {\em J. \'{E}c. polytech. Math.}, 8:1099--1132, 2021.

\bibitem{Bogachev}
V.~I. Bogachev.
\newblock {\em Differentiable measures and the {M}alliavin calculus}, volume
  164 of {\em Mathematical Surveys and Monographs}.
\newblock American Mathematical Society, Providence, RI, 2010.

\bibitem{BrianiCardaliaguet}
A.~Briani and P.~Cardaliaguet.
\newblock Stable solutions in potential mean field game systems.
\newblock {\em NoDEA Nonlinear Differential Equations Appl.}, 25(1):Paper No.
  1, 26, 2018.

\bibitem{Burzoni}
M.~Burzoni, V.~Ignazio, A.~M. Reppen, and H.~M. Soner.
\newblock Viscosity solutions for controlled {M}c{K}ean-{V}lasov
  jump-diffusions.
\newblock {\em SIAM J. Control Optim.}, 58(3):1676--1699, 2020.

\bibitem{cannarsa}
P.~Cannarsa and C.~Sinestrari.
\newblock {\em Semiconcave functions, {H}amilton-{J}acobi equations, and
  optimal control}, volume~58 of {\em Progress in Nonlinear Differential
  Equations and their Applications}.
\newblock Birkh\"{a}user Boston, Inc., Boston, MA, 2004.

\bibitem{CardaliaguetCirantPorretta}
P.~Cardaliaguet, M.~Cirant, and A.~Porretta.
\newblock Splitting methods and short time existence for the master equations
  in mean field games.
\newblock {\em arXiv}, https://arxiv.org/abs/2001.10406, 2020.

\bibitem{CardaliaguetDelarueLasryLions}
P.~Cardaliaguet, F.~Delarue, J.-M. Lasry, and P.-L. Lions.
\newblock {\em The master equation and the convergence problem in mean field
  games}, volume 201 of {\em Annals of Mathematics Studies}.
\newblock Princeton University Press, Princeton, NJ, 2019.

\bibitem{CardaliaguetGraberPorrettaTonon}
P.~Cardaliaguet, J.~Graber, A.~Porretta, and D.~Tonon.
\newblock Second order mean field games with degenerate diffusion and local
  coupling.
\newblock {\em NoDEA}, 22:1287--1317, 2015.

\bibitem{cardaliaguetporretta-cetraro}
P.~Cardaliaguet and A.~Porretta.
\newblock An introduction to mean field game theory.
\newblock In {\em Mean Field Games, chapter 1, Cetraro, Italy 2019,
  Cardaliaguet, Pierre, Porretta, Alessio (Eds.)}, LNM 2281, pages 203--248.
  Springer, 2021.

\bibitem{CardaQuinca}
P.~Cardaliaguet and M.~Quincampoix.
\newblock {Deterministic differential games under probability knowledge of
  initial condition}.
\newblock {\em International Game Theory Review}, 10(1).

\bibitem{CardaSou2020}
P.~Cardaliaguet and P.~Souganidis.
\newblock {On first order mean field game systems with a common noise}.
\newblock {\em To appear on Annals of applied probability}, 2020.

\bibitem{CardaliaguetSouganidis2}
P.~Cardaliaguet and P.~Souganidis.
\newblock Weak solutions of the master equation for mean field games with no
  idiosyncratic noise.
\newblock {\em arXiv}, https://arxiv.org/abs/2109.14911, 2021.

\bibitem{cardaliaguet-souganidis:2}
P.~Cardaliaguet and P.~Souganidis.
\newblock Regularity of the value function and quantitative propagation of
  chaos for mean field control problems.
\newblock 2022.

\bibitem{CarmonaDelarue_book_I}
R.~Carmona and F.~Delarue.
\newblock {\em Probabilistic theory of mean field games with applications.
  {I}}, volume~83 of {\em Probability Theory and Stochastic Modelling}.
\newblock Springer, Cham, 2018.
\newblock Mean field FBSDEs, control, and games.

\bibitem{CarmonaDelarue_book_II}
R.~Carmona and F.~Delarue.
\newblock {\em Probabilistic theory of mean field games with applications.
  {II}}, volume~84 of {\em Probability Theory and Stochastic Modelling}.
\newblock Springer, Cham, 2018.
\newblock Mean field games with common noise and master equations.

\bibitem{cecdaifispel}
A.~Cecchin, P.~Dai~Pra, M.~Fischer, and G.~Pelino.
\newblock On the convergence problem in mean field games: a two state model
  without uniqueness.
\newblock {\em SIAM J. Control Optim.}, 57(4):2443--2466, 2019.

\bibitem{Cecchin:Delarue:CPDE}
A.~Cecchin and F.~Delarue.
\newblock Selection by vanishing common noise for potential finite state mean
  field games.
\newblock {\em Comm. Partial Differential Equations}, 47(1):89--168, 2022.

\bibitem{cha-cri-del_AMS}
J.-F. Chassagneux, D.~Crisan, and F.~Delarue.
\newblock A probabilistic approach to classical solutions of the master
  equation for large population equilibria.
\newblock {\em Memoirs of the AMS}, To appear.

\bibitem{conforti}
G.~Conforti, R.~Kraaij, and D.~Tonon.
\newblock Hamilton--jacobi equations for controlled gradient flows: the
  comparison principle.
\newblock {\em arXiv}, https://arxiv.org/abs/2111.13258, 2021.

\bibitem{CossoGozziKharroubiPham}
A.~Cosso, F.~Gozzi, I.~Kharroubi, H.~Pham, and M.~Rosestolato.
\newblock Master bellman equation in the wasserstein space: Uniqueness of
  viscosity solutions.
\newblock {\em arXiv}, https://arxiv.org/abs/2107.10535, 2021.

\bibitem{CossoPham}
A.~Cosso and H.~Pham.
\newblock Zero-sum stochastic differential games of generalized
  {M}c{K}ean-{V}lasov type.
\newblock {\em J. Math. Pures Appl. (9)}, 129:180--212, 2019.

\bibitem{delfog2019}
F.~Delarue and R.~Foguen~Tchuendom.
\newblock {Selection of equilibria in a linear quadratic mean field game}.
\newblock {\em Stochastic Processes and their Applications}, 130(2):1000--1040,
  2020.

\bibitem{DelarueTse}
F.~Delarue and A.~Tse.
\newblock Uniform in time weak propagation of chaos on the torus.
\newblock {\em arXiv}, https://arxiv.org/abs/2104.14973, 2021.

\bibitem{DelloSchiavo}
L.~Dello~Schiavo.
\newblock A {R}ademacher-type theorem on {$L^2$}-{W}asserstein spaces over
  closed {R}iemannian manifolds.
\newblock {\em J. Funct. Anal.}, 278(6):108397, 57, 2020.

\bibitem{djete2019mckean}
M.~F. Djete, D.~Possama{\"\i}, and X.~Tan.
\newblock Mckean-vlasov optimal control: the dynamic programming principle.
\newblock {\em arXiv preprint arXiv:1907.08860}, 2019.

\bibitem{Douglis}
A.~Douglis.
\newblock The continuous dependence of generalized solutions of non-linear
  partial differential equations upon initial data.
\newblock {\em Comm. Pure Appl. Math.}, 14:267--284, 1961.

\bibitem{GangboMeszaros}
W.~Gangbo and A.~R. Mészáros.
\newblock Global well-posedness of master equations for deterministic
  displacement convex potential mean field games.
\newblock {\em arXiv}, https://arxiv.org/abs/2004.01660, 2020.

\bibitem{GangboMeszarosMouZhang}
W.~Gangbo, A.~R. Mészáros, C.~Mou, and J.~Zhang.
\newblock Mean field games master equations with non-separable hamiltonians and
  displacement monotonicity.
\newblock {\em arXiv}, https://arxiv.org/abs/2101.12362, 2021.

\bibitem{gan-swi}
W.~Gangbo and A.~\'{S}wiech.
\newblock Existence of a solution to an equation arising from the theory of
  mean field games.
\newblock {\em J. Differential Equations}, 259(11):6573--6643, 2015.

\bibitem{GangboTudorascu}
W.~Gangbo and A.~Tudorascu.
\newblock On differentiability in the {W}asserstein space and well-posedness
  for {H}amilton-{J}acobi equations.
\newblock {\em J. Math. Pures Appl. (9)}, 125:119--174, 2019.

\bibitem{Golubov}
B.~I. Golubov.
\newblock Multiple series and {F}ourier integrals.
\newblock In {\em Mathematical analysis, {V}ol. 19}, pages 3--54, 232. Akad.
  Nauk SSSR, Vsesoyuz. Inst. Nauchn. i Tekhn. Informatsii, Moscow, 1982.

\bibitem{gomes_survey}
D.~A. Gomes and J.~a. Sa\'{u}de.
\newblock Mean field games models---a brief survey.
\newblock {\em Dyn. Games Appl.}, 4(2):110--154, 2014.

\bibitem{HuangCainesMalhame1}
M.~Huang, P.~Caines, and R.~Malham{\'{e}}.
\newblock Individual and mass behavior in large population stochastic wireless
  power control problems: centralized and {N}ash equilibrium solutions.
\newblock pages 98 -- 103, 2003.

\bibitem{Huang2006}
M.~Huang, R.~P. Malham{\'e}, and P.~E. Caines.
\newblock {Large population stochastic dynamic games: Closed-loop
  {M}c{K}ean-{V}lasov systems and the {N}ash certainty equivalence principle}.
\newblock {\em Commun. Inf. Syst.}, 6(3):221--251, 2006.

\bibitem{IseriZhang}
M.~Iseri and J.~Zhang.
\newblock Set values for mean field games.
\newblock {\em arXiv}, https://arxiv.org/abs/2107.01661, 2021.

\bibitem{jean:hal-03564787}
F.~Jean, O.~Jerhaoui, and H.~Zidani.
\newblock {Deterministic optimal control on Riemannian manifolds under
  probability knowledge of the initial condition}.
\newblock {\em HAL},
  https://hal-ensta-paris.archives-ouvertes.fr//hal-03564787, Feb. 2022.
\newblock working paper or preprint.

\bibitem{JiMaQu}
A.~Q.~M. Jimenez, Chlo\'{e};~Marigonda.
\newblock {Optimal control of multiagent systems in the Wasserstein space}.
\newblock {\em Calc. Var. Partial Differential Equations}, 59(2).

\bibitem{JKO}
R.~Jordan, D.~Kinderlehrer, and F.~Otto.
\newblock The variational formulation of the {F}okker-{P}lanck equation.
\newblock {\em SIAM J. Math. Anal.}, 29(1):1--17, 1998.

\bibitem{Kruzhkov1960}
S.~N. Kru\v{z}kov.
\newblock The {C}auchy problem in the large for certain non-linear first order
  differential equations.
\newblock {\em Soviet Math. Dokl.}, 1:474--477, 1960.

\bibitem{kruzkov}
S.~N. Kru\v{z}kov.
\newblock Generalized solutions of nonlinear equations of the first order with
  several independent variables. {II}.
\newblock {\em Mat. Sb. (N.S.)}, 72 (114):108--134, 1967.

\bibitem{Lacker2017}
D.~Lacker.
\newblock Limit theory for controlled mckean-vlasov dynamics.
\newblock {\em SIAM J. Control Optim.}, 55:1641--1672, 2017.

\bibitem{Lacker_superposition}
D.~Lacker, M.~Shkolnikov, and J.~Zhang.
\newblock Superposition and mimicking theorems for conditional
  {M}c{K}ean-{V}lasov equations.
\newblock {\em arXiv}, https://arxiv.org/abs/2004.00099, 2020.

\bibitem{Lasry2006}
J.-M. Lasry and P.-L. Lions.
\newblock Jeux \`a champ moyen. {I}. {L}e cas stationnaire.
\newblock {\em C. R. Math. Acad. Sci. Paris}, 343(9):619--625, 2006.

\bibitem{LasryLions2}
J.-M. Lasry and P.-L. Lions.
\newblock Jeux \`a champ moyen. {II}. {H}orizon fini et contr\^ole optimal.
\newblock {\em C. R. Math. Acad. Sci. Paris}, 343(10):679--684, 2006.

\bibitem{LasryLions}
J.-M. Lasry and P.-L. Lions.
\newblock Mean field games.
\newblock {\em Jpn. J. Math.}, 2(1):229--260, 2007.

\bibitem{Lauriere-Pironneau}
M.~Lauri\`ere and O.~Pironneau.
\newblock Dynamic programming for mean-field type control.
\newblock {\em J. Optim. Theory Appl.}, 169(3):902--924, 2016.

\bibitem{Lionscollege2}
P.-L. Lions.
\newblock Cours au coll\`ege de france, equations aux d\'eriv\'ees partielles
  et applications.
\newblock
  https://www.college-de-france.fr/site/pierre-louis-lions/course-2010-2011.htm,
  2010-11.

\bibitem{Lionsvideo}
P.-L. Lions.
\newblock Estim\'ees nouvelles pour les \'equations quasilin\'eaires.
\newblock Seminar in {A}pplied {M}athematics at the {C}oll\`ege de {F}rance.
  http://www.college-de-france.fr/site/pierre-louis-lions/seminar-2014-11-14-1%
  1h15.htm, 2014.

\bibitem{Lions_HJB}
P.-L. Lions.
\newblock {\em Generalized solutions of {H}amilton-{J}acobi equations},
  volume~69 of {\em Research Notes in Mathematics}.
\newblock Pitman (Advanced Publishing Program), Boston, Mass.-London, 1982.

\bibitem{MouZhang}
C.~Mou and J.~Zhang.
\newblock Wellposedness of second order master equations for mean field games
  with nonsmooth data.
\newblock {\em arXiv}, https://arxiv.org/abs/1903.09907, 2019.

\bibitem{MouZhang2}
C.~Mou and J.~Zhang.
\newblock Mean field game master equations with anti-monotonicity conditions.
\newblock {\em arXiv}, https://arxiv.org/abs/2201.10762, 2022.

\bibitem{MeszarosMou}
A.~R. Mészáros and C.~Mou.
\newblock Mean field games systems under displacement monotonicity.
\newblock {\em arXiv}, https://arxiv.org/abs/2109.06687, 2021.

\bibitem{Pham-Wei}
H.~Pham and X.~Wei.
\newblock Bellman equation and viscosity solutions for mean-field stochastic
  control problem.
\newblock {\em ESAIM Control Optim. Calc. Var.}, 24(1):437--461, 2018.

\bibitem{Wisniewski}
A.~Wi\'{s}niewski.
\newblock The structure of measurable mappings on metric spaces.
\newblock {\em Proc. Amer. Math. Soc.}, 122(1):147--150, 1994.

\bibitem{WuZhang}
C.~Wu and J.~Zhang.
\newblock Viscosity solutions to parabolic master equations and
  {M}c{K}ean-{V}lasov {SDE}s with closed-loop controls.
\newblock {\em Ann. Appl. Probab.}, 30(2):936--986, 2020.

\bibitem{Zhizhiashvili}
L.~Zhizhiashvili.
\newblock {\em Trigonometric {F}ourier series and their conjugates}, volume 372
  of {\em Mathematics and its Applications}.
\newblock Kluwer Academic Publishers Group, Dordrecht, 1996.
\newblock Revised and updated translation of {{\it} Some problems of the theory
  of trigonometric Fourier series and their conjugate series} (Russian)
  [Tbilis. Gos. Univ., Tbilisi, 1993], Translated from the Russian by George
  Kvinikadze.

\bibitem{Zalinescu}
C.~Z\u{a}linescu.
\newblock On uniformly convex functions.
\newblock {\em J. Math. Anal. Appl.}, 95(2):344--374, 1983.

\end{thebibliography}
\end{document}